\documentclass[10pt]{article}
\usepackage[margin = 1in]{geometry}
\usepackage{xparse}

\usepackage{amsmath, amsthm, amssymb}
\usepackage[mathscr]{euscript}
\usepackage{nccmath}
\usepackage{booktabs}
\usepackage{empheq}
\usepackage{mdframed}
\usepackage{pifont}
\usepackage{enumitem}
\usepackage{graphicx} % more modern
\usepackage{wrapfig}
\usepackage{subfigure}
\usepackage{empheq}

\usepackage{thmtools}
\usepackage{thm-restate}

\usepackage{nicefrac}
\usepackage{booktabs}
\usepackage{multirow}
\usepackage{rotating}
\usepackage{bm}

\usepackage{tikz}
\usetikzlibrary{shapes}
\usetikzlibrary{decorations.markings}
\usetikzlibrary{plotmarks}
\usetikzlibrary{patterns}
\usepackage{ctable}

\usepackage{color, colortbl}
\definecolor{mydarkblue}{rgb}{0,0.08,0.45}
\usepackage[colorlinks=true,
    linkcolor=mydarkblue,
    citecolor=mydarkblue,
    filecolor=mydarkblue,
    urlcolor=mydarkblue,
    pdfview=FitH]{hyperref}
    
\usepackage{empheq}

\definecolor{myteal}{RGB}{27,158,119}
\definecolor{myorange}{RGB}{217,95,2}
\definecolor{myred}{RGB}{231,41,138}
\definecolor{mypurple}{RGB}{152,78,163}
\definecolor{myblue}{RGB}{55,126,184}
\definecolor{mygreen}{RGB}{0,100,0}

\newtheorem{definition}{Definition}[section]

\newtheorem{proposition}{Proposition}[section]
\newtheorem{lemma}{Lemma}[section]
\newtheorem{theorem}{Theorem}[section]
\newtheorem{remark}{Remark}[section]
\newtheorem{corollary}{Corollary}[section]

\definecolor{myblue}{HTML}{D2E4FC}
\definecolor{Gray}{gray}{0.92}

\newif\ifshowcomments
\showcommentsfalse

\ifshowcomments
\newcommand{\ba}[1]{\textcolor{olive}{[BA: #1]}}
\newcommand{\jp}[1]{\textcolor{blue}{[JP: #1]}}
\else
\newcommand{\jp}[1]{}
\newcommand{\ba}[1]{}
\fi
\ifshowcomments
\title{\textcolor{red}{[NB: COMMENTS DISPLAYED!]\\}Homogenization of SGD in high-dimensions: \\ {\Large Exact dynamics and generalization properties}
}
\else
\title{Homogenization of SGD in high-dimensions:\\ {\Large Exact dynamics and generalization properties}}
\fi

\date{}

\newtheorem{assumption}{Assumption}

\graphicspath{{./figures/}}

\usepackage[sort]{natbib}
\def\N{\mathbb{N}}
\def\R{\mathbb{R}}
\def\Filt{\mathscr{F}}

\def\aa{{\boldsymbol a}}
\def\bb{{\boldsymbol b}}
\def\hh{{\boldsymbol h}}

\def\xx{{\boldsymbol x}}

\def\XX{{\boldsymbol X}}
\def\YY{{\boldsymbol Y}}
\def\ZZ{{\boldsymbol Z}}
\def\aa{{\boldsymbol a}}
\def\bb{{\boldsymbol b}}

\def\ee{{\boldsymbol e}}

\def\WW{{\boldsymbol W}}
\def\JJ{{\boldsymbol J}}

\def\II{{\text{\textbf{I}}}}
\def\yy{{\boldsymbol y}}
\def\vv{{\boldsymbol v}}
\def\uu{{\boldsymbol u}}
\def\ww{{\boldsymbol w}}
\def\zz{{\boldsymbol z}}
\def\SS{{\boldsymbol S}}
\def\BB{{\boldsymbol B}}
\def\AA{{\boldsymbol A}}
\def\CC{{\boldsymbol C}}
\def\FF{{\boldsymbol F}}
\def\GG{{\boldsymbol G}}
\def\MM{{\boldsymbol M}}

\def\PP{{\boldsymbol P}}
\def\TT{{\boldsymbol T}}

\def\QQ{{\boldsymbol Q}}
\def\UU{{\boldsymbol U}}
\def\VV{{\boldsymbol V}}

\def\xxi{  {\boldsymbol \xi} } 
\def\SSigma{{\boldsymbol \Sigma}}

\def\YY{ {\boldsymbol Y} }

\def\nnu{{\boldsymbol \nu}}

\def\eeta{{\boldsymbol \eta}}
\def\DDelta{{\boldsymbol \Delta}}

\def\bbeta{{\boldsymbol \beta}}

\def\dif{\mathop{}\!\mathrm{d}}

\def\RR{{\widehat{\bm {R}}}}
\def\EE{{\mathbb E}\,}

\DeclareMathOperator*{\argmin}{{arg\,min}}

\DeclareMathOperator{\tr}{tr}
\def\defas{\stackrel{\text{def}}{=}}

\DeclareDocumentCommand{\Prto} {o} {
  \IfNoValueTF {#1}
  {\overset{\Pr}{\longrightarrow}}
  { \xrightarrow[ #1 \to \infty]{\Pr }}
}

\DeclareDocumentCommand{\Asto} {o} {
  \IfNoValueTF {#1}
  {\overset{\text{\rm a.s.}}{\longrightarrow}}
  { \xrightarrow[ #1 \to \infty]{\text{\rm a.s.} }}
}

\DeclareDocumentCommand{\law} {o} {
  \IfNoValueTF {#1}
  {\overset{\text{law}}{=}}
  { \xrightarrow[ #1 \to \infty]{\Pr }}
}

\DeclareMathOperator{\E}{\mathbb{E}}

\newcommand{\Id}{\text{\textbf{I}}}
\DeclareMathOperator{\Exp}{\mathbb{E}}

\DeclareMathOperator{\ntr}{\bar{tr}}
\newcommand{\gcal}{\mathcal{R}}
\newcommand{\gcaltil}{\tilde{\mathcal{R}}}
\newcommand{\gcalm}[1]{\mathcal{R}^{(#1)}}
\newcommand{\gcaltilm}[1]{\tilde{\mathcal{R}}^{(#1)}}
\newcommand{\deq}{\mathrel{\mathop:}=}

\newcommand{\al}[1]{\begin{align}#1\end{align}}

\newcommand{\pa}[1]{\left({#1}\right)}

\newcommand{\h}[1]{\{{#1}\}}

\newcommand{\ha}[1]{\left\{{#1}\right\}}

\newcommand{\abs}[1]{\lvert #1 \rvert}

\newcommand{\absa}[1]{\left\lvert #1 \right\rvert}

\newcommand{\norm}[1]{\lVert #1 \rVert}

\newcommand{\norma}[1]{\left\lVert #1 \right\rVert}

\newcommand{\vertiii}[1]{{\left\vert\kern-0.25ex\left\vert\kern-0.25ex\left\vert #1 
    \right\vert\kern-0.25ex\right\vert\kern-0.25ex\right\vert}}

\makeatletter
\newcommand{\vast}{\bBigg@{4}}
\newcommand{\Vast}{\bBigg@{5}}
\makeatother

\begin{document}

\author{%
    Courtney Paquette\footnotemark[2] \thanks{Google Research, Brain Team}
  \and Elliot Paquette\thanks{Department of Mathematics and Statistics, McGill University, Montreal, QC; CP is a CIFAR AI chair, MILA and CP was supported by a Discovery Grant from the
Natural Science and Engineering Research Council (NSERC) of Canada; website \url{https://cypaquette.github.io/} and email \url{courtney.paquette@mcgill.ca}. Research by EP was supported by a Discovery Grant from the
Natural Science and Engineering Research Council (NSERC) of Canada; website \url{https://elliotpaquette.github.io/} and email: \url{elliot.paquette@mcgill.ca}. }% 
  \and Ben Adlam\footnotemark[1]%
  \and Jeffrey Pennington\footnotemark[1]%
  }

% \editor{}

\maketitle
\begin{abstract} We develop a stochastic differential equation, called homogenized SGD, for analyzing the dynamics of stochastic gradient descent (SGD) on a high-dimensional random least squares problem with $\ell^2$-regularization. We show that \textit{homogenized SGD is the high-dimensional equivalence of SGD} -- for any quadratic statistic (e.g., population risk with quadratic loss), the statistic under the iterates of SGD converges to the statistic under homogenized SGD when the number of samples $n$ and number of features $d$ are polynomially related ($d^c < n < d^{1/c}$ for some $c > 0$). By analyzing homogenized SGD, we provide exact non-asymptotic high-dimensional expressions for the generalization performance of SGD in terms of a solution of a Volterra integral equation. Further we provide the exact value of the limiting excess risk in the case of quadratic losses when trained by SGD. The analysis is formulated for data matrices and target vectors that satisfy a family of resolvent conditions, which can roughly be viewed as a weak (non-quantitative) form of delocalization of sample-side singular vectors of the data. Several motivating applications are provided including sample covariance matrices with independent samples and random features with non-generative model targets.
\end{abstract}

\section{Introduction} \label{sec:intro}

A central component of many supervised learning methods is empirical risk minimization (ERM), formulated as
\begin{equation} \label{eq:erm}
\min_{\xx \in \mathbb{R}^d}~f(\xx) \defas \sum_{i=1}^n f_i(\xx)\,, 
\end{equation}
where each $f_i : \mathbb{R}^d \to \mathbb{R}$ represents the loss function due to the $i$-th training sample. %Here the loss functions are typically assumed to be iid random samples from some (unknown) distribution\jp{Do we really mean that the loss functions themselves are iid samples?}.
Modern applications of ERM seek to extract meaningful information from datasets where the number of samples $n$ and the dimensionality of each sample $d$ are both large. Moreover they often employ powerful models with a comparably large (or even larger) latent parameter space. Owing to the difficulty of performing ERM in such \emph{high-dimensional} settings, in practice, optimization is performed using a \emph{multi-pass} algorithm, in which each sample of the dataset is used more than once.
% This contrasts to streaming or single-pass algorithms for which at each iteration a fresh sample is drawn from the distribution.
% Nonetheless, in many settings the minimum of $n$ and $d$ are large, and so we will say that the problem is high-dimensional. We call algorithms for solving \eqref{eq:erm} which use each sample many times multi-pass methods. 
%A popular optimization algorithm for solving the high-dimensional ERM in \eqref{eq:erm} is the stochastic gradient descent (SGD) \citep{robbins1951} and its variants.  In this article, we develop a new mathematical theory for the analysis of multi-pass SGD on the high-dimensional ERM in the case that $f_i$ are quadratics, or more specifically, we study multi-pass SGD applied to the high-dimensional ridge--regression problem.
Stochastic gradient descent (SGD)~\citep{robbins1951} and its variants are among the most widely-used algorithms for solving the high-dimensional ERM problem in~\eqref{eq:erm}. In this article, we develop a new mathematical theory for the analysis of multi-pass SGD in the case that $f_i$ are quadratics, corresponding to the setting of high-dimensional $\ell^2$-regularized least squares.

A prevailing paradigm for analyzing stochatic optimizations methods is by connection to a corresponding stochastic differential equation (SDE) \citep{li2017stochastic, mandt2016variational,jastrzkebski2017three,Kushner,ljung1977,barrett2021implicit}. A central requirement in establishing such connections has been that the learning rate goes to zero so that the trajectory of the objective function over the lifetime of the algorithm converges to the solution of an SDE. Unfortunately, in the small learning rate limit, the only limiting process that results from SGD is actually an ODE \citep{yaida2018fluctuationdissipation}. 
 \begin{wrapfigure}[31]{r}{0.46\textwidth}
%\vspace{-1cm}
 \centering
       \includegraphics[width = 0.95\linewidth]{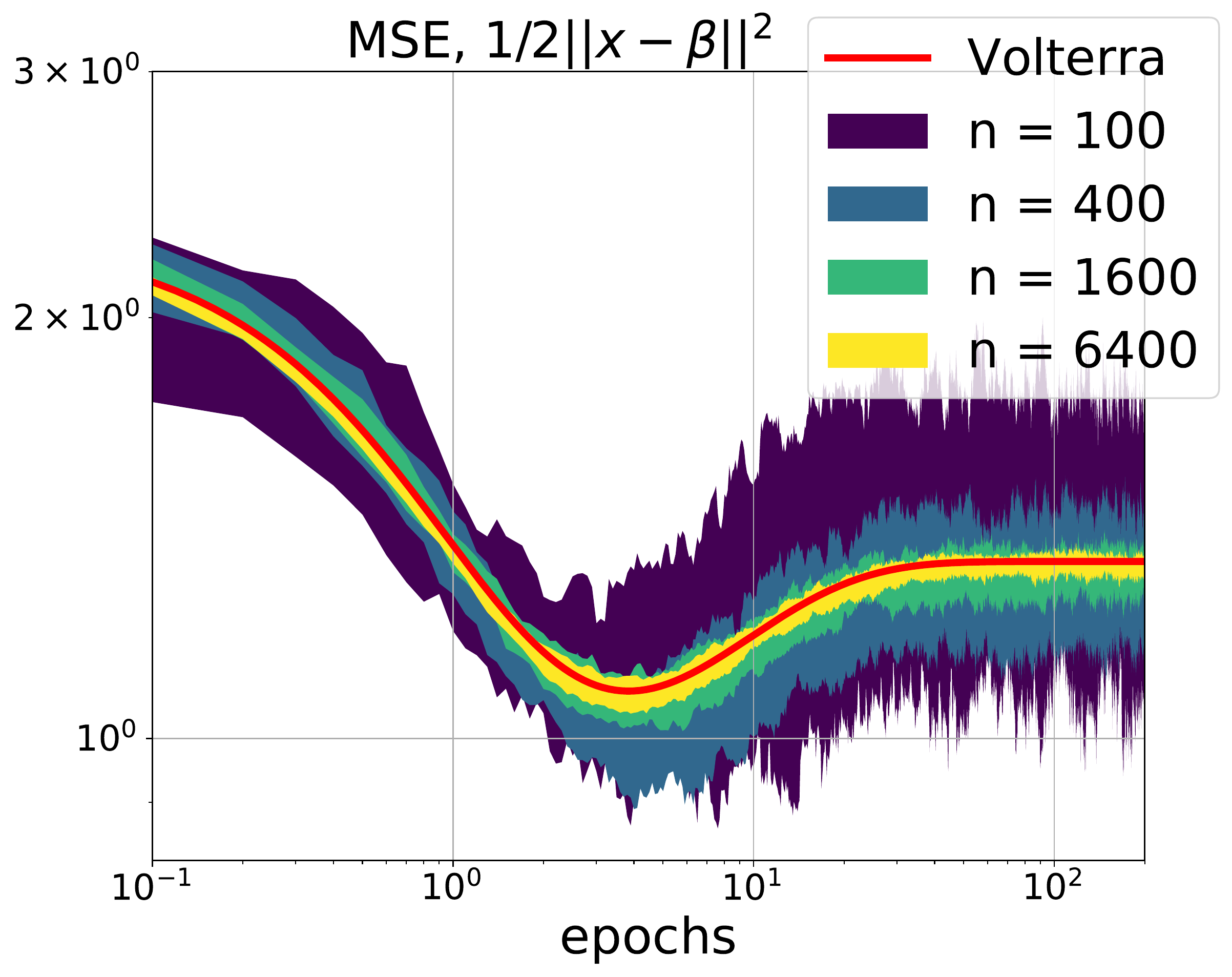}
     \caption{\textbf{Concentration of mean squared error (MSE), $\tfrac{1}{2}\|\xx-\bbeta\|_2^2$, for SGD} on a Gaussian random $\ell^2$-regularized least-squares problem (Section~\ref{sec:formal_setup}) where $\bbeta \sim N(\bm{0}, \II_d)$ is the ground truth signal and a generative model $\bb =  \AA \bbeta + \eeta$ where entries of $\eeta$ iid standard normal with $\|\eeta\|_2^2 = 2.25$, $n = 0.9d$ with $\ell^2$-regularization parameter $\delta = 0.1$, SGD was initialized at $\xx_0 \sim N(\bm{0}, 4 \II_d)$ (independent of $\AA$, $\bbeta$); an $80\%$ confidence interval (shaded region) over $10$ runs for each $n$, a constant learning rate for SGD was applied, $\gamma = 0.8$. Any quadratic statistic, such as the MSE, becomes non-random in the large limit and all runs of SGD converge to a deterministic function $\Omega_t$ (red) solving a Volterra equation \eqref{eqa:PLoss}.
     }
    \label{fig:MSE_Gaussian}
\end{wrapfigure}
Moreover, the learning rate schedules used in practical high-dimensional settings are often much larger than would be amenable to such analysis. These observations highlight the need for an alternative approach that does not require a vanishing learning schedule. 
%to . In this article, pursue a connection between SGD and an SDE that arises not as the result of small learning rate, but rather from the high-dimensionality of the problem.
%
% Much has been written identifying various aspects of SGD that are desirable for such high-dimensional problems, and to some extent, there exist convergence and risk guarantees for SGD which are valid in high-dimensions (literature review?).
%
We define a class of high-dimensional $\ell^2$-regularized least square problems, for which we can prove a quantitative comparison between SGD and an SDE, which we call \emph{homogenized SGD} (c.f.\ \eqref{eqF:HSGD}) introduced in \citep{paquette2021dynamics,paquetteSGD2021}, which improves with the dimensionality of the problem without sending learning rate to zero.  Furthermore, homogenized SGD is exactly solvable, in the sense that its training loss follows a Volterra integral equation, with parameters given by the least-squares problem.  Moreover, using this description, we give an exact expression for the generalization performance of SGD in high dimensions, for arbitrary learning rate schedules.% \jp{perhaps instead: ``for arbitrary learning rate schedules"?}. 

Many common ERM problems fit into our framework and are amenable to exact analysis and we examine two examples in detail. First, we show that high-dimensional linear regression in the $d \propto n$ regime falls into this class. Second, we analyze the random features model, with Gaussian data and Gaussian features, again in the proportionate scaling regime where the number of random features $d$ is proportionate to the size of the dataset $n$. % We will not focus on the statistical implications of this approximation here; for those, see the companion article \citep{Us}, although we give a flavor of the statistical applications here. 
The focus of this paper is on the mathematical aspects of the problem. For the statistical implications of homogenized SGD, we give a few motivating statistical applications here.
% Add this back when we have NeurIPS
% For the statistical implications of homogenized SGD, see our companion article \citep{Us}, although we also give a few motivating statistical applications here.

\subsection{Formal problem setup} \label{sec:formal_setup}
To formalize the analysis, we define the \textit{$\ell^2$-regularized least squares problem}:  %(for empirical risk minimization):
 \begin{equation}\label{eq:rr}
    \argmin_{\xx\in\mathbb{R}^d} \Big\{ f(\xx) \defas  \frac{1}{2}\|\AA \xx - \bb\|^2 + \frac{\delta}{2} \|\xx\|^2 =\sum_{i=1}^n \underbrace{ \frac{1}{2} \left ( (\aa_i \xx - b_i)^2  + \frac{\delta}{n} \|\xx\|^2 \right)}_{\defas f_i(\xx)} \Big\}.
\end{equation}
The fixed parameter $\delta > 0$ controls the regularization strength and it is independent of $n$ and $d$.  We focus on setups where the parameter choices $n$ and $d$ are large, but we do not require that they are proportional.  Instead, we need the following:
\begin{assumption}[Polynomially related]\label{ass:poly}
There is an $\alpha \in (0,1)$ so that
\[
d^{\alpha} \leq n \leq d^{1/\alpha}.
\]
\end{assumption}
\noindent Moreover, our results only gain power when one (and hence both) of these parameters are large.

The data matrix $\AA \in \mathbb{R}^{n \times d}$ and the labels $\bb$ may be deterministic or random; we formulate our theorems for deterministic $\AA$ and $\bb$ in \eqref{eq:rr} satisfying various assumptions, and in the applications of these theorems to statistical settings, we shall show that random $\AA$ and $\bb$ \eqref{eq:rr} satisfy those assumptions. These assumptions are motivated by the empirical risk minimization problem (ERM) and in particular the case where the augmented matrix $[\AA ~|~\bb]$ has rows that are independent and sampled from some common distribution. We also note that the problem \eqref{eq:rr} is homogeneous, in that if we divide all of $\AA$, $\bb$ and $\sqrt{\delta}$ by any desired scalar, we produce an equivalent optimization problem. As such, we may also adopt the following normalization convention without loss of generality.

%For a simple motivating example, in the \emph{Gaussian random ridge regression problem}, each row of $[\AA ~|~ \bb]$ is sampled independently from a multivariate normal distribution, with standardized rows ($\Exp \aa_i = 0$ and $\EE[\|\aa_i\|^2]=1$).  We will not need this standardization assumption per se, but rather we assume the following normalization convention.

\begin{assumption}[Data--target normalization] \label{assumption:Target} 
There is a constant $C>0$ independent of $d$ and $n$ such that
the spectral norm of $\AA$ is bounded by $C$
and
the target vector $\bb \in \mathbb{R}^n$ is normalized so that $\|\bb\|^2 \leq C$.
\end{assumption}
%\noindent We note that the problem \eqref{eq:rr} is homogeneous, in that if we divide all of $\AA$, $\bb$ and $\sqrt{\delta}$ by any desired scalar, we produce an equivalent optimization problem.  Hence, Assumption \ref{assumption:Target} is rather a normalization convention which so far does not limit the generality of the problem.
More importantly, we also assume that the data and targets resemble typical unstructured high-dimensional random matrices. One of the principal qualitative properties of high-dimensional random matrices is the \emph{delocalization of their eigenvectors}, which refers to the statistical similarity of the eigenvectors to uniform random elements from the Euclidean sphere. The precise mathematical description of this assumption is most easily given in terms of resolvent bounds. The resolvent $R(z; \MM)$ of a matrix $\MM \in \mathbb{R}^{d \times d}$ is
\[
  R(z; \MM) =  (z\II_d-\MM)^{-1} \quad \text{for $z \in \mathbb{C}$.}
\]
In terms of the resolvent, we suppose the following:
\begin{assumption}\label{ass: laundry_list}
  Suppose $\Omega$ is the contour enclosing $[0,1+\|\AA\|^2]$
  at distance $1/2$.
  %Let $\Omega'$ be a positively oriented smooth contour enclosing $[0,1]$ of length at most $100\pi$ and contained in the complex disk of radius $3$.  Let $\Omega$ be the dilation of $\Omega'$ by $(1+\|\AA\|_{op}^2)$, i.e.\ $\Omega = 
  %(1+\|\AA\|_{op}^2) \times \Omega'$.
  %Suppose spectral norm of $\AA^T\AA$ is bounded by $1$ with high probability.
  Suppose there is a $\theta \in (0,\tfrac 12)$ for which
  %and an event $\mathcal{G}$ that holds with high
  %probability on which
  \begin{enumerate}
    \item
      \( \displaystyle
	\max_{z \in \Omega} \max_{1 \leq i \leq n} |\ee_i^T R(z; \AA\AA^T) \bb| \leq n^{\theta-1/2}.
      \)
    \item
      \( \displaystyle
	\max_{z \in \Omega} \max_{1 \leq i \neq j \leq n} |\ee_i^T R(z; \AA\AA^T) \ee_j^T| \leq n^{\theta-1/2}.
      \)
    \item
      \( \displaystyle
      \max_{z \in \Omega} \max_{1 \leq i \leq n} |\ee_i^T R(z; \AA\AA^T) \ee_i - \tfrac 1n\tr R(z; \AA\AA^T)| \leq n^{\theta-1/2}.
      \)
  \end{enumerate}
\end{assumption}
\noindent Only the resolvent of $\AA \AA^T$ appears in these assumptions, and so in effect we are only assuming statistical properties on the left singular-vectors of $\AA$.  This assumption reflects the common formulation of ERM in which the rows of $\AA$ are independent, and so the left singular-vectors of $\AA$ are expected to be delocalized (under some mildness assumptions on the distributions of the rows). The first condition, which involves the interaction between $\AA\AA^T$ and $\bb$, can be understood as requiring that $\bb$ is not too strongly aligned with the left singular-vectors of $\AA$. The other two conditions can be viewed as corollaries of delocalization of the left singular-vectors.

\paragraph{Notation.} In this paper, we adhere whenever possible to the following notation. We denote vectors in lowercase boldface $(\xx)$ and matrices in upper boldface $(\AA)$ with processes such as gradient flow in calligraphic script $\bm{\mathscr{X}}$. Unless otherwise specified, the norm $\|\cdot\|$ is taken to be the standard Euclidean norm if it is applied to a vector and the operator 2-norm if it is applied to a matrix. The vector $\ee_i$ is a standard basis vector with a $1$ in the ith coordinate and otherwise $0$ and the matrix $\II_d$ is the $(d \times d)$ identity matrix. For a matrix $\AA$ and a vector $\bb$, we denote constants depending on $\AA$ and $\bb$, $C(\AA, \bb)$, as those bounded by an absolute constant multiplied by $\|\AA\|$ and $\|\bb\|$. For convenience we will also use the subgaussian norm $\|\cdot\|_{\psi_2}$ (see e.g., \citep{vershynin2018high} for more details) which is equivalent up to universal constants to the optimal variance proxy in a Gaussian tail bound for a random variable $X$ i.e.,
\[
\| X \|_{\psi_2} 
\asymp
\inf \{ V  > 0 : \forall~t > 0~\Pr( |X| > t) \leq 2 e^{-t^2/V^2}\}.
\]
We say an event $B$ holds with overwhelming probability (w.o.p.) if, for every fixed $D > 0$, $\Pr(B) \ge 1 -C_D d^{-D}$ for some $C_D$ independent of $d$.  

\subsection{Algorithmic setup}

We solve the ERM problem \eqref{eq:rr} using stochastic gradient descent (SGD) with learning rate $\gamma_k$: for an initial vector $\xx_0 \in \mathbb{R}^d$, we define a sequence of iterates $\{\xx_k\}$ which obey the recurrence,
\begin{equation} \begin{aligned} \label{eq:sgd}
    \xx_{k+1} = \xx_{k} - \gamma_k \nabla f_{i_k}(\xx_k) = \xx_k - \gamma_k \AA^T \ee_{i_k} \ee_{i_k}^T ( \AA \xx_k - \bb) - \tfrac{\gamma_k \delta}{n} \xx_k\,.
\end{aligned} \end{equation}
The rows $\{i_1,i_2,\dots\}$ are chosen uniformly at random, and thus the batch size is one.  Earlier work \citep{paquetteSGD2021} suggests that under similar (albeit more restrictive) assumptions, minibatch SGD with batch-size $\beta = o(n)$ produces the same dynamical behavior as SGD after sampling single-batch SGD at iteration counts $\beta \N$.  Therefore, we content ourselves with the simpler case with batch size equal to one.

As we want to give descriptions of the dynamics of SGD which are consistent across increasing dimensions, we suppose that $\gamma_k$ has a smoothly varying schedule.  Specifically, we suppose:
\begin{assumption}\label{assumption:lr}
There is a continuous bounded function $\gamma : [0, \infty)\to [0,\infty)$ such that $\gamma_k = \gamma(k/n)$ for all $k$.  As such 
\[
\widehat{\gamma} \defas \sup_{t} \gamma(t)  < \infty.
\]
\end{assumption}
\noindent Although the classic Robbins-Monro $\gamma_k = \tfrac1k$ does not technically fit into this framework, for problems in which Assumptions \ref{assumption:Target} and \ref{ass: laundry_list} are in effect (or more generally where some non-trivial fraction of the samples are needed to commence learning), the classic $1/k$ rate is often too slow to produce any practically relevant results. Moreover, from a theoretical point of view, such a rate produces behavior similar to gradient flow (see \eqref{eq:GF}) \citep{}, and it could be viewed as effectively non-stochastic. In our high-dimensional setting, a suitable analogue of the Robbins-Monro schedule that does satisfy our assumptions and yields nontrivial behavior is $\gamma_k = \tfrac{n}{n+k}=\tfrac{1}{1+k/n}$.

As for the initialization $\xx_0$, we need to suppose that it, does not interact too strongly with the \emph{right} singular-vectors of $\AA$.  %\jp{Changed from $\AA \AA^T$}.
In the spirit of Assumption \ref{ass: laundry_list}, it suffices to assume the following:
\begin{assumption}\label{assumption:init}
Let $\Omega$ be the same contour as in Assumption \ref{ass: laundry_list} and let $\theta \in (0,\tfrac 12)$.  Then
      \[
	\displaystyle \max_{z \in \Omega} \max_{1 \leq i \leq d} |\ee_i^T R(z; \AA^T\AA) \xx_0| \leq n^{\theta-1/2}.
      \]
\end{assumption}
\noindent Note that, as a simple but common case, this assumption is surely satisfied for $\xx_0 = \bm{0}$.  In principle, this assumption is general enough to allow for $\xx_0$ which are correlated with $\AA$ in a nontrivial way, but we do not have an application for such an initialization.  For a large class of nonzero initializations independent from $(\AA,\bb)$, this assumption is satisfied, as a corollary of Assumption \ref{ass: laundry_list}:
\begin{lemma}\label{lem:xo}
    Suppose that Assumption \ref{ass: laundry_list} holds with some $\theta_0 \in (0,\tfrac12)$ and that $\xx_0$ is chosen randomly, independent of $(\AA,\bb)$, and with independent coordinates in such a way that for some $C$ independent of $d$ or $n$
    %so that for some $R > 0$ independent of $n$ 
    \[
    \|\Exp \xx_0\|_\infty \leq C/n
    \quad\text{and}\quad
    \max_i\|(\xx_0-\Exp \xx_0)_i\|^2_{\psi_2} \leq Cn^{2\theta_0-1}.
    \]
    For any $\theta >\theta_0$, Assumption \ref{assumption:init} holds with any $\theta > \theta_0$ on an event of probability tending to $1$ as $n \to \infty$.
\end{lemma}
\noindent Note that this assumption allows for deterministic $\xx_0$ having maximum norm $\mathcal{O}(1/n)$, as well as iid centered subgaussian vectors of Euclidean norm $\mathcal{O}(1)$.

\subsection{Homogenized SGD} \label{sec:homogenized_SGD_intro_1}
Our main result is a comparison of the dynamical behavior of SGD \eqref{eq:sgd} to another process, \emph{homogenized SGD} (HSGD) applied to the $\ell^2$-regularized least-squares \eqref{eq:rr}.  To formulate HSGD, we will refer frequently to the empirical risk
\begin{equation}\label{eq:ERM}
\mathscr{L}(\xx)\defas\frac{1}{2}\|\AA \xx - \bb\|^2,
\end{equation}
which differs from the $\ell^2$-regularized objective function $f$ in \eqref{eq:rr} in that no regularizer has been added.
\emph{Homogenized SGD} is defined as the strong solution of the stochastic differential equation:
\begin{equation}\label{eqF:HSGD}
\dif \XX_t \defas
-\gamma(t) \nabla f(\XX_t) \dif t
+ \gamma(t) \sqrt{\tfrac{2}{n}\mathscr{L}(\XX_t)\nabla^2\mathscr{L}(\XX_t)}\dif \BB_t,
\end{equation}
where the initial conditions given by $\XX_0 = \xx_0$ and $(\BB_t : t \geq 0)$ a $d$--dimensional standard Brownian motion. The time variable of HSGD is defined so that one unit of time corresponds to $n$ steps of SGD \eqref{eq:sgd}.  Because SGD performs sampling with replacement, after $n$ steps, SGD will not have used each of the datapoints with high probability, but rather a constant fraction (approximately $1-1/e$) of them.

A natural point of comparison to HSGD is \emph{gradient flow}, which is the low-noise limit of \eqref{eqF:HSGD}.  Specifically, we define
\begin{equation}\label{eq:GF}
\dif \bm{\mathscr{X}}_t^{\text{gf}} \defas
-\nabla f(\bm{\mathscr{X}}_t^{\text{gf}}) \dif t,
\quad
\bm{\mathscr{X}}_0^{\text{gf}} = \xx_0\,.
\end{equation}
In the case that the objective function $f$ is the $\ell^2$-regularized least squares problem \eqref{eq:rr}, the gradient flow ODE is explicitly solvable, 
\begin{equation}\label{eq:GFS}
    \bm{\mathscr{X}}_t^{\text{gf}} = e^{-(\AA^T \AA + \delta \II_d)t} \xx_0 + (\AA^T \AA + \delta \II_d)^{-1} \big [ \II_d - e^{-(\AA^T \AA + \delta \II_d)t} \big ] \AA^T \bb.
\end{equation}
We can adjust the gradient flow solution to account for a learning rate $\gamma(t)$. If we let $\Gamma(t) = \int_0^t \gamma(s)\,\dif s$, then $\bm{\mathscr{X}}^{\text{gf}}_{\Gamma(t)}$ solves the ODE,
\(
\dif \bm{\mathscr{X}}_{\Gamma(t)}^{\text{gf}}=
-\gamma(t)\nabla f(\bm{\mathscr{X}}_{\Gamma(t)}^{\text{gf}})\,.
\)
Hence in \eqref{eqF:HSGD}, if we were to set the Brownian noise to $0,$ we would have nothing but gradient flow taken at time $\Gamma(t).$ This observation plays a significant role when comparing to the dynamics of SGD (see Section~\ref{sec:exact_solvability_SGD}).

Diffusion approximations to SGD have a long history.  In the stochastic approximation literature, it appears as a natural counterpart to ODE methods (c.f.\ \cite{Kushner}, \cite{LjungPflug}).  However, these are methods that require the vanishing learning rate (such as $\gamma_k = 1/k$) and moreover, in the setup we have suggested here, the resulting SDE (which only arises in an asymptotic comparison, as is standard with stochastic approximation theory) has a vanishing diffusion term for such an aggressive learning-rate decay --- the asymptotic trajectory of SGD is also approximated by the ODE gradient flow with time change $\Gamma(n) \approx \log n.$

A more natural point of comparison is the \emph{stochastic modified equation} of \citep{mandt2016variational,li2017stochastic,li2018measuring} which has been rigorously compared to the behavior of SGD \citep{li2019stochastic}.   To make a comparison with \citep{li2019stochastic}, we fix the learning rate $\gamma$ and we rescale time to be on the order of epochs.  With these changes, the SME solves,
\begin{equation}\label{eqF:SME}
\dif \bm{\mathscr{M}}_t \defas
-\gamma \nabla f(\bm{\mathscr{M}}_t) \dif t
+ \gamma \sqrt{n \SSigma(\bm{\mathscr{M}_t}})\dif \BB_t,
\quad\text{where}
\quad
\left\{
\begin{aligned}
&\SSigma(\xx) \defas \Exp\bigl(\nabla g_{I}(\xx) \otimes \nabla g_{I}(\xx)\bigr), \\
&g_I(\xx) \defas f_I(\xx) - \Exp f_I(\xx), \\
&I \defas \operatorname{Uniform}\{1,2,\ldots, n\}.
\end{aligned}
\right.
\end{equation}
The diffusion matrix $\SSigma$ of the SME is chosen to exactly match the covariance of the increments of SGD \eqref{eq:sgd}.  When applied to the $\ell^2$-regularized problem \eqref{eq:rr},
this matrix becomes (with $\aa_i$ the $i$-th row of $\AA$)
\[
\SSigma(\xx) 
=
\frac{1}{n}\sum_{i=1}^n (\aa_i\cdot \xx - \bb_i)^2 \aa_i^T \aa_i
-\frac{1}{n^2} \AA^T (\AA \xx - \bb)(\AA \xx - \bb)^T \AA.
\]
HSGD and the SME can be compared by replacing second term by $0$ and the first term by
\[
n\SSigma(\xx) 
\approx \frac{1}{n}
\biggl(
\sum_{i=1}^n (\aa_i\cdot \xx - \bb_i)^2 
\biggr)
\times
\biggl(
\sum_{i=1}^n \aa_i^T \aa_i
\biggr)
= \frac{2}{n} \mathscr{L}(\xx) \nabla^2 \mathscr{L}(\xx),
\]
which is the diffusion coefficient in HSGD.  

The SME has been used for a variety of purposes, such as optimal learning rate scheduling \citep{li2018measuring}, analysis of momentum terms \citep{li2017stochastic}, and prediction of test risk behavior \citep{smith2020on}; however, analysis of the SME is itself difficult, as the diffusion coefficient involves interactions between the functions $f_i$, and, to our knowledge, while the theory developed in \citep{li2019stochastic} provides dimension-independent comparisons, the resulting SME has not been analyzed in any high-dimensional setting.  Furthermore, the mathematical comparison which is proven in \citep{li2019stochastic}, on the time scale in \eqref{eqF:SME},  %\jp{which time scale?},
gives a comparison for time of order $\mathcal{O}(1/n)$ and with an error which is bounded by $\mathcal{O}(\gamma)$.  As such, in a high-dimensional setting, the comparison that exists between SME and SGD only provides a non-vanishing error over a vanishing window of time.

Because the SME naturally matches the drift and diffusion matrix of SGD, the above remarks might lead to a conjecture that in fact no comparison is possible between SGD and SDEs; indeed, in a fixed-dimensional analysis, \citet{yaida2018fluctuationdissipation} showed that there is no small learning-rate limit of SGD that produces nontrivial stochastic behavior. In contrast, we will show that in a high-dimensional limit, this is precisely what occurs (although when univariate statistics of this high-dimensional SDE are taken, almost deterministic behavior is seen).

\subsection{Exact Solvability of Homogenized SGD} \label{sec:exact_solvability_SGD}

In the $\ell^2$-regularized least square problem \eqref{eq:rr}, all nontrivial interactions between the coordinates of $\XX_t$ are mediated through a single scalar, the empirical risk $\mathscr{L}(\XX_t)$. In a high-dimensional setting, this empirical risk concentrates around a deterministic path $\Psi_t$. To define this path, we introduce the integrated learning rate $\Gamma$ and kernel $K$, for any $d \times d$ matrix $\PP$,
\begin{equation}\label{eqa:V}
\Gamma(t) = \int_0^t \gamma(s)\,\dif s,
\quad
\text{and}
\quad
K(t,s ; \PP) = 
\tfrac{1}{n}
\gamma^2(s)
\tr
\biggl(
\PP
(\nabla^2 \mathscr{L})
\exp\bigl( -2(\nabla^2 \mathscr{L} + \delta \II_d )( \Gamma(t) - \Gamma(s))\bigr)
\biggr).
\end{equation}

% \begin{wrapfigure}[21]{r}{0.5\textwidth}
% \vspace{-0.3cm}
%  \centering \includegraphics[width = 0.85\linewidth]{figures/Fig2.pdf}
%  \vspace{-0.25cm}
%      \caption{Time-dependent test error of the random feature model under SGD, given in Thm.~\ref{thm:random_features}, as a function of the integrated learning rate $\Gamma(t) = \gamma t$ for various constant learning rates $\gamma$, with zero regularization ($\delta = 0$), no additive noise ($\eta = 0$), and normalized ReLU activation function. Good agreement is observed between the asymptotic theoretical predictions (solid curves) and finite-size empirical simulations with $n=4800$, $d=8000$, and $n_0=4000$ (markers). For small learning rates, the curves approach that of gradient flow (black line).}
%     \label{fig:RF_gradient_flow}
% \end{wrapfigure}

The path $\Psi_t$ satisfies the Volterra integral equation: %(see Theorem \ref{thm:conc} for a precise formulation):
\begin{equation}\label{eqa:VLoss}
\Psi_t
=
\mathscr{L}\bigl( 
\bm{\mathscr{X}}^{\text{gf}}_{\Gamma(t)}\bigr)
+
\int_0^t 
K(t,s; \nabla^2 \mathscr{L}) 
\Psi_s
\dif s.
\end{equation}
Moreover for any other quadratic, $\mathcal{R} : \R^d \to \R$ independent of the Brownian path, the trajectory $\mathcal{R}(\XX_t)$ concentrates around
\begin{equation}\label{eqa:PLoss}
\Omega_t
=
\mathcal{R}\bigl( 
\bm{\mathscr{X}}^{\text{gf}}_{\Gamma(t)}\bigr)
%+
%\int_0^t 
%A(t,s; \nabla^2 \mathcal{E})
%\dif s +
+\int_0^t 
K(t,s; \nabla^2 \mathcal{R})
\Psi_s
\dif s.
\end{equation}
Note the trajectories of gradient flows can computed explicitly using \eqref{eq:GFS}.

Under relatively weak assumptions, % Assumption \ref{ass:train} and \ref{ass:risk}, 
we can precisely connect homogenized SGD to the solutions of these equations.
These assumptions can roughly be summarized as stating that the functionals $\mathscr{L}$ and $\mathcal{R}$ depend on sufficiently many coordinates, so that some concentration of measure can take place.  The statistics we consider of the SGD/HSGD path are all of the following form:
\begin{definition}\label{def:quad}
A function $\mathcal{R} :\R^d \to \R$ is quadratic 
if it is a degree-$2$ polynomial or equivalently if can be represented by 
\[
\mathcal{R}(\xx) = \tfrac{1}{2}\xx^T \TT \xx + \uu^T \xx + c
\]
for some $d \times d$ matrix $\TT$, vector $\uu \in \R^d$ and scalar $c \in \R$.
For any quadratic, define the $H^2$--norm:
\[
\|\mathcal{R}\|_{H^2}
\defas
\|\nabla^2 \mathcal{R}\|
+
\|\nabla \mathcal{R}(0)\|
+|\mathcal{R}(0)|
=
\|\TT\|
+
\|\uu\|
+|c|.
\]
\end{definition}
\noindent We note that under Assumption \ref{assumption:Target} the empirical risk will have bounded $H^2$--norm.  Then:
\begin{theorem}\label{thm:trainrisk} Suppose $\mathcal{R} :\R^d \to \R$ is a quadratic with $\|\mathcal{R}\|_{H^2} \leq C$ for some $C >0$.  Suppose that $(\AA,\bb)$ satisfies for some $\epsilon > 0$
\begin{equation}\label{eq:trainrisk}
\tr(\AA^T\AA) \leq C n,
\quad
\|\bb\| \leq C,
\quad\text{and}\quad
\|\AA^T\AA\| \leq n d^{-\epsilon} \le C.
\end{equation}
Then for any $T > 0$ and for any $D>0$ there is an $C'>0$ sufficiently large that for all $d>0$
\[
\Pr\biggl[
\sup_{0 \leq t \leq T}\biggl\|
\begin{pmatrix} \mathscr{L}(\XX_t) \\ \mathcal{R}(\XX_t)\end{pmatrix}
-
\begin{pmatrix}\Psi_t \\ \Omega_t \end{pmatrix}
\biggr\| > d^{-\epsilon/2} 
\biggr] \leq C'd^{-D},
\]
where $\Psi_t$ and $\Omega_t$ solve \eqref{eqa:VLoss} and \eqref{eqa:PLoss}.
\end{theorem}

\begin{figure}[t]
    \centering
    \includegraphics[scale =0.18]{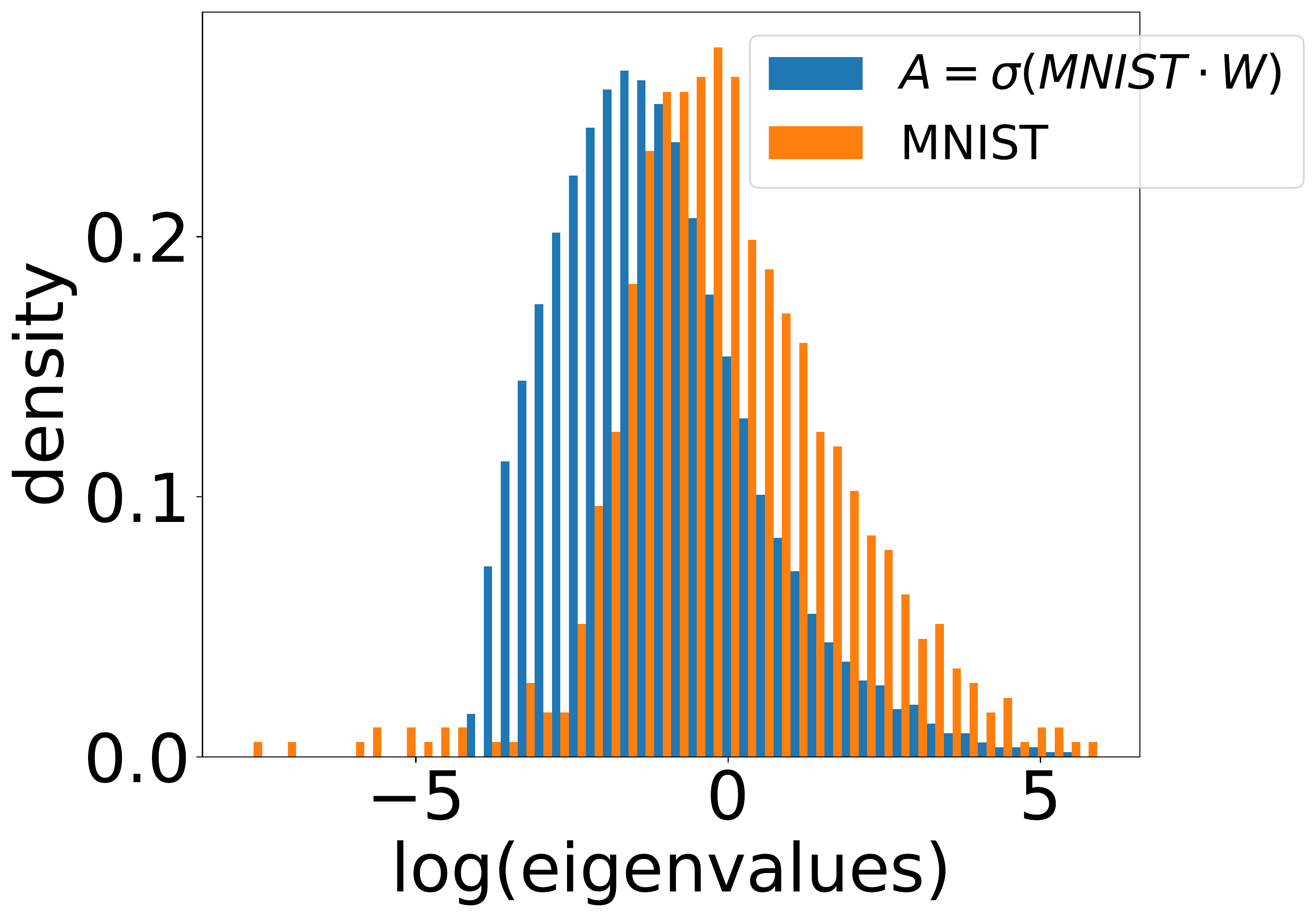} \includegraphics[scale =0.18]{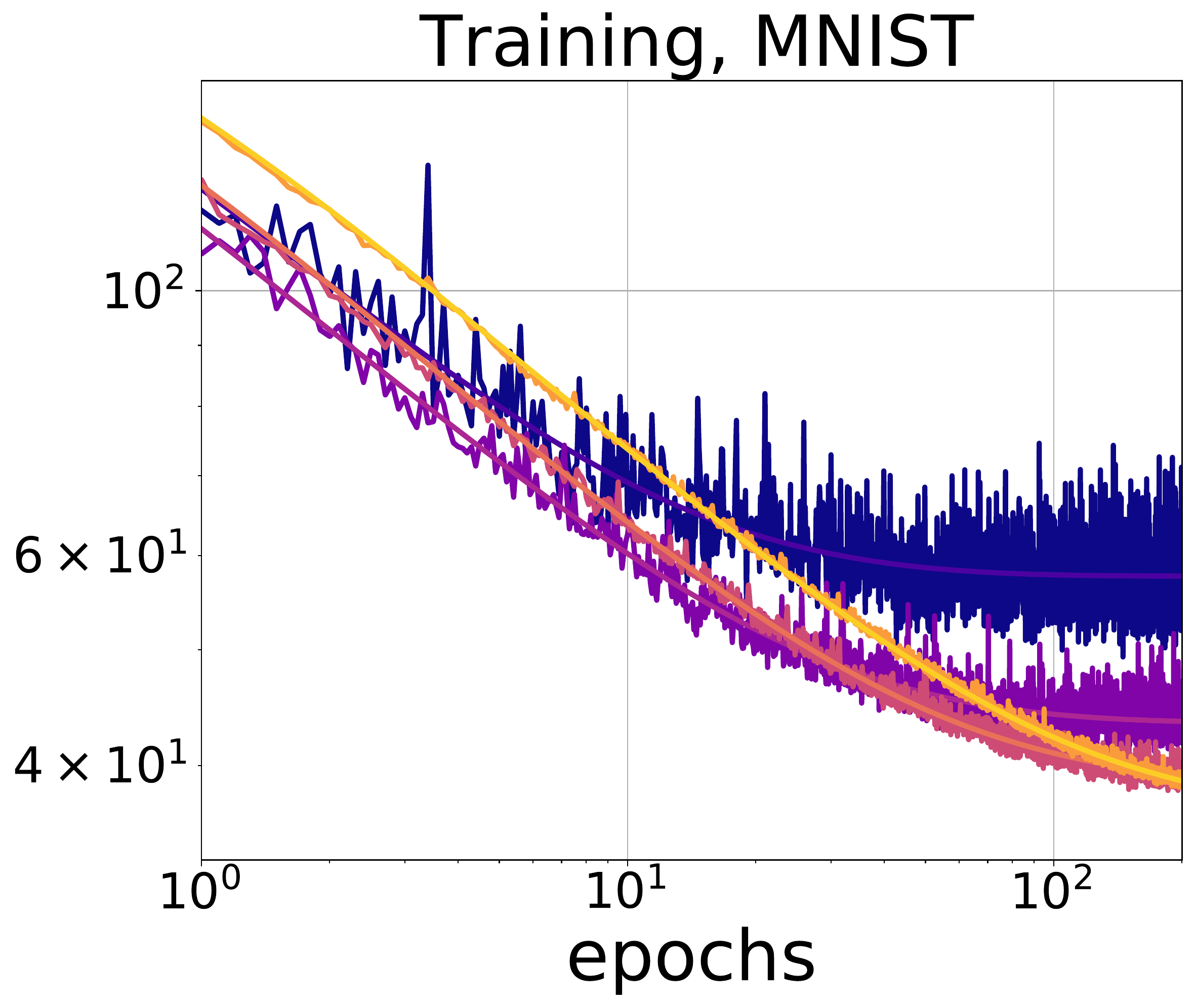} 
    \includegraphics[scale =0.18]{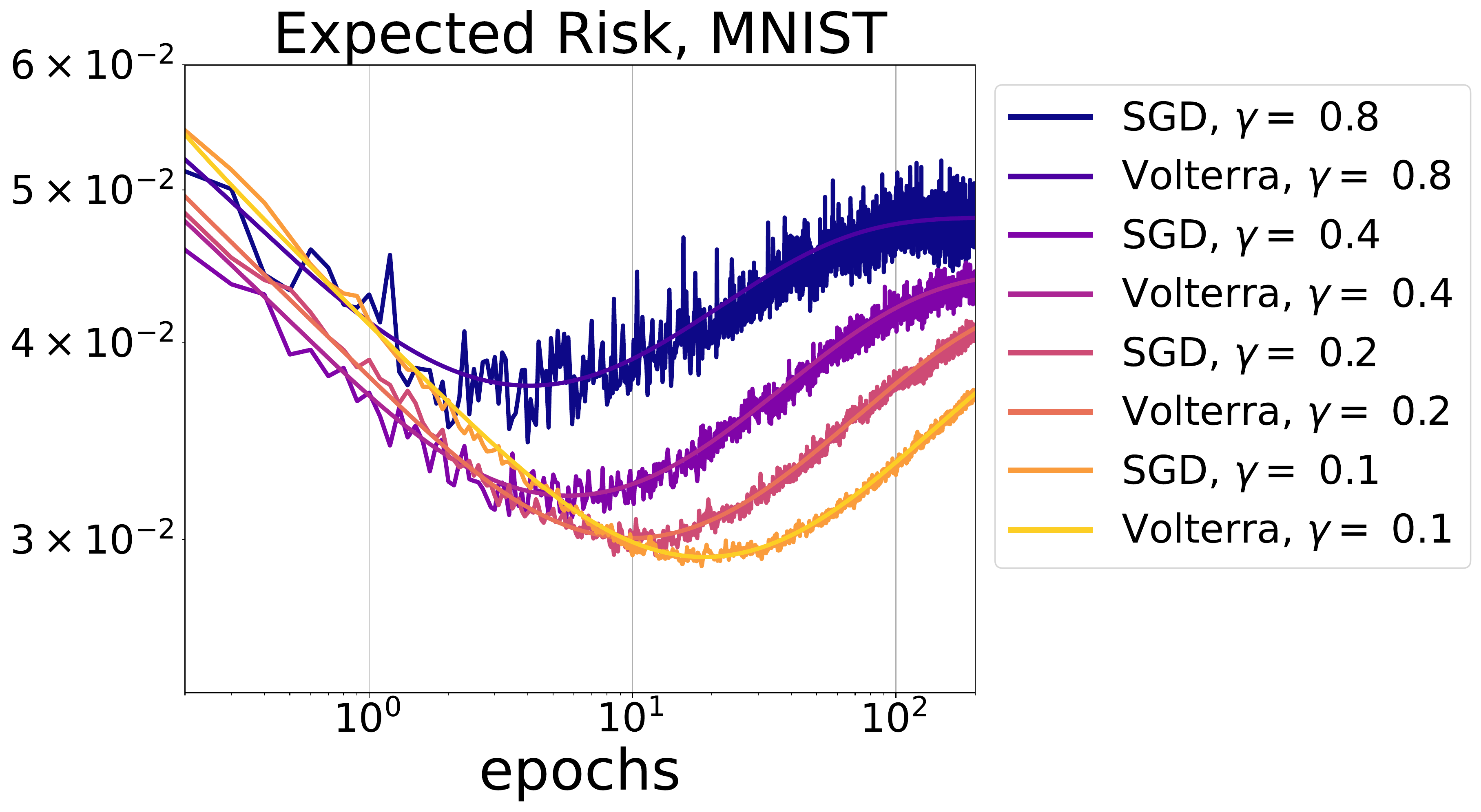}
    \caption{\textbf{SGD vs Theory on MNIST}: MNIST $(60000 \times 28 \times 28)$ images. Random features model on MNIST used with $n=4000$ images, random features $d = 2000$, and $n_0 = 28 \times 28$ was trained with one run of SGD (middle) for various learning rates and regularization parameter $0.01$; entries of the random features $\WW_{ij} \sim N(0,1)$ and a normalized ReLu activation function $\sigma(\cdot) = (\max\{0, \cdot\} - a)/b$ was applied. The Volterra equation matches the dynamics of the training loss (least-squares), $\mathcal{L}$, even with only one run of SGD. The log(eigenvalues) of the covariance of the MNIST dataset and the random features matrix used in the regression displayed (left). The expected risk, $\mathcal{R}(\xx) = \frac{1}{2}\mathbb{E}[(b - \xx^T \sigma(\xx_i \WW))^2]$ where $\xx_i$ is an image from the MNIST test set, follows the predicted behavior $\Omega_t$. Both the predicted $\Psi_t$ and $\Omega_t$ match the performance of SGD in this non-idealized setting.}
    % duplicated to account for space
    \vspace{-0.5cm} \label{fig:MNIST}
\end{figure}

\noindent We give a formal proof of Theorem~\ref{thm:trainrisk} in Section~\ref{sec:concentration_hsgd} (see \citep[Theorem 11]{Us} slightly generalizes the setting).
The functions $\Psi_t$ and $\Omega_t$ can be viewed as the expected behavior of the loss function $\mathscr{L}(\XX_t)$ and any quadratic statistic $\mathcal{R}(\XX_t)$ under homogenized SGD. Theorem~\ref{thm:trainrisk} then shows concentration around the mean. We remark that to solve \eqref{eqa:V} we need as input $\mathscr{L}(\bm{\mathscr{X}}_{\Gamma(t)}^{\text{gf}})$ which be computed using \eqref{eq:GFS}. The solution of $\Psi$ is can then be found by repeatedly convolving the forcing term $\mathscr{L}(\bm{\mathscr{X}}_{\Gamma(t)}^{\text{gf}})$ with the kernel $K$ (provided $\sup_{t\ge0} \sup_{s \ge0} K(t,s; \nabla^2 \mathscr{L})$ is bounded \citep{gripenberg1980volterra}), that is, 
\begin{align*}
    \Psi_t = g(t) + (\mathcal{K} \star g)(t) + (\mathcal{K} \star \mathcal{K} \star g)(t) + \cdots
    \quad\text{where}\quad
    \left\{
    \begin{aligned}
    g(t) &\defas \mathscr{L}(\bm{\mathscr{X}}_{\Gamma(t)}^{\text{gf}}),  \\
    (\mathcal{K} \star h)(t) &\defas \int_0^t K(t,s; \nabla^2 \mathscr{L}) h(s)\,\dif s,\\
    \quad \forall~h &\,\in C([0,\infty)).
    \end{aligned}
    \right.
\end{align*}
Moreover, numerical approximations to \eqref{eqa:V} can be found by taking a large but finite number of convolutions in the expression above.  The boundedness of this solution corresponds precisely to learning rate choices for which SGD is convergent.  

In the case of constant learning rate $\gamma(s) \equiv \gamma$, more can be said. The Volterra equation \eqref{eqa:V} is of convolution--type, and in fact is a special case of the renewal equation \citep{Asmussen} (allowing for \emph{defective} and \emph{excessive} variants). Specifically, the expression in \eqref{eqa:VLoss} simplifies to
\begin{align} \label{eq:constant_gamma_psi}
    \Psi_t = \mathscr{L}(\bm{\mathscr{X}}_{\Gamma(t)}^{\text{gf}}) + \frac{\gamma^2}{n}  \int_0^t \tr \left ( (\AA^T \AA)^2 e^{-2 (\AA^T\AA + \delta \II_d)(t-s)} \right ) \Psi_s \, \dif s. 
\end{align}
In addition to fixed point algorithms, one can also use Laplace transform techniques for deriving analytical expressions for $\Psi_t$ in this case. These solutions to \eqref{eq:constant_gamma_psi} can be analyzed explicitly for convergence guarantees and  rates of convergence, see \citep{paquette2021dynamics,paquetteSGD2021}.
% Add this back when NeurIPS
%and \citep{Us} for an analysis of these equations for general statistical considerations. 
As a simple example writing $K(t,s; \PP) = K(t-s; \PP)$, the convergence of \eqref{eq:constant_gamma_psi} occurs precisely when $\int_0^\infty K(t; \nabla^2 \mathscr{L})\,\dif t \leq 1$.\footnote{See \citep[Chapter V]{Asmussen} for a general discussion.  In the case that the norm is exactly $1$,  this remains true as it is a special case of the Blackwell renewal theorem.  When the norm is larger than $1$, in the event that the empirical risk of gradient flow is bounded away from $0,$ the training loss is divergent.}
% See especially \citep[Theorem 2]{Us} for the mathematical formulation and proof and of the concentration of the paths of $\mathscr{L}(\XX_t)$ and $q(\XX_t)$ around $\Psi_t$ and $\Omega_t$.

Under the assumption that $\gamma(s)$ stabilizes eventually, i.e.\ $\gamma(s) \to \gamma$ as $s\to\infty,$ we may still characterize the eventual behavior of solution.  In particular we can formulate the eventual behavior of $\Omega_t$ as $t \to \infty$.  In the case that $\mathcal{R}$ represents the population risk, then the difference $\Omega_t-\mathcal{R}\bigl( 
\bm{\mathscr{X}}_{\Gamma(t)}^{\text{gf}}\bigr)$ gains the interpretation of the excess risk of SGD over gradient flow.  On taking $t\to\infty$, this is thus the excess risk of the SGD estimator over the ridge regression estimator:
\begin{theorem}\label{thm:eventualrisk} 
If $\gamma(t) \to 0$ but $\Gamma(t) \to \infty$ as $t\to\infty$ (c.f.\ the Robbins-Monro setting), then 
\(
\Omega_t-\mathcal{R}\bigl( 
\bm{\mathscr{X}}_{\Gamma(t)}^{\text{gf}} \bigr)
\xrightarrow[t\to\infty]{ } 
0.\) 
If on the other hand $\gamma(t) \to \widetilde{\gamma} > 0$,
where the limiting learning rate satisfies 
\begin{equation} \label{eq:convergence_threshold}
\widetilde{\gamma} < 2 \big ( \tfrac{1}{n} \tr ( (\AA^T \AA)^2 (\AA^T \AA + \delta \II_d)^{-1} ) \big )^{-1},
\end{equation}
 then with $\Psi_\infty$ given by the limiting empirical risk:
\[
\Psi_\infty 
=
\mathscr{L}\bigl( 
\bm{\mathscr{X}}_{\infty}^{\text{gf}}\bigr)
\times
\biggl(
1
-
\frac{\widetilde{\gamma}}{2n}
\tr
\bigl(
(\nabla^2 \mathscr{L})^2
\bigl( \nabla^2 \mathscr{L} + \delta \II_d \bigr)^{-1}
\bigr)
\biggr)^{-1},
\]
the limiting excess risk of SGD over ridge regression is given by
\[
\Omega_t-\mathcal{R}\bigl( 
\bm{\mathscr{X}}_{\Gamma(t)}^{\text{gf}}\bigr)
\xrightarrow[t\to\infty]{ } \frac{\widetilde{\gamma}}{2n}  \Psi_\infty \times 
\tr
\biggl(
(\nabla^2 \mathcal{R}) (\nabla^2 \mathscr{L})
\bigl( \nabla^2 \mathscr{L} + \delta \II_d \bigr)^{-1}
\biggr).
\]
\end{theorem}
\noindent We elaborate more on the statistical interpretations in Section \ref{sec:stats}. 
% Add this back on when we have the NeurIPS paper
% (see also \cite{Us} for further statistical conclusions).

\subsection{Comparison between SGD and HSGD}

To make the comparison between SGD and HSGD, we shall apply maps from the paths of SGD and HSGD to $\R$.\footnote{In fact, we expect that it is untrue that the processes can be realized in such a way that $\|\xx_{\lfloor tn \rfloor} - \XX_t\|$ is vanishingly small.  It is only after observing a low--dimensional shadow that the processes become indistinguishable.}  The maps of interest are chosen to have some statistical consequence, and so we will consider a general class of quadratic functions. 
%We will be interested in analyzing the dynamics of the iterates $\xx_k$ applied to different statistics, that is general quadratic statistics (see below). 
By doing so, we can capture finite-time behavior of SGD on metrics of suboptimality (c.f., function values, distance to optimality, and gradients). Moreover by allowing for generic quadratics, we can capture generalization performance of SGD at finite time under differing test and training covariance structures.  See Section \ref{sec:stats} for further development and applications.  To execute the mathematical comparison between SGD and HSGD, we require an additional assumption on the quadratic in the same spirit as Assumption \ref{ass: laundry_list}:% implies some independence of the quadratic from the left-singular-vectors of $\AA$:

\begin{assumption}[Quadratic statistics] \label{assumption: quadratics} Suppose $\mathcal{R} : \mathbb{R}^d \to \mathbb{R}$ is quadratic, i.e. there is a symmetric matrix $\TT \in \mathbb{R}^{d \times d}$, a vector $\uu \in \mathbb{R}^d$, and a constant $c \in \mathbb{R}$ so that
\begin{equation} \label{eq:statistic}
    \mathcal{R}(\xx_t) = \tfrac{1}{2} \xx_t^T \TT \xx_t + \uu^T \xx_t + c.
\end{equation}
We assume that $\mathcal{R}$ satisfies $\|\mathcal{R}\|_{H^2} \le C$ for some $C$ independent of $n$ and $d$. Moreover, we assume the following (for the same $\Omega$ and $\theta$) as in Assumption \ref{ass: laundry_list}:
\begin{equation} \label{eq:key_lemma_ass}
    	\max_{z,y \in \Omega} \max_{1 \leq i \leq n} 
    	|\ee_i^T \AA \widehat\TT  \AA^T \ee_i
    	-\tfrac1n\tr(\AA \widehat\TT  \AA^T)
    	| \leq \|\TT\| n^{-\epsilon}
    	\quad\text{where}\quad 
    	\left\{
    	\begin{aligned}
    	&\widehat\TT =  R(z) \TT R(y) + R(y) \TT R(z), \\
    	&R(z) = R(z; \AA^T\AA)
    	\end{aligned}
    	\right.
\end{equation}
\end{assumption}

\noindent This assumption ensures that quadratic $\mathcal{R}$ has a Hessian which is not too correlated with any of the left singular--vectors of $\AA$.  Establishing Assumption \ref{assumption: quadratics} can be non--trivial in the cases when the quadratic has complicated dependence on $\AA$.  In simple cases, (especially for the case of the empirical risk and the norm) it follows automatically from Assumption \ref{ass: laundry_list}.
\begin{lemma}\label{lem:qs}
    Suppose that 
    \(
    \mathcal{R}
    \)
    satisfies \eqref{eq:statistic} with $\TT$ given by a polynomial $p$ in $\AA^T \AA$ (especially  $\II$ and the monomial $\AA^T\AA$) having bounded coefficients, and suppose $\uu$ and $c$ are norm bounded independently of $n$ or $d$.  Then supposing Assumptions \ref{assumption:Target} and \ref{ass: laundry_list} for some $\theta_0 \in (0,\tfrac 12)$, for all $n$ sufficiently large and for any $\theta > \theta_0$, Assumption \ref{assumption: quadratics} holds.
\end{lemma}
\noindent Thus for example $\mathcal{R} = \mathscr{L}$ will satisfy Assumption \ref{assumption: quadratics} and the simple Euclidean vector norm $\mathcal{R} = \| \cdot \|^2$ under the assumptions invoked.  The other main application to consider is when $\mathcal{R}$ is a fixed matrix and the rows of $\AA$ are independent.  In this case, provided the rows of the random matrix satisfy a certain ``quadratic concentration property'', this follows (see Section \ref{sec:quad_con} for formal statements).

Our main comparison theorem is the following:
\begin{theorem}[Homogenized SGD and SGD] \label{thm:homogenized_SGD_SGD}
Suppose $n$ and $d$ are related by Assumption \ref{ass:poly}.  Suppose the $\ell^2$-regularized least-squares problem \eqref{eq:rr} satisfies Assumptions~\ref{assumption:Target} and \ref{ass: laundry_list} where $n \ge d^{\tilde{\varepsilon}}$ for some $\tilde{\varepsilon} > 0$.
Suppose the learning rate schedule $\gamma$ satisfies Assumption \ref{assumption:lr}, and the initialization $\xx_0$ satisfies Assumption \ref{assumption:init}.  
Let $\mathcal{R} \, : \, \mathbb{R}^d \mapsto \mathbb{R}$ be any quadratic statistic satisfying Assumption~\ref{assumption: quadratics}. For any deterministic $T > 0$ and any $D > 0$, there is a $C > 0$ such that
\[ \Pr \bigg [  \sup_{0 \le t \le T} |\mathcal{R}(\xx_{\lfloor tn \rfloor} )-\mathcal{R}(\XX_t)| > d^{-\tilde{\varepsilon}/2} \, \big | \, \AA, \bb, \xx_0 \bigg ] \le C d^{-D} , \]
where $\xx_t$ and $\XX_t$ are the iterates SGD and homogenized SGD respectively.
\end{theorem}

%  \begin{wrapfigure}[28]{r}{0.4\textwidth}
% \vspace{-0.7cm}
%  \centering \includegraphics[width = 0.9\linewidth]{figures/training_loss_gaussian_small.pdf}
%       \includegraphics[width = 0.95\linewidth]{figures/testing_error_signal_small.pdf}
%      \caption{\textbf{Concentration of SGD on training loss and expected risk}, $\EE[(\aa \cdot \xx-b)^2]$. Same set-up as in Fig.~\ref{fig:MSE_Gaussian}; $\aa$ generated from some covariance as $\AA$. More volatility in the expected risk across runs even for large $n$ in comparison to the training error. }
%     \label{fig:Gaussian}
% \end{wrapfigure}

\noindent The processes $\{\xx_k:k\}$ and $(\XX_t : t)$ are independent of each other, conditionally on $\AA,\bb,\xx_0,$ and so this statement is also implicitly a concentration of measure result.  More to the point, using Theorem~\ref{thm:trainrisk} (see also \citep[Theorem 2]{Us}), we can further compare $\mathcal{R}(\XX_t)$ to the deterministic path $\Omega_t$:

\begin{theorem}[Concentration of SGD]\label{thm:expectation} Under the same assumptions as Theorem~\ref{thm:homogenized_SGD_SGD} and under the further assumption that $\mathcal{R}$ and $\mathscr{L}$ have bounded $\|\mathcal{R}\|_{H^2}$ and $\|\mathscr{L}\|_{H^2}$ independent of $n$ or $d$, for some $C'$ sufficiently large
\[ \Pr \bigg [  
\sup_{0 \leq t \leq T}\biggl\|
\begin{pmatrix} \mathscr{L}(\xx_{\lfloor tn \rfloor}) \\ \mathcal{R}(\xx_{\lfloor tn \rfloor})\end{pmatrix}
-
\begin{pmatrix}\Psi_t \\ \Omega_t \end{pmatrix}
\biggr\|
 > d^{-\tilde{\varepsilon}/2} \, \big | \, \AA, \bb, \xx_0 \bigg ] \le C' d^{-D} , \]
where $\Omega_t$ solves \eqref{eqa:PLoss}. %$\EE_0[\cdot]$ is the expectation conditioned on the problem and initialization (i.e., matrix $\AA$, targets $\bb$, and initialization $\xx_0$). 
\end{theorem}

\begin{figure}[t]
    \centering
    \includegraphics[scale = 0.18]{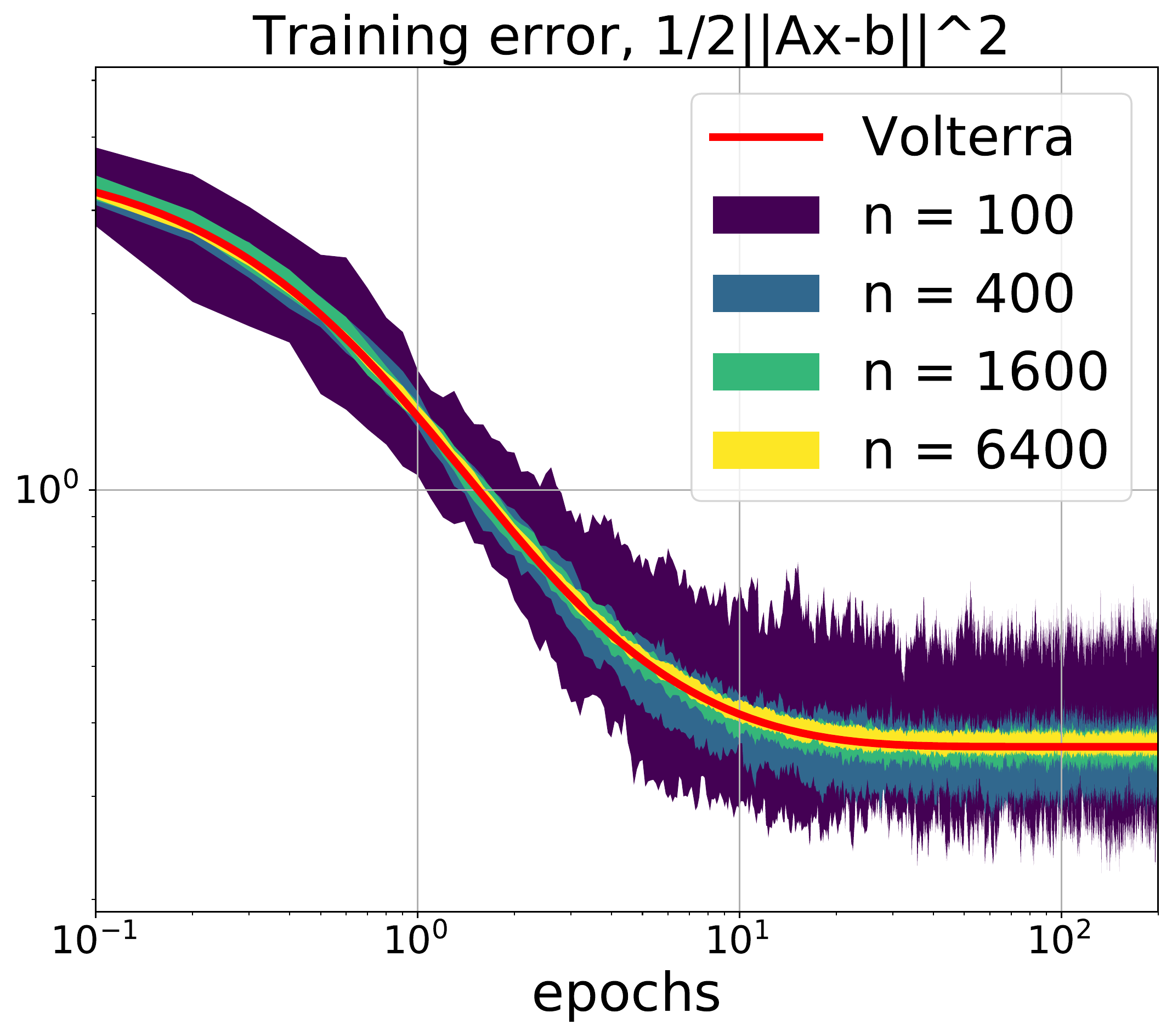}
       \quad \includegraphics[scale = 0.18]{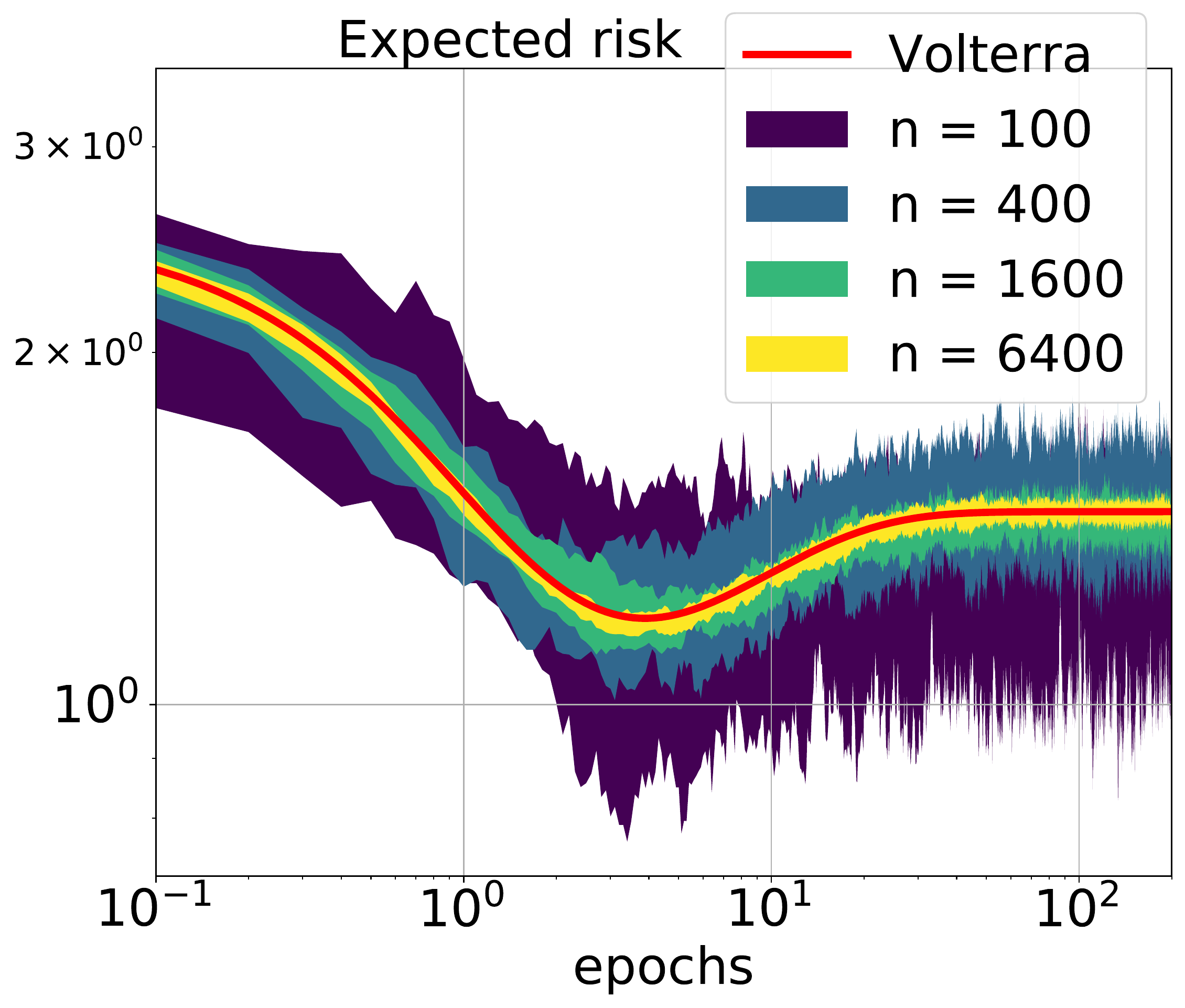}
       \quad \includegraphics[scale = 0.18]{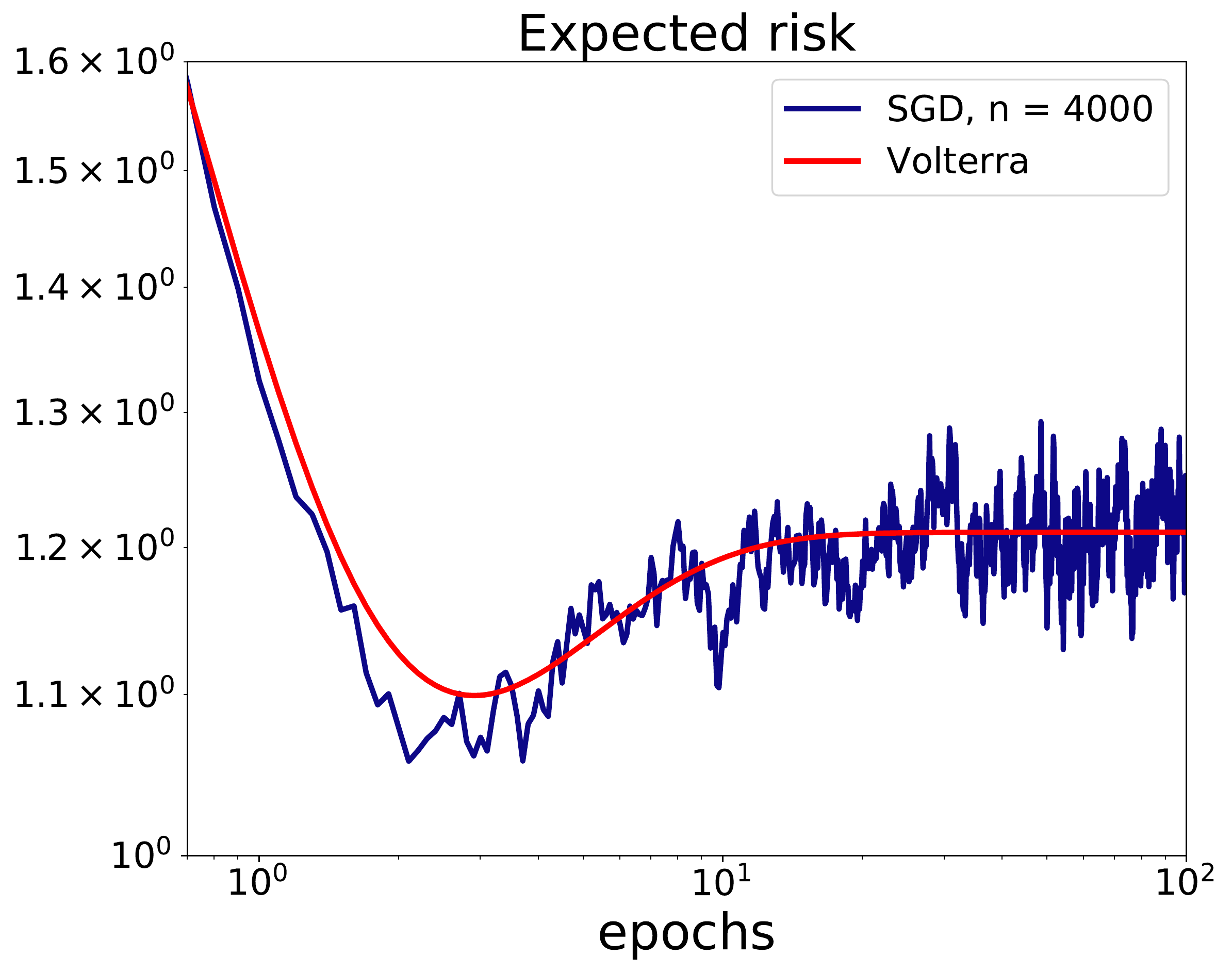}
     \caption{\textbf{Concentration of SGD on training loss and expected risk}, $1/2 \EE[(\aa \cdot \xx-b)^2]$. Set-up as in Fig.~\ref{fig:MSE_Gaussian}; For expected risk, the samples $\aa$ generated from same covariance as $\AA$. More volatility in the expected risk across runs even for large $n$ in comparison to the training error (left and center). The predicted $\Omega_t$ matches the performance of SGD on the expected risk even for a single run (right). }
    \label{fig:Gaussian}
\end{figure}

\subsection{Motivating applications}\label{sec:stats}
We give some motivating problem setups that illustrate the versatility of our setup as well as some common statistics. Of particular note is that our setup admits any quadratic test error. 

\subsubsection{Training loss} \label{sec:training_loss_intro}

One important (nonstatistical) quadratic statistic, which allows analysis of the optimization aspects of SGD in high dimensions, is the $\ell^2$-regularized loss function $f$ in \eqref{eq:rr}.  
%In that case, we take 
%\[
%\TT = \AA^T \AA + \delta \II_d,
%\quad
%\uu = -\AA^T \bb,
%\quad\text{and}\quad
%c=\tfrac{1}{2}\bb^T \bb.
%\]
%Using these choices, we have
%\[
%\tilde{q}(\xx_k) = \frac12 \| \AA \xx_k - %\bb\|^2 + \frac \delta2 \| \xx_k \|^2 = f(\xx_k).
%\]
Then provided that $\AA,\bb$ satisfy Assumptions \ref{assumption:Target} and \ref{ass: laundry_list}, $\xx_0$ is iid subgaussian, Lemmas \ref{lem:xo} and  \ref{lem:qs} and Theorem~\ref{thm:expectation} show that $f(\xx_k)$ concentrates around the solution of a Volterra integral equation.
A natural setup under which Assumptions \ref{assumption:Target} and \ref{ass: laundry_list} are satisfied is the following:

\begin{assumption}\label{ass:sc}
Suppose $M > 0$ is a constant.
Suppose that $\SSigma$ is a positive semi-definite $d \times d$ matrix with $\tr \SSigma = 1$ and $\|\SSigma\| \leq M / \sqrt{d} < \infty.$
Suppose that $\AA$ is a random matrix $\AA = \ZZ \sqrt{\SSigma}$ where $\ZZ$ is an $n \times d$ matrix of independent, mean $0$, variance $1$ entries with subgaussian norm at most $M < \infty$, and suppose $n \leq M d$.  Finally suppose that $\bb = \AA \bbeta + \xxi$ for $\bbeta,\xxi$ iid centered subgaussian satisfying $\|\bbeta\|^2 =R$ and $\|\xxi\|^2 = \widetilde{R} \frac{n}{d}$.
\end{assumption}
\noindent These assumptions naturally lead to random matrices that satisfy Assumption \ref{ass:sc} with good probability:
\begin{lemma}\label{lem:sc}
    If $(\AA,\bb)$ satisfy Assumption \ref{ass:sc}, then $(\AA,\bb)$ satisfies Assumptions \ref{assumption:Target} and \ref{ass: laundry_list} with probability tending to $1 - e^{-\Omega(d)}$.
\end{lemma}
\noindent Hence, under these assumptions, we conclude:
\begin{theorem}\label{thm:Ab}
Suppose $(\AA, \bb)$ satisfy Assumption \ref{ass:sc}, $\delta >0$ and $\xx_0$ is iid centered subgaussian with $\Exp \|\xx_0\|^2 = \widehat{R}.$ Suppose that $\gamma(\cdot)$ satisfies Assumption \ref{assumption:lr}, then for some $\epsilon > 0$, for all $T >0$, and for all $D>0$ there is a $C>0$ such that
\[
\Pr
\biggl(
\sup_{0 \leq t \leq T}
\left\| 
\begin{pmatrix}
\mathscr{L}(\xx_{\lfloor tn \rfloor}) \\
\tfrac{1}{2} \|\xx_{\lfloor tn \rfloor} - \bbeta \|^2
\end{pmatrix} - 
\begin{pmatrix}
\Psi_t  \\ 
\Omega_t
\end{pmatrix}
\right\| > d^{-\epsilon}
\biggr)
\leq Cd^{-D},
\]
where $\Psi_t$ solves \eqref{eqa:V} and $\Omega_t$ solves \eqref{eqa:PLoss} with $\mathcal{R} = \frac{1}{2}\|\cdot - \bbeta\|^2$.
\end{theorem}
\noindent We discuss generalization implications in the the next section. Other works (see e.g., \citep{bordelon2022learning,ziyin2022strength}) also have loss dynamics of SGD but in the streaming or one-pass setting. This result, on the other hand, holds for the multi-pass SGD setting. 

Theorem \ref{thm:Ab} generalizes \citep{paquetteSGD2021} in that it allows for varying training rates, adds a regularization parameter, and allows for non--orthogonally--invariant designs $\AA$.  We further note that under the assumptions of Theorem \ref{thm:Ab}, we can further approximate the behavior of gradient flow to show that
\begin{equation}\label{eq:fullPsi_training}
\begin{aligned}
\Psi_t &=  \mathscr{L}(\bm{\mathscr{X}}_{\Gamma(t)}^{\text{gf}}) + \frac{1}{n} \int_0^t \gamma^2(s) \tr \bigg ( (\AA^T \AA)^2 e^{-2(\AA^T \AA + \delta \II_d)(\Gamma(t)-\Gamma(s))} \bigg ) \Psi_s \, \dif s\\
   \text{where} \qquad  \mathscr{L}(\bm{\mathscr{X}}_{\Gamma(t)}^{\text{gf}})
    &=
    \frac{R}{2d} \tr \bigg [ (\AA^T\AA) \bigg ( \AA^T \AA (\AA^T \AA + \delta \II_d)^{-1} \big (\II_d-e^{-(\AA^T\AA + \delta \II_d)\Gamma(t)} \big )- \II_d \bigg )^2 \bigg ] \\
    &+
   \frac{ \widetilde{R}}{2d} \tr \bigg [ \bigg ( \AA(\AA^T \AA + \delta \II_d)^{-1}\big [ \II_d - e^{-(\AA^T \AA + \delta \II_d)\Gamma(t)} \big ]\AA^T - \II_n \bigg )^2 \bigg ] \\
    &+
    \frac{\widehat{R}}{2d} \tr \big ( \AA^T \AA e^{-2 (\AA^T \AA + \delta \II_d) \Gamma(t)} \big ).
\end{aligned}
\end{equation}
For the risk $\mathcal{R}(\cdot) = 1/2 \|\cdot - \bbeta\|^2$, we have following expression
\begin{equation}\label{eq:fullOmega_training}
\begin{aligned}
    \Omega_t &= \mathcal{R}(\bm{\mathscr{X}}_{\Gamma(t)}^{\text{gf}}) + \frac{1}{n} \int_0^t \gamma^2(s) \tr \bigg ( (\AA^T \AA) e^{-2(\AA^T \AA + \delta \II_d)(\Gamma(t)-\Gamma(s))} \bigg ) \Psi_s \, \dif s\\
    \text{where} \qquad \mathcal{R}(\bm{\mathscr{X}}_{\Gamma(t)}^{\text{gf}}) &= 
    \frac{R}{2d} \tr \bigg [ \bigg ( \AA^T \AA (\AA^T \AA + \delta \II_d)^{-1} \big (\II_d-e^{-(\AA^T\AA + \delta \II_d)\Gamma(t)} \big )- \II_d \bigg )^2 \bigg ] \\
    &+
   \frac{ \widetilde{R}}{2d} \tr \bigg [ \bigg ( (\AA^T \AA + \delta \II_d)^{-1}\big [ \II_d - e^{-(\AA^T \AA + \delta \II_d)\Gamma(t)} \big ]\AA^T \bigg )^2 \bigg ] \\
    &+
    \frac{\widehat{R}}{2d} \tr \big ( e^{-2 (\AA^T \AA + \delta \II_d) \Gamma(t)} \big ).
\end{aligned}
\end{equation}
Under the learning rate assumptions in Theorem~\ref{thm:eventualrisk}, the limiting gradient flow terms simplify
\begin{equation}
    \begin{aligned}
    \mathscr{L}(\bm{\mathscr{X}}_{\infty}^{\text{gf}})
    &=
    \frac{R}{2d} \tr \bigg [ (\AA^T\AA) \bigg ( \AA^T \AA (\AA^T \AA + \delta \II_d)^{-1} - \II_d \bigg )^2 \bigg ] +
   \frac{ \widetilde{R}}{2d} \tr \bigg [ \bigg ( \AA(\AA^T \AA + \delta \II_d)^{-1}\AA^T - \II_n \bigg )^2 \bigg ] \\
    \mathcal{R}(\bm{\mathscr{X}}_{\infty}^{\text{gf}}) 
    &= 
    \frac{R}{2d} \tr \bigg [ \bigg ( \AA^T \AA (\AA^T \AA + \delta \II_d)^{-1} - \II_d \bigg )^2 \bigg ] +
   \frac{ \widetilde{R}}{2d} \tr \bigg [ \bigg ( (\AA^T \AA + \delta \II_d)^{-1} \AA^T \bigg )^2 \bigg ].
    \end{aligned}
\end{equation}

\subsubsection{Excess risk of SGD for ERM in linear regression} \label{sec:erm_excess_risk_intro}
In the standard linear regression setup, we suppose that $\AA$ is generated by taking $n$ independent $d$--dimensional samples from a centered distribution $\mathcal{D}_f$ which we assume to be standardized (mean $0$ and expected sample-norm-squared $1$).  We let the matrix $\SSigma_f \in \mathbb{R}^{d \times d} $ be the feature covariance of $\mathcal{D}_f$, that is
\begin{equation} \label{eq:train_covariance}
    \SSigma_f \defas \Exp [ \aa \aa^T ],\quad \text{where}\quad \aa \sim \mathcal{D}_f.
\end{equation}
Suppose there is a linear (``ground truth'' or ``signal'') function $\beta : \mathbb{R}^d \to \mathbb{R}$, which for simplicity we suppose to have $\beta(0) = 0$.  In this case, we identify $\beta$ with a vector using the representation $\aa \mapsto \bbeta^T\aa$.  We suppose that our data is drawn from a distribution $\mathcal{D}$ on $\mathbb{R}^d \times \mathbb{R}$, with the property that
\[
\Exp[ \, b \, | \,  \aa \, ] = \bbeta^T \aa, \quad \text{where}\quad (\aa, b) \sim \mathcal{D},
\]
and the data $\aa \sim \mathcal{D}_f$.

Hence we suppose that $[\AA ~|~ \bb]$ is a $\R^{n \times d} \times \R^{n \times 1}$ matrix on independent samples from $\mathcal{D}$.  The vector $\xx_t$ represents an estimate of $\bbeta$, and the population risk is
\[
%\tilde{q}(\xx_t) 
\mathcal{R}(\xx_t)
\defas \frac{1}{2} \Exp [ (b - \xx_t^T \aa)^2 | \xx_t] \quad \text{where}\quad (\aa, b) \sim \mathcal{D},
\]
where $(\aa,b)$ is an sample independent of $\xx_t$.  This can be evaluated in terms of the feature covariance matrix $\SSigma_f$ and the noise $\eta^2 \defas \Exp[ \, (b-\bbeta^T \aa)^2 \, ]$ to give
\begin{equation}\label{eq:lrq}
%\tilde{q}(\xx_t) = 
\mathcal{R}(\xx_t)
=\frac{1}{2} \eta^2 + \frac{1}{2} (\bbeta-\xx_t)^T \SSigma_f (\bbeta-\xx_t).
\end{equation}
It is important to note that the sequence $\{\xx_t\}_{ t \ge 0}$ is generated from the iterates of SGD applied to the $\ell^2$-regularized least-squares problem \eqref{eq:rr}.

In the case that $(\aa, b)$ is jointly Gaussian, it follows that we may represent
\[
\aa = \SSigma_f^{1/2}\zz, \quad b = \bbeta^T \aa + \eta w,
\quad\text{where}\quad (\zz,w) \sim N(0, \II_d \oplus 1).
\]
Therefore, it follows that the iterates $\xx_t$ are generated from the SGD algorithm applied to the problem:
\[
\min_{\xx} \frac{1}{2} \| \AA \xx - \bb\|^2_2 + \frac{\delta}{2} \|\xx\|^2
\quad\text{where}\quad
\bb = \AA \bbeta + \eta \ww,
\]
and the vector $\ww$ is iid $N(0,1)$ random variables, independent of $\AA$.  This is also known as the generative model with noise.

Moreover, if $\mathcal{D}$ satisfies Assumption \ref{ass:sc} (with $\SSigma=\SSigma_f$) then the population risk $\mathcal{R}(\xx_k)$ is well approximated by $\Omega$:
\begin{theorem}\label{thm:AbR}
Suppose $(\AA, \bb)$ satisfy Assumption \ref{ass:sc}, $\delta >0$ and $\xx_0$ is iid centered subgaussian with $\Exp \|\xx_0\|^2 = \widehat{R}.$ Suppose that $\gamma(\cdot)$ satisfies Assumption \ref{assumption:lr}, then for some $\epsilon > 0$, for all $T >0$, and for all $D>0$ there is a $C>0$ such that
\[
\Pr
\biggl(
\sup_{0 \leq t \leq T}
\left\| 
\begin{pmatrix}
\mathscr{L}(\xx_{\lfloor tn \rfloor}) \\
\mathcal{R}(\xx_{\lfloor tn \rfloor}) 
\end{pmatrix} - 
\begin{pmatrix}
\Psi_t  \\ 
\Omega_t
\end{pmatrix}
\right\| > d^{-\epsilon}
\biggr)
\leq Cd^{-D},
\]
where $\Psi_t$ solves \eqref{eqa:V} and $\Omega_t$ solves \eqref{eqa:PLoss} with $\mathcal{R}$ given by \eqref{eq:lrq}.
\end{theorem}

\noindent We remark that under Assumption~\ref{ass:sc} (and in-distribution) that $\eta^2 =  \tfrac{\widetilde{R}}{d}$. In the case of out-of-distribition regression (see section below), we have that $\eta^2 \neq \frac{\widetilde{R}}{d}$ as the $\eta$ represents the population noise.  \\

\noindent The loss function $\mathcal{L}$ evaluated at gradient flow is the same as in \eqref{eq:fullOmega_training} as is the limiting loss $\Omega_{\infty}$. For the test risk $\mathcal{R}$ in \eqref{eq:lrq} evaluated at gradient flow, we have the following expressions for 
\begin{equation}
\begin{aligned} \label{eq:excess_risk_ERM_limit}
    \mathcal{R}(\bm{\mathscr{X}}_{\Gamma(t)}^{\text{gf}}) 
    &= \frac{R}{2d} \tr \bigg ( \SSigma_f \bigg ( \CC(t) \AA^T \AA-\II_d \bigg )^2 \bigg ) + \frac{\widetilde{R}}{2d} \tr \bigg ( \SSigma_f \AA^T\AA \CC^2(t) \bigg )\\
    &+ \frac{\widehat{R}}{2d} \tr \bigg ( \SSigma_f \exp \big (-2 (\AA^T \AA + \delta \II_d) \Gamma(t) \big ) \bigg ) + \frac{1}{2} \eta^2 \\
    \text{and} \quad \mathcal{R}(\bm{\mathscr{X}}_{\infty}^{\text{gf}})
    &= \frac{R}{2d} \tr \bigg ( \SSigma_f \big (\AA^T\AA (\AA^T\AA + \delta \II_d)^{-1} - \II_d \big )^2 \bigg ) + \frac{\widetilde{R}}{2d} \tr \bigg ( \SSigma_f \AA^T \AA \big ( \AA^T\AA + \delta \II_d \big )^{-2} \bigg )  + \frac{1}{2} \eta^2 \\
    \text{where} \quad \CC(t) 
    &\defas (\AA^T \AA + \delta \II_d)^{-1} \bigg (\II_d - \exp\big (-(\AA^T\AA + \delta \II_d)\Gamma(t) \big) \bigg ).
\end{aligned}
\end{equation}

\noindent Using Theorem \ref{thm:eventualrisk}, we conclude that in the case that $\gamma(t) \to 0$ as $t \to \infty$, the excess risk of SGD tends to $0$.  More interestingly, in the interpolation regime, $\mathscr{L}(\bm{\mathscr{X}}_\infty^\text{gf})=0$, i.e.\ the empirical risk tends to $0$.  In this case, even without taking $\gamma \to 0,$ the excess risk of SGD tends to $0$.  If on the other hand it does not tend to $0$ (i.e., $\gamma(t) \to \widetilde{\gamma}$), we arrive at the formula for excess risk of SGD over the ridge estimator risk:
\begin{equation}\label{eq:ererm}
\Omega_\infty- \mathcal{R}(\bm{\mathscr{X}}_\infty^{\text{gf}})
=
\mathcal{L}(\bm{\mathscr{X}}_{\infty}^\text{gf}) \times \frac{\widetilde{\gamma}}{2n} 
\frac{
\tr
\bigl(
(\nabla^2 \mathscr{L})
\SSigma_f
\bigl( \nabla^2 \mathscr{L} + \delta \II_d \bigr)^{-1}
\bigr)}
{
1
-
\frac{\widetilde{\gamma}}{2n}
\tr
\bigl(
(\nabla^2 \mathscr{L})^2
\bigl( \nabla^2 \mathscr{L} + \delta \II_d \bigr)^{-1}
\bigr)
}
=
\Psi_\infty
\times
\frac{\widetilde{\gamma}}{2n} 
\tr
\biggl(
\tfrac{
(\nabla^2 \mathscr{L})}
{\bigl( \nabla^2 \mathscr{L} + \delta \II_d \bigr)}
\SSigma_f
\biggr)
.
\end{equation}
\noindent We note that the right-hand-side is proportional to $\Psi_\infty$ (c.f.\ Theorem \ref{thm:eventualrisk}), and hence this excess risk due to SGD will be small if the limiting empirical risk $\Psi_\infty$ is small.  This also shows that the regularization term $\delta$ interacts with the excess risk due to SGD: if the spectrum of $\nabla^2 \mathcal{R}$ is heavy in that it has slowly decaying eigenvalues, the reduction in excess risk due to the $\ell^2$--regularization $\delta$ can be large.

\subsubsection{(Out-of-distribution) linear regression}

As before, we suppose that the data matrix $\AA$ is generated by taking $n$ independent $d$-dimensional samples from a centered distribution $\mathcal{D}_f$ with feature covariance $\SSigma_f$ (see \eqref{eq:train_covariance}). We also suppose, as in the previous in-distribution example, that there is a linear (``ground truth" or ``signal") function $\beta : \mathbb{R}^d \to \mathbb{R}$ which we identify with the vector $\bbeta \in \mathbb{R}^d$ and for which $\mathbb{E}[b | \aa] = \bbeta^T \aa$ where $(\aa, b) \sim \mathcal{D}$ and the data $\aa \sim \mathcal{D}_f$. We will generate our target $b$ from the distribution $(\aa, b) \sim \mathcal{D}$. We then let $\xx_t$ be the iterates generated by SGD applied to the optimization problem
\[
    \min_{x \in \mathbb{R}^d} \, \frac{1}{2} \|\AA \xx-\bb\|^2 + \frac{\delta}{2} \|\xx\|^2,
\quad \text{where} \quad  (\aa_i, b_i) \sim \mathcal{D}.
\]

The main distinction from the previous example is that we measure our generalization error using a different distribution than $\mathcal{D}$. Explicitly, there exists another centered distribution $\widehat{\mathcal{D}_f}$ (standardized) with covariance features matrix $\widehat{\SSigma}_f \in \mathbb{R}^{d \times d}$ from which we generate a vector $\widehat{\aa} \sim \widehat{\mathcal{D}}_f$. Moreover, we generate a test point $(\widehat{\aa}, \widehat{b})$ from a new distribution $\widehat{\mathcal{D}}$ such that 
$\EE[ \widehat{b} | \widehat{\aa} ] = \bbeta^T \widehat{\aa}$ 
with the same $\bbeta$ as before and the distribution $\widehat{\mathcal{D}}$ has $\widehat{\aa}$-marginal $\widehat{\mathcal{D}}_f$. We measure the population risk, $\mathcal{R}: \mathbb{R}^d \to \mathbb{R}$ as 
\begin{equation}
    \mathcal{R}(\xx_t) \defas \frac{1}{2} \EE[ (\widehat{b}- \xx_t^T \widehat{\aa} )^2 | \xx_t] = \frac{1}{2} \eta^2 + (\xx_t - \bbeta)^T \widehat{\SSigma}_f (\xx_t-\bbeta) \quad \text{where} \quad \eta^2 \defas \EE[(\widehat{b}-\bbeta^T \widehat{\aa})^2 ].
\end{equation}
In this setting, we can again derive the limiting excess risk, which has a similar formula for $\mathcal{R}(\bm{\mathscr{X}}_{\Gamma(t)}^{\text{gf}})$ as in \eqref{eq:excess_risk_ERM_limit} by replacing $\SSigma_f$ with $\widehat{\SSigma}_f$.

% Here $(\widehat{\aa}, \widehat{b}) \sim \widehat{\mathcal{D}}$ and $\EE[\widehat{\aa} \widehat{\aa}^T ] = \widehat{\SSigma}_f$
% \begin{align*} 
%      \EE[(\widehat{b}-\xx_t^T \widehat{\aa})^2 | \xx_t] &= \EE[ \widehat{b}^2 - 2 \xx_t^T \widehat{\aa} \widehat{b} + (\xx_t^T \widehat{\aa})^2 | \xx_t] = \EE[ \widehat{b}^2 ] - 2 \EE[\widehat{b} \widehat{\aa}^T]\xx_t + \xx_t^T \widehat{\SSigma}_f \xx_t\\
%     &= \widehat{\eta}^2  -2 \EE[\bbeta^T \widehat{\aa} \widehat{\aa}^T ] \xx_t + \xx_t^T \widehat{\SSigma}_f \xx_t + \EE[\bbeta^T \widehat{\aa} \widehat{\aa}^T \beta ]\\
%     & = \widehat{\eta}^2 + (\xx_t-\bbeta)^T \widehat{\SSigma}_f (\xx_t -\bbeta)
% \end{align*}

% Here $\widehat{\eta}^2 = \EE[ (\widehat{b}-\bbeta^T \widehat{\aa})^2]$ and $\EE[\widehat{b} | \widehat{a}] = \bbeta^T \widehat{\aa}$ 

% Here we are solve $(\aa, b) \sim \mathcal{D}$ such that $\EE[b | \aa] = \bbeta^T \aa$ and the data $a \sim \mathcal{D}_f$ with covariance $\SSigma_f$. The iterates $\xx_t$ are generated from SGD applied to the problem
% \[ \min_{\xx \in \R^d} \frac{1}{2}\|\AA \xx- \bb\|^2 + \frac{\delta}{2} \|\xx\|^2\]
% where the rows of $\AA$ are $\aa_i \sim \mathcal{D}_f$ and the $(\aa_i, b_i) \sim \mathcal{D}$.

\subsubsection{Random features model of a linear ground ground truth} \label{sec:random_features_intro}
We follow a setup based upon \citep{mei2019generalization,adlam2020neural}. As before, we suppose that the data matrix $\XX$ is generated by taking $n$ independent $n_0$-dimensional samples from a centered distribution $\mathcal{D}_f$ with feature covariance 
\[\SSigma_f \defas \EE[\XX_i^T \XX_i], \qquad \text{where $\XX_i \in \mathbb{R}^{1 \times n_0}$ and $\XX_i \sim \mathcal{D}_f$.}\]
We suppose for simplicity that $\XX$ is a data matrix having dimension $n \times n_0$ whose iid rows are drawn from a multivariate Gaussian with covariance $\SSigma_f$ and nice covariance structure:
%Suppose for simplicity the covariance $\SSigma_f$ satisfies

\begin{figure}[t]
%\vspace{-1cm}
 \centering 
  \includegraphics[width=\textwidth]{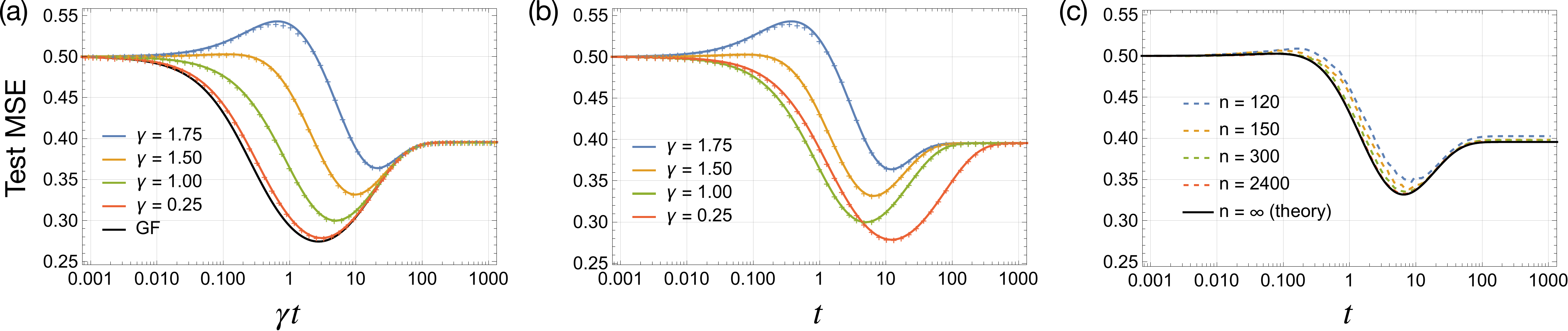}
     \caption{Time-dependent test error of the random feature model under SGD, as described by Thm.~\ref{thm:random_features}, with zero regularization ($\delta = 0$), no additive noise ($\eta = 0$), normalized ReLU activation function $\sigma(\cdot) \propto \max\{0,\cdot\}$, and aspect ratios $n/d = 3/5$ and $n_0/d = 1/2$. {\bf (a,b)} Good agreement is observed between the asymptotic theoretical predictions (solid curves) and finite-size empirical simulations with $n=4800$ (markers). {\bf(a)} As a function of the integrated learning rate $\Gamma(t) = \gamma t$, the curves for small constant learning rates $\gamma$ approach that of gradient flow (black line). {\bf (b)} As a function of time $t$, there can be a tradeoff between large $\gamma$ which enables faster progress and small $\gamma$ which generates lower noise. {\bf(c)} Empirical simulations at various finite sizes (dashed curves) for $\gamma=1.5$: for large $n$, the curves closely approximate the asymptotic predictions (solid black curve); for moderate $n$, the deviations appear most significant for intermediates times.}
    \label{fig:RF_convergence}
\end{figure}

\begin{assumption}\label{ass:RF_cov}
The distribution $\mathcal{D}_f$ is multivariate normal and
the covariance matrix $\SSigma_f$ of the random features data satisfies for some $C>0$
\[
%\overline{\tr}{(\SSigma_f)}=1
\tfrac{1}{n_0}{\tr}{(\SSigma_f)}=1
\quad\text{and}\quad
\|\SSigma_f\| \leq C.
\]
\end{assumption}
\noindent This allows $\XX$ to be represented equivalently as $\XX = \ZZ\SSigma^{1/2}/\sqrt{n_0}$ for a iid standard Gaussian matrix $\ZZ$.
We suppose that $\WW$ is an $(n_0 \times d)$ iid feature matrix having standard Gaussian entries and independent of $\ZZ$ so that $\ZZ \SSigma^{1/2} \WW/\sqrt{n_0}$ is a matrix whose rows are standardized.  

We let $\sigma$ be an activation function satisfying:
\begin{assumption}\label{ass:RF_sigma}
The activation function satisfies for $C_0,C_1 \geq 0$ 
\[
%|\sigma(x)| \leq e^{C_|x|},
%\quad\text{and}\quad
|\sigma'(x)| \leq C_0e^{C_1|x|},
\quad\text{for all}\quad x \in \R,
\quad\text{and for standard normal $Z$,}\quad
\Exp \sigma(Z) = 0.
\]
\end{assumption}
\noindent  We note that from the outset, the growth rate of the derivative of the activation function implies a similar bound on the growth rate of the underlying activation function $\sigma$.
As before, we suppose the data $[\XX ~|~ \bb]$ is arranged in the matrix $\R^n \times (\R^{n_0} \times \R)$ where each row is an independent sample from $\mathcal{D}$. We now transform the data $\XX \in \mathbb{R}^{n \times n_0}$ by putting
\[
\AA =  \sigma(\XX\WW / \sqrt{n_0} ) \in \mathbb{R}^{n \times d},
\]
where $\WW \in \mathbb{R}^{n_0 \times d}$ is a matrix independent of $[\XX~|~\bb]$ of independent standard normals.\footnote{In \cite{mei2019generalization}, the distribution of the columns are taken as independent uniform vectors on the sphere $\sqrt{d}\,\mathbb{S}^{d-1}$.  The activation function $\sigma$ is a 1-Lipschitz function from $\R \to \R$ which is applied entrywise to the underlying matrix.}
%Note that we assume that the rows of $\AA$ are normalized. 
The activation function $\sigma \, : \, \mathbb{R} \to \mathbb{R}$ is applied element-wise.% and it is assumed that $|\psi(x)|, |\psi'(x)| = \mathcal{O}(\exp(Cx))$ for some positive $C$. 
% 
%We further assume that $\psi$ is shifted so that
%\[
%\Exp \psi(Z) = 0 \quad\text{where}\quad Z \sim N(0,1).
%\]

We introduce the following notation
\begin{equation}
    \SSigma_{\sigma}(\WW) \defas \EE[ \sigma(\XX_i \WW / \sqrt{n_0} )^T \sigma(\XX_i \WW / \sqrt{n_0}) \, | \, \WW] \quad \text{and} \quad \widehat{\sigma}(\WW) \defas \EE[\XX_i^T \sigma(\XX_i \WW / \sqrt{n_0}) | \WW]
\end{equation}
The population risk, $\mathcal{R} : \mathbb{R}^d \to \mathbb{R}$ as a random variable in $\XX$ and $\WW$, is
\begin{equation}\label{eq:RF_Risk}
\begin{aligned}
\mathcal{R}(\xx_t) &\defas
\Exp [ (b - \xx_t^T \sigma(\XX_i \WW / \sqrt{n_0}))^2 | \xx_t, \WW]\\
&= \eta^2 + \Exp[(\XX_i \bbeta -  \sigma(\XX_i \WW / \sqrt{n_0}) \xx_t )^2 \, | \, \xx_t, \WW]\\
&= \eta^2 + \bbeta^T \SSigma_f \bbeta + \xx_t^T \SSigma_{\sigma}(\WW) \xx_t - 2 \bbeta^T \widehat{\sigma}(\WW) \xx_t, \quad \text{where $(\XX_i, b) \sim \mathcal{D}$ and $\Exp[ \, b \, | \, \XX_i] = \XX_i \bbeta$.}
\end{aligned}
\end{equation}
The $\ell^2$-regularized least-squares problem is now
\[
\min_{\xx} \frac{1}{2} \| \AA \xx - \bb\|^2_2 + \frac{\delta}{2} \|\xx\|^2
\quad\text{where}\quad
\bb = \XX \bbeta + \eta \ww,
\]
which is the random features regression.  This should be compared to a two--layer neural network model, in which the hidden layer has dimension $n_0$.  However, the hidden layer weights are simply generated randomly in advance and are left untrained.  The optimization is only performed on the final layers' weights ($\xx$). 

\begin{theorem} \label{thm:random_features}
Suppose that $n,d,n_0$ are proportionally related.
Suppose that the data matrix $\XX$ satisfies Assumption \ref{ass:RF_cov},
and the random features $\WW$ are iid standard normal.
Suppose $\bb = \XX \bbeta + \eta \ww$ with $\bbeta,\ww$ independent isotropic subgaussian vectors with $\Exp \|\bbeta\|^2 = 1/n_0$ and $\Exp \|\ww\|^2 = 1$ and $\eta$ bounded independent of $n$.
Suppose the activation function satisfies Assumption \ref{ass:RF_sigma}.
Suppose the initialization $\xx_0$ is iid centered subgaussian with $\Exp \|\xx_0\|^2 = \widehat{R}.$ Suppose that $\gamma(\cdot)$ satisfies Assumption \ref{assumption:lr}.
Then for some $\epsilon > 0$, for all $T >0$, and for all $D>0$ there is a $C>0$ such that
\[
\Pr
\biggl(
\sup_{0 \leq t \leq T}
\left\| 
\begin{pmatrix}
\mathscr{L}(\xx_{\lfloor tn \rfloor}) \\
\mathcal{R}(\xx_{\lfloor tn \rfloor}) 
\end{pmatrix} - 
\begin{pmatrix}
\Psi_t  \\ 
\Omega_t
\end{pmatrix}
\right\| > d^{-\epsilon}
\biggr)
\leq Cd^{-D},
\]
where $\Psi_t$ solves \eqref{eqa:V} and $\Omega_t$ solves \eqref{eqa:PLoss} with $\mathcal{R}$ given by \eqref{eq:RF_Risk}.
\end{theorem}

Finally, as in \eqref{eq:ererm} we derive the excess risk of SGD ($\gamma(t) \to \widetilde{\gamma}$) over ridge regression:
\begin{equation}\label{eq:RF_er}
\Omega_\infty- \mathcal{R}(\bm{\mathscr{X}}_\infty^{\text{gf}})
= \mathcal{L}(\bm{\mathscr{X}}_{\infty}^{\text{gf}}) \times
\frac{\widetilde{\gamma}}{2n} 
\frac{
\tr
\bigl(
(\nabla^2 \mathscr{L})
\SSigma_{\sigma}(\WW)
\bigl( \nabla^2 \mathscr{L} + \delta \II_d \bigr)^{-1}
\bigr)}
{
1
-
\tfrac{\widetilde{\gamma}}{2n}
\tr
\bigl(
(\nabla^2 \mathscr{L})^2
\bigl( \nabla^2 \mathscr{L} + \delta \II_d \bigr)^{-1}
\bigr)
}
=
\Psi_\infty
\times
\frac{\widetilde{\gamma}}{2n} 
\tr
\biggl(
\frac{
(\nabla^2 \mathscr{L})}
{\bigl( \nabla^2 \mathscr{L} + \delta \II_d \bigr)}
(\SSigma_{\sigma}(\WW))
\biggr)
.
\end{equation}
We again recall that $\Psi_\infty$ is the limiting empirical risk of SGD.
The risk of the ridge regression estimator $\mathcal{R}(\bm{\mathscr{X}}_\infty^{\text{gf}})$
was first given by~\citet{mei2019generalization}.

\subsection{ Related work } 

We highlight recent progress of research in the high-dimensional setting, specifically random features, as well as recent advances in analyzing the excess risk of SGD. As these areas are highly active, we present below a non-exhaustive list of the current progress. 

%\paragraph{SGD and excess risk}
In the literature, convergence guarantees and risk bounds are available for analyzing SGD and its variants \citep{shalev2016accelerated,needell2016stochastic,moulines2011nonasymptotic,defossez2015averaged,varre2021last}. A popular paradigm for analyzing risk bounds of SGD is the one-pass or streaming setting where one supposes that the gradient estimators are independent with a common distribution \citep{gurbuzbalaban2020heavy-tail,jain2018accelerating}. Such a setting was considered in a series of works \citep{jain2018accelerating,dieuleveut2017harder} which explored `one-pass' SGD on a least-squares under a design condition on the data matrix. Extending this idea, \citep{zou2021benefits} provide upper and lower excess risk bounds for constant stepsize SGD on the $\ell^2$-regularized least-squares problem which extends the work of \citep{bartlett2020benign,tsigler2020benign}. These bounds are characterized by the full eigenspectrum of the population covariance matrix. Beyond the confines of the streaming setting, much less is known about the risk bounds for multi-pass SGD \citep{lin2017optimal, pillaud2018statistical, lei2021generalization, zou2022risk}. Like in this work, \citep{lin2017optimal,pillaud2018statistical,zou2022risk} consider the simplified setting of analyzing the behavior of multi-pass SGD on a (high-dimensional) $\ell^2$-regularized least-squares problem. In contrast to our exact dynamics of the excess risk, previous works provide only bounds. Other approaches to analyzing generalization performance of multi-pass SGD include uniform stability (see c.f. \citep{hardt2016train}). Excess risk and uniform stability are related through a (loose) triangle inequality; we preferred to focus on the excess risk in this paper and leave discussions of uniform stability for future directions. 

% \paragraph{SGD and implicit regularization.} 
The exact features of SGD that are responsible for the success of SGD on high-dimensional problems are the subject of extensive research. These features are often labelled under the umbrella term the \textit{implicit regularization effects of SGD}. The implicit regularization mechanism has been primarily hypothesized to exist in highly nonconvex settings. In such a setting, empirical observations have led to the conclusion that the noise generated by small batch \citep{keskar2016on, hoffer2017train} and/or large learning rate \citep{lewkowycz2020large} SGD leads to better generalization performance. A proposed mechanism for this improvement is that the stochasticity inherent in SGD allows the optimizer to escape traps (c.f. \citep{bouchaud1990anomalous} or \citep{zhu2018anisotropic}) which poorly generalize and are stationarity points for gradient flow. A related point of view on the proposed escape mechanism is that SGD demonstrates a preference for flat minima, long considered to be a preferable solution for generalization properties \citep{hochreiter1997flat}. For the least-squares problem, multi-pass SGD converges to the minimial norm solution \citep{gunasekar2018characterizing,neyshabur2014in,zhang2016accelerated} which is widely cited as the implicit bias of SGD \citep{kobak2020optimal,derezinksi2020exact}.

% \paragraph{SGD and stochastic differential equations.} 
In this work, we analyze a new mathematical tool, homogenized SGD, introduced simultaneously in \citep{paquette2021dynamics,mori2021logarithmic}, and we show that this SDE is the high-dimensional equivalence of SGD. This technique of using SDEs to analyze SGD is not new (see, for example, stochastic modified equation (SME) \citep{li2017stochastic,mandt2016variational} and Langevin dynamics \citep{cheng2018sharp} and other SDE formulations and intepretations  \citep{ jastrzkebski2017three,Kushner,ljung1977,barrett2021implicit}). A common construction of previously studied SDEs (such as the SME) require that the learning rate $\gamma \to 0$ which is not common in practical application. In contrast, our homogenized SGD allows for bounded step sizes and it is easily analyzable. Moreover in \citep{yaida2018fluctuationdissipation} the author showed there is no small learning rate limit of SGD that produces nontrival stochastic behavior. While this is true for a fixed dimension, in the high-dimensional setting, there is an equivalence between SGD and a nontrival stochastic SDE (we call homogenized SGD, Section~\ref{sec:homogenized_SGD_intro_1}). 

% \paragraph{High-dimensionality.} 
Our framework is inspired by the the phenomenology of random matrix theory and high dimensionality. In the context of linear overparameterized models, a large body of works \citep{hastie2019surprises,hsu2012random,dobriban2018high,wu2020on,xu2019on} analyze the excess risk (test error) of ridge regression at the optimum in the asymptotic regime where both sample size $n$ and dimension $d$ go to infinity as $d/n \to r \in (0,\infty)$. Different structures on the data matrix $\AA$ have been considered (for general covariance structures on the data matrix and a general non-isotropic source condition on the parameters \citep{wu2020on,richards2021asymptotics} which were built on early observations analyzing the minimum-norm interpolated least squares and ridge regression in the random design setting \citep{belkin2019reconciling,dobriban2018high,hastie2019surprises}). Beyond the least squares setting, exact high-dimensional asymptotics of first order methods on random data exist for among objectives such as logistic regression  \citep{celentano2021high,mignacco2020dynamical}) and rank-one matrix completion \citep{bodin2021rank}, and a general lower bound on the high-dimensional generalization performance of 2-layer networks \citep{seroussi2021lower}.

% \paragraph{Random features.} 
A central example where our results hold is the random feature setting, introduced by \citep{Rahimi2008Random} for scaling kernel machines. Random features models provide a rich but tractable class of models to gain further insights into generalization phenomena \citep{mei2019generalization,liao2020Random,adlam2020understanding,adlam2020neural,tripuraneni2021covariate}. These models are particularly of interest because of their connection to neural networks where the number of random features corresponds to model complexity \citep{jacot2018neural,neal1996priors,lee2018deep} and because of its use as a practical method for data analysis \citep{Rahimi2008Random,shankar2020neural}. 

From a technical perspective, our analysis of random features requires tools and recent results from random matrix theory. A central challenge in analyzing random features stems from the fact that matrices of the form $\sigma(\XX \WW)$ have nonlinear dependencies between elements. Known results about these matrices such as spectral information \citep{peche2019note,benigni2019eigenvalue,pennington2017nonlinear,louart2018random} and linearization simplification \citep{adlam2020neural}. In this paper, we verify that random features model satisfies the resolvent conditions in Assumption~\ref{ass: laundry_list} by using techniques similar to sample covariance matrices (see \citep{bai2010spectral,couillet2022random} for an introduction to sample covariance matrices).  

The present paper builds on the earlier works of \citep{paquetteSGD2021,paquette2021dynamics} and aims to establish a unifying framework for analyzing different statistics of SGD on the high-dimensional $\ell^2$-regularized least-squares regression problem.  In \citep{paquetteSGD2021}, the authors derived a convolution-type Volterra equation that gave the exact dynamics of SGD for the least-squares loss in the large-scale asymptotic limit $(d/n \to \infty)$. Such a result required a left orthogonal invariance condition on the data matrix $\AA$. Building on this work (under left orthogonal invariance), the authors then  showed that there was an SDE, called homogenized SGD, whose dynamics on least-squares loss matched those of SGD \citep{paquette2021dynamics}. Under the assumption that homogenized SGD held, they analyzed the dynamics of a class of stochastic momentum algorithms.

In this work, we give a complete picture of the relationship between homogenized SGD and SGD in the high-dimensional setting on a $\ell^2$-regularized least-squares problem. We focus on the mathematical aspects, that is, we rigorously show that homogenized SGD is the high-dimensional equivalence of SGD under any quadratic statistics (see Assumption~\ref{assumption: quadratics}). We  reduces the data matrix assumptions to a simply family of resolvent conditions, which can roughly be viewed as a weak (non-quantitative) form of delocalization of sample-side singular vectors of the data. In this way, we can incorporate a wider variety of models often used in machine learning such as random features and sample covariance matrices. 

%Add this when we have the NeurIPS paper out
% In this series of works, we give a complete picture of the relationship between homogenized SGD and SGD in the high-dimensional setting on a $\ell^2$-regularized least-squares problem; both mathematically (this paper) and its statistical implications \citep{Us}. As the first in this series, we focus on the mathematical aspects, that is, we rigorously show that homogenized SGD is the high-dimensional equivalence of SGD under any quadratic statistics (see Assumption~\ref{assumption: quadratics}). We  reduces the data matrix assumptions to a simply family of resolvent conditions, which can roughly be viewed as a weak (non-quantitative) form of delocalization of sample-side singular vectors of the data. In this way, we can incorporate a wider variety of models often used in machine learning such as random features and sample covariance matrices. In the companion paper, \citep{Us}, we delve into the statistical interpretation of this high-dimensional equivalence by looking at the excess risk under homogenized SGD. In particular, we give conclusions about implicit regularization effects of multi-pass SGD and contrast the excess risks of streaming and multi-pass SGD.    

\paragraph{Organization.} The remainder of the article is structured as follows: in Section~\ref{sec:resolvents} we give a summary properties of the resolvent and prove simple cases where Assumptions~\ref{assumption:init}, \ref{assumption: quadratics}, and \ref{ass:sc} hold (specifically the proofs of Lemmas~\ref{lem:xo}, \ref{lem:qs}, and \ref{lem:sc}). Section~\ref{sec:dynamics} introduces homogenized SGD and proves some properties of it. In particular, the concentration of homogenized SGD around its mean, Theorem~\ref{thm:trainrisk}, is discussed in Section~\ref{sec:concentration_hsgd}. Our main results are then described and proved in Section~\ref{sec:comp_SGD_HSGD}. For instance, the technical argument for the high-dimensional equivalence to SGD, Theorem~\ref{thm:homogenized_SGD_SGD}, and the concentration of statistics of SGD to a deterministic function, Theorem~\ref{thm:expectation}, are in this section. Section~\ref{sec:martingale_errors} details bounds on the martingale error terms that arise in proving Theorem~\ref{thm:homogenized_SGD_SGD}, the high-dimensional equivalence of SGD. We highlight how the assumptions hold for the random features setting in Section~\ref{sec:random_features}. 
%The article concludes on showing some numerical simulations in Section~\ref{sec:numerical_simulations}.

\section{Resolvents} \label{sec:resolvents}

We have formulated many of our assumptions in terms of resolvents.   Resolvents offer many advantages, especially for the analysis of random matrices.  In this section, we discuss a few properties of resolvents generally, and then we prove the lemmas in the body of the text that show simple setups under which these assumptions are satisfied (and which are essentially resolvent exercises) -- Lemmas \ref{lem:xo}, \ref{lem:qs}, \ref{lem:sc} which are established in Section \ref{sec:rlemmas}.

Recalling the definition of the resolvent
$R(z; \MM)$ of a matrix $\MM \in \mathbb{R}^{d \times d}$ is
\begin{equation}\label{eqr:def}
  R(z; \MM) =  (z\II_d-\MM)^{-1} \quad \text{for $z \in \mathbb{C}$.}
\end{equation}
The resolvent encodes the spectral properties of the matrix $\MM$ as an analytic (matrix--valued) function.  For symmetric matrices $\MM$, the resolvent has poles at each eigenvalue of the matrix $\MM$ on the real line, but is analytic in $\mathbb{C} \setminus\{ \lambda(\MM)\}$.  Moreover, if $\MM = \UU\operatorname{diag}(\lambda_1,\lambda_2,\dots,\lambda_d) \UU^T$ is a diagonalization of $\MM$, then we have a representation of the resolvent as 
\begin{equation}\label{eqr:diag}
R(z; \MM) = \UU \operatorname{diag}(z-\lambda_1,z-\lambda_2,\dots,z-\lambda_d)^{-1} \UU^T.
\end{equation}

This in particular allows for the representation of analytic functions of $\MM$ by contour integration in the complex plane.  If we suppose that $\Omega$ is any simple contour that encloses the eigenvalues of $\MM$, we have the representation for analytic functions $\varphi$
\begin{equation}\label{eqr:int}
\varphi(\MM) = \frac{1}{2\pi i}\oint_\Omega  \varphi(z) R(z; \MM)\dif z.
\end{equation}
In particular when $\varphi$ is an entire function, and so $\varphi(\MM)$ can be represented by a convergent power series of $\MM$, this gives an alternative representation.  However \eqref{eqr:int} is more powerful and correctly generalizes to the application of analytic functions which are analytic in the interior of $\Omega$.  We shall be especially interested in $\varphi$ such as those that appear in \eqref{eq:GFS}.

%Besides its ability to represent analytic functions of the matrix argument, the resolvent also enjoys \emph{a priori} %stability properties and algebraic identities.  
Resolvents also enjoy some \emph{a priori} estimates which are convenient for probabilistic analyses.
\begin{lemma}\label{lem:res}
    Under Assumptions \ref{assumption:Target} and \ref{ass: laundry_list}, the resolvent satisfies the following estimates. \begin{enumerate}
        \item There is a constant $C=C(\Omega)>0$ such that
    \(
    \max_{z \in \Omega} \|R(z; \MM)\| \leq C.
    \)
    \item For the same constant $C$,
    \(
    \max_{z,y \in \Omega} \|\tfrac{R(z; \MM)-R(y; \MM)}{z-y}\| \leq C^2.
    \)
    \end{enumerate}
\end{lemma}
\begin{proof}
For the first point, using the diagonalization \eqref{eqr:diag}, the operator norm of the resolvent $R(z; \MM)$ is precisely given by the maximum of $|z-\lambda_j|^{-1}$ over all $z\in \Omega$ and all eigenvalues $\lambda_j$ of $\AA^T\AA$.  As we have assumed that $z$ and $\lambda_j$ are separated by some absolute constant (depending on $\Omega'$), we conclude that the first point holds.  For the second point, the difference of resolvents satisfies the identity
\[
\frac{R(z; \MM)-R(y; \MM)}{z-y}
=
-R(z; \MM)R(y; \MM),
\]
which can be verified by multiplying both sides by $(y-\MM)$ and $(z-\MM)$.  Thus in fact, we may bound the operator norm of this expression by the square of the maximum of the operator norms of the resolvents, as claimed.
\end{proof}

One particular identity which simplifies the stochastic analysis of the resolvent is the \emph{Ward identity}, which states
\begin{equation}\label{eq:Ward}
\sum_{j}
|R(z; \MM)_{ij}|^2
= (\Im z)^{-1} \Im R(z; \MM)_{ii}
\quad\text{for all}\quad i \in \{1,2,\dots,n\}.
\end{equation}
This can be verified just by computing the imaginary part of the diagonal resolvent entry.

In part because of this equation, it is simpler to work in domains which are bounded away from the real axis.  More generally, it is convenient to have some tools for comparing resolvent expressions in different parts of the plane.
\begin{lemma}\label{lem:fb}
Suppose that for some complex numbers $\{w_j\}$ and all $z \in \mathbb{C}\setminus \{ \lambda_1, \lambda_2, \dots, \lambda_n\},$
\[
f(z) = \sum_{j=1}^n \frac{w_j}{z-\lambda_j}
\quad\text{satisfies}\quad |f^{(k)}(z)| \leq  \|R(z; \AA\AA^T)\|^{-k-1}_{op} \leq 
\bigl(\min_i |z-\lambda_i|\bigr)^{-k-1}
\quad
k\in\{0,1\}.
\]
Suppose Assumption \ref{assumption:Target} holds and suppose that $\Omega$ is as in Assumption \ref{ass: laundry_list}.  
\begin{enumerate}
    \item 
    For any $\varepsilon,\theta \in (0,\tfrac 12)$ and all $\delta > 0$ sufficiently small, there is a finite set $D \subset \mathbb{C}$ having $\Im z \geq \delta$, $|D| \leq C(\delta)(1+ \|\AA \AA^T\|) \times n^{1/2-\theta}$ and $D$ contained in a $\mathcal{O}(\sqrt{\delta})$ neighborhood of $\Omega$ such that
    \[
    \begin{aligned}
    &\max_{z \in D}
    |f(z)|
    \leq n^{\theta-1/2}
    \implies
    \max_{z \in \Omega}
    |f(z)|
    \leq n^{\theta+\varepsilon-1/2}
    \quad\text{and}\\
    &
    \max_{z \in \Omega}
    |f(z)|
    \leq n^{\theta-1/2}
    \implies
    \max_{z \in D}
    |f(z)|
    \leq n^{\theta+\varepsilon-1/2} \\
    \end{aligned}
    \]
    for all $n$ sufficiently large.
    \item 
     For any $\varepsilon,\theta \in (0,\tfrac 12)$ and all $\delta > 0$ sufficiently small, if $\Omega_\delta$ is the set of $z \in \mathbb{C}$ such that $\min_{y \in \Omega} |z-y| = \delta$ and $z$ is connected to $\infty$ in $\mathbb{C}\setminus \Omega$ 
    \[
    \max_{z \in \Omega_\delta}
    |f(z)|
    \leq n^{\theta-1/2}
    \implies
    \max_{z \in \Omega}
    |f(z)|
    \leq n^{\theta+\varepsilon-1/2}
    \]
    for all $n$ sufficiently large.
\end{enumerate}
\end{lemma}
\noindent In particular to verify Assumption \ref{ass: laundry_list}, Assumption \ref{assumption:init}, or Assumption \ref{assumption: quadratics} it suffices to check it on the finite set $D$.  For the final assumption, which depends on two parameters, one may take both $z,y \in D$.
\begin{proof}
\noindent \emph{Part 1.} By assumption, the curve $\Omega$ encloses the spectrum, and hence we can find two other curves $\Omega_1,\Omega_2$ so that $\Omega_1$ encloses the spectrum, $\Omega$ encloses $\Omega_1$ and $\Omega_2$ encloses $\Omega$.  Furthermore, there is some $\mu > 0$ which is the minimum distance between $\Omega_1$ and $\Omega$, $\Omega$ and $\Omega_2$ and $\Omega_1$ and the spectrum.  By the construction of $\Omega$ we may assume that $\mu$ is bounded below by an absolute constant.  Let $\mathcal{A}$ be the annular region contained in $\Omega_2$ but not contained in $\Omega_1$.

We will pick $D$ to be a $n^{1/2-\theta}$--net of the set
\[
U_\delta\defas
\Omega \cap \{ z : |\Im z| \geq \delta\}
\bigcup
\mathcal{A} 
\cap \{z : |\Im z| = \delta\} 
\cap \{z : d(z,\Omega) \leq \sqrt{\delta}\}.
\]
We then suppose that $f$ is bounded by $n^{\theta-1/2}$ uniformly over $D$.
By the bound on $f^{(k)}$ we note that on this set, we have
\[
\max_{z \in U_\delta} |f^{(1)}(z)| \leq 2 \mu^{-2}.
\]
Hence, we conclude that from the bound on $f$,
\[
\max_{z \in U_\delta} |f(z)| \leq
\max_{z \in D} |f(z)| 
+ 2\mu^{-2}n^{\theta-1/2}
\leq n^{\theta+\varepsilon/2-1/2}
\]
for all $n$ sufficiently large and all $\delta >0$.

Now we suppose that $(B_t : t \geq 0)$ is a standard complex Brownian motion, and we let $z \in \Omega$ be arbitrary.  Let $T$ be the first time that $B_t + z$ hits $U_\delta \cup \Omega_1 \cup \Omega_2$.  Note that if $z \in U_\delta$, then $T=0$.  The function $\log |f(z)|$ is harmonic in $\mathcal{A}$, and hence we have that 
\[
\begin{aligned}
\log |f(z)| 
&= \Exp \log |f(B_T+z)|
\leq 
({\theta+\varepsilon/2-1/2})\log n \times \Pr( B_T + z \in U_\delta) \\
&+
\bigl\{ \max_{z \in \Omega_1 \cup \Omega_2} \log |f(z)| \bigr\}
\times (1-\Pr( B_T + z \in U_\delta)).
\end{aligned}
\]
Setting $p_\delta(z) = \Pr( B_T + z \in U_\delta)$ and using that $f$ is bounded by $2/\mu$, we have that
\[
\log |f(z)| \leq 
({\theta+\varepsilon/2-1/2})\log n \times p_\delta(z)
+\log (2/\mu).
\]
By sending $\delta \to 0,$ we have that $p_\delta(z)$ converges to $1$ uniformly in $z \in \Omega$ as $z \in \Omega$ is separated uniformly from $\Omega_1$ and $\Omega_2$ by $\mu$, but on sending $\delta \to 0$, the set $U_\delta$ contains two horizontal line segments at distance $\mathcal{O}(\delta)$ from $z$ which extend a distance to the left and to the right a distance bounded below by $\mathcal{O}({\sqrt{\delta}})$.  Hence we can pick $\delta_0$ sufficiently small that we can make $p_\delta(z)$ satisfy
\[
|p_\delta(z)|({\theta+\varepsilon/2-1/2}) \geq ({\theta+3\varepsilon/4-1/2}).
\]
Then for all $n$ sufficiently large, we have
\[
\max_{z \in \Omega}
\log |f(z)| \leq 
({\theta+\varepsilon-1/2})\log n,
\]
which completes the proof of the implication.

For the second implication, we have that all of $U_\delta$ is within some neighborhood $\mathcal{O}(\sqrt{\delta})$ of $\Omega$.  This in effect is the same argument as the next part, and so we do not enter into the details.

\noindent \emph{Part 2.} 
Let $\Omega_\delta$ be as in the statement and let $\Omega_1$ be as in the previous part.  Let $B_t$ be standard complex Brownian motion, and let $T$ be the first time $B_t + z$ hits $\Omega_1 \cup \Omega_\delta$, for any $z \in \Omega$.  Then 
\[
\begin{aligned}
\log |f(z)| 
&= \Exp \log |f(B_T+z)|
\leq 
({\theta-1/2})\log n \times \Pr( B_T + z \in \Omega_\delta) \\
&+
\bigl\{ \max_{z \in \Omega_1} \log |f(z)| \bigr\}
\times (1-\Pr( B_T + z \in \Omega_\delta)).
\end{aligned}.
\]
Once more, on sending $\delta \to 0$ the probability that Brownian motion hits $\Omega_\delta$ tends to $1$.  
\end{proof}

\subsection{Resolvent lemmas}\label{sec:rlemmas}
We now turn to the proof of the Lemmas~\ref{lem:xo}, \ref{lem:qs}, and \ref{lem:sc}. 
\begin{proof}[Proof of Lemma \ref{lem:xo}]
We recall for convenience the statement to be proven.
We suppose that Assumption \ref{ass: laundry_list} holds with some $\theta_0 \in (0,\tfrac12)$ and that $\xx_0$ is chosen randomly, independent of $(\AA,\bb)$, in such a way that for some $C$ independent of $d$ or $n$
    %so that for some $R > 0$ independent of $n$ 
    \[
    \|\Exp \xx_0\|_\infty \leq C/n
    \quad\text{and}\quad
    \max_j \|(\xx_0-\Exp \xx_0)_j\|_{\psi_2}^2 \leq Cn^{2\theta_0-1}.
    \]
We wish to show that for any $\theta >\theta_0$, Assumption \ref{assumption:init} holds on an event of probability tending to $1$ as $n \to \infty$, i.e.
\[
	\displaystyle \max_{z \in \Omega} \max_{1 \leq i \leq n} |\ee_i^T R(z; \AA^T\AA) \xx_0| \leq n^{\theta-1/2}.
\]
We note that using Lemma \ref{lem:fb} we may instead prove the claim for $z \in D$ for some set $D$ having imaginary part bounded away from the axis.
We begin by observing that for any fixed $z \in D$,
\[
\ee_i^T R(z; \AA^T\AA) \xx_0
=
\ee_i^T R(z; \AA^T\AA) \Exp(\xx_0)
+
\sum_{j=1}^n
\ee_i^T R(z; \AA^T\AA) (\xx_0 - \Exp \xx_0)_j.
\]
The first of these terms can be controlled solely using Assumption \ref{ass: laundry_list} and the norm bound on $\xx_0$.  For the second, we need a stochastic estimate.  By the Pythagorean theorem for the subgaussian norm and the Ward identity,
\[
\biggl\|
\sum_{j=1}^n
\ee_i^T R(z; \AA^T\AA) (\xx_0 - \Exp \xx_0)_j
\biggr\|_{\psi_2}^2
\leq
\frac{n^{2\theta_0}}{n}
\sum_{j=1}^n |\ee_i^T R(z; \AA^T\AA)\ee_j|^2 \leq n^{-1+2\theta_0} \delta^{-1} |\Im \ee_i^T R(z; \AA^T\AA) \ee_i|.
\]
Hence bounding the resolvent using the imaginary part
\[
\biggl\|
\sum_{j=1}^n
\ee_i^T R(z; \AA^T\AA) (\xx_0 - \Exp \xx_0)_j
\biggr\|_{\psi_2}^2
\leq \delta^{-2}n^{-1+2\theta_0}.
\]
It follows that taking a union bound over $D$ for any $\theta > \theta_0$
\[
\max_{z \in D}
\bigl|
\sum_{j=1}^n
\ee_i^T R(z; \AA^T\AA) (\xx_0 - \Exp \xx_0)_j
\bigr|
\leq C(\delta)n^{\theta-1/2}
\]
with overwhelming probability.
\end{proof}

\begin{proof}[Proof of Lemma \ref{lem:qs}]
    We recall that which we wish to show.
    Suppose that 
    \(
    \mathcal{R}
    \)
    satisfies \eqref{eq:statistic} with $\TT$ given by a polynomial $p$ in $\AA^T \AA$. 
    Suppose $\uu$ and $c$ are norm bounded independently of $n$ or $d$.  
    Suppose Assumptions \ref{assumption:Target} and \ref{ass: laundry_list} for some $\theta \in (0,\tfrac 12)$.
    We shall show that 
    for all $n$ sufficiently large and for any $\theta' > \theta$, $\mathcal{R}$ satisfies Assumption \ref{assumption: quadratics}.  For a meromorphic function $f$
    \[
    \ee_i^T \AA f(\AA^T\AA) \AA^T \ee_i
    =
    \ee_i^T \tilde{f}(\AA\AA^T) \ee_i,
    \]
    where $\tilde f(x) = xf(x)$ for all $x \in \mathbb{C}$.  Thus it suffices to show for any monomial $p$, any $\varepsilon > 0$ and all $n$ sufficiently large that
    \[
       	\max_{z,y \in \Omega} \max_{1 \leq i \leq n} 
    	|\ee_i^T R(z; \AA\AA^T) p(\AA\AA^T)R(y; \AA\AA^T) \ee_i
    	-\tfrac1n\tr(R(z; \AA\AA^T) p(\AA\AA^T)R(y; \AA\AA^T))
    	| \leq (1/2)n^{\theta+\varepsilon-1/2}.
    \]
    Let $\Omega_\delta$ be as in Lemma \ref{lem:fb}, part 2.  Then it suffices to show for $z,y \in \Omega_\delta$ that 
    \begin{equation}\label{eq:ORdelta}
       	\max_{z,y \in \Omega_\delta} \max_{1 \leq i \leq n} 
    	|\ee_i^T R(z; \AA\AA^T) p(\AA\AA^T)R(y; \AA\AA^T) \ee_i
    	-\tfrac1n\tr(R(z; \AA\AA^T) p(\AA\AA^T)R(y; \AA\AA^T))
    	| \leq n^{\theta+\varepsilon/2-1/2},
    \end{equation}
    by applying Lemma \ref{lem:fb} twice: once to move $z$ to $\Omega$ $y$-by-$y$ and once more to move $y$ to $\Omega$ $z$-by-$z$.
    Let $f(x) = (z-x)^{-1}(y-x)^{-1}p(x)$.  Then by Cauchy's integration formula
    \[
    \begin{aligned}
    &\ee_i^T R(z; \AA\AA^T) p(\AA\AA^T)R(y; \AA\AA^T) \ee_i
    	-\tfrac1n\tr(R(z; \AA\AA^T) p(\AA\AA^T)R(y; \AA\AA^T) \\
    &=
    \frac{1}{2\pi i}
    \oint_{\Omega}
    f(x)\bigl\{
    \ee_i^T R(x; \AA\AA^T) \ee_i
    -\tfrac1n\tr(R(x; \AA\AA^T))
    \bigr\}\,\dif x.
    \end{aligned}
    \]
    Hence we conclude \eqref{eq:ORdelta} for any $\varepsilon >0$ using Assumption \ref{ass: laundry_list} and the bound $|f(x)| = \mathcal{O}(\delta^{-2})$.

\end{proof}

\begin{proof}[Proof of Lemma \ref{lem:sc}]
We recall that we suppose that $\AA$ is a random matrix $\AA = \ZZ \sqrt{\SSigma}$ where $\ZZ$ is an $n \times d$ matrix of independent, mean $0$, variance $1$ entries with subgaussian norm at most $M < \infty$, and suppose $n \leq M d$.  We also suppose that $\bb = \AA \bbeta + \eeta$ for $\bbeta,\eeta$ iid centered subgaussian satisfying $\|\bbeta\|^2 =R$ and $\|\eeta\|^2 = \tilde{R}$.

We should show that Assumptions \ref{assumption:Target} and \ref{ass: laundry_list} hold with probability tending to $1 - e^{-\Omega(d)}$.
For Assumption \ref{ass: laundry_list}, the proof strategy is the same as the more complicated random features case which is fully elaborated in detail in Section \ref{sec:random_features}.  Hence, we do not enter into the details. The norm bound on $\bb$ and $\AA^T\AA$ can be derived from subgaussian estimates and a standard net argument, see \cite{vershynin2018high}.

\end{proof}

\subsection{Quadratic concentration property}\label{sec:quad_con}

In this section, we show how Assumption \ref{assumption: quadratics} follows for random matrices $\AA$ with independent rows.  Recall that we have assumed that each row is standardized to have $\Exp \aa = 0$ and $\Exp \|\aa\|^2 = 1$.
Beyond this, we will need to assume:
\begin{definition} \label{def:quad_con}
Say that a random vector $\aa \in \R^d$ has the \emph{quadratic concentration property} if there is a $\theta \in (0,1/2)$ and a $C>0$ so that for any $D > 0$ such that for any deterministic matrix $\WW \in \R^{d\times d}$ with operator norm $1$
\[
\Pr\biggl(
\bigl|\aa^T \WW \aa - \Exp \aa^T \WW \aa \bigr| > d^{\theta-1/2}
\biggr)
\leq C d^{-D}.
\]
\end{definition}
\noindent A simple example of a random vector satisfying this is any vector with iid subgaussian entries, which follows from the Hanson--Wright inequality (c.f.\ Theorem \ref{thm:adamczak}).  Moreover, the image of a $d$--dimensional vector of iid mean $0$ subgaussian entries by a bounded matrix also satisfies this inequality.

We show that for random matrices that satisfy Definition \ref{def:quad_con}, Assumption \ref{assumption: quadratics} holds.
\begin{lemma}\label{lem:stdrows}
Suppose $n,d$ satisfy Assumption \ref{ass:poly}.
Suppose that $\AA$ is an $n\times d$ random matrix whose rows are iid, are standardized (satisfy $\Exp \|\ee_i^T\AA\|^2 = 1$) and have the quadratic concentration property with parameter $\theta \in (0,1/2)$.  Then for any deterministic $d\times d$ matrix $\TT$ and for any $\epsilon < \alpha(\theta-1/2)$, with $\Omega$ as in Assumption \ref{assumption: quadratics},
\begin{equation}
    	\max_{z,y \in \Omega} \max_{1 \leq i \leq n} 
    	|\ee_i^T \AA \widehat\TT  \AA^T \ee_i
    	-\tfrac1n\tr(\AA \widehat\TT  \AA^T)
    	| \leq \|\TT\|_{op} n^{-\epsilon}
    	\quad\text{where}\quad 
    	\left\{
    	\begin{aligned}
    	&\widehat\TT =  R(z) \TT R(y) + R(y) \TT R(z), \\
    	&R(z) = R(z; \AA^T\AA)
    	\end{aligned}
    	\right.
\end{equation}
with overwhelming probability.
\end{lemma}
\begin{proof}
As the rows of $\AA$ have the quadratic concentration property, we have that with overwhelming probability the operator norm of $\AA$ is $\mathcal{O}(n)$, and in particular is polynomially bounded.  By Lemma \ref{lem:fb}, it suffices to show the claim for $z,y \in D$, a set of $z$ in the upper half plane with imaginary part at least some $\delta>0$.
As this set of $\{(z,y)\}$ has polynomial in $n$ cardinality, it suffices to show it for a fixed pair of $(z,y)$ with imaginary part greater than $\delta>0$.

It suffices to show the claim for a fixed $i$.  
Consider the matrix $\widehat{\AA}$ in which the row $\aa \defas \ee_i^T\AA$ has been set to 0.
Then 
\[
\widehat{\AA}^T\widehat{\AA} = \AA^T\AA 
- \aa^T \aa.% + \widehat{\aa}^T\widehat{\aa}.
\]
From the Sherman--Morrison--Woodbury formula
\begin{equation}\label{eq:loo}
R(z; {\AA}^T{\AA})
=
R(z; \widehat{\AA}^T\widehat{\AA})
-
R(z; \widehat{\AA}^T\widehat{\AA})
\aa^T \aa
R(z; \widehat{\AA}^T\widehat{\AA})
\bigl(
1+\aa 
R(z; \widehat{\AA}^T\widehat{\AA}) \aa^T
\bigr)^{-1}.
\end{equation}
We abbreviate $R(z)=R(z;\AA^T\AA)$ and $\widehat{R}(z)=R(z;\widehat{\AA}^T\widehat{\AA})$.
Hence with $u(z) = \aa 
\widehat R \aa^T
$ and
\[
\ee_i^T \AA R(z) \TT R(y)  \AA^T \ee_i
=
\aa \widehat{R}\TT \widehat{R} \aa^T
\biggl(\frac{1}{1+u(z)}\biggr)
\biggl(\frac{1}{1+u(y)}\biggr).
\]
Thus by independence of the rows, and the quadratic concentration property 
\[
|\aa \widehat{R}(z)\TT \widehat{R}(y) \aa^T
-
\tr(K\widehat{R}(z)\TT \widehat{R}(y))| \leq C(\delta)d^{\theta-1/2}
\]
with overwhelming probability, where $K$ is the covariance matrix $\Exp(\aa \otimes \aa)$.  Likewise 
\[
|u(z) - \tr(K \widehat{R}(z))| \leq  C(\delta)d^{\theta-1/2}
\]
with overwhelming probability.  

Now using that $\tr(K) = \Exp \|\ee_i\AA\|^2= 1,$ we can express
\[
1+\tr(K \widehat{R}(z))
=\tr\bigl(K (z+1 - \widehat{\AA}^T\widehat{\AA})\widehat{R}(z)\bigr).
\]
We bound this below in modulus by either using the real part or the imaginary part of $(z+1 - \widehat{\AA}^T\widehat{\AA})\widehat{R}(z)$.
Observe that
\[
\frac{x+i\delta + 1}{x+i\delta} = 
\frac{x^2+\delta^2+x}{x^2+\delta^2}
- 
\frac{i\delta}{x^2+\delta^2}.
\]
We apply this function to $x$ being $\Re z - \lambda_j(\widehat{\AA}^T\widehat{\AA})$ with $j$ running over all eigenvalues and with $z$ in a neighborhood of $\Omega$.  The real part thus is bounded below provide $x \geq 0$ and otherwise in which case the imaginary part is at least $\mathcal{O}(\delta)$ (using that $x$ is bounded below).  Thus letting $\Pi_z$ be the orthogonal projection map into the eigenspaces of $\widehat{\AA}^T\widehat{\AA}$ for which $\Re z - \lambda_j(\widehat{\AA}^T\widehat{\AA}) \geq 0$, we have
\[
\Re(1+\tr(K \Pi_z \widehat{R}(z))) \geq \tr(K \Pi_z)
\quad\text{and}\quad
|\Im(1+\tr(K (1-\Pi_z) \widehat{R}(z)))| \geq c(\Omega)\delta \tr(K (1-\Pi_z)).
\]
As we have that 
\[
\tr(K \Pi_z) + \tr(K (1-\Pi_z)) = 1,
\]
we conclude that for some constant $c$ depending on $\delta,\Omega$
\begin{equation}
|1+\tr(K \widehat{R}(z))| \geq c(\Omega,\delta).
\end{equation}

Moreover, we may then (using the same \eqref{eq:loo} representation) compare
\[
|
\tr(K\widehat{R}(z)\TT \widehat{R}(y))
-
\tr(K {R}(z)\TT {R}(y))
| \leq C(\delta)d^{\theta-1/2}
\]
as well as
\[
|
\tr( K \widehat{R}(z)) -
\tr( K {R}(z))|
\leq C(\delta)d^{\theta-1/2}.
\]
Combining everything, we arrive at a concentration inequality
\[
\biggl|
\ee_i^T \AA R(z) \TT R(y)  \AA^T \ee_i
-
w(z,y)
\biggr| \leq C(\delta)d^{\theta-1/2}
\quad\text{with}
\quad
w(z,y)=
\frac{\tr(K\widehat{R}(z)\TT \widehat{R}(y))}{(1+\tr( K {R}(z)))(1+\tr( K {R}(y)))}
.
\]
As $w$ does not depend on $i$, we conclude that
\[
|w(z,y) - \frac{1}{n}\tr \AA R(z) \TT R(y) \AA^T| \leq C(\delta)d^{\theta-1/2}
\]
with overwhelming probability, and so we've completed the proof at a single $(z,y).$
Taking $\epsilon$ as in the statement of the Theorem concludes the proof.

\end{proof}

%\section{Arc of the proof: comparing HSGD and SGD} \label{sec:dynamics}
\section{SGD and HSGD under the statistic} \label{sec:dynamics}

In this section we make the first steps to the proof of the comparison theorem, Theorem \ref{thm:homogenized_SGD_SGD}.  In particular, we reformat and reformulate the problem. Throughout this Section~\ref{sec:dynamics}, Section~\ref{sec:comp_SGD_HSGD}, and Section~\ref{sec:martingale_errors}, we  normalize our matrix $\AA$ so that it has row sum always $1$ without loss of generality.

%In this section, we introduce a diffusion approximation, denoted by \textit{homogenized SGD}, as the high-dimensional equivalence of SGD (\textit{c.f.}, \cite{paquette2021dynamics}). While analyzing SGD using diffusion processes is not new (e.g. \cite{li2019stochastic,li2017stochastic,mandt2016variational}), homogenized SGD is constructed by sending $n \to \infty$ as compared with sending the stepsize $\gamma(t) \to 0$. This ensures that the learning rate is bounded away from $0$ and thus homogenized SGD allows us to analyze properties of the learning rate in the diffusion process. Moreover it is explicitly solvable. We define this process and some properties in Section~\ref{sec:homogenized_SGD_intro}. Finally in Section~\ref{sec:comp_SGD_HSGD}, we show the high-dimensional equivalence of SGD to homogenized SGD (Theorem~\ref{thm:homogenized_SGD_SGD}). 

We decompose $\AA = \UU \SSigma \VV^T$ where $\UU \in \mathbb{R}^{n \times n}$, $\VV^T \in \mathbb{R}^{d \times d}$ orthogonal matrices and $\SSigma \in \mathbb{R}^{n \times d}$ is a diagonal matrix with the singular values of $\AA$ on the diagonal. We make the following change of variables by $\widehat{\nnu}_k \defas \VV^T\xx_{k}$,
\begin{equation} 
    \begin{aligned}
        \widehat{\nnu}_{k+1} = \widehat{\nnu}_k - \gamma_k \SSigma^T \UU^T \PP_k ( \UU \SSigma \nnu_k - \bb) - \tfrac{\gamma_k \delta}{n} \widehat{\nnu}_k, \qquad \PP_k = \ee_{i_k} \ee_{i_k}^T.
    \end{aligned}
\end{equation}

The least squares term in the objective function \eqref{eq:rr} plays an important role owing to the randomness that is not present in the $\ell^2$-regularization term. To make this explicit, we will denote the following
\begin{equation} \label{eq:lsq}
    \mathscr{L}(\xx) = \frac{1}{2} \|\AA \xx-\bb\|^2 \quad \text{and} \quad \widehat{\mathscr{L}}(\nnu) \defas \frac{1}{2} \|\SSigma \nnu - \UU^T \bb\|^2.
\end{equation}
Note that if $\nnu = \VV^T \xx$, then $\mathscr{L}(\xx) = \widehat{\mathscr{L}}(\nnu)$.

\subsection{Embedding into continuous time} We next consider an embedding of the process $\widehat{\nnu}_k$ into a continuous time. This is done to simplify the analysis and does not change the underlying behavior of SGD. Let $\mathbb{N}_0 \defas \mathbb{N} \cup \{0\}$. We define the infinite random sequence of times $\{\tau_k \, : \, k \in \mathbb{N}_0\}$ with $0 = \tau_0 < \tau_1 < \tau_2 < \hdots $, which will record the time at which the $k$-th update of SGD occurs. The distribution of these times $\{\tau_k \, : \, k \in \mathbb{N}_0\}$ will follow a standard rate $n$-Poisson process. This means that the interarrival times are i.i.d. Exp$(n)$ random variables, \textit{i.e.}, those with mean $\tfrac{1}{n}$, and we note this randomization is independent of both SGD, the matrix $\AA$, and vector $\bb$. The function $N_t$ will count the number of arrivals of the Poisson process before time $t$, that is
\[ N_t \defas \sup \{k \in \mathbb{N}_0 \, : \, \tau_k \le t\}. \]
Then for any $t > 0$, $N_t$ is Poisson$(nt)$. 

We embed the process $\widehat{\nnu}$ into continuous time by taking $\nnu_t = \widehat{\nnu}_{N_t}$. We scaled time (by choosing the rate of the Poisson process) so that in a single unit of time $t$, the algorithm has done one complete pass (in expectation) of the data set. 

\subsection{SGD under the statistic} \label{sec:Doob}

We compute the Doob decomposition for quasi-martingales \cite[Thm. 18, Chapt. 3]{protter2005stochastic} for $\nnu_t$, the iterates of SGD in the eigenspace. We will apply this to derive an exact expression for the behavior of SGD applied to the quadratics using It\^{o}'s formula \cite[Thm. 33, Chapt. 2]{protter2005stochastic}. Here we let $\mathcal{F}_t$ be the $\sigma$-algebra of information available to the process at time $t \ge 0$. We compute the compensator for the quadratic $q$ applied to $\nnu_t$. So we take, for any $j \in [d]$,
\begin{equation} 
    \begin{aligned}
        \mathcal{A}_{t,j} \defas \partial_t \mathbb{E}[\nu_{t,j} | \mathcal{F}_t] = \lim_{\varepsilon \downarrow 0} \varepsilon^{-1} \mathbb{E}[\nu_{t + \varepsilon, j} - \nu_{t,j} | \mathcal{F}_t]
    \end{aligned} 
\end{equation}
Thus the compensator for the $\nu_{t,j} - \nu_{t,0}$ process is $\int_0^t \mathcal{A}_{s,j} \, \dif s$. We then have the decomposition
\begin{equation} \label{eq:nnu}
    \begin{aligned}
        \nu_{t,j} &= \nu_{0,j} + \int_0^t \mathcal{A}_{s,j} \, \dif s + \mathcal{M}_{t,j}\\
        \dif \nu_{t,j} &= \mathcal{A}_{t,j} + \dif \mathcal{M}_{t,j},
    \end{aligned}
\end{equation}
where $\{\mathcal{M}_{t,j} \, : \, t \ge 0\}$ are $\mathcal{F}_t$-adapted martingales. For the computation of $\mathcal{A}_{t,j}$, we observe that as $\varepsilon \to 0$, $\mathcal{A}_{t,j}$ is dominated by the contribution of a single Poisson point arrival; as in time $\varepsilon$, the probability of having multiple Poisson point arrivals is $\mathcal{O}(n^2 \varepsilon^2)$, whereas the probability of having a single arrival is $1-e^{-n \varepsilon} \sim n \varepsilon$ as $\varepsilon \to 0$. For notational simplicity, we let the projection matrix $\PP \in \mathbb{R}^{d \times d}$ be an i.i.d. copy of $\PP_1$, which is independent of all the randomness so far. It follows that 
\begin{equation} \label{eq:nu_change}
    \begin{aligned} 
        \mathcal{A}_t &= n \mathbb{E} \left [ \nnu_{t} - \gamma(t) \SSigma^T \UU^T \PP (\UU \SSigma \nnu_t - \bb) - \tfrac{\gamma(t) \delta}{n} \nnu_{t} -  \nnu_{t} \, | \, \mathcal{F}_t \right ]\\
        &= -\gamma(t) \left ( \SSigma^T \SSigma \nnu_t - \SSigma^T \UU^T \bb + \delta \nnu_t \right ).
    \end{aligned}
\end{equation}

Next we apply It\^{o}'s formula to a quadratic $q$, that is, a 2nd-degree polynomial with complex coefficients, or $q(\xx) = \xx^T \widehat{\SS} \xx + \widehat{\hh}^T \xx + \widehat{c}$ where the matrix $\widehat{\SS} \in \mathbb{C}^{d \times d}$ (not necessarily symmetric), the vector $\widehat{\hh} \in \mathbb{C}^d$, and $\widehat{c}$ is a constant. We make note that this quadratic $q$ is not the same as the statistic in Assumption~\ref{assumption: quadratics}. We must apply It\^{o}'s formula to a larger class of quadratics in order to deduce the equivalence of homogenized SGD and SGD. Because of this, we will need to consider matrices $\widehat{\SS}$ which are possibly non-symmetric and complex. Recall It\^{o}'s formula \cite[Thm. 33, Chapt. 2]{protter2005stochastic} applied to a quadratic $q$ evaluated at $\nnu_t$:
\begin{equation} \label{eq: statistics_nu}
    \begin{aligned}
        &q(\nnu_t) = q(\nnu_0) + \int_0^t \nabla q(\nnu_{s-}) \cdot \dif \nnu_{s}\\
        % \sum_{j=1}^d \int_0^t \frac{\partial g}{\partial z_j} (\nnu_{s-}) \, \dif \nu_{s,j}\\
        & \quad + \sum_{0 < s \le t} \big \{ q(\nnu_{s}) - q(\nnu_{s-}) - \nabla q(\nnu_{s-})^T \Delta \nnu_{s} \big \} \\
        % \sum_{j=1}^d \frac{\partial g}{\partial z_j} (\nnu_{s-}) \Delta \nu_{s,j} \}\\
        & = q(\nnu_0) +  \int_0^t \nabla q (\nnu_{s-}) \cdot ( \dif \nnu_{s} - \mathcal{A}_{s} \dif s + \mathcal{A}_{s} \dif s) +  \frac{1}{2}\sum_{0 \le s \le t}  ( \Delta \nnu_s) ^T \nabla^2 q(\nnu_{s-}) \Delta \nnu_s   \\
        &= q(\nnu_0) - \int_0^t \gamma(s) \nabla q(\nnu_{s})^T \left (\SSigma^T \SSigma \nnu_{s} - \SSigma^T \UU^T \bb + \delta \nnu_{s} \right ) \dif s +  \frac{1}{2} \sum_{0 \le s \le t} ( \Delta \nnu_s) ^T \nabla^2 q(\nnu_{s-}) \Delta \nnu_s\\
        & \quad + \int_0^t \nabla q(\nnu_{s-}) \cdot \, \dif \mathcal{M}_{s}.
    \end{aligned}    
\end{equation}
Here $\Delta \nu_{s,j} = \nu_{s,j} - \nu_{s-,j}$ which captures the jumps at $s$ and $\nnu_{s-} = \nnu_{s-\varepsilon}$ for $\varepsilon > 0$, that is, the value of $\nnu$ right before the jump. We define the martingale
\begin{equation} \label{eq:martingale_grad}
\mathcal{M}_{t}^{\text{grad}}(q) \defas \int_0^t \nabla q(\nnu_{s-}) \cdot \, \dif \mathcal{M}_{s}.
\end{equation}

Next for the sum $\sum_{0 \le s \le t} \tfrac{1}{2} (\Delta \nnu_s)^T \nabla^2 q(\nnu_{s-}) \Delta \nnu_s$ in \eqref{eq: statistics_nu}, we perform a Doob decomposition,
\begin{equation} \label{eq:quadratic}
    \begin{aligned}
        \sum_{0 < s \le t} (\Delta \nnu_s)^T \nabla^2 q(\nnu_{s-}) \Delta \nnu_s 
        %&= \sum_{0 < s \le t} (\Delta \nnu_s)^T (\nabla^2 q) \Delta \nnu_s\\
        &= \underbrace{(\Delta \nnu_0)^T (\nabla^2 q) \Delta \nnu_0}_{= 0} + \int_0^t \mathcal{B}_s \, \dif s + \mathcal{M}_t^{\text{quad}}(q)
    \end{aligned}
\end{equation}
where $\{\mathcal{M}_t^{\text{quad}}(q) \, : \, t \ge 0\}$ are $\mathcal{F}_t$-adapted martingales and
\begin{equation}\begin{aligned} \label{eq:martingale_quad}
    \mathcal{M}_t^{\text{quad}}(q) &\defas         \sum_{0 < s \le t} (\Delta \nnu_s)^T \nabla^2 q(\nnu_{s-}) \Delta \nnu_s - \int_0^t \mathcal{B}_s \dif s\\
   \text{and} \quad  \mathcal{B}_t &\defas \lim_{\varepsilon \downarrow 0} \varepsilon^{-1} \mathbb{E}[ (\Delta \nnu_{t+\varepsilon})^T (\nabla^2 q) (\Delta \nnu_{t+\varepsilon}) | \mathcal{F}_{t} ].
\end{aligned}
\end{equation}
As in $\mathcal{A}_{t,j}$, for the computation of $\mathcal{B}_{t}$, we observe that as $\varepsilon \to 0$, $\mathcal{B}_{j}$ is dominated by the contribution of a single Poisson point arrival; as in time $\varepsilon$, the probability of having multiple Poisson point arrivals is $\mathcal{O}(n^2 \varepsilon^2)$, whereas the probability of having a single arrival is $1-e^{-n \varepsilon} \sim n \varepsilon$ as $\varepsilon \to 0$. As again, we let the projection matrix $\PP$ be an i.i.d. copy of $\PP_1$, which is independent of all the randomness so far. It follows from \eqref{eq:quadratic}
\begin{equation}
    \begin{aligned} \label{eq:conditional_expectation_difference}
        \mathcal{B}_t &= \gamma(t)^2 n \mathbb{E} \left [ (\SSigma^T \UU^T \PP (\UU \SSigma \nnu_{t}-\bb) - \tfrac{\delta}{n} \nnu_{t})^T (\nabla^2 q)(\SSigma^T \UU^T \PP (\UU \SSigma \nnu_{t}-\bb) - \tfrac{\delta}{n} \nnu_{t})  | \mathcal{F}_{t}\right ]\\
        &= \gamma^2(t) n \mathbb{E}[ \big(\SSigma^T \UU^T \PP(\UU \SSigma \nnu_t - \bb ) \big)^T (\nabla^2 q) \big(\SSigma^T \UU^T \PP(\UU \SSigma \nnu_t - \bb ) \big) | \mathcal{F}_{t} ] \\
        % &= \gamma^2 n \mathbb{E}[ \| \SS^{1/2} \SSigma^T \UU^T \PP (\UU \SSigma \nnu_{t} - \bb) \|^2 | \mathcal{F}_{t} ] \\
        &\qquad + \frac{\gamma^2(t) \delta^2}{n} \nnu_{t}^T (\nabla^2 q) \nnu_{t-} - 2 \gamma^2(t) \delta \mathbb{E} \big [\nnu_{t}^T (\nabla^2 q) (\SSigma^T \UU^T \PP (\UU \SSigma \nnu_{t}-\bb) ) | \mathcal{F}_{t}  \big ]\\
        &= \gamma^2(t) \sum_{i=1}^n (\ee_i^T \UU \SSigma (\nabla^2 q)  \SSigma^T \UU^T \ee_i) (\ee_i^T (\UU \SSigma \nnu_{t} - \bb) )^2 \\
        % &= \textcolor{red}{\gamma^2 n \sum_{j=1}^d \EE [ (\ee_j^T \SS^{1/2}  \SSigma^T \UU^T \ee_i)^2 (\ee_i^T (\UU \SSigma \nnu_{t} - \bb) )^2 | \mathcal{F}_{t} ] \quad \text{need to change}} \\
        & \qquad + \frac{\gamma^2(t) \delta^2}{n} \nnu_{t}^T (\nabla^2 q) \nnu_{t} - \frac{2 \gamma^2(t) \delta}{n} \nnu_{t}^T (\nabla^2 q) \SSigma^T (\SSigma \nnu_{t} - \UU^T \bb).
    \end{aligned}
\end{equation}
% To compute this conditional expectation, we record the following the lemma.

% \begin{lemma} \label{lem:conditional_expectation} Suppose that $\uu$ and $\vv$ are fixed vectors in $\mathbb{R}^n$. Then 
% \[ \mathbb{E}_{i} \left (  u_i v_i \right )^2 = \frac{1}{n} \sum_{i=1}^n (u_i v_i)^2.   \]
% \end{lemma}

Using \eqref{eq:conditional_expectation_difference}, we express the Hessian term \eqref{eq:quadratic} as
\begin{equation}
    \begin{aligned}
        \sum_{0 \le s \le t} (\Delta \nnu_s)^T &(\nabla^2 q) (\Delta \nnu_s) = \int_0^t \mathcal{B}_s \dif s + \mathcal{M}_t^{\text{quad}}(q)\\
        & = \sum_{i=1}^n \int_0^t \gamma^2(s) (\ee_i^T \UU \SSigma (\nabla^2 q) \SSigma^T \UU^T \ee_i)  (\ee_i^T (\UU \SSigma \nnu_s -\bb))^2 \, \dif s\\
        % &= \textcolor{red}{\sum_{j=1}^d \sum_{i=1}^n \int_0^t \gamma^2 (\ee_j^T \SS^{1/2} \SSigma^T \UU^T \ee_i)^2 (\ee_i^T (\UU \SSigma \nnu_s -\bb))^2 \dif s}\\
        & \quad + \frac{1}{n} \int_0^t \gamma^2(s) \left ( \delta^2 \nnu_s^T (\nabla^2 q) \nnu_s - 2 \delta \nnu_s^T (\nabla^2q) \SSigma^T (\SSigma \nnu_s - \UU^T \bb) \right ) \, \dif s + \mathcal{M}_t^{\text{quad}}(q).
    \end{aligned}
\end{equation}

Finally, the quadratic $q(\nnu_t)$ in its Doob decomposition is
\begin{align}
    &q(\nnu_t) = q(\nnu_0) - \int_0^t \gamma(s) \nabla q(\nnu_s)^T (\SSigma^T \SSigma \nnu_s - \SSigma^T \UU \bb + \delta \nnu_s) \, \dif s \label{eq:blah_1} \\
    & + \frac{1}{2} \sum_{i=1}^n \int_0^t \gamma^2(s) (\ee_i^T \UU \SSigma (\nabla^2 q) \SSigma^T \UU^T \ee_i)(\ee_i^T (\UU \SSigma \nnu_s - \bb))^2 \dif s \label{eq:blah_2}\\
    % & \qquad + \textcolor{red}{\gamma^2 \sum_{j=1}^d \sum_{i=1}^n \int_0^t  (\ee_j^T \SS^{1/2} \SSigma^T \UU^T \ee_i)^2 (\ee_i^T (\UU \SSigma \nnu_s -\bb))^2 \dif s} \label{eq:blah_2}\\
    & + \frac{1}{2n} \int_0^t \gamma^2(s) \left ( \delta^2 \nnu_s^T (\nabla^2 q) \nnu_s - 2 \delta \nnu_s^T (\nabla^2 q) \SSigma^T (\SSigma \nnu_s - \UU^T \bb) \right ) \, \dif s + \mathcal{M}_t^{\text{grad}}(q) + \mathcal{M}_t^{\text{quad}}(q), \label{eq:blah_3}
\end{align}
where $\mathcal{M}_t^{\text{grad}}(q)$ and $\mathcal{M}_t^{\text{quad}}$ are defined in \eqref{eq:martingale_grad} and \eqref{eq:martingale_quad}, respectively. In the large $n$-limit, we expect that the first term in \eqref{eq:blah_3} to vanish. For the martingales $\mathcal{M}_t^{\text{grad}}$ and $\mathcal{M}_t^{\text{quad}}$, we will use concentration to show that it vanishes. We now return to \eqref{eq:blah_2}. The \textit{key lemma} to simplifying \eqref{eq:blah_2} is that $\ee_i^T \UU \SSigma (\nabla^2 q) \SSigma^T \UU^T \ee_i$ self-averages to $\frac{\tr(\SSigma (\nabla^2 q) \SSigma^T)}{n}$ (see Assumption~\ref{assumption: quadratics}). The error incurred by the key lemma, we denote by
    \begin{equation}  \begin{aligned} \label{eq:keylemma_error}
        \mathcal{E}_t^{\text{KL}}(q) &\defas \frac{1}{2} \sum_{i=1}^n \int_0^t \gamma^2(s) (\ee_i^T \UU \SSigma (\nabla^2 q) \SSigma^T \UU^T \ee_i) (\ee_i^T (\UU \SSigma \nnu_s-\bb))^2 \, \dif s  \\
        &- \frac{1}{n} \tr((\nabla^2 q) \SSigma^T \SSigma ) \int_0^t \gamma^2(s) \widehat{\mathscr{L}}(\nnu_s) \dif s\\
        &\le \max_{1 \le i \le n} \left \{ (\ee_i^T \UU \SSigma (\nabla^2 q) \SSigma^T \UU^T \ee_i) - \frac{1}{n} \tr((\nabla^2 q) \SSigma^T \SSigma ) \right \} \int_0^t \gamma^2(s) \widehat{\mathscr{L}}(\nnu_s) \, \dif s
        % &\text{where $f_1(\nnu_s) \defas \frac{1}{2} \|\SSigma \nnu_s - \bb\|^2$.}
        % \textcolor{red}{\mathcal{E}_t^{\text{KL}} \defas \gamma^2 \int_0^t \sum_{j=1}^d \sum_{i=1}^n (\ee_j^T \SS^{1/2} \SSigma^T \UU^T \ee_i)^2 (\ee_i^T (\UU \SSigma \nnu_s-\bb))^2 \, \dif s - \gamma^2 \frac{2}{n} \tr(\SS \SSigma^T \SSigma ) \int_0^t f_1(\nnu_s) \dif s}
    \end{aligned}
    \end{equation}
Here we use the definition of $\widehat{\mathscr{L}}$ in \eqref{eq:lsq}. We will apply Assumption~\ref{assumption: quadratics} (i.e., key lemma) to the above expression with many different quadratics. However they will all have the form \eqref{eq:key_lemma_ass} (see Section~\ref{sec:comp_SGD_HSGD} for specifics). With this, we have that $\mathcal{E}_t^{\text{KL}}$ will go to $0$ as $n \to \infty$ provided $\int_0^t \gamma^2(s) \widehat{\mathscr{L}}(\nnu_s) \, \dif s$ is bounded. 

We now express $q(\nnu_t)$ below as a sum of two terms (1st line and 2nd line of \eqref{eq:SGD_statistic}). When $q$ is the statistic $g$, the first line of \eqref{eq:SGD_statistic} will equal the behavior of SGD under a different process which we call homogenized SGD (see Section~\ref{sec:homogenized_SGD_intro}). The 2nd line are the error terms which will vanish as $n \to \infty$:
% as a sum of two terms, the first of which is the value of the statistic $g$ under homogenized SGD and the second term are errors which vanish as $n$ gets large
\begin{equation} \label{eq:SGD_statistic}
    \begin{aligned}
    &q(\nnu_t) = q(\nnu_0) - \int_0^t \gamma(s) \nabla q(\nnu_s)^T ( \SSigma^T \SSigma \nnu_s - \SSigma^T \UU^T \bb + \delta \nnu_s) \, \dif s + \frac{1}{n} \tr( (\nabla^2 q) \SSigma^T \SSigma) \int_0^t \gamma^2(s) \widehat{\mathscr{L}}(\nnu_s) \, \dif s\\
    & + \frac{1}{2n} \int_0^t \gamma^2(s) \left ( \delta^2 \nnu_s^T (\nabla^2 q) \nnu_s - 2 \delta \nnu_s^T (\nabla^2 q) \SSigma^T (\SSigma \nnu_s - \UU^T \bb) \right ) \, \dif s + \mathcal{M}_t^{\text{grad}}(q) + \mathcal{M}_t^{\text{quad}}(q) + \mathcal{E}_t^{\text{KL}}(q).
    \end{aligned}
\end{equation}

\subsection{HSGD in the spectral basis} \label{sec:homogenized_SGD_intro}
In the spectral basis, HSGD takes a slightly simpler form: namely the gradient flow generator becomes diagonal.  All interaction between the different coordinates of the solution of the SDE in the spectral basis arises through the empirical risk $\mathscr{L}$.  Under the change of variables $\YY_t \defas \VV^T \XX_t$ where $\VV^T$ are the right singular vectors of $\AA$, we have that
\begin{equation} \begin{aligned} \label{eq:HSGD_dif}
    \dif \YY_t &= - \gamma(t) \VV^T \nabla f(\XX_t) + \gamma(t) \sqrt{ \tfrac{2}{n} \mathscr{L}(\XX_t) \SSigma^T \SSigma} \dif (\VV^T \BB_t)\\
    &= \left ( -\gamma(t) \SSigma^T \SSigma \YY_t + \gamma(t) \SSigma^T \UU^T \bb -  \delta \gamma(t) \YY_t \right ) \, \dif t + \gamma(t) \sqrt{ \tfrac{2}{n} \widehat{\mathscr{L}}(\YY_t) \SSigma^T \SSigma } \dif (\VV^T \BB_t).
\end{aligned}
\end{equation}
Here we used $\widehat{\mathscr{L}}(\YY_t)  = \frac{1}{2}\|\SSigma \YY_t - \UU^T \bb\|^2 = \frac{1}{2} \| \AA \XX_t - \bb\|^2 = \mathscr{L}(\XX_t)$ and $\VV^T \nabla f(\XX_t) = \VV^T \AA^T (\AA \XX_t - \bb) + \delta \VV^T \XX_t = \SSigma^T ( \SSigma \YY_t - \UU^T \bb ) + \delta \YY_t$. We also note that $\VV^T \BB_t$ is another Brownian motion. Fix an arbitrary quadratic $q : \mathbb{R}^d \to \mathbb{C}$ as in Section~\ref{sec:Doob}. By using It\^{o}'s formula \cite[Thm. 33, Chapt. 2]{protter2005stochastic}, we deduce that
\begin{equation} \label{eq:ito_diffusion}
    \begin{aligned}
        &q(\YY_t) = q(\YY_0) + \int_0^t \nabla q(\YY_{s}) \cdot \dif \YY_s + \frac{1}{2} \int_0^t \dif \YY_s^T \nabla^2 q(\YY_{s-}) \dif \YY_s\\
        &= q(\YY_0) - \int_0^t \gamma(s) \nabla q(\YY_{s})^T \big( \SSigma^T \SSigma \YY_s - \SSigma^T \UU^T \bb +  \delta \YY_s \big) \, \dif s +  \tr( (\nabla^2 q) \SSigma^T \SSigma) \frac{2}{n} \int_0^t \gamma^2(s) \widehat{\mathscr{L}}(\YY_s) \, \dif s \\
        & \quad + \int_0^t \gamma(s) \nabla q(\YY_s) \cdot \sqrt{ \tfrac{2}{n} \widehat{\mathscr{L}}(\YY_s) \SSigma^T \SSigma } \,  \dif (\VV^T \BB_s)\\
        &= q(\YY_0) - \int_0^t \gamma(s) \nabla q(\YY_{s})^T \big ( \SSigma^T \SSigma \YY_s - \SSigma^T \UU^T \bb + \delta \YY_s \big ) \, \dif s +  \frac{2}{n} \tr( (\nabla^2 q) \SSigma^T \SSigma) \int_0^t \gamma^2(s) \widehat{\mathscr{L}}(\YY_s)  \, \dif s \\
        & \quad + \mathcal{M}_t^{\text{hSGD}}(q),
    \end{aligned}
\end{equation}
where $\mathcal{M}_t^{\text{hSGD}}(q) \defas \int_0^t \gamma(s) \nabla q(\YY_s) \cdot \sqrt{ \tfrac{2}{n} \widehat{\mathscr{L}}(\YY_s) \SSigma^T \SSigma } \,  \dif (\VV^T \BB_s)$ is a martingale. 

\subsection{Properties of homogenized SGD and SGD} 

We begin our analysis by discussing bounds on the norms of the iterates of SGD and homogenized SGD which will be fruitful in the next section. 

We remark to control the errors, we will need to make an \textit{a priori} estimate that effectively shows that the iterates remain bounded. Thus, we define the stopping time, for any fixed $\varepsilon > 0$, by
\begin{equation} \label{eq:stopping_time_norm}
    \vartheta_{\varepsilon} \defas \inf \big \{ t \ge 0 \, : \, \|\nnu_t\|^2_2 > n^{\varepsilon} \quad \text{or} \quad \|\YY_t\|^2 > n^{\varepsilon} \big \} \quad \text{for some $\varepsilon > 0$}.
\end{equation}
The choice of $\varepsilon$ will be determined later, and as such, we simplify notation by $\vartheta \defas \vartheta_{\varepsilon}$. In particular, the value of $\varepsilon = \min\{\theta/2, (1/2 - \theta)/4\}$ will suffice where $\theta$ is defined in Assumption~\ref{ass: laundry_list}. It will be convenient to work with homogenized SGD and SGD under the stopped processes, that is, $\nnu_t^{\vartheta} \defas \nnu_{t \wedge \vartheta}$ and $\YY_t^{\vartheta} \defas \YY_{t \wedge \vartheta}$. We will later remove this stopping time by showing it does not occur with overwhelming probability (c.f. Corollary~\ref{lem:stopping_time_function_values}). Recall the bound on the learning rate (see Assumption~\ref{assumption:lr}), 
\begin{equation}
    \widehat{\gamma} \defas \sup_t \gamma(t).
\end{equation}

First with overwhelming probability, we show that $\vartheta$ does not occur for homogenized SGD, $\YY_t$.

\begin{lemma}[Boundedness of loss and norm under Homogenized SGD] \label{lem:boundedness_loss_norm_homogenized_SGD} Fix a constant $T > 0$. With overwhelming probability, for any $\varepsilon > 0$,
\begin{align*}
    \sup_{0 \le t \le T} \|\YY_t\|^2 \le n^{\varepsilon}, \quad \text{where $\YY_t$ is homogenized SGD defined in \eqref{eq:HSGD_dif}.}
\end{align*}
% In fact, fix any quadratic, $\widehat{q} \, : \, \mathbb{R}^d \to \mathbb{R}$, such that $\widehat{q}(\xx) = \xx^T \widehat{\SS} \xx + \widehat{\hh}^T \xx + \widehat{c}$, with $\widehat{\SS}$, $\widehat{\hh}$, and $c$ bounded norm independent of $n$. With overwhelming probability,
% \[ \sup_{0 \le t \le T} \widehat{q}(\YY_t) \le n^{\varepsilon}. \]
\end{lemma}

\begin{proof} First, we show that there exists constants $c_1, c_2 > 0$ such that 
\begin{align}
    g(t, \YY_t) = (1 + \|\YY_t\|^2)^{\lambda} e^{-(c_1 \lambda + c_2 \lambda^2)t}
\end{align}
is a supermartingale for every $\lambda \ge 0$. Let $c_1, c_2 \ge 0$ be constants to be determined later. From It\^{o}'s formula and \eqref{eq:HSGD_dif}, we have that 
\begin{align} \label{eq:super_martingale}
    & \dif  g(t,\YY_t)  = \partial_t g(t, \YY_t) + \nabla g(t, \YY_t)^T \dif \YY_t + \dif \YY_t^T \nabla^2_{yy} g(t, \YY_t) \dif \YY_t\\
    &= -(c_1 \lambda + c_2 \lambda^2) e^{-(c_1 \lambda + c_2 \lambda^2)t} (1 + \|\YY_t\|^2)^{\lambda} \label{eq:super_martingale_2} \\
    & + e^{-(c_1 \lambda + c_2 \lambda^2)t} \lambda (1 + \|\YY_t\|^2)^{\lambda -1} \YY_t^T \big [ (-\gamma(t) ( \SSigma^T \SSigma + \delta \II_d ) \YY_t + \gamma(t) \SSigma^T \UU^T \bb ) \dif t + \sqrt{\tfrac{2}{n} \widehat{\mathscr{L}}(\YY_t) \SSigma^T \SSigma} \dif \BB_t \big ] \label{eq:super_martingale_3}\\
    & + \tfrac{\gamma^2(t)}{n} e^{-(c_1 \lambda + c_2 \lambda^2)t } \widehat{\mathscr{L}}(\YY_t) (1+ \|\YY_t\|^2)^{\lambda -1} \tr \big ( [ \lambda(\lambda-1) (1 + \|\YY_t\|^2 )^{-1} \YY_t \YY_t^T + \lambda \II_d ] \SSigma^T \SSigma \big ).\label{eq:super_martingale_1}
\end{align}
For each term in \eqref{eq:super_martingale_3} and \eqref{eq:super_martingale_1}, we upper bound that only depend on the norms of $\SSigma$, $\bb$, $\UU$, and $\delta$ multiplied by $g(t, \YY_t)$. Then because \eqref{eq:super_martingale_2} is negative, we can choose $c_1$ and $c_2$ large enough, based on the bounds of \eqref{eq:super_martingale_3} and \eqref{eq:super_martingale_1}, so that the drift term is overall negative. For the bounds on \eqref{eq:super_martingale_3} and \eqref{eq:super_martingale_1}, we have 
\begin{equation*} \begin{aligned}
    &| \gamma(t) \YY_t^T (\SSigma^T \SSigma + \delta \II_d) \YY_t | \le \widehat{\gamma} \|\SSigma^T \SSigma + \delta \II_d \| (1 + \|\YY_t\|^2)\\
    &| \gamma(t) \YY_t^T \SSigma^T \UU^T \bb | \le \widehat{\gamma} \|\SSigma^T \UU^T \bb\|_2 \|\YY_t\|_2 \le \widehat{\gamma} \|\SSigma^T \UU^T \bb\|_2 (1 + \|\YY_t\|^2)\\
    &\widehat{\mathscr{L}}(\YY_t) 
    % = \|\SSigma \YY_t - \UU^T \bb\|^2 = \YY_t^T \SSigma^T \SSigma \YY_t - 2 \bb^T \UU \SSigma \YY_t + \|\bb\|^2 
    \le \|\SSigma^T \SSigma\| (1 + \|\YY_t\|^2) + 2 \|\SSigma \UU^T \bb\|_2 (1 + \|\YY_t\|^2) + \|\bb\|^2_2 (1 + \|\YY_t\|^2)\\
    &\lambda (\lambda-1) (1 + \|\YY_t\|^2)^{-1} \tr( \YY_t \YY_t^T \SSigma^T \SSigma)  \le \lambda (\lambda -1) \|\SSigma^T \SSigma\|.
\end{aligned} \end{equation*}
It immediately follows that the terms in \eqref{eq:super_martingale_3} and \eqref{eq:super_martingale_1} are upper bounded by constants multiplied by $g(t,\YY_t)$. We can now define the constants $c_1$ and $c_2$ sufficiently large enough so that the drift term in $\partial_t g(t, \YY_t)$ is negative, 
\begin{equation}
\begin{aligned} \label{eq:constants_supermartingale}
    c_1 & > \widehat{\gamma} \|\SSigma^T \SSigma + \delta \II_d\| + \widehat{\gamma} \|\SSigma^T \UU^T \bb\|_2 + \tfrac{\widehat{\gamma}^2}{n} 2\|\SSigma^T \SSigma\| (\|\SSigma^T \SSigma\| + 2 \|\SSigma^T \UU^T \bb\|_2 + \|\bb\|^2)\\
    c_2 & > \tfrac{\widehat{\gamma}^2}{n} \|\SSigma^T \SSigma\| (\|\SSigma^T \SSigma\| + 2 \|\SSigma^T \UU^T \bb\|_2 + \|\bb\|^2).
\end{aligned}
\end{equation}
This ensures that the drift term is strictly less than $0$ and hence we have that $g(t, \YY_t)$ is a supermartingale.

Fix a constant $T > 0$ and let 
\[ S_T \defas \sup_{0 \le t \le T} \|\YY_t\|^2 \le \sup_{0\le t \le T} 1 + \|\YY_t\|^2.\] 
For convenience, denote the supermartingale $M_t^{\lambda} \defas (1 + \|\YY_t\|^2)^{\lambda} e^{-(c_1 \lambda + c_2 \lambda^2)t}$ for any $\lambda > 0$ and constants $c_1$, $c_2$ defined in \eqref{eq:constants_supermartingale}. It follows that $(S_T)^{\lambda} e^{-(c_1 \lambda + c_2 \lambda^2)T} \le \displaystyle \sup_{0 \le t \le T} M_t^{\lambda}$. By optional stopping ($M_t^{\lambda}$ is a positive supermartingale) and Doob's $L^1$-inequality \citep[Chapt 2., Theorem 1.7]{revuz1999continuous}, we deduce for any $a > 0$ and $\lambda > 0$
\begin{align*}
    \Pr(S_T \ge a \, | \, \mathcal{F}_0) &\le \Pr(\sup_{0 \le t \le T} M_t^{\lambda} \ge a^{\lambda} e^{-(c_1\lambda + c_2 \lambda^2)T} \, | \, \mathcal{F}_0)\\
     &\le a^{-\lambda} e^{(c_1 \lambda + c_2 \lambda^2) T} \sup_{0 \le t \le T} \EE[M_t^{\lambda} \, | \, \mathcal{F}_0] \\
     &\le  a^{-\lambda} e^{(c_1 \lambda + c_2 \lambda^2) T} \EE[M_0^{\lambda} \, | \, \mathcal{F}_0].
\end{align*}
The result follows by setting $a = n^{\varepsilon}$. 

% For a generic quadratic $\widehat{q}$, we see that $\displaystyle \sup_{0 \le t \le T} \widehat{q}(\YY_t) \le C(\SS, \hh, c) (1 + \|\YY_t\|^2)$ for some constant $C(\SS, \hh, c)$. The result immediately follows from the same argument. 
\end{proof}

Next, we get a bound on the size of the iterates of SGD \eqref{eq:nnu}, $\nnu_t$, at the stopping time. When $t = \vartheta$, the norm $\|\nnu_t\|^2_2$ is not too large.

\begin{lemma}[Size of $\nnu_t$] \label{lem:boundedness_loss_norm_SGD} Let Assumptions~\ref{assumption:Target}, \ref{ass: laundry_list}, and \ref{assumption: quadratics} hold. Choose $\varepsilon = \min\{\theta/2, 1/4(1/2-\theta)\}$. Fix a constant $T > 0$ and suppose the stopping time $\vartheta \le T$. For $t = \vartheta$, $\displaystyle \|\nnu_t^{\vartheta}\|^2_2 \le C(\delta, \gamma, \SSigma, \UU, \bb, T) n^{\varepsilon}$ with overwhelming probability. 
\end{lemma}

\begin{proof} We apply \eqref{eq:SGD_statistic} to the quadratic $\|\cdot\|^2$, that is, 
\begin{equation} 
    \begin{aligned}
    \|\nnu_t^{\vartheta}\|^2 &= \|\nnu_0\|^2 - 2 \int_0^{\vartheta} \gamma(s) \nnu_s^T ( \SSigma^T \SSigma \nnu_s - \SSigma^T \UU^T \bb + \delta \nnu_s) \, \dif s + \frac{2}{n} \tr(\SSigma^T \SSigma) \int_0^{\vartheta} \gamma^2(s) \widehat{\mathscr{L}}(\nnu_s) \, \dif s\\
    & \qquad + \frac{1}{n} \int_0^{\vartheta} \gamma^2(s) \left ( \delta^2 \|\nnu_s\|^2 - 2 \delta \nnu_s^T \SSigma^T (\SSigma \nnu_s - \UU^T \bb) \right ) \, \dif s\\
    & \qquad + \mathcal{M}_{t\wedge \vartheta}^{\text{grad}}(\|\cdot\|^2) + \mathcal{M}_{t \wedge \vartheta}^{\text{quad}}(\|\cdot\|^2) + \mathcal{E}_{t \wedge \vartheta}^{\text{KL}}(\|\cdot\|^2).
    \end{aligned}
\end{equation}
We can bound the terms $\nnu_s^T ( \SSigma^T \SSigma \nnu_s - \SSigma^T \UU^T \bb + \delta \nnu_s)$, $\widehat{\mathscr{L}}(\nnu_s)$, and $\delta^2 \|\nnu_s\|^2 - 2 \delta \nnu_s^T \SSigma^T (\SSigma \nnu_s - \UU^T \bb)$ can all be bounded by terms depending on $C(\delta, \gamma, \SSigma, \bb) \|\nnu_s\|^2_2$. Since these occur for $0 \le s < \vartheta$, we deduce that there exists a constant such that for $t = \vartheta$ the norm 
\[\|\nnu_t^{\vartheta}\|^2 \le C(\delta, \gamma, \SSigma, \bb, T) \int_0^{\vartheta} \|\nnu_s\|^2_2 \, \dif s + |\mathcal{M}_{t \wedge \vartheta}^{\text{grad}}(\|\cdot\|^2)| + |\mathcal{M}_{t \wedge \vartheta}^{\text{quad}}(\|\cdot\|^2)| + |\mathcal{E}_{t \wedge \vartheta}^{\text{KL}}(\|\cdot\|^2)|.\]
Let $\varepsilon = \min\{\theta/2, 1/4(1/2-\theta)\}$ and $\alpha = \theta + 1/4(1/2-\theta)$. Since $\vartheta \le T$, we have from Proposition~\ref{prop:martingale_errors_SGD} and Assumption~\ref{assumption: quadratics} that each of the events $|\mathcal{M}_{t\wedge \vartheta}^{\text{grad}}| \le n^{-1/2 + \alpha + \varepsilon}$, $|\mathcal{M}_{t \wedge \vartheta}^{\text{quad}}| \le n^{-1/2 + \alpha + \varepsilon}$, and $|\mathcal{E}_{t \wedge \vartheta}^{\text{KL}}| \le n^{\varepsilon + \theta - 1/2}$ occur with overwhelming probability. The choice of $\alpha$ and $\varepsilon$ ensure that $-1/2 + \alpha + \varepsilon < 0$ and $\varepsilon + \theta - 1/2 < 0$. The intersection of these events also occurs with overwhelming probability. We observe that the integral occurs up to but not including $\vartheta$, and thus the integrand $\|\nnu_s\|^2$ is bounded by $n^{\varepsilon}$.  The result immediately follows. 
\end{proof}

\subsection{Concentration of homogenized SGD} \label{sec:concentration_hsgd}
In this section, we show that homogenized SGD applied to any quadratic statistic $\mathcal{R} \, : \, \mathbb{R}^d \to \mathbb{R}$ with $\|\mathcal{R}\|_{H^2} \le C$, $C > 0$ (see Definition~\ref{def:quad}) concentrates around its mean (see Theorem~\ref{thm:trainrisk}), that is, 
for any $T > 0$ and for any $D>0$ there is an $C'>0$ sufficiently large that for all $d>0$
\begin{equation} \label{eq:concentration_hsgd_1}
\Pr\biggl[
\sup_{0 \leq t \leq T}\biggl\|
\begin{pmatrix} \mathscr{L}(\XX_t) \\ \mathcal{R}(\XX_t)\end{pmatrix}
-
\begin{pmatrix}\Psi_t \\ \Omega_t \end{pmatrix}
\biggr\| > d^{-\epsilon/2} 
\biggr] \leq C'd^{-D},
\end{equation}
where $\Psi_t$ and $\Omega_t$ solve \eqref{eqa:VLoss} and \eqref{eqa:PLoss}. In particular, we will show that $\Psi_t$ is the solution to a Volterra integral equation, 
\begin{equation}
    \Psi_t = \mathscr{L}(\bm{\mathscr{X}}_{\Gamma(t)}^{\text{gf}}) + \frac{1}{n} \int_0^t \gamma^2(s)  \tr \big ( (\AA^T \AA)^2 e^{-2(\AA^T \AA + \delta \II_d)(\Gamma(t) - \Gamma(s))} \big ) \Psi_s \, \dif s.
\end{equation}
And the expression $\Omega_t$ satisfies the integral equation,
\begin{equation}
    \Omega_t = \mathcal{R}(\bm{\mathscr{X}}_{\Gamma(t)}^{\text{gf}}) + \frac{1}{n} \int_0^t \gamma^2(s) \tr \big ( \nabla^2 \mathcal{R} (\AA^T \AA) e^{-2 (\AA^T \AA + \delta \II_d)(\Gamma(t)-\Gamma(s))} \big ) \Psi_s \, \dif s.
\end{equation}

An important input into $\Psi_t$ and $\Omega_t$ will be the value of the respective function under \textit{gradient flow} (GF). Recall gradient flow defined in \eqref{eq:GFS} as the process $\bm{\mathscr{\XX}}_t^{\text{gf}}$ which solves the ODE
\begin{equation} \label{eq:gf}
\dif \bm{\mathscr{\XX}}_{\Gamma(t)}^{\text{gf}} = - \gamma(t) \nabla f(\bm{\mathscr{\XX}}_{\Gamma(t)}^{\text{gf}}),
\end{equation}
for an objective function $f$ initialized with $\mathscr{X}_0^{\text{gf}} = \xx_0$ and learning rate $\gamma(t)$. For the $\ell^2$-regularized least-squares problem \eqref{eq:rr}, we solve this ODE \eqref{eq:gf}, 
\begin{equation}
    \bm{\mathscr{X}}_{\Gamma(t)}^{\text{gf}} = e^{-(\AA^T \AA + \delta \II_d) \Gamma(t)} \XX_0 + \int_0^t e^{-(\AA^T \AA + \delta \II_d)(\Gamma(t)-\Gamma(s))} \AA^T \bb\, \dif s.
\end{equation}
We note that gradient flow under $\mathcal{R}$ and $\mathcal{L}$ is explicitly computable from initialization, data matrix, and target vector information.  

With this in hand, we can proceed to evaluate homogenized SGD (see \eqref{eqF:HSGD}) under the $\mathcal{R}$ and $\mathscr{L}$, which we recall below
\[
\dif \XX_t = \gamma(t)(- (\AA^T \AA + \delta \II_d) \XX_t + \AA^T \bb ) \, \dif t + \gamma(t) \sqrt{\tfrac{2}{n} \mathscr{L}(\XX_t) \AA^T \AA} \,  \dif \BB_t,
\]
where $\XX_0 = \xx_0$, the process $(\BB_t \, : \, t\ge 0)$ is a $d$-dimensional standard Brownian motion, and the least squares loss function $\mathscr{L}(\XX_t) = \tfrac{1}{2} \|\AA \XX_t - \bb\|^2_2$. 

The following lemma and its proof is an extension of the result in \citep[Lemma B.2]{paquetteSGD2021} adapted to accommodate a regularization parameter and a time dependent learning rate. Lemma~\ref{lemma:iterates_hSGD} gives us a recursive expression for homogenized SGD $\XX_t$.

\begin{lemma}[Iterates of homogenized SGD, Lemma B.2, \cite{paquetteSGD2021}] \label{lemma:iterates_hSGD} Let $\XX_t$ be the iterates generated under homogenized SGD. For any $t \ge 0$, 
\begin{equation} \begin{aligned} \label{eq:coordinate}
    \XX_t &= \bm{\mathscr{X}}_{\Gamma(t)}^{\text{gf}} + \int_0^t \gamma(s) \sqrt{ \tfrac{2}{n} \mathscr{L}(\XX_s) \AA^T \AA } e^{-(\AA^T \AA + \delta \II_d)(\Gamma(t)-\Gamma(s))} \, \dif \BB_s.
    % \YY_t &= e^{- (\SSigma^T \SSigma + \delta \II_d) \Gamma(t)} \YY_0 + \int_0^t \gamma(s) e^{-(\SSigma^T \SSigma + \delta \II_d) (\Gamma(t)-\Gamma(s))}  (\SSigma^T \UU^T \bb) \, \dif s \\
    % & \quad + \int_0^t \gamma(s) \sqrt{ \tfrac{2}{n} \widehat{\mathscr{L}}(\YY_s) \SSigma^T \SSigma } e^{-(\SSigma^T \SSigma + \delta \II_d)(\Gamma(t)-\Gamma(s))} \, \dif \BB_s.
\end{aligned}
\end{equation}
If the learning rate $\gamma(t) \equiv \gamma$ is constant, the expression simplifies to
\begin{equation}
\begin{aligned} \label{eq:constant_stepsize}
\XX_t &= \bm{\mathscr{X}}_{\gamma t}^{\text{gf}} + \gamma \int_0^t \sqrt{ \tfrac{2}{n} \mathscr{L}(\XX_s) \AA^T \AA } e^{-\gamma(\AA^T \AA + \delta \II_d)(t-s)} \, \dif \BB_s.
% \YY_t &= e^{-\gamma (\SSigma^T \SSigma + \delta \II_d)t} \YY_0 + (\SSigma^T \SSigma + \delta \II_d)^{-1} (\II_d - e^{-\gamma(\SSigma^T \SSigma + \delta \II_d)t})(\SSigma^T \UU^T \bb) \\
%     & \quad + \gamma \int_0^t \sqrt{ \tfrac{2}{n} \widehat{\mathscr{L}}(\YY_s) \SSigma^T \SSigma } e^{-\gamma(\SSigma^T \SSigma + \delta \II_d)(t-s)} \, \dif \BB_s.
\end{aligned}
\end{equation}
\end{lemma}

\begin{proof} Applying an integrating factor to \eqref{eq:HSGD_dif}, we have that
\begin{equation}
    \dif ( e^{\Gamma(t) (\AA^T \AA + \delta \II_d)} \XX_t ) = \gamma(t) e^{(\AA^T \AA + \delta \II_d)\Gamma(t)} \AA^T \bb \dif t + \gamma(t) \sqrt{\tfrac{2}{n} \mathscr{L}(\XX_t) \AA^T \AA } e^{\Gamma(t) (\AA^T \AA + \delta \II_d)} \dif \BB_{t}.
\end{equation}
By integrating both sides, we get that
\begin{align*}
    \XX_t &= e^{-(\AA^T \AA + \delta \II_d) \Gamma(t)} \XX_0 + \int_0^t \gamma(s) e^{-(\AA^T \AA + \delta \II_d)(\Gamma(t) - \Gamma(s))} \, \AA^T \bb \, \dif s\\
    & \quad + \int_0^t \gamma(s) \sqrt{ \tfrac{2}{n} \mathscr{L}(\XX_s) \AA^T \AA } e^{-(\AA^T \AA + \delta \II_d)(\Gamma(t)-\Gamma(s))} \, \dif \BB_s.
\end{align*}
By a simple change of variables (here set $s = \Gamma(s)$), the first two terms reduce to gradient flow but evaluated at time $\Gamma(t)$. This is precisely $\bm{\mathscr{X}}_{\Gamma(t)}^{\text{gf}}$. \end{proof}
% For constant learning rate (see \eqref{eq:constant_stepsize}), we used that 
% \[\int_0^t e^{-\gamma (\SSigma^T \SSigma + \delta \II_d) (t-s)} \, \dif s = \tfrac{1}{\gamma} (\SSigma^T \SSigma + \delta \II_d)^{-1} (1-e^{-\gamma(\SSigma^T \SSigma + \delta \II_d) t} ).\]

% The second expression immediately follows after squaring out the first expression, taking expectations, and using the isotropic initialization. 

% \begin{equation} \begin{aligned}
%     \YY_{t,j} = e^{-\gamma (\sigma_j^2 + \delta)t} \YY_{0,j} + \gamma \int_0^t e^{-\gamma(\sigma_j^2 + \delta)(t-s)} (\SSigma^T \UU^T \bb)_j \dif s + \gamma \sigma_j \sqrt{\tfrac{2}{n}}  \int_0^t \sqrt{\widehat{f}_1(\YY_s)} e^{-\gamma (\sigma_j^2 + \delta) (t-s)} \dif \BB_{s,j}
% \end{aligned}
% \end{equation}

The SDE \eqref{eqF:HSGD} evolves according to the least squares loss function $\mathscr{L}$ applied to homogenized SGD.  Consequently, the proof of Theorem~\ref{thm:trainrisk} (see \eqref{eq:concentration_hsgd_1}) begins by deriving the dynamics of the loss function $\mathscr{L}$ (see \cite{paquetteSGD2021, paquette2021dynamics} for a similar formula). From this, we give an expression for the quadratic statistic $\mathcal{R}$. We now prove Theorem~\ref{thm:trainrisk}.

% Using Lemma~\ref{lemma:iterates_hSGD} formulation of the iterates of SGD, we derive an expression for the exact dynamics of SGD applied to the least squares loss function $\mathscr{L}$. The result is below.  
% To account for the time dependent learning rate, we introduce the accumulated learning rate defined as 
% \begin{equation}
%     \Gamma(t) \defas \int_0^t \gamma(s) \, \dif s.
% \end{equation}
% One can view this quantity as the amount of time we run gradient flow (see Lemma~\ref{lemma:iterates_hSGD}). 

% In order to simplify some of the expressions, we place an independent assumption on the initialization $\xx_0$:

% \begin{assumption}[Initialization] \label{assumption: initialization} The initial vector $\xx_0 \in \mathbb{R}^d$ is chosen so that $\xx_0$ is a centered sub-Gaussian with $\|\xx_0\|^2 = \tfrac{d}{n} R^2$ where $R \in \mathbb{R}$.  
% \end{assumption}

\begin{proof}[Proof of Theorem~\ref{thm:trainrisk}] Define $\QQ_t \defas \exp( (\AA^T \AA + \delta \II_d) \Gamma(t))$. The proof will be broken down into six steps. \\

\noindent \emph{Step 1. Volterra equation for the expected loss}. It follows from Lemma~\ref{lemma:iterates_hSGD}
% Applying It\^o's rule to $\QQ_t \XX_t$, we derive
% \[
% \dif(\QQ_t \XX_t) = \gamma(t) \QQ_t \AA^T \bb \dif t
% + \gamma(t)\QQ_t \sqrt{\tfrac{2}{n}\mathscr{L}(\XX_t) \AA^T \AA}\dif \BB_t.
% \]
% Hence
% \[
% \QQ_t\XX_t= \QQ_0 \XX_0 + \int_0^t 
% \gamma(s) \QQ_s \AA^T \bb \dif s
% + \int_0^t \gamma(s)\QQ_s \sqrt{\mathscr{L}(\XX_s) \mathscr{M}_s + \mathscr{A}_s}\dif \BB_s.
% \]
% Note that on setting $\mathscr{M}_s=\mathscr{A}_s=0$, this gives gradient flow $\mathscr{X}_{\Gamma(t)}$, and hence we have representation
\[
\XX_t= \bm{\mathscr{X}}_{\Gamma(t)}^{\text{gf}}
+ \QQ_t^{-1} \int_0^t \gamma(s)\QQ_s \sqrt{\tfrac{2}{n} \mathscr{L}(\XX_s) \AA^T \AA}\dif \BB_s.
\]
Expanding the quadratic,
\begin{equation}\label{eq:rawL}
\begin{aligned}
\mathscr{L}( \XX_t)
&=
\mathscr{L}\bigl( 
\bm{\mathscr{X}}_{\Gamma(t)}^{\text{gf}} \bigr)
+
\nabla \mathscr{L}(\bm{\mathscr{X}}_{\Gamma(t)}^{\text{gf}})^T\QQ_t^{-1} \int_0^t \gamma(s)\QQ_s \sqrt{\tfrac{2}{n} \mathscr{L}(\XX_s) \AA^T \AA } \, \dif \BB_s
\\
&+
\frac{1}{2} \biggl\|\AA\QQ_t^{-1} \int_0^t \gamma(s)\QQ_s \sqrt{\tfrac{2}{n} \mathscr{L}(\XX_s) \AA^T \AA} \, \dif \BB_s\biggr\|^2.
\end{aligned}
\end{equation}
It follows that with $\Filt_t$ the sigma--algebra generated by $(\XX_0,\mathscr{L}, (\BB_s: 0 \leq s \leq t))$
if we compute the $\Filt_0$--conditional expectation, the Brownian integral vanishes, and we are left with two contributions from the second norm--squared process
\begin{equation}\label{eq:rawV}
\begin{aligned}
\Exp \bigl[ \mathscr{L}( \XX_t) ~\vert~ \Filt_0]
&=
\mathscr{L}\bigl( 
\bm{\mathscr{X}}_{\Gamma(t)}^{\text{gf}} \bigr)
% +
% \int_0^t 
% \gamma^2(s)
% \tr
% \biggl(
% \AA^T\AA
% \QQ_t^{-2}\QQ_s^2\mathscr{A}_s
% \biggr)
% \dif s
+
\int_0^t 
\gamma^2(s) n^{-1}
\tr
\biggl(
(\AA^T\AA)^2
\QQ_t^{-2}\QQ_s^2
\biggr)
\Exp \bigl[ \mathscr{L}( \XX_s) ~\vert~ \Filt_0]
\dif s.
\end{aligned}
\end{equation}
This is the claimed Volterra equation, $\Psi_t \defas \EE[\mathscr{L}(\XX_t) \, | \Filt_0]$ in \eqref{eqa:VLoss} with
\[K(t,s; \nabla^2 \mathscr{L}) = \gamma^2(s) n^{-1} \tr \bigg ( (\AA^T \AA)^2 \QQ_t^{-2} \QQ_s^2 \bigg ).
\]

\noindent \emph{Step 2. High probability boundedness of $\mathscr{L}$}.
We observe before beginning that many of the quantities that appear in the expressions above are bounded.
The gradient flow $\bm{\mathscr{X}}_{\Gamma(t)}^{\text{gf}}$ satisfies a uniform bound, solely in terms of its initial conditions and in particular the boundedness of $\mathscr{L}(\bm{\mathscr{X}}_{\Gamma(t)}^{\text{gf}})$ satisfies $\mathscr{L}(\bm{\mathscr{X}}_{\Gamma(t)}^{\text{gf}}) \le \mathscr{L}(\bm{\XX}_0)$.  The matrix $\QQ_t^{-1}\QQ_s$ is uniformly bounded in norm by $1$ for all $t \geq s$.  We have also assumed that $\|\AA\|$ and $\|\bb\|$ are bounded.
By applying It\^o's formula to the norm $u_t = \frac12 \|\XX_t\|^2$, we have from \eqref{eqF:HSGD} that
\[
\dif u_t
= -\gamma(t) \XX_t^T (\AA^T (\AA \XX_t - \bb) + \delta \XX_t) \dif t
+ \gamma(t)\XX_t^T \sqrt{\tfrac{2}{n} \mathscr{L}(\XX_t) \AA^T \AA}\dif \BB_t
+ \frac{\gamma^2(t)}{n}\tr\bigl(\mathscr{L}(\XX_t) \AA^T \AA\bigr)\dif t.
\]

%This has the same law as the strong solution of
%\[
%5\dif u_t
%= -\gamma(t) \XX_t^T (\AA^T (\AA \XX_t - \bb)) \dif t
%+ \sqrt{\tr \XX_t \XX_t^T \bigl(\mathscr{L}(\XX_t) \mathscr{M}_t + \mathscr{A}_t\bigt) }\dif W_t
%+ \frac{1}{2}\tr\bigl(\mathscr{L}(\XX_t) \mathscr{M}_t + \mathscr{A}_t\bigr)\dif t,
%\]
%for some univariate Brownian motion $(W_t : t \geq 0)$.
From the norm boundedness of $\AA$ and $\bb$, we can bound $\mathscr{L}(\XX_t) \leq 2\|\AA\|^2u_t + 2\|\bb\|^2 \leq C(u_t + 1)$.  Likewise, increasing $C$ as need be, using the boundedness of $\gamma(t)$ and $n^{-1} \tr (\AA^T \AA)$, we conclude
\[
\dif \langle u_t \rangle
= 2\gamma^2(t) n^{-1} \tr\bigl( \XX_t \XX_t^T \mathscr{L}(\XX_t) \AA^T \AA \bigr)
\leq C(u_t+1)^2.
\]
where $\langle \cdot \rangle$ is the quadratic variation \citep[Chapter IV]{revuz1999continuous}.
It follows that $z_t \defas \log( 1 + u_t ) - Ct$ is supermartingale with $\langle z_t \rangle \leq C$ for some sufficiently large $C$ and all $t \leq T$.  Hence with probability at least $1-e^{-2C(T)(\log d)^{3/2}}$,
\[
z_t \leq (\log d)^{3/4}
\]
for all $t \leq T$.
On this same event it follows for a sufficiently large constant $C>0$
\[
f(t) = \mathscr{L}(\XX_t) + \delta u_t 
\quad\text{and}
\quad
\mathscr{L}(\XX_t) \leq C(u_t + 1) \leq C^2e^{Ct+(\log d)^{3/4}}
\]
for all $t \leq T$.\\

%From boundedness of $\|\AA\|,$ $\|\QQ_t^{-1}\QQ_s\|$, and the boundedness of gradient flow, we have that there is a constant $C>0$ sufficiently large that in \eqref{eq:rawV},
%\[
%| \Exp \bigl[ \mathscr{L}( \XX_t) ~\vert~ \Filt_0] | \leq C^t.
%\]
%we compare the loss under $\mathscr{L}$
%Taking the difference

\noindent \emph{Step 3. Concentration of the loss}.  We may now control the difference of the loss from its expectation.  Specifically, in comparing 
\eqref{eq:rawL} and \eqref{eq:rawV}, we may express the difference $\Delta_t \defas \mathscr{L}(\XX_t)-\Exp[ \mathscr{L}(\XX_t) ~|~ \Filt_0]$ as
\begin{equation}\label{eq:Deltat}
\begin{aligned}
\Delta_t &= \MM_t^{(1)} + \MM_t^{(2)} + \int_0^t
\gamma^2(s) n^{-1}
\tr
\biggl(
(\AA^T\AA)^2
\QQ_t^{-2}\QQ_s^2
\biggr)
\Delta_s
\dif s, \quad\text{where} \\
\MM_t^{(1)}
&=\nabla \mathscr{L}(\bm{\mathscr{X}}_{\Gamma(t)}^{\text{gf}})^T\QQ_t^{-1} \int_0^t 
\gamma(s)\QQ_s \sqrt{\tfrac{2}{n} \mathscr{L}(\XX_s) \AA^T \AA } \, \dif \BB_s, \quad\text{and} \\
\MM_t^{(2)}
&=
\frac{1}{2} \biggl\|\AA\QQ_t^{-1} \int_0^t 
\gamma(s)\QQ_s \sqrt{ \tfrac{2}{n} \mathscr{L}(\XX_s) \AA^T \AA}\dif \BB_s\biggr\|^2
-
\int_0^t 
\gamma^2(s) n^{-1}
\tr
\biggl(
(\AA^T\AA)^2
\QQ_t^{-2}\QQ_s^2
\biggr) \mathscr{L}(\XX_s)
\dif s.
\end{aligned}
\end{equation}

We claim that both processes $\MM_t^{(1)}$ and $\MM_t^{(2)}$ are small, whose proof we defer.  Specifically, with probability $1-C(T)e^{-(\log d)^{3/2}}$ we have
\[
\max_{0 \leq t \leq T}\bigl\{ |\MM_t^{(1)}| + |\MM_t^{(2)}| \bigr\} \leq d^{-3\epsilon/4}.
\]
From the uniform boundedness in norm of $\AA^T\AA$, we then conclude from \eqref{eq:Deltat} for all $t \leq T$
\[
|\Delta_t|
\leq 
d^{-3\epsilon/4}
+\int_0^t \|\AA^T\AA\| |\Delta_s|\dif s.
\]
Using Gronwall's inequality, 
\[
|\Delta_t| \leq \|\AA^T\AA\|^{-1} \bigl( e^{\|\AA^T\AA\| t}-1\bigr) d^{-3\epsilon/4}.
\]
Thus we conclude by increasing the constants in the claimed bound that the desired inequality holds.\\

\noindent \emph{Step 4 (Deferred). Concentration of the martingales}. The quantities $\MM_t^{(1)}$ and $\MM_t^{(2)}$ are, unfortunately, not martingales. We rectify this by decoupling the integral dependence on $t$ and the other time $t$ dependent terms.  A meshing argument is then applied. We introduce two martingales, for each fixed $t \in [0,T]$,
\[
\begin{aligned}
\MM_u^{(1,t)}
&=
\nabla \mathscr{L}(\bm{\mathscr{X}}_{\Gamma(t)}^{\text{gf}})^T\QQ_t^{-1} \int_0^u 
\gamma(s)\QQ_s \sqrt{\tfrac{2}{n} \mathscr{L}(\XX_s) \AA^T \AA}\dif \BB_s. \\
\MM_u^{(2,t)}
&=
\frac{1}{2} \biggl\|\AA\QQ_t^{-1} \int_0^u 
\gamma(s)\QQ_s \sqrt{\tfrac{2}{n}\mathscr{L}(  \XX_s) \AA^T \AA} \dif \BB_s\biggr\|^2
-
\int_0^u 
\gamma^2(s)
n^{-1}\tr
\biggl(
(\AA^T\AA)^2
\QQ_t^{-2}\QQ_s^2(\mathscr{L}(\XX_s))
\biggr)
\dif s.
\end{aligned}
\]
We first show that if we fix any $t \leq T$, then for all $d$ sufficiently large with respect to $T$ and with probability at least $1-2e^{-(\log d)^{3/2}}$,
\[
\max_{0 \leq u \leq t}\bigl\{ |\MM_u^{(1,t)}| + |\MM_u^{(2,t)}| \bigr\} \leq d^{-7\epsilon/8}.
\]
We will then need to use a meshing argument to complete the argument.  We show the details for the first.  Those for the second are similar.

We simply need to bound the quadratic variation of each.  Note
\[
\langle \MM_u^{(1,t)} \rangle
=
\int_0^u 
2\gamma^2(s)
n^{-1} \tr\bigl(  \QQ_t^{-1} \QQ_s \nabla \mathscr{L}(\bm{\mathscr{X}}_{\Gamma(t)}^{\text{gf}}) \nabla \mathscr{L}(\bm{\mathscr{X}}_{\Gamma(t)}^{\text{gf}})^T \QQ_t^{-1} \QQ_s \AA^T \AA \mathscr{L}(\XX_s) \bigr)\dif u.
\]
Here we use the norm boundedness of $n^{-1} \|\AA^T \AA\|$ by $d^{-\epsilon}$.  We further bound the other terms in norm to produce
\[
\langle \MM_u^{(1,t)} \rangle
\leq
2d^{-\epsilon} \|\AA^T\AA\| \bigl(C^2 e^{Cu + (\log d)^{3/4}}\bigr) \mathscr{L}(\bm{\mathscr{X}}_{\Gamma(t)}^{\text{gf}}). %\tr(\mathscr{X}_{\Gamma(t)} \mathscr{X}_{\Gamma(t)}^T). 
\]
We note that $\mathscr{L}(\bm{\mathscr{X}}_{\Gamma(t)}^{\text{gf}}) \le \mathscr{L}(\bm{\mathscr{X}}_0)$. Hence with probability at least $1 - e^{-(\log d)^{3/2}}$ (for all $d$ sufficiently large with respect to $T,\|\AA\|,\|\bb\|, \epsilon$),
\[
\max_{0 \leq u \leq t}
|
\MM_u^{(1,t)}
|
\leq
d^{-7\epsilon/8}/2.
\]

\noindent \emph{Step 5 (Deferred). Mesh argument}. Finally, we use a union bound to gain the control from Step 4 over a mesh of $[0,T]$ of spacing $d^{-100}$.  From the union bound, we therefore have for all these mesh points $\{t_k\}$
\[
\max_k \max_{0 \leq u \leq t_k}\bigl\{ |\MM_u^{(1,t_k)}| + |\MM_u^{(2,t_k)}| \bigr\} \leq d^{-7\epsilon/8},
\]
and this holds with probability $1-2Td^{100}e^{-(\log d)^{3/2}}$.  For $t \in [t_k,t_{k+1}]$, we just use that
\[
\|
% \mathscr{X}_{\Gamma(t)}^T\AA^T\AA\QQ_t^{-1}
% -
% \mathscr{X}_{\Gamma(t_{k+1})}^T\AA^T\AA\QQ_{t_{k+1}}^{-1}\|
\nabla \mathscr{L}(\bm{\mathscr{X}}_{\Gamma(t)}^{\text{gf}})^T \QQ_t^{-1}
-
\nabla \mathscr{L}(\bm{\mathscr{X}}_{\Gamma(t_{k+1})}^{\text{gf}})^T \QQ_{t_{k+1}}^{-1}\|
\leq C(T,\AA, \bb)d^{-100},
\]
and thus on the event that $\mathscr{L}(\XX_s)$ is bounded, we have for $t \in [t_k,t_{k+1}]$
\[
|\MM_t^{(1)} - \MM_t^{(1,t_{k+1})}| \leq C(T,\AA, \bb)d^{-50}
\]
for all $d$ sufficiently large with respect to $T$, $\|\AA\|,$ and $\|\bb\|$.\\

\noindent \emph{Step 6. Other quadratics}. Hence, if we take $\Psi_t$ as a solution to the Volterra equation
\[
\Psi_t
=
\mathscr{L}\bigl( 
\bm{\mathscr{X}}_{\Gamma(t)}^{\text{gf}}\bigr)
+
\int_0^t 
\gamma^2(s) n^{-1}
\tr
\biggl(
(\AA^T\AA)^2
\QQ_t^{-2}\QQ_s^2
\biggr)
\Psi_s
\dif s,
\]
then we have a high--quality approximation for the loss $\mathscr{L}(\XX_t)$, and moreover, applying It\^o's equation, we may always represent another quadratic $\mathcal{R} :\R^d \to \R$,
\[
\begin{aligned}
\mathcal{R}( \XX_t)
&=
\mathcal{R}\bigl( 
\bm{\mathscr{X}}_{\Gamma(t)}^{\text{gf}}\bigr)
+
\nabla \mathcal{R}(\bm{\mathscr{X}}_{\Gamma(t)}^{\text{gf}})^T \QQ_t^{-1} \int_0^t 
\gamma(s)
\QQ_s \sqrt{\tfrac{2}{n} \AA^T \AA \mathscr{L}(\XX_s)}\dif \BB_s
% \mathscr{X}_{\Gamma(t)}^T(\nabla^2 \mathcal{E})\QQ_t^{-1} \int_0^t 
% \gamma(s)
% \QQ_s \sqrt{\mathscr{L}(\XX_s) \mathscr{M}_s + \mathscr{A}_s}\dif \BB_s
\\
&+ \frac{1}{2} \bigg (\QQ_t^{-1} \int_0^t \gamma(s)
\QQ_s \sqrt{\tfrac{2}{n} \AA^T \AA \mathscr{L}(\XX_s)}\dif \BB_s \bigg )^T (\nabla^2 \mathcal{R}) \QQ_t^{-1} \int_0^t \gamma(s)
\QQ_s \sqrt{\tfrac{2}{n} \AA^T \AA \mathscr{L}(\XX_s)}\dif \BB_s
% \frac{1}{2} \biggl\|\sqrt{ (\nabla^2 \mathcal{E}) }\QQ_t^{-1} \int_0^t \gamma(s)
% \QQ_s \sqrt{\mathscr{L}(\XX_s) \mathscr{M}_s + \mathscr{A}_s}\dif \BB_s\biggr\|^2.
\end{aligned}
\]
By comparing this to the same expression, where we replace the losses $\mathscr{L}(\XX_s)$ by $\Psi_s$ and compute expectations over the Brownian terms, we arrive at (compare \eqref{eq:rawV})
\[
\begin{aligned}
\mathcal{R}( \XX_t)
&=
\MM_t^{(3)}
+
\mathcal{R}\bigl( 
\bm{\mathscr{X}}_{\Gamma(t)}^{\text{gf}}\bigr)
+ 
\int_0^t 
\gamma^2(s) n^{-1}
\tr
\biggl(
(\nabla^2 \mathcal{R})
\AA^T \AA \QQ_t^{-2}\QQ_s^2 
\biggr)
\Psi_s
\dif s.
\end{aligned}
\]
Provided the Hessian $(\nabla^2 \mathcal{R})$ and gradient $\nabla \mathcal{R}(0)$ are bounded independently of $d$ uniformly on $T$, the concentration of $\MM_t^{(3)}$ now follows exactly as in Steps 4 and 5.
\end{proof}

\section{Main technical argument}\label{sec:comp_SGD_HSGD} 
In this section, %we show that the statistic $\tilde{q}$ defined in Assumption~\ref{assumption: quadratics} applied to the iterates of SGD $\xx_t$ and homogenized SGD $\XX_t$ are close, that is, the difference $\tilde{q}(\xx_t)-\tilde{q}(\XX_t)$ is small with overwhelming probability. We make this precise in below.
we will prove Theorem~\ref{thm:homogenized_SGD_SGD} and Theorem~\ref{thm:expectation} below after we have introduced some notation and lemmas. Note that the statement is equivalent to proving a result about the statistic $g$ under the change of basis iterates $\nnu$ and $\YY$. The proof of Theorem~\ref{thm:homogenized_SGD_SGD} will use an elaborate net argument --- we construct a set of quadratics which contains our chosen statistic $g$ (or $\tilde{q}$ after a change of basis) and we show that the difference between homogenized SGD \eqref{eq:ito_diffusion} and SGD \eqref{eq:SGD_statistic} over this entire class of quadratics $Q$ is small. 

\begin{remark} In what proceeds, we work in the basis formed by the vectors of $\VV^T$ (spectral decomposition on the data matrix $\AA = \UU \SSigma \VV^T$), or equivalently, $\vv = \VV^T \xx$. Consequently the statistic \eqref{eq:statistic} in the basis of $\VV^T$ is
\begin{equation} \label{eq:statistic_1}
    \mathcal{R}(\VV \vv) \defas \widehat{\mathcal{R}}(\vv) = \frac{1}{2} \vv^T \SS \vv + \hh^T \vv + c, \quad \text{where $\SS \defas \VV^T\TT \VV$ and $\hh \defas \VV^T \uu$.}
\end{equation}
When referring to Assumption~\ref{assumption: quadratics} we will use this formulation \eqref{eq:statistic_1} of the statistic.
\end{remark}

It will be convenient to introduce the function $\widehat{f}: \mathbb{R}^d \to \mathbb{R}$ by
\begin{equation} \label{eq:rr_simplified}
    \widehat{f}(\vv) \defas \widehat{\mathscr{L}}(\vv) + \tfrac{\delta}{2} \|\vv\|^2 =  \tfrac{1}{2} \|\SSigma \vv - \UU^T \bb\|^2 + \tfrac{\delta}{2} \|\vv\|^2. 
\end{equation}  

First to define this set of quadratics, let $\widehat{g} : \mathbb{R}^d \to \mathbb{R}$ be any generic quadratic $\widehat{g}(\xx) = \xx^T \widehat{\SS} \xx + \widehat{\hh}^T \xx + c$ for which $\widehat{\SS}$ is symmetric, deterministic and, the norm $\|\widehat{\SS} \|$ is bounded independent of $n$ and $\|\widehat{\hh}\|_2 = \|\nabla \widehat{g} (0)\|_2$ is deterministic, bounded independent of $n$.
% for which $\widehat{g}(\xx) = \xx^T \widehat{\SS} \xx + \widehat{\hh}^T \xx + \widehat{c}$ where the $\widehat{\SS}$ is a deterministic symmetric positive semi-definite matrix and the deterministic vector $\widehat{\hh} \in \mathbb{R}^d$ and constant $\widehat{c}$ are arbitrary. 
We want to define a set of quadratics that contains $\widehat{g}$ and whose cardinality is small compared with the ``ball" of quadratics. The set of quadratics must satisfy the property that for any quadratic $q$ in the set the expression $\nabla q(\xx)^T \nabla \widehat{f}(\xx) - \nabla q(\yy)^T \nabla \widehat{f}(\yy)$ is small if $q(\xx) - q(\yy)$ is small. Note here we can not appeal to continuity because we do not know that $\xx$ and $\yy$ are close. To do so, we introduce some notation
\begin{align}
    \MM \defas \SSigma^T \SSigma + \delta \II_d \qquad \text{and} \qquad \RR(z) \defas (\II_d - z \MM)^{-1} \quad \text{for $z \in \mathbb{C}$}.
\end{align}
We can relate $\RR(z)$ to the resolvent of $\AA^T \AA$ by the following
\begin{equation} \label{eq:resolvent_relation}
    z \VV^T [ R(z) - \delta^{-1} \II_d] \VV = \RR(1/z) \quad \text{for $z \in \mathbb{C} \setminus \{0\}$.}
\end{equation}
For $z,y \in \mathbb{C}$, we define a set of quadratics based on the fixed quadratic $\widehat{g}$ by
\begin{equation} \begin{aligned} \label{eq:quadratic_set_hatg}
      Q_{\widehat{g}} &\defas \bigg \{ \xx^T (\nabla^2 \widehat{g}) \xx, \{\xx^T \RR(z) (\nabla^2 \widehat{g}) \RR(y) \xx \, : \, |z| = |y| = (2\|\MM\|)^{-1}\},\\
      & \qquad \quad  (\nabla \widehat{g}(0))^T\xx, \{ (\nabla \widehat{g}(0))^T \RR(z) \xx \, : \, |z| = (2\|\MM\|)^{-1} \}, ( (\nabla^2 \widehat{g}) \SSigma^T \UU^T \bb)^T \xx, \\
      & \qquad \quad \{ \xx^T \RR(z) (\nabla^2 \widehat{g}) \RR(y) \SSigma^T \UU^T \bb \, : \, |z| = |y| =  (2\|\MM\|)^{-1}\}, \widehat{g}(\xx) \bigg \}.
    %   Q_{\widehat{g}} &\defas \bigg \{ \xx^T \widehat{\SS} \xx, \{\xx^T \RR(z) \widehat{\SS} \RR(y) \xx \, : \, |z| = |y| = (2\|\MM\|)^{-1}\},\\
    %   & \qquad \quad  (\nabla \widehat{g}(0))^T\xx, \{ (\nabla \widehat{g}(0))^T \RR(z) \xx \, : \, |z| = (2\|\MM\|)^{-1} \}, ( \widehat{\SS} \SSigma^T \UU^T \bb)^T \xx, \\
    %   & \qquad \quad \{ \xx^T \RR(z) \widehat{\SS} \RR(y) \SSigma^T \UU^T \bb \, : \, |z| = |y| =  (2\|\MM\|)^{-1}\}, \widehat{g}(\xx) \bigg \}.
\end{aligned} \end{equation}
We recall the gradient of the loss function \eqref{eq:rr}, $\nabla \widehat{f}(\xx) = (\SSigma^T \SSigma + \delta \II_d) \xx - \SSigma^T \UU^T \bb = \MM \xx - \SSigma^T \UU^T \bb$. Using this set $Q_{\widehat{g}}$, we show it is ``closed" under this ``gradient mapping". As we will need the stopped processes $\nnu_t^{\vartheta}$ and $\YY_t^{\vartheta}$ later to show Theorem~\ref{thm:homogenized_SGD_SGD}, we will work under this stopping time (see \eqref{eq:stopping_time_norm}). 

\begin{proposition}[$Q_{\widehat{g}}$ closed under gradient mapping] \label{prop:quad_set} Let $\widehat{g} : \mathbb{R}^d \to \mathbb{R}$ be any quadratic such that $\widehat{g}(\xx) = \xx^T \widehat{\SS} \xx + \widehat{\hh}^T \xx + c$ for which $\widehat{\SS}$ is symmetric with norm $\|\widehat{\SS} \|$ bounded, independent of $n$ and $\|\widehat{\hh}\|_2 = \|\nabla \widehat{g} (0)\|_2$ is bounded, independent of $n$. For any $T > 0$, define
\begin{equation} \label{eq:max_quadratic}
    \mathcal{Q}_T^{\widehat{g}} \defas \sup_{q \in Q_{\widehat{g}}} \sup_{0 \le t \le T} |q(\nnu_t^{\vartheta}) - q(\YY_t^{\vartheta})|,
\end{equation}
where $\nnu_t$ and $\YY_t$ are the iterates of SGD and homogenized SGD in the $\VV$ basis, respectively and $\vartheta$ is the stopping time defined in \eqref{eq:stopping_time_norm}. For the function $\widehat{f}(\cdot) = \tfrac{1}{2} \|\SSigma \cdot - \UU^T \bb\|^2 + \frac{\delta}{2} \|\cdot\|^2$, 
\begin{align}
    \sup_{q \in Q_{\widehat{g}}} &\sup_{0 \le t \le T} | \nabla q(\nnu_t^{\vartheta})^T \nabla \widehat{f}(\nnu_t^{\vartheta}) - \nabla q(\YY_t^{\vartheta})^T \nabla \widehat{f}(\YY_t^{\vartheta}) | \le C \cdot \mathcal{Q}_T^{\widehat{g}}, \label{eq:grad_condition}\\
     \sup_{q \in Q_{\widehat{g}}} &\sup_{0 \le t \le T} | (\nnu_t^{\vartheta})^T \nabla^2 q(\nnu_t^{\vartheta}) \SSigma^T( \SSigma \nnu_t^{\vartheta}- \UU^T \bb) -  (\YY_t^{\vartheta})^T \nabla^2 q(\YY_t^{\vartheta}) \SSigma^T( \SSigma \YY_t^{\vartheta}- \UU^T \bb)| \le C \cdot \mathcal{Q}_T^{\widehat{g}}, \label{eq:function_condition} \\
     \sup_{q \in Q_{\widehat{g}}} &\sup_{0 \le t \le T} | (\nnu_t^{\vartheta})^T \nabla^2 q(\nnu_t^{\vartheta}) \nnu_t^{\vartheta} - (\YY_t^{\vartheta})^T \nabla^2 q(\YY_t^{\vartheta}) \YY_t^{\vartheta}| \le 2\mathcal{Q}_T^{\widehat{g}} \label{eq:hessian_bound}
\end{align}
where the constants $C \defas C(\delta, \SSigma, \bb)$ are independent of $n$ and depend only on the matrix $\SSigma$, the targets $\bb$, and the $\ell^2$-regularization parameter $\delta$. Moreover, we have that 
\begin{equation} \label{eq:operator_bound}
    \sup_{q \in Q_{\widehat{g}}} \|\nabla^2 q\| \le 4 \|\nabla^2 \widehat{g}\| \quad \text{and} \quad \sup_{q \in Q_{\widehat{g}}} \|\nabla q(0)\|_2 \le C \cdot (\|\nabla^2 \widehat{g}\| + \|\nabla \widehat{g}(0)\|_2), 
\end{equation}
where the constant $C = C(\delta, \SSigma, \bb)$ is independent of $n$. 
\end{proposition}

\begin{proof} For any $q \in Q_{\widehat{g}}$ with $q(\xx) = \xx^T \tilde{\SS} \xx$ for some matrix $\tilde{\SS}$, we see that \eqref{eq:function_condition} can be written as
\begin{align*}
    \xx^T \nabla^2q(\xx) \SSigma^T (\SSigma \xx - \UU^T \bb) = \nabla q(\xx)^T \nabla \widehat{f}(\xx) - \delta \xx^T (\tilde{\SS} + \tilde{\SS}^T) \xx . 
\end{align*}
If we show that \eqref{eq:grad_condition} holds, we immediately conclude \eqref{eq:function_condition} as all terms can be bounded by $\mathcal{Q}_T^{\widehat{g}}$ as $\xx^T \tilde{\SS} \xx, \xx^T \tilde{\SS}^T \xx \in Q_{\widehat{g}}$. Note we immediately satisfy \eqref{eq:function_condition} for linear functions of the form $\tilde{\hh}^T \xx$ since the Hessian is identically $0$ in this case. The only other $q \in Q_{\widehat{g}}$ is $\widehat{g}$. By construction, there exists a $q \in Q_{\widehat{g}}$ such that $q \neq \widehat{g}$ and it has the form $q(\xx) = \xx^T \tilde{\SS} \xx$ which has the same Hessian as that of $\widehat{g}$. It follows that \eqref{eq:function_condition} holds for $\widehat{g}$ provided \eqref{eq:grad_condition}. A similar argument holds for \eqref{eq:hessian_bound}: for any $q\in Q_{\widehat{g}}$ with $q(\xx) = \xx^T \tilde{\SS} \xx$, one has that $ \xx^T (\nabla^2 q) \xx = \xx^T \tilde{\SS} \xx + \xx^T \tilde{\SS}^T \xx$ and $\xx^T \tilde{\SS} \xx, \xx^T \tilde{\SS}^T \xx \in Q_{\widehat{g}}$. Pure linear functions and the fixed quadratic $\widehat{g}$ satisfy \eqref{eq:hessian_bound} for the same reason as \eqref{eq:function_condition}.

We now turn to showing \eqref{eq:grad_condition} holds and we do so by cases. Let $q(\xx) = \xx^T (\nabla^2\widehat{g}) \xx$ so that $\nabla q(\xx)^T \nabla \widehat{f}(\xx) = 2\xx^T (\nabla^2 \widehat{g}) \MM \xx - 2\xx^T (\nabla^2 \widehat{g}) \SSigma^T \UU^T \bb$ and $\nabla^2 q(\xx) = 2 (\nabla^2 \widehat{g})$. We need to show that the individual quadratics are (up to constants) quadratics that live in the set $Q_{\widehat{g}}$. It is clear that $\xx^T (\nabla^2 \widehat{g}) \SSigma^T \UU^T \bb$ is in the set $Q_{\widehat{g}}$. Now we need to show $\xx^T (\nabla^2 \widehat{g}) \MM \xx \in Q_{\widehat{g}}$. First, we remark that $\RR(z) = (\II - z \MM)^{-1} = \sum_{k=0}^\infty z^k \MM^k$ for $|z| < 1/ \|\MM\|$. Let $\Omega$ be the circle in $\mathbb{C}$ centered at $0$ with radius $(2 \|\MM\|)^{-1}$, which ensures that $(\II - z \MM)^{-1}$ can be represented by the power series. By Cauchy's integral formula, we have
\begin{equation}
    \xx^T (\nabla^2 \widehat{g}) \MM \xx = \frac{1}{4 \pi i} \int_{\Omega} \frac{\xx^T \RR(z) (\nabla^2 \widehat{g}) \RR(z) \xx}{z^2} \, \dif z.
\end{equation}
In particular, we have that
\begin{align*}
    |(\nnu_t^{\vartheta})^T (\nabla^2 \widehat{g}) \MM \nnu_t^{\vartheta} &- (\YY_t^{\vartheta})^T  (\nabla^2 \widehat{g}) \MM \YY_t^{\vartheta} | \\
    &\le \frac{1}{4 \pi} \int_{\Omega} \frac{|(\nnu_t^{\vartheta})^T \RR(z) (\nabla^2 \widehat{g}) \RR(z) \nnu_t^{\vartheta} - (\YY_t^{\vartheta})^T \RR(z) (\nabla^2 \widehat{g}) \RR(z) \YY_t^{\vartheta} |}{|z|^2} \, |\dif z|\\
    &\le \frac{2 \|\MM\| \mathcal{Q}_T^{\widehat{g}}}{4\pi}  \int_0^{2\pi} \, \dif t =  \|\MM\| \mathcal{Q}_T^{\widehat{g}}. 
\end{align*}
This gives \eqref{eq:grad_condition} result for $\xx^T (\nabla^2 \widehat{g}) \xx$.

Suppose $q(\xx) = \xx^T \RR(z) (\nabla^2 \widehat{g}) \RR(y) \xx$ for some $z, y$ such that $|z| = |y| = (2 \|\MM\|)^{-1}$. A simple computation shows that 
\begin{equation}
\begin{aligned} \label{eq:RSR_1}
    \nabla q(\xx)^T \nabla \widehat{f}(\xx) &= \xx^T \RR(y) (\nabla^2 \widehat{g}) \RR(z) \MM \xx + \xx^T \RR(z) (\nabla^2 \widehat{g}) \RR(y) \MM \xx\\
    & \qquad - \xx^T \RR(y) (\nabla^2 \widehat{g}) \RR(z) \SSigma^T \UU^T \bb - \xx^T \RR(z) (\nabla^2 \widehat{g}) \RR(y) \SSigma^T \UU^T \bb.
\end{aligned}
\end{equation}
The third and fourth terms in the sum of \eqref{eq:RSR_1} exist in the set $Q_{\widehat{g}}$. For the first term, we have that
\begin{align} \label{eq:RSR_2}
    \xx^T \RR(y) (\nabla^2 \widehat{g}) \RR(z) \MM \xx &= \xx^T \RR(y) (\nabla^2 \widehat{g}) \RR(z) (\MM - \tfrac{1}{z}\II_d + \tfrac{1}{z} \II_d ) \xx\\
    &= -\frac{\xx^T \RR(y) (\nabla^2 \widehat{g}) \xx}{z} + \frac{\xx^T \RR(y) (\nabla^2 \widehat{g}) \RR(z) \xx}{z}.
\end{align}
The second term in \eqref{eq:RSR_2} with $|z| = (2 \|\MM\|)^{-1}$ can be bounded by a term in $Q_{\widehat{g}}$ so we only need to consider the first term. By Cauchy's integral formula, we deduce that
\begin{align}
    |z|^{-1} |(\nnu_t^{\vartheta})^T &\RR(y) (\nabla^2 \widehat{g}) \nnu_t^{\vartheta} - (\YY_t^{\vartheta})^T \RR(y) (\nabla^2 \widehat{g}) \YY_t^{\vartheta}|\\
    &\le 
    2 \|\MM\| \left | \frac{1}{2 \pi i} \int_{\Omega} \frac{(\nnu_t^{\vartheta})^T \RR(y) (\nabla^2 \widehat{g}) \RR(z) \nnu_t^{\vartheta} - (\YY_t^{\vartheta})^T \RR(y) (\nabla^2 \widehat{g}) \RR(z) \YY_t^{\vartheta}}{z} \, \dif z \right |\\
    &\le 2\|\MM\| \mathcal{Q}_T^{\widehat{g}}.
\end{align}
This completes the result for $\xx^T \RR(y) (\nabla^2 \widehat{g}) \xx$. By reversing the roles of $z$ and $y$, we also have shown the result for $\xx^T \RR(z) (\nabla^2 \widehat{g}) \xx$. Consequently, we have show that $q(\xx) = \xx^T \RR(z) (\nabla^2 \widehat{g}) \RR(y) \xx$ satisfies \eqref{eq:grad_condition}.

Next we let $q(\xx) = (\nabla g(0))^T \xx$ so that $\nabla q(\xx)^T \nabla \widehat{f}(\xx) = (\nabla \widehat{g}(0))^T \MM \xx - (\nabla \widehat{g}(0))^T \SSigma^T \UU^T \bb$. The constant cancels when we look at the difference; hence it is only the first term we consider. As before, we will write $(\nabla \widehat{g}(0))^T \MM \xx$ in terms of the resolvent. Consequently, we deduce
\begin{align*}
    |(\nabla \widehat{g}(0))^T \MM \nnu_t^{\vartheta} - (\nabla \widehat{g}(0))^T \MM \YY_t^{\vartheta}| &= |(2 \pi i)^{-1}| \left | \int_{\Omega} \frac{(\nabla \widehat{g}(0))^T \RR(z) \nnu_t^{\vartheta} - (\nabla \widehat{g}(0))^T \RR(z) \YY_t^{\vartheta}}{z^2} \, \dif z \right |\\
    &\le 2\|\MM\| \mathcal{Q}_T^{\widehat{g}},
\end{align*}
for $(\nabla \widehat{g}(0))^T \RR(z) \xx$ is in the set $Q_{\widehat{g}}$. The result follows for $q(\xx) = (\nabla \widehat{g}(0))^T \xx$. 

Next, let $q(\xx) = ((\nabla^2 \widehat{g}) \SSigma^T \UU^T \bb)^T \xx$. Analogous to $(\nabla g(0))^T \xx$, we only need to show that $((\nabla^2 \widehat{g}) \SSigma^T \UU^T \bb)^T \MM \xx \in Q_{\widehat{g}}$. For this term, we use Cauchy's integral formula
\begin{align*}
    (\SSigma^T \UU^T \bb)^T (\nabla^2 \widehat{g}) \MM \xx = \frac{1}{(2\pi i)^2} \int_{\Omega} \int_{\Omega} \frac{\xx^T \RR(z) (\nabla^2 \widehat{g}) \RR(y) \SSigma^T \UU^T \bb}{z^2 y} \, \dif z \dif y.
\end{align*}
From this, we deduce the following
\begin{align*}
    |4 \pi^2|^{-1} \left |\int_{\Omega} \int_{\Omega} \frac{(\nnu_t^{\vartheta})^T \RR(z) (\nabla^2 \widehat{g}) \RR(y) \SSigma^T \UU^T \bb - (\YY_t^{\vartheta})^T \RR(z) (\nabla^2 \widehat{g}) \RR(y)\SSigma^T \UU^T \bb}{z^2 y} \, \dif z \dif y \right | \le 2 \|\MM\| \mathcal{Q}_T^{\widehat{g}}.
\end{align*}
It follows that \eqref{eq:grad_condition} holds for $q(\xx) = ((\nabla^2 \widehat{g}) \SSigma^T \UU^T \bb)^T \xx$.

We let $q(\xx) = \xx^T \RR(z) (\nabla^2 \widehat{g}) \RR(y) \SSigma^T \UU^T \bb$ so that 
\[
\nabla q(\xx)^T \nabla \widehat{f}(\xx) = \xx^T \MM \RR(z) (\nabla^2 \widehat{g}) \RR(y) \SSigma^T \UU^T \bb - (\SSigma^T \UU^T \bb)^T \RR(z) (\nabla^2 \widehat{g}) \RR(y) \SSigma^T \UU^T \bb.
\]
The last term is a constant which will disappear when we take the difference. Now by adding and subtracting $z^{-1} \II_d$, we can write

\begin{align*}
    \xx^T \MM \RR(z) (\nabla^2 \widehat{g}) \RR(y) \SSigma^T \UU^T \bb = -z^{-1} \xx^T (\nabla^2 \widehat{g}) \RR(y) \SSigma^T \UU^T \bb + z^{-1} \xx^T \RR(z) (\nabla^2 \widehat{g}) \RR(y) \SSigma^T \UU^T \bb.
\end{align*}
Note that $|z|^{-1} = 2 \|\MM\|$ and thus the second term lives in the set $Q_{\widehat{g}}$. For the first term, we need to use Cauchy's Integral formula, that is, 
\begin{align*}
    |z|^{-1} |(&\nnu_t^{\vartheta})^T (\nabla^2 \widehat{g}) \RR(y) \SSigma^T \UU^T \bb - (\YY_t^{\vartheta})^T (\nabla^2 \widehat{g}) \RR(y) \SSigma^T \UU^T \bb |\\
    &\le 2 \|\MM\| \frac{1}{2\pi} \left | \int_{\Omega} \frac{(\nnu_t^{\vartheta})^T \RR(z) (\nabla^2 \widehat{g}) \RR(y) \SSigma^T \UU^T \bb-(\YY_t^{\vartheta})^T \RR(z) (\nabla^2 \widehat{g}) \RR(y) \SSigma^T \UU^T \bb}{z} \, \dif z \right |\\
    &\le 2 \|\MM\| \mathcal{Q}_{T}^{\widehat{g}}.
\end{align*}

Lastly, we let $q(\xx) = \widehat{g}(\xx)$. By construction, each term of $\nabla \widehat{g}(\xx)^T \nabla \widehat{f}(\xx)$ is in $Q_{\widehat{g}}$. This proves \eqref{eq:grad_condition}.

It remains to show that the norm of the Hessian and the norm of the gradient are uniformly bounded. We observe that for $|z| = (2 \|\MM\|)^{-1}$, the operator norm of the resolvent is $\|\RR(z)\| \le \sum_{k=0}^\infty |z|^{k} \|\MM\|^k \le 2$. It follows by Cauchy-Schwarz and submultiplicative norm bounds that,
$\|\RR(y) (\nabla^2 \widehat{g}) \RR(z)\| \le 4 \|\nabla^2 \widehat{g}\|$, which is sufficient for showing \eqref{eq:operator_bound} and the corresponding uniform bound for the gradient.
\end{proof}

\begin{corollary} \label{cor:f_1_bar_q} Fix a constant $T > 0$ and let $\widehat{g}(\cdot) = \|\cdot\|^2_2$. For this $\widehat{g}$ construct the set $Q_{\|\cdot\|^2_2}$ and constant $\mathcal{Q}_T^{\|\cdot\|^2_2}$ as in \eqref{eq:quadratic_set_hatg} and \eqref{eq:max_quadratic} respectively. Then there exists a constant $C \defas C(\delta, \SSigma)$ such that
\begin{equation}
   \sup_{0 \le t \le T} \int_0^t | \widehat{\mathscr{L}}(\nnu_s^{\vartheta})-\widehat{\mathscr{L}}(\YY_s^{\vartheta}) | \, \dif  s \le C \int_0^T \mathcal{Q}_s^{\|\cdot\|^2} \, \dif s,
\end{equation}
where $\widehat{\mathscr{L}}(\cdot) = \|\SSigma \cdot - \UU^T \bb\|^2_2$.
\end{corollary}

\begin{proof} First, we observe that 
\[
    \widehat{\mathscr{L}}(\xx) = \xx^T \SSigma^T \SSigma \xx - 2 \xx^T \SSigma^T \UU^T \bb + \bb^T \bb = \xx^T \MM \xx - \delta \xx^T \xx - 2\xx^T \SSigma^T \UU^T \bb + \bb^T \bb.
\]
The constant $\bb^T \bb$ can be ignored for we take the difference. Both $\xx^T \SSigma^T \UU^T \bb$ and $\xx^T \xx$ are in $Q_{\|\cdot\|^2}$ so the difference evaluated at $\nnu_t^{\vartheta}$ and $\YY_t^{\vartheta}$ can be bounded by $\mathcal{Q}_T^{\|\cdot\|^2}$. As for $\xx^T \MM\xx$, we can express it using Cauchy's Integral formula,
\begin{equation} \label{eq:blah_11}
    \xx^T \MM \xx = \frac{1}{(2 \pi i)} \int_{\Omega} \int_{\Omega} \frac{\xx^T \RR(z) \RR(y) \xx}{z^2 y} \, \dif y \dif z,
\end{equation}
where $|z| = |y| = (2 \|\MM\|)^{-1}$ and $\Omega$ is a circle centered at $0$ of radius $(2 \|\MM\|)^{-1}$. It follows by \eqref{eq:blah_11} that 
\begin{equation} \begin{aligned}
    \sup_{0 \le t \le T} \int_0^t &| (\nnu_s^{\vartheta})^T \MM \nnu_s^{\vartheta} - (\YY_s^{\vartheta})^T \MM \YY_s^{\vartheta} | \, \dif s \\
    &\le \sup_{0 \le t \le T} (4 \pi^2)^{-1} \int_0^t \int_{\Omega} \int_{\Omega} \frac{|(\nnu_s^{\vartheta})^T \RR(z) \RR(y) \nnu_s^{\vartheta} - (\YY_s^{\vartheta})^T \RR(z) \RR(y) \YY_s^{\vartheta}|}{|y| |z|^2} \, \dif |z| \dif |y| \dif s \\
    &\le 2 \|\MM\| \int_0^T \mathcal{Q}_s^{\|\cdot\|^2} \dif s.
\end{aligned} \end{equation}
The result immediately follows. 
\end{proof}

We are now ready to define the set of quadratics. Using the notation established in \eqref{eq:quadratic_set_hatg} and \eqref{eq:max_quadratic} for the statistic $\widehat{\mathcal{R}}$ (resp. $\mathcal{R}$ under a change of basis) and $\|\cdot\|^2_2$, for any $T > 0$ and sequences $\nnu_t$ and $\YY_t$ of SGD and homogenized SGD respectively, we define
\begin{equation} \begin{gathered}  \label{eq:quad_sup}
    Q \defas Q_{\widehat{\mathcal{R}}} \cup  Q_{\|\cdot\|^2_2}, \qquad
    \mathcal{Q}_T \defas \sup_{q \in Q} \sup_{0 \le t \le T} |q(\nnu_t^{\vartheta}) - q(\YY_t^{\vartheta})| = \sup \{ \mathcal{Q}_T^{\widehat{\mathcal{R}}}, \mathcal{Q}_T^{\|\cdot\|^2} \},\\
    \sup_{q \in Q} \|\nabla^2 q\| \le C \max \{\|\nabla^2 \widehat{\mathcal{R}}\|, 2\}, \quad \text{and} \quad \sup_{q \in Q} \|\nabla q(0)\|_2 \le C \max\{\|\nabla^2 \widehat{\mathcal{R}} \| + \|\nabla \widehat{\mathcal{R}}(0)\|_2, 2\},
\end{gathered}
\end{equation}
where $C = C(\delta, \SSigma, \bb)$ is a constant independent of $n$.

\begin{proposition} \label{prop:bar{q}} Let $q \in Q$ where $Q$ is defined in \eqref{eq:quad_sup} and define $\mathcal{Q}_T$ as in \eqref{eq:quad_sup}. Suppose the quadratic statistic $\widehat{\mathcal{R}} : \mathbb{R}^d \to \mathbb{R}$ in \eqref{eq:statistic_1} (resp. $\mathcal{R}$ under a change of basis) satisfying Assumption~\ref{assumption: quadratics}. For any $T > 0$, 
\begin{equation} 
\begin{aligned}
    \sup_{0 \le t \le T} |q(\nnu_t^{\vartheta}) - q(\YY_t^{\vartheta})| &\le C(\delta, \widehat{\gamma}, \nabla^2 \widehat{\mathcal{R}}, \SSigma, \bb)  \int_0^T \mathcal{Q}_s \, \dif s\\
    & \qquad + \widehat{C}(\delta, \widehat{\gamma}, \nabla^2 \widehat{\mathcal{R}}, \SSigma, \bb) \cdot \left ( \int_0^T \mathcal{Q}_s \, \dif s + T \right ) \cdot n^{\varepsilon-1} + \mathcal{E}_{\text{errors}}^T(q),
\end{aligned}
\end{equation}
where the errors $\mathcal{E}_{\text{errors}}^T$ are defined by
\begin{equation}
    \begin{aligned} \label{eq:errors}
        \mathcal{E}_{\text{errors}}^T(q) \defas  \sup_{0 \le t \le T} \big \{  |\mathcal{M}_{t \wedge \vartheta}^{\text{grad}}(q)| + |\mathcal{M}_{t\wedge \vartheta}^{\text{quad}}(q)| + | \mathcal{M}_{t \wedge \vartheta}^{\text{hSGD}}(q)| + |\mathcal{E}_{t \wedge \vartheta}^{\text{KL}}(q)| \big \}.
    \end{aligned}
\end{equation}
Here the constants $C(\delta, \widehat{\gamma}, \nabla^2 \widehat{\mathcal{R}}, \SSigma, \bb)$, $\widehat{C}(\delta, \widehat{\gamma}, \nabla^2 \widehat{\mathcal{R}}, \SSigma, \bb)$ are independent of $n$ and the quadratic $q$.
\end{proposition}

\begin{proof} Recall the definitions for $\widehat{f}(\vv) = \widehat{\mathscr{L}}(\vv) + \tfrac{\delta}{2} \|\vv\|^2$ and $\widehat{\mathscr{L}} = \frac{1}{2} \|\SSigma \nnu - \UU^T \bb\|^2$, \eqref{eq:rr_simplified} and \eqref{eq:lsq}, respectively. For $0 \le t \le T$, It\^{o}'s formula (c.f. \eqref{eq:SGD_statistic} and \eqref{eq:ito_diffusion} in Sections~\ref{sec:Doob} and \ref{sec:homogenized_SGD_intro} respectively) yields the following 
\begin{align}
    &\sup_{0 \le t \le T} |q(\nnu_t^\vartheta)-q(\YY_t^\vartheta)| \nonumber \\
    &\le \sup_{0 \le t \le T} \widehat{\gamma} \int_0^{t \wedge \vartheta}  |\nabla q(\YY_s^\vartheta)^T \nabla \widehat{f}(\YY_s^\vartheta)  - \nabla q(\nnu_s^\vartheta)^T \nabla \widehat{f}(\nnu_s^\vartheta) | \, \dif s \label{eq:bound_1}\\
    &  + \sup_{0 \le t \le T} \frac{2\widehat{\gamma}^2}{n} | \tr( (\nabla^2 q) \SSigma^T \SSigma) | \int_0^{t \wedge \vartheta} | \widehat{\mathscr{L}}(\nnu_s^\vartheta)-\widehat{\mathscr{L}}(\YY_s^\vartheta) | \, \dif s \label{eq:bound_2}\\
    & + \sup_{0 \le t \le T} \frac{\widehat{\gamma}^2}{2n} \left | \int_0^{t \wedge \vartheta} \left ( \delta^2 (\nnu_s^\vartheta)^T (\nabla^2 q) \nnu_s^\vartheta - 2 \delta (\nnu_s^\vartheta)^T (\nabla^2 q) \SSigma^T (\SSigma \nnu_s^\vartheta - \UU^T \bb) \right ) \, \dif s \right | + \mathcal{E}_{\text{errors}}^T(q). \label{eq:bound_3}
\end{align}
% where the errors $\mathcal{E}_{\text{errors}}^T$ are defined by
% \begin{equation}
%     \begin{aligned}
%         \mathcal{E}_{\text{errors}}^T(\bar{q}) \defas  \sup_{0 \le t \le T} \big \{  |\mathcal{M}_t^{\text{grad}}(\bar{q})| + |\mathcal{M}_t^{\text{quad}}(\bar{q})| + | \mathcal{M}_t^{\text{hSGD}}(\bar{q})| + |\mathcal{E}_{\text{KL}}(\bar{q})| \big \}.
%     \end{aligned}
% \end{equation}
Using Proposition~\ref{prop:quad_set}, we will bound the quantity \eqref{eq:bound_1} in terms of $\mathcal{Q}_s$. From \eqref{eq:grad_condition} in Proposition~\ref{prop:quad_set}, we conclude that
\begin{equation} \label{eq:tilde_g}
    \begin{aligned}
        \sup_{0 \le t \le T} \int_0^{t \wedge \vartheta} | \nabla q(\YY_s^\vartheta)^T \nabla \widehat{f}(\YY_s^\vartheta) - \nabla \bar{q}(\nnu_s^\vartheta)^T \nabla \widehat{f}(\nnu_s^\vartheta) | \, \dif s
        &\le C(\delta, \SSigma, \bb) \int_0^T \mathcal{Q}_s \, \dif s.
    \end{aligned}
\end{equation}
We now turn to \eqref{eq:bound_2}. Using Holder's inequality, $|\tr((\nabla^2 q) \SSigma^T \SSigma)| \le \|\nabla^2 q\| |\tr(\SSigma^T \SSigma)|$. Moreover, by Corollary~\ref{cor:f_1_bar_q}, it follows that
\begin{equation}
    \begin{aligned}
       \sup_{0 \le t \le T} &\frac{2 \widehat{\gamma}^2 |\tr((\nabla^2 q) \SSigma^T \SSigma)|}{n} \int_0^{t \wedge \vartheta}  |\widehat{\mathscr{L}}(\nnu_s^\vartheta)-\widehat{\mathscr{L}}(\YY_s^\vartheta)| \, \dif s\\
       &\le 2 \widehat{\gamma}^2 \|\nabla ^2 q\| \frac{|\tr(\SSigma^T \SSigma)|}{n} \int_0^T  |\widehat{\mathscr{L}}(\nnu_s^\vartheta)-\widehat{\mathscr{L}}(\YY_s^\vartheta)|  \, \dif s\\
       &\le C(\delta, \SSigma) \widehat{\gamma}^2 \max\{\|\nabla^2 \widehat{\mathcal{R}} \|, 2\}  \frac{|\tr(\SSigma^T \SSigma)|}{n}  \int_0^T \mathcal{Q}_s \, \dif s,
    \end{aligned}
\end{equation}
Here we used that the operator norm of $\nabla^2 q$ is bounded by \eqref{eq:quad_sup} and Corollary~\ref{cor:f_1_bar_q} with \eqref{eq:quad_sup}. 

Next we bound the first term in \eqref{eq:bound_3}. Let $\widehat{q}(\xx) = \delta^2 \xx^T (\nabla^2 q) \xx - 2 \delta \xx^T (\nabla^2 q) \SSigma^T (\SSigma \xx-\UU^T \bb) $. By Proposition~\ref{prop:quad_set}, in particular \eqref{eq:function_condition} and \eqref{eq:hessian_bound}, a simple computation yields
\begin{align*}
    \sup_{0 \le t \le T} \bigg | \int_0^{t \wedge \vartheta} (\delta^2 (\nnu_s^\vartheta)^T (\nabla^2 q) \nnu_s^\vartheta &- 2 \delta (\nnu_s^\vartheta)^T (\nabla^2 q) \SSigma^T (\SSigma \nnu_s^\vartheta-\UU^T \bb) ) \, \dif s \bigg |\\
    & \le \int_0^T | \widehat{q}(\nnu_s^\vartheta)-\widehat{q}(\YY_s^\vartheta) | \, \dif s +  \int_0^T |\widehat{q}(\YY_s^\vartheta)| \, \dif s \\
    &\le \tilde{C}(\delta, \SSigma, \bb)  \int_0^T \mathcal{Q}_s \, \dif s + T \max_{0 \le s \le T} |\widehat{q}(\YY_s^\vartheta)|. 
\end{align*}
By definition of the stopping time $\vartheta$, we have that
\begin{align*}
    \max_{0 \le s \le T} |\widehat{q}(\YY_s^\vartheta)| &\le \max_{0 \le s \le T} \|\nabla^2 q\| \|\delta^2 \II_d - 2\delta \SSigma^T \SSigma\| \|\YY_s^\vartheta\|^2_2 + \|\nabla^2 q\| \|\SSigma^T \UU^T \bb\|_2 \|\YY_s^\vartheta\|_2\\
    &\le 4 \max\{ \|\nabla^2 \widehat{\mathcal{R}} \|, 2\} \big ( \|\delta^2 \II_d - 2 \delta \SSigma^T \SSigma\| n^{\varepsilon} + \|\SSigma^T \UU^T \bb\|_2 n^{\varepsilon} \big ).
\end{align*}
Here we used that $\|\nabla^2 q\| \le 4 \max\{ \|\nabla^2 \widehat{\mathcal{R}} \|, 2\}$ (c.f., \eqref{eq:operator_bound} in Proposition~\ref{prop:quad_set}). 
% We also need to bound $\|\widehat{g}\|_{H^2}$,
% \begin{align*}
%     \|\widehat{g}\|_{H^2} \le \|\SS\| \|\delta^2 \II - 2\delta \SSigma^T \SSigma\| + \|\SS\| \|\SSigma^T \UU^T \bb\|_2 \le \|\delta^2 \II - 2\delta \SSigma^T \SSigma\| + \|\SSigma^T \UU^T \bb\|_2.
% \end{align*}
As a consequence, we have 
\begin{align*}
   \frac{\widehat{\gamma}^2}{n} \sup_{0 \le t \le T} &\bigg | \int_0^{t \wedge \vartheta} (\delta^2 (\nnu_s^\vartheta)^T (\nabla^2 q) \nnu_s^\vartheta - 2 \delta (\nnu_s^\vartheta)^T (\nabla^2 q) \SSigma^T (\SSigma \nnu_s^\vartheta-\UU^T \bb) ) \, \dif s \bigg | \\
   &\le \frac{\widehat{\gamma}^2}{n} \tilde{C}(\delta, \SSigma, \bb)  \int_0^T \mathcal{Q}_s \, \dif s + T \max_{0 \le s \le T} |\widehat{q}(\YY_s^\vartheta)|  \\
   &\le \frac{\widehat{\gamma}^2}{n}  \widehat{C}(\delta, \nabla^2 \widehat{\mathcal{R}}, \SSigma, \bb)\left ( \int_0^T \mathcal{Q}_s \, \dif s + T \cdot n^{\varepsilon} \right ),
\end{align*}
where $\widehat{C}(\delta, \nabla^2 \widehat{\mathcal{R}}, \SSigma, \bb)$ is some constant independent of $q$. The result follows. 
\end{proof}

The cardinality of the set of quadratics $Q$ is large. We show below that we can approximate any quadratic in $Q$ with quadratics from a set of quadratics $\bar{Q}$ of smaller cardinality. This will enable us to perform a net argument. For any quadratic $q$, we introduce a $H^2$-norm, that is,
\begin{equation} \label{eq:quad_norm}
    \|q\|_{H^2} \defas |q(0)| + \|\nabla q(0)\|_2 + \|\nabla^2 q\|. 
\end{equation}
In the next lemma, we show that we can approximate any $q \in Q$ with a class of quadratics with finite cardinality.

\begin{lemma}[Approximation of $Q$] \label{lem:approx_Q} Fix an $\varepsilon > 0$. There exists a set $\bar{Q} \subseteq Q$ of cardinality $Cn^{4 \varepsilon}$ for some absolute constant (independent of $n$ and $\varepsilon$) so that for any $q \in Q$ there exists a $\bar{q} \in \bar{Q}$ such that $\|q - \bar{q}\|_{H^2} \le \widehat{C}(\delta, \SSigma, \bb) n^{-2 \varepsilon}$ where $\widehat{C}$ is some constant and $q(0) = \bar{q}(0)$. 
\end{lemma}

\begin{proof}  First, we make a grid $\mathcal{G}$ of $|z| = (2 \|\MM\|)^{-1}$ such that for any $z$ with $|z| = (2 \|\MM\|)^{-1}$ there exists $\bar{z} \in \mathcal{G}$ with $|z-\bar{z}| \le n^{-2 \varepsilon}$. For all $\widehat{\SS} \in \{\II_d, \nabla^2 g\}$, set $\bar{Q}$ to be quadratics $\bar{q}$ of the form $\bar{q}(\xx) = \xx^T \RR(\bar{z}) \widehat{\SS} \RR(\bar{y})$ with $\bar{z},\bar{y} \in \mathcal{G}$, $\bar{q}(\xx) = (\nabla \widehat{\mathcal{R}} (0))^T \RR(\bar{z}) \xx$ with $\bar{z} \in \mathcal{G}$, and $\bar{q}(\xx) = \xx^T \RR(\bar{z}) \widehat{\SS} \RR(\bar{y}) \SSigma^T \UU^T \bb$ with $\bar{z}, \bar{y} \in \mathcal{G}$. It will be convenient to also add the quadratics $\xx^T \widehat{\SS} \xx$, $(\widehat{\SS} \SSigma^T \UU^T \bb)^T\xx$, $(\nabla \widehat{\mathcal{R}} (0))^T \xx$, and $\widehat{\mathcal{R}}$ where $\widehat{\SS} \in \{ \II_d, (\nabla^2 \widehat{\mathcal{R}})\}$. We note that this set $\bar{Q}$ has cardinality on the order of $n^{4 \varepsilon}$.  

We proceed by cases with all cases quite similar. Let $q(\xx) = \xx^T \RR(z) \widehat{\SS} \RR(y) \xx$ for some matrix $\widehat{\SS} \in \{\II_d, (\nabla^2 \widehat{\mathcal{R}})\}$. Choose $\bar{z}, \bar{y} \in \mathcal{G}$ such that $|z-\bar{z}| \le n^{-2\varepsilon}$ and $|y-\bar{y}| \le n^{-2 \varepsilon}$. Consider $\bar{q}(\xx) = \xx^T \RR(\bar{z}) \widehat{\SS} \RR(\bar{y}) \xx$. A simple computation shows that
\begin{equation}
\begin{aligned} \label{eq:q_case_1}
    \|q-\bar{q}\|_{H^2} &= \|\RR(z)\widehat{\SS} \RR(y) - \RR(\bar{z}) \widehat{\SS} \RR(\bar{y})\|\\
    &\le \|\widehat{\SS} \RR(y)\| \|\RR(z)-\RR(\bar{z})\| + \|\widehat{\SS} \RR(\bar{z})\| \|\RR(y)-\RR(\bar{y})\|.
\end{aligned}
\end{equation}
Since $|y| = |\bar{z}| = (2 \|\MM\|)^{-1}$, we have that $\|\RR(y)\| \le 2$. Using the identity $(\II_d - z \MM)^{-1} - (\II_d - \bar{z} \MM)^{-1} = (\II_d- z \MM)^{-1} (z-\bar{z}) \MM (\II_d-\bar{z} \MM)^{-1}$, we get a bound on the difference of $\RR(z)$, 
\begin{equation}
    \|\RR(z) - \RR(\bar{z})\| \le |z-\bar{z}| \|\II_d -z \MM\|^{-1} \|\MM\| \|\II_d-\bar{z} \MM\|^{-1} \le 4 \|\MM\| n^{-2 \varepsilon}.
\end{equation}
It immediately follows from \eqref{eq:q_case_1} that $\bar{q}$ satisfies $\|q-\bar{q}\|_{H^2} \le \widehat{C} n^{-2\varepsilon}$ and $\bar{q} \in \bar{Q}$.

The other cases for quadratics $q(\xx) = (\nabla \widehat{\mathcal{R}}(0))^T \RR(z) \xx$ and $q(\xx) = \xx^T \RR(z) \widehat{\SS} \RR(y) \SSigma^T \UU^T \bb$ we approximate with $\bar{q}(\xx) = (\nabla \widehat{\mathcal{R}} (0))^T \RR(\bar{z}) \xx$ and $\bar{q}(\xx) = \xx^T \RR(\bar{z}) \widehat{\SS} \RR(\bar{y})$ with $\bar{z}, \bar{y} \in \mathcal{G}$ respectively. Similar bounds as in the prior case yield the conclusions of the lemma. 
\end{proof}

We are ready to prove Theorem~\ref{thm:homogenized_SGD_SGD} under the assumption that the error $
 \sup_{q \in Q} \mathcal{E}^T_{\text{errors}}(q)$ (see \eqref{eq:errors}) is sufficiently small. We prove this error is small in Section~\ref{sec:martingale_errors} and in Lemma~\ref{lem:key_lemma}. 

% \begin{theorem}[Homogenized SGD and SGD] Fix a constant $T > 0$ and let $g \, : \, \mathbb{R}^d \mapsto \mathbb{R}$ be the statistic we are interested in studying where $g(\xx) = \xx^T \SS \xx + \hh^T \xx + c$  for some matrix $\SS \in \mathbb{R}^{d \times}$, vector $\hh \in \mathbb{R}^d$, and constant $c \in \mathbb{R}$. Under Assumptions ,
% \[ \sup_{0 \le t \le T} |g(\nnu_t)-g(\YY_t)| \Prto[n] 0, \]
% where $\nnu_t$ and $\YY_t$ are the iterates SGD and homogenized SGD respectively.
% \end{theorem}
  
\begin{proof}[Proof of Theorem~\ref{thm:homogenized_SGD_SGD}] As $\mathcal{R}(\xx_t) = \widehat{\mathcal{R}}(\nnu_t)$ and $\mathcal{R}(\XX_t) = \widehat{\mathcal{R}}(\YY_t)$, it suffices to prove the result for the processes $\nnu_t$ and $\YY_t$ and $\widehat{\mathcal{R}}$. We recall the stopping time $\vartheta$ in \eqref{eq:stopping_time_norm}  
\begin{equation}
    \vartheta \defas \inf \big \{ t \ge 0 \, : \, \|\nnu_t\|^2_2 > n^{\varepsilon} \quad \text{or} \quad \|\YY_t\|^2 > n^{\varepsilon} \big \}.
\end{equation}
where $\varepsilon = \min\{\theta/2, 1/4(1/2-\theta)\}$ and $0 < \theta < 1/2$ as in Assumption~\ref{ass: laundry_list}.
Fix a quadratic function $q \in Q$ where the set $Q$ is defined in \eqref{eq:quad_sup}. (We also define the value of $\mathcal{Q}_T$ in \eqref{eq:quad_sup}). It will suffice to show a bound on the difference of the quadratic $q$ applied to the stopped processes $\nnu_{t \wedge \vartheta} \defas \nnu_t^{\vartheta}$ and $\YY_{t \wedge \vartheta} \defas \YY_t^{\vartheta}$. By Lemma~\ref{lem:approx_Q}, there exists a subset $\bar{Q} \subseteq Q$ such that for every $q \in Q$ there is a $\bar{q} \in \bar{Q}$ with $\|q - \bar{q}\|_{H^2} \le C(\delta, \SSigma, \bb) n^{-2 \varepsilon}$ and $q(0) = \bar{q}(0)$.  
We then deduce that
\begin{align*}
    |q(\nnu_t^{\vartheta}) - q(\YY_t^{\vartheta})| &\le |q(\nnu_t^{\vartheta}) - \bar{q}(\nnu_t^{\vartheta})| + |\bar{q}(\YY_t^{\vartheta}) - q(\YY_t^{\vartheta})| + |\bar{q}(\nnu_t^{\vartheta})-\bar{q}(\YY_t^{\vartheta})|\\
    & \le \|q-\bar{q}\|_{H^2} (\|\nnu_t^{\vartheta}\|^2 + \|\nnu_t^{\vartheta}\|_2) + \|q-\bar{q}\|_{H^2} (\|\YY_t^{\vartheta}\|^2 + \|\YY_t^{\vartheta}\|_2) + |\bar{q}(\nnu_t^{\vartheta})-\bar{q}(\YY_t^{\vartheta})|.
\end{align*}
By our choice of $\bar{q}$, we have that $\|q-\bar{q}\|_{H^2} \le C(\delta, \SSigma, \bb) n^{-2\varepsilon}$. Moreover using the definition of the stopping time $\vartheta$, we have that up until time $t = \vartheta$, both processes $\|\YY_t^{\vartheta}\|^2_2, \|\nnu_t^{\vartheta}\|^2 \le n^{\varepsilon}$. The same bound holds the for the processes at $t = \vartheta$ by Lemma~\ref{lem:boundedness_loss_norm_homogenized_SGD} and Lemma~\ref{lem:boundedness_loss_norm_SGD}, with overwhelming probability 
(w.o.p). It follows that $\sup_{0 \le t \le T} \|\YY_t^{\vartheta}\|^2, \sup_{0\le t \le T} \|\nnu_t^{\vartheta}\|^2 \le C(\delta, \widehat{\gamma}, \SSigma, \bb) n^{\varepsilon}$, w.o.p. We deduce that w.o.p.
\begin{equation} \label{eq:blah_thm_1}
    \sup_{0 \le t \le T } |q(\nnu_t^{\vartheta}) - q(\YY_t^{\vartheta})| \le C(\delta, \widehat{\gamma}, \SSigma, \bb) n^{-\varepsilon} + \sup_{0 \le t \le T} |\bar{q}(\nnu_t^{\vartheta}) - \bar{q}(\YY_t^{\vartheta})|.
\end{equation}

It suffices from \eqref{eq:blah_thm_1} to bound $\sup_{0 \le t \le T} |\bar{q}(\nnu_t^{\vartheta}) - \bar{q}(\YY_t^{\vartheta})|$. We apply Proposition~\ref{prop:bar{q}} to the quadratic $\bar{q}$. As the constants in Proposition~\ref{prop:bar{q}} do not depend on the choice of the quadratic $\bar{q}$ and $\bar{Q} \subseteq Q$ so that $\mathcal{\bar{Q}}_s \le \mathcal{Q}_s$. Thus
\begin{equation} \begin{aligned}
    \sup_{\bar{q} \in \bar{Q}} |\bar{q}(\nnu_t^{\vartheta}) - \bar{q}(\YY_t^{\vartheta})| &\le C(\delta, \widehat{\gamma}, \nabla^2 \widehat{\mathcal{R}}, \SSigma, \bb) \int_0^T \mathcal{Q}_s \, \dif s \\
    & \qquad + \widehat{C}(\delta, \widehat{\gamma}, \nabla^2 \widehat{\mathcal{R}}, \SSigma, \bb) \cdot \left ( \int_0^T \mathcal{Q}_s \, \dif s + T \right ) n^{\varepsilon -1} + \sup_{\bar{q} \in \bar{Q}} \mathcal{E}^T_{\text{errors}}(\bar{q}),\\
    \text{where} \quad \mathcal{E}^T_{\text{errors}}(\bar{q}) &= \sup_{0 \le t \le T} \big \{ |\mathcal{M}_{t \wedge \vartheta}^{\text{grad}}(\bar{q})| + |\mathcal{M}_{t \wedge \vartheta}^{\text{quad}}(\bar{q})| + |\mathcal{M}_{t \wedge \vartheta}^{\text{hSGD}}(\bar{q})| + |\mathcal{E}_{t \wedge \vartheta}^{\text{KL}}(\bar{q})|  \big \}.
\end{aligned} \end{equation}
By Proposition~\ref{prop:martingale_errors_SGD}, there exists constants independent of $n$ and $q$ such that w.o.p.
\begin{equation} \begin{gathered} \label{eq:thm_martingale_error}
    \sup_{0 \le t \le T} |\mathcal{M}_{t \wedge \vartheta}^{\text{grad}}| \le C(\delta, \widehat{\gamma}, \SSigma,  T) (\|\nabla^2 q \| + \|\nabla q(0)\|_2) n^{-1/2 + \alpha + \varepsilon}\\
    \text{and} \quad \sup_{0 \le t \le T} | \mathcal{M}_{t \wedge \vartheta}^{\text{quad}}(q) | \le C(\delta, \widehat{\gamma}, \SSigma)  T^{1/2} \| \nabla^2 q\| n^{-1/2 + \alpha + \varepsilon}, 
\end{gathered} \end{equation}
By choosing $\alpha = \theta + 1/4(1/2-\theta)$ and $\varepsilon = \min\{\theta/2, 1/4(1/2-\theta)\}$, we have that $-1/2 + \alpha + \varepsilon < 0$. From \eqref{eq:quad_sup}, we have that $\|\nabla^2 q\|$ and $\|\nabla \bar{q}(0)\|_2$ are uniformly bounded. As \eqref{eq:thm_martingale_error} holds w.o.p and the cardinality of the set $\bar{Q}$ is $n^{4 \varepsilon}$, then it follows that $\sup_{0 \le t \le T} |\mathcal{M}_{t \wedge \vartheta}^{\text{grad}}|$ and $\sup_{0 \le t \le T} |\mathcal{M}_{t \wedge \vartheta}^{\text{quad}}|$ are uniformly bounded in $\bar{q}$ by $n^{-1/2 + \alpha + \varepsilon}$ w.o.p., that is,
\begin{equation} \begin{gathered}
    \sup_{\bar{q} \in \bar{Q}} \sup_{0 \le t \le T} |\mathcal{M}_{t \wedge \vartheta}^{\text{grad}}| \le C(\delta, \widehat{\gamma}, \SSigma, \bb, T)n^{-1/2 + \alpha + \varepsilon}\\
    \text{and} \quad 
 \sup_{\bar{q} \in \bar{Q}} \sup_{0 \le t \le T} |\mathcal{M}_{t \wedge \vartheta}^{\text{quad}}|  \le C(\delta, \widehat{\gamma}, \SSigma, \bb, T)n^{-1/2 + \alpha + \varepsilon} \quad \text{w.o.p.},
\end{gathered}
\end{equation}
and $-1/2 + \alpha + \varepsilon < 0$. Similarly, we deduce from Proposition~\ref{prop:martingale_errors_hSGD} and the key lemma (Lemma~\ref{lem:key_lemma})
\[ \sup_{\bar{q} \in \bar{Q}} \sup_{0 \le t \le T} |\mathcal{M}_{t \wedge \vartheta}^{\text{hSGD}}| \le C(\delta, \widehat{\gamma}, \SSigma, \bb, T) n^{\varepsilon - 1/2} \quad \text{and} \quad \sup_{\bar{q} \in \bar{Q}} |\mathcal{E}_{t \wedge \vartheta}^{\text{KL}}(\bar{q})| \le C(\delta, \widehat{\gamma}, \SSigma, \bb, \|\widehat{\mathcal{R}}\|_{H^2}, T) n^{\theta+\varepsilon-1/2} \quad \text{w.o.p.} \]
By construction of $\varepsilon$, we have that both $\varepsilon - 1/2 < 0$ and $\theta + \varepsilon - 1/2 < 0$. There exists constants $\tilde{\varepsilon}, \tilde{c} > 0$ such that $\tilde{\varepsilon}-\tilde{c} < 0$ and
\begin{equation}
    \sup_{\bar{q} \in \bar{Q}} \mathcal{E}_{\text{errors}}^T( \bar{q}) \le C(\delta, \widehat{\gamma}, \SSigma, \bb, \|\widehat{\mathcal{R}}\|_{H^2}, T) n^{\tilde{\varepsilon}-\bar{c}} \quad \text{w.o.p.}
\end{equation}
Returning now to \eqref{eq:blah_thm_1}, we have thus shown 
\begin{equation} \begin{aligned}
    \mathcal{Q}_T &= \sup_{q \in Q} \sup_{0 \le t \le T} |q(\nnu_t^{\vartheta})-q(\YY_t^{\vartheta})| \\
    &\le C(\delta, \widehat{\gamma}, \SSigma, \bb, \|\widehat{\mathcal{R}}\|_{H^2}, T) n^{\tilde{\varepsilon}-\bar{c}} + \widehat{C}(\delta, \widehat{\gamma}, \SSigma, \bb, \|\widehat{\mathcal{R}}\|_{H^2}, T) (n^{\varepsilon-1} + 1) \int_0^T \mathcal{Q}_s \, \dif s,
    \end{aligned}
\end{equation}
for some constants $C$ and $\widehat{C}$. We apply Gronwall's inequality (and the statistic $\widehat{\mathcal{R}} \in Q$) to conclude
\begin{equation} \label{eq:thm11_g}
   \sup_{0 \le t \le T} |\widehat{\mathcal{R}}(\nnu_t^{\vartheta})-\widehat{\mathcal{R}}(\YY_t^{\vartheta})| \le \mathcal{Q}_T \le C n^{\tilde{\varepsilon} - \tilde{c}} \exp \left ( \widehat{C} (n^{\varepsilon-1} + 1) T \right ) \quad \text{w.o.p.}
\end{equation}
Here we simplified notation so that $C \defas C(\delta, \gamma, \SSigma, \bb, \|\widehat{\mathcal{R}} \|_{H^2}, T)$ and $\widehat{C} \defas \widehat{C}(\delta, \gamma, \SSigma, \bb, \|\widehat{\mathcal{R}}\|_{H^2}, T)$. Note also that $\tilde{\varepsilon}-\tilde{c} < 0$ and thus $\sup_{0 \le t \le T} |\widehat{\mathcal{R}}(\nnu_t^{\vartheta})-\widehat{\mathcal{R}}(\YY_t^{\vartheta})| \to 0$ w.o.p. The result is almost complete except for the stopping time. We now remove this. Since $\|\cdot\|^2 \in Q$, we have that $ \displaystyle \sup_{0 \le t \le T} | \|\nnu_t^{\vartheta}\|^2- \|\YY_t^{\vartheta}\|^2 |\le  C n^{\tilde{\varepsilon} - \tilde{c}} \exp \left ( \widehat{C} (n^{\varepsilon-1} + 1) T \right )$. Moreover we observe that
\[\Pr(\vartheta > T) \ge \Pr( \{ \sup_{0\le t \le T} \|\YY_t\|^2 \le n^{\varepsilon/2} \} \cap \{ \sup_{0 \le t \le T} |\|\YY_t^{\vartheta} \|^2 - \|\nnu_t^{\vartheta}\|^2| \le n^{\varepsilon /2} \} ).\]
Since the two events on the RHS occur w.o.p. (c.f. Lemma~\ref{lem:boundedness_loss_norm_homogenized_SGD}), we have that $\vartheta > T$ occurs w.o.p. Hence we may remove the stopping time from \eqref{eq:thm11_g}. By Assumption~\ref{ass:poly}, we can convert any with overwhelming probability statement in $n$ to an with overwhelming probability statement in $d$ and the result follows.  
\end{proof} 
It immediately follows from the proof of Theorem~\ref{thm:homogenized_SGD_SGD} that $\vartheta > T$. 

% We remark to control the errors, we will need to make an \textit{a priori} estimate that effectively shows that the iterates remain bounded. Thus, we define the stopping time, for any fixed $\varepsilon > 0$, by
% \begin{equation}
%      \vartheta = \inf \big \{ t \ge 0 \, : \, \|\nnu_t\|^2_2 > n^{\varepsilon} \quad \text{or} \quad \|\YY_t\|^2 > n^{\varepsilon} \big \}.
%     % \vartheta \defas \inf \, \{t \ge 0 \, : \, \|\UU \SSigma \nnu_t - \eeta\| > n^{\theta}\}.
% \end{equation}
% We then show: 
\begin{corollary}[Similar to Lemma B.4 in \cite{paquetteSGD2021}] \label{lem:stopping_time_function_values} \label{lem:bound_function_values} For any $\varepsilon > 0$, and for any $T > 0$, $\vartheta > T$ with high probability.
\end{corollary}

\begin{proof} See proof of Theorem~\ref{thm:homogenized_SGD_SGD}.
\end{proof}

We end this section by showing that the error induced from the key lemma, disappears as $n \to \infty$ provided that $\varepsilon$ is chosen sufficiently small. 

\begin{lemma}[Key lemma] \label{lem:key_lemma} Fix $T > 0$ and let the set $Q$ be defined as in \eqref{eq:quad_sup}. Consider the error $|\mathcal{E}_{t \wedge \vartheta}^\text{KL}(q)|$ in \eqref{eq:keylemma_error} for $q \in Q$. Under Assumptions~\ref{ass: laundry_list} and \ref{assumption: quadratics},
\begin{equation}
\begin{aligned}
\sup_{0 \le t\le T} |\mathcal{E}_{t \wedge \vartheta}^{\text{KL}}(q)|
     &\le C(\delta, \widehat{\gamma}, \SSigma, \bb, T) n^{\theta + \varepsilon - 1/2},
\end{aligned}
\end{equation}
for some constant $C(\delta, \widehat{\gamma}, \SSigma, \bb, T)$ independent of $n$. 
\end{lemma}

\begin{proof} First observe that
     \begin{equation}
\begin{aligned}
     \sup_{0 \le t\le T} |\mathcal{E}_{t \wedge \vartheta}^{\text{KL}}(q)| &\le \max_{1 \le i \le n} \left \{ \big | (\ee_i^T \UU \SSigma (\nabla^2 q) \SSigma^T \UU^T \ee_i) - \frac{1}{n} \tr((\nabla^2 q) \SSigma^T \SSigma ) \big | \right \} \int_0^{t \wedge \vartheta} | \gamma^2(s) \widehat{\mathscr{L}}(\nnu_s)| \, \dif s.
\end{aligned}
\end{equation}
By the definition of the stopping time $\vartheta$, $|\mathscr{L}(\nnu_{s \wedge \vartheta})| \le C(\delta, \SSigma, \bb) n^{\varepsilon}$. Every $q \in Q$ with $\nabla^2 q \neq 0$ can be expressed as $\RR(z)\RR(y) + \widehat{R}(y) \RR(z)$ or $\RR(z)(\nabla^2 \widehat{\mathcal{R}} )\RR(y) + \RR(y) (\nabla^2 \widehat{\mathcal{R}})\RR(z)$ for $|z| = |y| = (2 \|\MM\|)^{-1}$. The resolvent of $\AA \AA^T$ and $\RR(z)$ are related (see \eqref{eq:resolvent_relation}). The result immediately follows by applying \eqref{eq:key_lemma_ass} in Assumption~\ref{ass: laundry_list} and Assumption~\ref{assumption: quadratics}.
\end{proof}

% \section{Key Lemma}

% If we use the fact that $\nnu_t = \VV^T \xx_t$, then we can define a matrix $\TT \defas \VV (\nabla^2 q) \VV^T$. We do not need the key lemma to hold for all choices of $\nabla^2 q$, but only for a class of quadratics. In this was, the key lemma is equivalent to proving the following.

% \begin{lemma}[Key Lemma without Orthogonal Invariance] \label{lemma:key} Fix a symmetric matrix $\tilde{\SS} \in \mathbb{R}^{d \times d}$ that is bounded in operator norm and deterministic and the matrix $\MM = \AA^T \AA + \delta \II$. Define $\TT \defas \RR(z) \tilde{\SS} \RR(y) + \RR(y) \tilde{\SS} \RR(z)$ with $|z| = |y| = (2 \|\MM\|)^{-1}$ where the matrix $\RR(z) = (\II - z \MM)^{-1}$ for $z \in \mathbb{C}$.  With overwhelming probability, there exists an $\varepsilon > 0$ such that
% \begin{equation}
%     \left |   (\ee_i^T \AA \TT \AA^T \ee_i)  - \frac{\tr(\AA \TT \AA^T )}{n} \right | \le n^{-\varepsilon} \qquad \text{for all $i = 1, \hdots n$}.
% \end{equation}
% % where $\TT = \RR(z) \tilde{\SS} \RR(y)$ with $|z| = |y| = (2 \|\MM\|)^{-1}$ and $\MM = \AA^T \AA + \delta \II$ and $\RR(z) = (\II - z \MM)^{-1}$ for $z \in \mathbb{C}$. Here the matrix $\tilde{\SS} \in \mathbb{R}^{d \times d}$ is symmetric, bounded in operator norm, and deterministic.
% \end{lemma}

\subsection{Concentration and limiting behavior of SGD under the statistic} \label{sec:concentration_SGD} 
In this section, we show the concentration of SGD under any quadratic statistic $\mathcal{R}$ and specifically, the least-square loss $\mathcal{L}$ (Theorem~\ref{thm:expectation}). Next, we prove the limiting behavior of SGD on the excess risk with both a learning rate $\gamma(t) \to 0$ (Robbins-Monro setting) and constant $\gamma(t) \equiv \gamma$ (see Theorem~\ref{thm:eventualrisk}). Finally, we will show the result for our three motivating examples: training loss (Section~\ref{sec:training_loss_intro} Theorem~\ref{thm:Ab}), empirical risk minimization in linear regression (Section~\ref{sec:erm_excess_risk_intro}, Theorem~\ref{thm:AbR}), and random features (Section~\ref{sec:random_features_intro}, Theorem~\ref{thm:random_features}).

We proceed to find the generalization error (or loss function) for SGD. As shown in Theorem~\ref{thm:homogenized_SGD_SGD}, we do not need to work directly with the iterates of SGD, but rather can use homogenized SGD \eqref{eq:HSGD_dif} instead. Moreover in Theorem~\ref{thm:trainrisk}, homogenized SGD applied to the loss $\mathcal{L}$ and the risk $\mathcal{R}$ concentrate around their means, which we recall below, respectively, 
\begin{equation} \label{eq:lsq_hSGD}
    \begin{aligned}
         \Psi_t &= \mathscr{L}(\bm{\mathscr{X}}_{\Gamma(t)}^{\text{gf}}) + \frac{1}{n} \int_0^t \gamma^2(s) \text{\rm tr}\big ((\AA^T \AA)^2 e^{-2 (\AA^T \AA + \delta \II_d) (\Gamma(t)-\Gamma(s))} \big ) \Psi_t \, \dif s,\\
        \Omega_t &= \mathcal{R}(\bm{\mathscr{X}}_{\Gamma(t)}^{\text{gf}})
        + \frac{1}{n} \int_0^t \gamma^2(s) \text{\rm tr} \big ((\nabla^2 \mathcal{R}) \AA^T \AA e^{-2(\AA^T \AA + \delta \II_d)(\Gamma(t)-\Gamma(s))} \big ) \Psi_s \, \dif s,
        % \EE_0[ \mathscr{L}(\XX_t)] &= \frac{R^2}{2n} \text{\rm tr}(\AA^T \AA e^{-2 (\AA^T \AA + \delta \II_d) \Gamma(t)} )\\
        % & + \frac{1}{2} \left \| \left [ \AA \left ( \int_0^t \gamma(s) e^{-(\AA^T \AA + \delta \II_d) (\Gamma(t)-\Gamma(s))}  \, \dif s \right ) \AA^T - \II_n \right ] \bb \right \|^2\\
        % & + \frac{1}{n} \int_0^t \gamma^2(s) \text{\rm tr}\big ((\AA^T \AA)^2 e^{-2 (\AA^T \AA + \delta \II_d) (\Gamma(t)-\Gamma(s))} \big ) \EE[\mathscr{L}(\XX_s) | \HH, \bb] \, \dif s.
    \end{aligned}
\end{equation}
where $\bm{\mathscr{X}}_{\Gamma(t)}^{\text{gf}}$ is gradient flow after $\Gamma(t)$ amount of time. 

\begin{remark} \rm{ When the learning rate is a constant $\gamma(t) \equiv \gamma$, the expression \eqref{eq:lsq_hSGD} simplifies significantly and $\Psi_t$ reduces to a solution of the convolution-type Volterra equation, that is, 
\begin{equation} \label{eq:lsq_hSGD_constant}
    \begin{aligned}
        \Psi_t = \mathscr{L}(\bm{\mathscr{X}}_{\Gamma(t)}^{\text{gf}})
        + \frac{\gamma^2}{n} \int_0^t \text{\rm tr}\big ((\AA^T \AA)^2 e^{-2 \gamma (\AA^T \AA + \delta \II_d) (t-s)} \big ) \Psi_s \, \dif s.
        % \EE_0[ \mathscr{L}(\XX_t)] &= \frac{R^2}{2n} \text{\rm tr}(\AA^T \AA e^{-2\gamma (\AA^T \AA + \delta \II_d) t} )\\
        % & + \frac{1}{2} \| \big [ \AA( \AA^T \AA + \delta \II_d)^{-1}(\II_d - e^{-\gamma (\AA^T \AA + \delta \II_d) t}) \AA^T - \II_n \big ] \bb\|^2\\
        % & + \frac{\gamma^2}{n} \int_0^t \text{\rm tr}\big ((\AA^T \AA)^2 e^{-2 \gamma (\AA^T \AA + \delta \II_d) (t-s)} \big ) \EE[\mathscr{L}(\XX_s) | \HH, \bb] \, \dif s.
    \end{aligned}
\end{equation}
}
\end{remark}
With these in hand, the main result -- concentration of SGD -- can be shown.

\begin{proof}[Proof of Theorem~\ref{thm:expectation}] 
% By Assumption~\ref{}, we can convert any with overwhelming probability statement in $n$ to an with overwhelming probability statement in $d$. 
Instead of considering SGD, Theorem~\ref{thm:homogenized_SGD_SGD} shows that we can directly work with homogenized SGD with the difference between the two vanishingly small as $d \to \infty$. In fact, $|\mathcal{R}(\xx_{\lfloor tn \rfloor }) - \mathcal{R}(\XX_t)| \le d^{-\tilde{\varepsilon}/2}$ with overwhelming probability for some $\tilde{\varepsilon}$. By Theorem~\ref{thm:trainrisk}, homogenized SGD concentrates around its mean with overwhelming probability. Combining these two theorems, proves the result after noting that the mean behavior of homogenized SGD is $\Psi_t$ and $\Omega_t$.
\end{proof}

We now consider specific examples of the loss function $\mathcal{L}$ and the excess risk $\mathcal{R}$ as discussed in Section~\ref{sec:stats}. In particular, we consider the training loss (Section~\ref{sec:training_loss_intro} Theorem~\ref{thm:Ab}), empirical risk minimization in linear regression (Section~\ref{sec:erm_excess_risk_intro}, Theorem~\ref{thm:AbR}), and random features (Section~\ref{sec:random_features_intro}, Theorem~\ref{thm:random_features}). The proofs of these theorem essentially follows from the same reasoning: concentration of the gradient flow. 

\begin{proof}[Proof of Theorem~\ref{thm:Ab}, ~\ref{thm:AbR}, and \ref{thm:random_features}] We will do Theorem~\ref{thm:Ab} in detail. The other theorems are a similar argument. To distinguish the different $\Psi_t$ and $\Omega_t$ functions, we denote the $\Psi_t$ from Theorem~\ref{thm:trainrisk} as $\Psi_t = \EE_0[\mathscr{L}(\XX_t)] \defas \EE[\mathscr{L}(\XX_t) \, | \, \AA, \bb, \xx_0]$. We recall that the targets come from a generative model, that is $\bb = \AA \bbeta + \eeta$ where the vector $\bbeta$ is an unknown signal and $\eeta$ is some additive noise. As $\bbeta, \eeta,$ and the initialization $\xx_0$ are assumed to be iid subgaussian and independent, $\EE[\mathscr{L}(\XX_t) \, | \, \AA, \bb, \xx_0]$ will concentrate around its mean $\EE[\mathscr{L}(\XX_t) \, | \, \AA]$. This amounts to concentration of the gradient flow term. By independence assumptions, all cross terms between $\eeta$, $\bbeta$, and $\xx_0$ are $0$ so only quadratic forms remain. The iid assumption between coordinates then finishes the proof. 
\end{proof}

\subsubsection{Limiting loss and risk values} 
We now analyze the limiting loss (risk) values, $\Psi_{\infty}$ ($\Omega_\infty$), respectively. The expression for $\Psi_t$ in \eqref{eq:lsq_hSGD} is a linear Volterra equation. Much is known about properties of linear Volterra equations and their convergence properties (\textit{c.f.} \citep{gripenberg1980volterra} or \citep{Asmussen}). In Volterra  terminology (see \citep{gripenberg1980volterra}), the forcing term $g(t)$ and kernel $K(t,s; \AA^T \AA)$, respectively, are explicitly
\begin{equation}
    \begin{aligned}
        g(t) &\defas  \mathscr{L}(\bm{\mathscr{X}}_{\Gamma(t)}^{\text{gf}}) \quad \text{and} \quad K(t,s; \AA^T \AA) \defas \frac{\gamma^2(s)}{n} \text{\rm tr}\big ( (\AA^T \AA)^2 e^{-2(\AA^T \AA + \delta \II_d)(\Gamma(t)-\Gamma(s))} \big ).
    \end{aligned}
\end{equation}
The kernel represents the inherent noise produced by the algorithm itself and it is what makes SGD different than, say gradient flow. Provided that $\sup_{t \ge 0} \int_0^t K(t,s; \AA^T \AA) \dif s \le 1$, there is a unique bounded solution $\Psi_t$ to \eqref{eq:lsq_hSGD} (see proof of Theorem~\ref{thm:eventualrisk}). The solution then is given by $\Psi_t = g + r \star g$ where the function $r : \mathbb{R}^2 \to \mathbb{R}$ is called the \textit{resolvent} of $K(t,s; \AA^T\AA)$ and the operation $\star$ acts on functions by $(f \star g) (t) = \int_0^t f(t,s) g(s)  \dif s$ (see \citep[Chapter 9]{gripenberg1980volterra}). If we make additional assumptions on the learning rate \eqref{eq:convergence_threshold}, the sufficient condition for uniqueness of the solution holds (i.e., $\sup_{t \ge 0} \int_0^t K(t,s; \AA^T \AA) \dif s \le 1$); we recall these assumptions below,
\begin{equation}\label{eq:learning_rate_convergence_threshold}
   \Gamma(t) \to \infty, \quad \gamma(t) \to \widetilde{\gamma} > 0, \quad \text{and} \quad  \widetilde{\gamma} < 2 \big ( n^{-1} \tr( (\AA^T \AA)^2 (\AA^T \AA + \delta \II_d)^{-1} \big )^{-1}
\end{equation}
We now will prove the limiting loss (risk) values under two different learning rate scenarios. We record the limiting loss values below for reference within the proof. First, when $\widetilde{\gamma} = 0$ (Robbins-Monro setting), it will follow that $\lim_{t \to \infty} \Psi_t - \mathscr{L}(\bm{\mathscr{X}}_{\Gamma(t)}^{\text{gf}}) = 0$ and $\lim_{t \to \infty} \Omega_t - \mathcal{R}(\bm{\mathscr{X}}_{\Gamma(t)}^{\text{gf}}) = 0$. On the other hand for an arbitrary $\widetilde{\gamma}$, the limiting empirical risk $\Psi_{\infty}$ is given by 
\begin{equation} \label{eq:limit_loss_main}
\Psi_\infty 
=
\mathscr{L}\bigl( 
\bm{\mathscr{X}}_{\infty}^{\text{gf}}\bigr)
\times
\biggl(
1
-
\frac{ \widetilde{\gamma}}{2n}
\tr
\bigl(
(\nabla^2 \mathscr{L})^2
\bigl( \nabla^2 \mathscr{L} + \delta \II_d \bigr)^{-1}
\bigr)
\biggr)^{-1}.
\end{equation}
The limiting excess risk of SGD over gradient flow is given by
\begin{equation} \label{eq:limit_risk_main}
\Omega_t-\mathcal{R}\bigl( 
\bm{\mathscr{X}}_{\Gamma(t)}^{\text{gf}}\bigr)
\xrightarrow[t\to\infty]{ } \frac{\widetilde{\gamma}}{2n}   \Psi_\infty \times 
\tr
\biggl(
(\nabla^2 \mathcal{R}) (\nabla^2 \mathscr{L})
\bigl( \nabla^2 \mathscr{L} + \delta \II_d \bigr)^{-1}
\biggr).
\end{equation}
With these in hand, we now prove Theorem~\ref{thm:eventualrisk}.

\begin{proof}[Proof of Theorem~\ref{thm:eventualrisk}] 
First suppose that the limiting loss $\Psi_t$ is bounded and it exists at infinity. Note that it is sufficient to prove \eqref{eq:limit_loss_main} for an arbitrary $\widetilde{\gamma}$ as we can set $\widetilde{\gamma} = 0$ to recover the Robbins-Monro result. We show under this condition on $\Psi_t$ that the limiting risk value holds for $\Omega_{\infty}$. A simple computation with a change of variables gives
\begin{equation}
\begin{aligned}
\lim_{t \to \infty} \Omega_t -& \mathcal{R}(\bm{\mathscr{X}}_{\Gamma(t)}^{\text{gf}})\\
&= \lim_{t \to \infty} n^{-1} \int_0^t \gamma^2(s) \text{tr}\bigg ( (\nabla^2 \mathcal{R}) \AA^T \AA \exp \big (-2(\AA^T \AA + \delta \II_d)(\Gamma(t)-\Gamma(s)) \big ) \bigg ) \Psi_s \, \dif s\\
&= \lim_{t \to \infty} n^{-1} \int_0^{\Gamma(t)} \gamma(s) \text{tr}\bigg ( (\nabla^2 \mathcal{R}) \AA^T \AA \exp \big (-2(\AA^T \AA + \delta \II_d)(\Gamma(t)-s) \big ) \bigg ) \Psi_{\Gamma^{-1}(s)} \, \dif s\\
&= \lim_{t \to \infty} n^{-1} \int_0^{\Gamma(t)} \gamma(\Gamma(t)-v) \text{tr}\bigg ( (\nabla^2 \mathcal{R}) \AA^T \AA \exp \big (-2(\AA^T \AA + \delta \II_d)v \big ) \bigg ) \Psi_{\Gamma^{-1}(\Gamma(t)-v)} \, \dif v.
\end{aligned}
\end{equation}
Dominated convergence theorem allows us to interchange the integral and limit as $\Psi_t$ and $\gamma(t)$ are bounded. We pull out the limiting values of $\lim_{t \to \infty} \gamma(t) = \widetilde{\gamma}$ and $\Psi_{\infty}$. By integrating, we deduce 
\begin{equation}
\begin{aligned}
    \lim_{t \to \infty} \Omega_t - \mathcal{R}(\bm{\mathscr{X}}_{\Gamma(t)}^{\text{gf}}) &= \lim_{t \to \infty} n^{-1} \int_0^{\Gamma(t)} \gamma(\Gamma(t)-v) \text{tr}\bigg ( (\nabla^2 \mathcal{R}) \AA^T \AA \exp \big (-2(\AA^T \AA + \delta \II_d)v \big ) \bigg ) \Psi_{\Gamma^{-1}(\Gamma(t)-v)} \, \dif v\\
    &= \widetilde{\gamma} \Psi_{\infty} n^{-1} \int_0^{\infty} \text{tr}\bigg ( (\nabla^2 \mathcal{R}) \AA^T \AA \exp \big (-2(\AA^T \AA + \delta \II_d)v \big ) \bigg ) \, \dif v\\
    &=  \widetilde{\gamma} \Psi_{\infty} n^{-1} \tr \bigg ( (\nabla^2 \mathcal{R}) \AA^T \AA (2(\AA^T \AA + \delta \II_d))^{-1} \bigg )
\end{aligned}
\end{equation}
The result for the limiting risk  value $\lim_{t \to \infty} \Omega_t - \mathcal{R}(\bm{\mathscr{X}}_{\Gamma(t)}^{\text{gf}})$ follows. 

It remains to show that $\Psi_t$ is bounded and exists at infinity with its limiting value given by \eqref{eq:limit_loss_main}. Recall the loss kernel for $\Psi_t$ given by
\begin{equation}
    K(t,s) \defas K(t,s; \nabla^2 \mathscr{L}) = \gamma^2(s) n^{-1} \tr\bigg ( (\nabla^2 \mathscr{L}) \AA^T \AA \exp \big (-2(\AA^T\AA + \delta\II_d)(\Gamma(t)-\Gamma(s)) \big) \bigg ),
\end{equation}
so that $\Psi_t$ is the solution to the Volterra equation
\begin{align} \label{eq:psi_volterra}
    \Psi_t = \mathscr{L}(\bm{\mathscr{X}}_{\Gamma(t)}^{\text{gf}}) + \int_0^t K(t,s) \Psi_s \, \dif s.
\end{align}
We show, under the upper bound on $\gamma(t)$ that is \eqref{eq:learning_rate_convergence_threshold}, that the kernel $K(s,t)$ is of $L^{\infty}$-type on $[0, \infty)$. A kernel is $L^\infty$-type if $\vertiii{K}_{L^{\infty}(J)} < \infty$ for a set $J \subset \mathbb{R}$ where $\vertiii{K}_{L^{\infty}(J)} = \sup_{t \in J} \int_J |K(s,t)| \, \dif s$ \citep[Chapter 9.2]{gripenberg1980volterra}. For this, we see that for each $t$ and $s$
\begin{equation} \label{eq:change_variables_volterra}
\begin{aligned}
K(t,s) \le n^{-1} \widehat{\gamma} \cdot \gamma(s) \tr\bigg ( (\nabla^2 \mathscr{L}) \AA^T \AA \exp \big (-2(\AA^T\AA + \delta\II_d)(\Gamma(t)-\Gamma(s)) \big) \bigg ).
\end{aligned}
\end{equation}
This implies by change of variables that
\begin{equation}
    \begin{aligned}
    \int_0^t K(t,s) \, \dif s &\le \int_0^{\Gamma(t)} n^{-1} \widehat{\gamma} \tr\bigg ( (\nabla^2 \mathscr{L}) \AA^T \AA \exp \big (-2(\AA^T\AA + \delta\II_d)(\Gamma(t)-s) \big) \bigg ) \dif s\\
    &\le \frac{\widehat{\gamma}}{2n} \tr \big ( (\AA^T\AA)^2 (\AA^T \AA + \delta\II_d)^{-1} \big ) < \infty. 
    \end{aligned}
\end{equation}
Hence, it follows that the kernel $K$ is $L^\infty$-type on $[0,\infty)$. To prove the boundedness assumption of $\Psi_t$, we will need something slightly stronger. We show that there exists a finite number of intervals $J_i$ such that $\cup_i J_i = [0, \infty)$ and $\vertiii{K}_{L^{\infty}(J_i)} \le 1$. From this and Theorem 9.3.13 in \citep{gripenberg1980volterra}, it will follow that the resolvent is also of type $L^\infty$ on $[0, \infty)$. Since $\gamma(t) \to \widetilde{\gamma}$, there exists a $t_0$ such that for all $t \ge t_0$, $\gamma(t) \le \widetilde{\gamma} + \varepsilon$. This $\varepsilon > 0$ can be chosen sufficiently small such that $\widetilde{\gamma} + \varepsilon < 2\big (n^{-1}\tr( (\AA^T \AA)^2 (\AA^T\AA + \delta \II_d)^{-1} ) \big )^{-1}$ (see \eqref{eq:learning_rate_convergence_threshold}, Assumption on $\widetilde{\gamma}$). First, we observe that 
\[\sup_t \sup_{0 \le s \le t_0} K(t,s) \le n^{-1}\widehat{\gamma}^2 \tr \big ( (\AA^T\AA)^2 e^{2(\AA^T \AA + \delta \II_d)\Gamma(t_0)} \big ) < \infty.\]
We break up the interval $[0, t_0]$ into finitely many intervals of length each of which has a length strictly less than $\big ( n^{-1}\widehat{\gamma}^2 \tr \big ( (\AA^T\AA)^2 e^{2(\AA^T \AA + \delta \II_d)\Gamma(t_0)} \big ) \big )^{-1}$. If we denote these intervals by $J_i$, then it immediately follows by bounding the integral using the sup of $K$ multiplied by the length of the interval $J_i$ that 
\[ \vertiii{K}_{L^{\infty}(J_i)} = \sup_{t \in J_i} \int_{J_i} K(t,s) \, \dif s < 1. \]
It only remains to show on the tail, that is, $J_{\infty} \defas (t_0, \infty)$, for which $\vertiii{K}_{L^{\infty}(J_{\infty})} < 1$. Using the same change of variables as in \eqref{eq:change_variables_volterra} and our choice of $t_0$, we have that for all $t \ge t_0$
\begin{align*}
    \int_{t_0}^t K(t,s) \, \dif s 
    &\le \int_{t_0}^t n^{-1} (\widetilde{\gamma} + \varepsilon) \gamma(s) \tr\bigg ( (\nabla^2 \mathscr{L}) \AA^T \AA \exp \big (-2(\AA^T\AA + \delta\II_d)(\Gamma(t)-\Gamma(s)) \big) \bigg ) \, \dif s\\
    &= \int_{\Gamma(t_0)}^{\Gamma(t)} n^{-1} (\widetilde{\gamma} + \varepsilon) \tr\bigg ( (\nabla^2 \mathscr{L}) \AA^T \AA \exp \big (-2(\AA^T\AA + \delta\II_d)(\Gamma(t)-s) \big) \bigg ) \dif s\\
    &\le \frac{\widetilde{\gamma} + \varepsilon}{2n} \tr \big ( (\AA^T\AA)^2 (\AA^T \AA + \delta\II_d)^{-1} \big ) < 1.
\end{align*}
The last inequality following by our assumption on $\widetilde{\gamma} + \varepsilon$ being sufficiently small. By Theorem 9.3.13 in \citep{gripenberg1980volterra}, we have that the resolvent is also of type $L^\infty$ on $[0, \infty)$. We also have that $K(t,s)$ is of bounded type, that is the kernel is bounded (see \citep[Definition 9.5.2]{gripenberg1980volterra} for precise definition). Since the forcing term $\mathscr{L}(\bm{\mathscr{X}}_{\Gamma(t)}^{\text{gf}})$ is bounded, then it follows by \citep[Theorem 9.5.4]{gripenberg1980volterra} that the solution to the Volterra equation \eqref{eq:psi_volterra}, $\Psi_t$, is bounded. 

We now show that $\Psi_t$ exists at infinity. Fix a $\varepsilon > 0$. By the assumptions on the learning rate, there exists a $t_0 > 0$ such that for all sufficiently large $t \ge s \ge t_0$
\begin{equation}
    \begin{aligned}
    \widetilde{\gamma}- \varepsilon  \le \gamma(t) \le \widetilde{\gamma} + \varepsilon \quad \text{and} \quad (\widetilde{\gamma} - \varepsilon)(t-s) \le \Gamma(t) - \Gamma(s) \le (\widetilde{\gamma}+ \varepsilon) (t-s). 
    \end{aligned}
\end{equation}
Using these inequalities for $\gamma(t)$, we get an upper bound and lower bound on the kernel $K(t,s)$ which we denote by $\overline{K}(t,s)$ and $\underline{K}(t,s)$, respectively. Specifically for all $t, s \ge t_0$, 
\begin{equation}
    \begin{aligned} \label{eq:kernel_bound_1}
    K(t,s) &\le \overline{K}(t,s) \defas n^{-1} (\widetilde{\gamma}+ \varepsilon)^2 \tr\bigg ( (\nabla^2 \mathscr{L}) \AA^T \AA \exp \big (-2(\AA^T\AA + \delta\II_d)(\widetilde{\gamma}-\varepsilon)(t-s) \big) \bigg )\\
    K(t,s) & \ge \underline{K}(t,s)\defas n^{-1} (\widetilde{\gamma}- \varepsilon)^2 \tr\bigg ( (\nabla^2 \mathscr{L}) \AA^T \AA \exp \big (-2(\AA^T\AA + \delta\II_d)(\widetilde{\gamma}+\varepsilon)(t-s) \big) \bigg ).
    \end{aligned}
\end{equation}
The kernels $\overline{K}(t,s)$ and $\underline{K}(t,s)$ are substantially nicer than the original $K(t,s)$ because they are proper convolution kernels. Here one can define $\overline{K} \, : \, [0, \infty) \to \mathbb{R}$ by
\[
\overline{K}(t) \defas  n^{-1} (\widetilde{\gamma} + \varepsilon)^2 \tr\bigg ( (\nabla^2 \mathscr{L}) \AA^T \AA \exp \big (-2(\AA^T\AA + \delta\II_d)(\widetilde{\gamma}-\varepsilon)t \big) \bigg ).
\]
Then it follows that $\overline{K}(t,s) = \overline{K}(t-s)$. A similar result holds for $\underline{K}(t,s)$. 

For ease of notation, define the forcing function: for $t \ge t_0$
\begin{equation}
    F(t) \defas \mathcal{L}(\bm{\mathscr{X}}_{\Gamma(t)}^{\text{gf}}) + \int_0^{t_0} K(t,s) \Psi_s \, \dif s,
\end{equation}
where $\Psi_t$ is a solution to \eqref{eq:psi_volterra}. Because $\Psi_s$ is bounded, it follows that $\lim_{t \to \infty} \int_0^{t_0} K(t,s) \Psi_s = 0$. Also it is clear that the $F(t)$ is bounded. 

Using the upper/lower bound on the kernel \eqref{eq:kernel_bound_1}, we can squeeze the value of $\Psi_t$ between two expressions: for $t, s\ge t_0$,
\begin{equation}
\begin{aligned} \label{eq:gronwall}
F(t) + \int_{t_0}^t \underline{K}(t,s) \Psi_s \, \dif s \le \Psi_t \le F(t) + \int_{t_0}^t \overline{K}(t,s) \Psi_s \, \dif s.
\end{aligned}
\end{equation}
 Using a similar argument for $K(t,s)$ and choosing $\varepsilon$ sufficiently small, $\overline{K}(t,s)$ and $\underline{K}(t,s)$ are $L^{\infty}$-type on $[0, \infty)$. Moreover using a similar argument as we did for $K$ itself, the norms $\vertiii{\overline{K}}_{L^\infty([0,\infty))} < 1$ and  $\vertiii{\underline{K}}_{L^\infty([0,\infty))} < 1$. Here we used the upper bound on $\widetilde{\gamma}$ in \eqref{eq:learning_rate_convergence_threshold} and a sufficiently small $\varepsilon$. Note we do not need to break up into finite intervals. As before, the resolvent then is of $L^{\infty}$-type on $[0, \infty)$ \citep[Corollary 9.3.10]{gripenberg1980volterra}. Further because of non-negativity, Proposition 9.8.1 in \citep{gripenberg1980volterra} yields that the resolvents are also non-negative. 
 
 Consider the upper bound (a similar argument will hold for the lower bound). We can apply Gronwall's inequality \eqref{eq:gronwall} \citep[Theorem 9.8.2]{gripenberg1980volterra}. It follows that $\Psi_t$ is upper bounded (lower bounded) by the solutions $\overline{\Psi}_t$ ($\underline{\Psi}_t$) to the following convolution Volterra equations
\[ \overline{\Psi_t} = F(t) + \int_{t_0}^t \overline{K}(t,s) \overline{\Psi}_s \, \dif s \quad \text{and} \quad \underline{\Psi}_t = F(t) + \int_{t_0}^t \underline{K}(t,s) \underline{\Psi}_s \, \dif s.\]
Specifically, we have $\underline{\Psi}_t \le \Psi_t \le \overline{\Psi}_t$ for all $t \ge t_0$. Since $\overline{\Psi}_t$ and $\underline{\Psi}_t$ are solutions to a proper convolution-type Volterra equation and $F(t)$ has a limit at infinity (denoted by $F(\infty)$), by \citep{Asmussen}, for $t \ge t_0$
\begin{equation}
    \limsup_{t \to \infty} \Psi_t \le \limsup_{t \to \infty} \overline{\Psi}_t = F(\infty) 
    \big (1- \vertiii{ \overline{K}}_{L^{\infty}([t_0, \infty))} \big )^{-1} \le F(\infty) \big (1- \vertiii{ \overline{K}}_{L^{\infty}([0, \infty))} \big )^{-1},
\end{equation}
and similarly, 
\begin{equation}
    \liminf_{t \to \infty} \Psi_t \ge \liminf_{t \to \infty} \underline{\Psi}_t  \le F(\infty) \big (1- \vertiii{ \underline{K}}_{L^{\infty}([0, \infty))} \big )^{-1} \quad \text{where} \quad F(\infty) \defas \lim_{t\to \infty} F(t) = \mathcal{L}(\bm{\mathscr{X}}_{\infty}^{\text{gf}}).
\end{equation}
A simple computation yields that 
\begin{equation}
    \begin{gathered}
    \vertiii{ \overline{K}}_{L^{\infty}([0, \infty))} = \frac{(\tilde{\gamma}+\varepsilon)^2}{\tilde{\gamma}-\varepsilon} G \quad \text{and} \quad  \vertiii{ \underline{K}}_{L^{\infty}([0, \infty))} = \frac{(\tilde{\gamma}-\varepsilon)^2}{\tilde{\gamma}+\varepsilon} G\\
    \text{where} \quad G \defas (2n)^{-1} \tr \big ( (\AA^T\AA)^2 (\AA^T\AA + \delta \II_d)^{-1} \big ). 
    \end{gathered}
\end{equation}
So for any sufficiently small $\varepsilon > 0$, we have that 
\begin{equation}
   \left ( 1 - \frac{(\tilde{\gamma}-\varepsilon)^2}{\tilde{\gamma}+\varepsilon} G \right )^{-1} \cdot \mathcal{L}(\bm{\mathscr{X}}_{\infty}^{\text{gf}}) \le \liminf_{t \to \infty} \Psi_t \le \limsup_{t \to \infty} \Psi_t \le \left ( 1- \frac{(\tilde{\gamma}+\varepsilon)^2}{\tilde{\gamma}-\varepsilon} G \right )^{-1} \cdot \mathcal{L}(\bm{\mathscr{X}}_{\infty}^{\text{gf}}).
\end{equation}
As this holds for any sufficiently small $\varepsilon$, the result follows by sending $\varepsilon \to 0$.
\end{proof}

\section{Martingale errors} \label{sec:martingale_errors}
The martingale errors $\mathcal{M}_t^{\text{hSGD}}$, $\mathcal{M}_t^{\text{grad}}$, and $\mathcal{M}_t^{\text{quad}}$,  \eqref{eq:ito_diffusion}, \eqref{eq:martingale_grad}, and \eqref{eq:martingale_quad} respectively, arise due to the stochastic processes governing homogenized SGD and the randomness in the algorithm itself. Controlling the error from homogenized SGD will be simple as the only randomness comes from the Brownian motion. On the other hand, the martingale errors from the randomness in the algorithm, that is, $\mathcal{M}_t^{\text{grad}}$ and $\mathcal{M}_t^{\text{quad}}$ are small, in part because the singular vector matrix $\UU$ is delocalized (Assumption~\ref{ass: laundry_list}). Estimating that the error generated by these martingales requires some substantial build-up (see Section~\ref{sec:martingale_errors_sgd}). Without loss of generality, we  normalize our matrix $\AA$ so that it has row sum always $1$ without loss of generality. 

First we control the martingale that arises in homogenized SGD, that is, 
\[ 
    \mathcal{M}_t^{\text{hSGD}}(q) = \int_0^t \gamma(s) \nabla q(\YY_s) \cdot \sqrt{\tfrac{2}{n} \widehat{\mathscr{L}}(\YY_s)\SSigma^T \SSigma} \cdot \dif \BB_s. 
\]
 To control the fluctuations of this martingales, we need to control its quadratic variation defined as follows. Consider a partition of time for $[0, t]$, that is, $0 = t_0 < t_1 < \hdots < t_n = t$. We define for any continuous process $Y$,
\begin{align*}
    \Delta Y_{t_k} \defas Y_{t_k} - Y_{t_{k-1}}
\end{align*}
The \textit{quadratic variation} $[Y_t(n)]$ is the limit of the sum of squares of all jumps of the process as the size of the partition $\Delta t = \max_i \{t_{i}-t_{i-1}\} \to 0$, that is, 
\begin{align} \label{eq:quad_variation}
    [Y_t] \defas \lim_{\Delta t \to 0} \sum_{k = 1}^{n} (\Delta Y_{t_k})^2.
\end{align}
% We can perform a Doob's decomposition on $[Y_t]$ and thus we define the \textit{compensator} for the quadratic variation, denoted by $\langle Y_t \rangle$, 
% \begin{align*}
%     \langle Y_t \rangle &\defas \int_0^t \mathcal{C}_s \, \dif s \quad \text{where} \quad [Y_t] = [Y_0] + \int_0^t \mathcal{C}_s \, \dif s + \mathcal{M}^{(1)}_t.
% \end{align*}
% Here $(\mathcal{M}^{(1)}_t \, : \, t\ge 0)$ are $\mathcal{F}_t$-adapted martingales. 
Using the quadratic variation of $\mathcal{M}_t^{\text{hSGD}}$, we show that the martingale from homogenized SGD is small. 

\begin{proposition}[Homogenized SGD martingale] \label{prop:martingale_errors_hSGD} Let $q$ be a quadratic with $\|\nabla^2 q\|$ and $\|\nabla q(0)\|_2$ bounded, independent of $n$. For any $T > 0$, with overwhelming probability, 
\begin{equation} \begin{gathered}
    \sup_{0 \le t \le T} | \mathcal{M}_{t \wedge \vartheta}^{\text{hSGD}}(q) | \le C(\delta, \widehat{\gamma}, \SSigma, \UU^T \bb, T) (\|\nabla^2 q \|^2 + \|\nabla q(0)\|^2)^{1/2} n^{\varepsilon - 1/2}
\end{gathered} \end{equation}
where $C(\delta, \widehat{\gamma}, \SSigma, \bb, T)$ is a constant depending only on the norms of the inputs and independent of $n$ and $q$.
\end{proposition}

\begin{proof} First, we compute the quadratic variation 
\begin{equation}
    \sup_{0 \le t \le T} [\mathcal{M}_{t \wedge \vartheta}^{\text{hSGD}}(q)] = \sup_{0 \le t \le T}  \frac{2}{n} \int_0^{t \wedge \vartheta} \gamma^2(s) \widehat{\mathscr{L}}(\YY_s) \nabla q(\YY_s)^T \SSigma^T \SSigma \nabla q(\YY_s) \, \dif s.
\end{equation}
A simple computations shows that 
\[
\widehat{\mathscr{L}}(\YY_s) \le 2 ( \|\SSigma^T \SSigma \| \|\YY_s\|^2 + \| \bb\|^2 ) \quad \text{and} \quad  \|\nabla q(\YY_s)\|^2_2 \le 2 (\|\nabla^2q\|^2 \|\YY_s\|^2 + \|\nabla q(0)\|^2).
\]
The definition of $\vartheta$ yields the following upper bound 
\begin{equation} \begin{aligned}
    \sup_{0 \le t \le T} [\mathcal{M}_{t \wedge \vartheta}^{\text{hSGD}}(q)] \le C(\widehat{\gamma}, \SSigma, \bb, T) (\|\nabla^2 q\|^2 + \|\nabla q(0)\|^2) n^{\varepsilon -1}.
\end{aligned}
\end{equation}
Since $\displaystyle \sup_{0 \le t \le T} [ \mathcal{M}_{t \wedge \vartheta}^{\text{hSGD}}(q) ] \le b$ a.s., then we have that $\displaystyle \text{Pr}(\sup_{0 \le t \le T} | \mathcal{M}_{t \wedge \vartheta}^{\text{hSGD}}(q)| > a) \le \exp(-a^2 / 2 b)$. Let $a = \sqrt{C(\widehat{\gamma}, \SSigma, \bb, T) (\|\nabla^2 q\|^2 + \|\nabla q(0)\|^2)} n^{\varepsilon -1/2}$ and the result immediately follows. 
\end{proof}

\subsection{Martingale errors in SGD} \label{sec:martingale_errors_sgd}
The martingale errors $\mathcal{M}_t^{\text{grad}}$ and $\mathcal{M}_t^{\text{quad}}$, \eqref{eq:martingale_grad} and \eqref{eq:martingale_quad} respectively, are due to the randomness in the algorithm itself. They in part are small because the singular vector matrix $\UU$ is delocalized, in that its off-diagonal entries in any fixed orthogonal basis are $n^{\varepsilon - 1/2}$ with overwhelming probability. 

Estimating that the error generated by the martingales requires substantial build-up.
As an input, we will use the stopping time $\vartheta_{\varepsilon} = \vartheta$ on the size of $\nnu_t$ and $\YY_t$ processes (see \eqref{eq:stopping_time_norm}), and so we work with the stopped process defined for any $t \ge 0$ by $\nnu_t^{\vartheta} \defas \nnu_{t \wedge \vartheta}$. The most important and technical input, which we will use in multiple places, is that the function values do not concentrate too heavily in any coordinate direction. In some sense, this is the most challenging and technical statement that we will prove:

\begin{proposition} \label{prop: entry_bound_function_values} Under Assumption \ref{ass: laundry_list}, for any $T \ge 0$, any $\alpha > \theta+\varepsilon$,
\begin{equation}
    \sup_{0 \le t \le T} \sup_{1 \le i \le n} \big (\ee_i^T (\UU \SSigma \nnu_t^{\vartheta} - \bb) \big )^2 \le n^{2\alpha -1}
%    \quad\text{and}\quad
%    \sup_{0 \le t \le T} \sup_{1 \le i \le n} \big (\ee_i^T \nnu_t^{\vartheta} \big )^2 \le n^{2\alpha -1}
\end{equation}
with overwhelming probability (conditioned on $\mathcal{G}$).
\end{proposition}

Proposition~\ref{prop: entry_bound_function_values} allows us to adopt a stopping time $\hbar$, defined as 
\begin{equation} \label{eq:stopping_time_hbar}
    \hbar \defas \inf \big \{ t \le \vartheta \, : \, \max_{1 \le i \le n} ( \ee_i^T (\UU \SSigma \nnu_t - \bb))^2 > n^{-1 + 2\alpha}  \big \}. 
\end{equation}
With this proposition in hand, we can bound the martingale errors.

\begin{proposition}[Martingale error bounds] \label{prop:martingale_errors_SGD} Suppose $0 < \theta < \alpha < 1/2$ and let $0 < \varepsilon < \alpha$ in the stopping time $\vartheta$ such that $ \alpha + \varepsilon < 1/2$. Let $q$ be a quadratic with $\|\nabla^2 q(\xx)\|$ and $\|\nabla q(0)\|_2$ bounded independent of $n$. For any $T > 0$, with overwhelming probability, 
\begin{equation} \begin{gathered}
    \sup_{0 \le t \le T} | \mathcal{M}_{t \wedge \vartheta}^{\text{grad}}(q) | \le C(\delta, \widehat{\gamma}, \SSigma, \bb, T) (\|\nabla^2 q \| + \|\nabla q(0)\|_2) n^{-1/2 + \alpha + \varepsilon}\\
    \text{and} \quad \sup_{0 \le t \le T} | \mathcal{M}_{t \wedge \vartheta}^{\text{quad}}(q) | \le C(\delta, \widehat{\gamma}, \SSigma, \bb) T^{1/2} \| \nabla^2 q\| n^{- 1/2 + \alpha + \varepsilon},
\end{gathered} \end{equation}
where $C(\delta, \widehat{\gamma}, \SSigma, \bb, T)$ is a constant depending only on the norms of the inputs and independent of $n$ and $q$. The choice of $\varepsilon$ in the stopping time $\vartheta$ ensures that $-1/2 + \alpha + \varepsilon < 0$.
\end{proposition}
The proofs of Propositions~\ref{prop: entry_bound_function_values}  and \ref{prop:martingale_errors_SGD} are deferred to Sections~\ref{sec:proof_prop_bootstrap} and \ref{sec:proof_prop_martingale_sgd} respectively.

% We recall that the martingale errors $\mathcal{M}_t^{\text{grad}}$ and $\mathcal{M}_t^{\text{quad}}$ are defined in \eqref{eq: statistics_nu} and \eqref{eq:quadratic} respectively.

\subsubsection{General martingale results}
We formulate some general concencetration lemmas for c\`adl\`ag, finite variation martingales $Y_t$ with jumps given exactly by $\{\tau_k \, : \, k \ge 0 \}$ (c.f. \cite[Appendix D]{paquetteSGD2021}). For such a process, the jumps entirely determine its fluctuations. Such general results will be applied to show that the martingales $\mathcal{M}_t^{\text{grad}}$ and $\mathcal{M}_t^{\text{quad}}$ go to $0$ as $n \to \infty$. We define for any c\`adl\`ag process $Y$, 
\begin{align*}
    \Delta Y_t \defas Y_t - Y_{t-}
\end{align*}
which is $0$ for all $t$ except $\{\tau_k, k \ge 0\}$. For reference throughout this section, we record the jumps of $\mathcal{M}_t^{\text{grad}}$ and $\mathcal{M}_t^{\text{quad}}$ are given by
\begin{equation}
    \begin{aligned} \label{eq:martingale_delta}
        \Delta \mathcal{M}_t^{\text{grad}} &= - \gamma(\tau_k-) \nabla q(\nnu_{\tau_k-})^T (\SSigma^T \UU^T \PP_k ( \UU \SSigma \nnu_{\tau_k-} - \bb) + \tfrac{\delta}{n} \nnu_{\tau_k-} )\\
        \Delta \mathcal{M}_{\tau_k}^{\text{quad}} &= \gamma^2(\tau_k-) (\SSigma^T \UU^T \PP_k (\UU \SSigma \nnu_{\tau_k-} - \bb) - \tfrac{\delta}{n} \nnu_{\tau_{k}-})^T (\nabla^2 q ) (\SSigma^T \UU^T \PP_k (\UU \SSigma \nnu_{\tau_k-} - \bb) - \tfrac{\delta}{n} \nnu_{\tau_{k}-}).
    \end{aligned}
\end{equation}
To control the fluctuations of these martingales, we need to control their quadratic variations. The \textit{quadratic variation} $[Y_t]$ is the sum of squares of all jumps of the process, that is, 
\begin{align} 
    [Y_t] \defas \sum_{k = 1}^{N_t} (\Delta Y_{\tau_k})^2.
\end{align}
We can perform a Doob's decomposition on $[Y_t]$ and thus we define the \textit{compensator} for the quadratic variation, denoted by $\langle Y_t \rangle$, 
\begin{align*}
    \langle Y_t \rangle &\defas \int_0^t \mathcal{C}_s \, \dif s \quad \text{where} \quad [Y_t] = [Y_0] + \int_0^t \mathcal{C}_s \, \dif s + \mathcal{M}^{(1)}_t.
\end{align*}
Here $(\mathcal{M}^{(1)}_t \, : \, t\ge 0)$ are $\mathcal{F}_t$-adapted martingales.

Moreover, for some of the martingales we consider here, it is possible to find good events on which the quadratic variation or the compensator variations are in control. Then it is a relatively standard fact that the fluctuations of these process are in control:

\begin{lemma}[Lemma D.1 in \cite{paquetteSGD2021}] \label{lem:general_martingale_bound} Suppose $(Y_t \, : \, t \ge 0)$ is a c\`adl\`ag finite variation martingale. Suppose there is an event $\mathcal{G}$ which is measurable with respect to $\mathcal{F}_0$ that holds with overwhelming probability, and so that for some $T > 0$
\begin{equation*}
    (i). \, \, [Y_T] \bm{1}_{\mathcal{G}} \le \tfrac{\beta}{nT} N_T; \quad \text{or} \quad (ii). \, \, \langle Y_T \rangle \bm{1}_{\mathcal{G}} \le 1 \quad \text{and} \quad \max_{0 \le t \le T} |Y_t - Y_{t-}| \bm{1}_{\mathcal{G}} \le 1.
\end{equation*}
Then for any $\widehat{\varepsilon} > 0$ with overwhelming probability, 
\[ \sup_{0 \le t \le T} |Y_t| \le n^{\widehat{\varepsilon}}. \]
\end{lemma}

% The martingale errors are due to the randomness in the algorithm itself. They are small in part because the singular vector matrix $\UU$ is delocalized, in that its off-diagonal entries in any fixed orthogonal basis are $n^{\varepsilon -1}$ with overwhelming probability. There are two martingale errors that require controlling
% \begin{align}
%     \mathcal{M}_t^{\text{grad}}(q) &= \int_0^t \nabla q(\nnu_{s-}) \cdot \dif \mathcal{M}_s \\
%     \mathcal{M}_t^{\text{quad}}(q) &= \sum_{0\le s \le t} (\nabla \nnu_s)^T (\nabla^2 q) (\nabla \nnu_s) - \int_0^t \mathcal{B}_s \, \dif s,
% \end{align}
% where the function $q$ is any quadratic and the martingales have jumps given by \eqref{eq:martingale_delta}. Estimating this error requires substantial build-up. The most important technical input, which we will use in multiple places, is that the function values do not concentrate too heavily in any coordinate direction. 

% \begin{proposition} \label{prop: entry_bound_function_values} For any $T \ge 0$, any $\varepsilon > 0$,
% \begin{equation}
%     \sup_{0 \le t \le T} \sup_{1 \le i \le n} \big (\ee_i^T (\UU \SSigma \nnu_t^{\vartheta} - \bb) \big )^2 \le n^{\varepsilon -1}
%     \quad\text{and}\quad
%     \sup_{0 \le t \le T} \sup_{1 \le i \le n} \big (\ee_i^T \nnu_t^{\vartheta} \big )^2 \le n^{\varepsilon -1}
% \end{equation}
% with overwhelming probability.
% \end{proposition}

\subsubsection{Proof of Proposition~\ref{prop: entry_bound_function_values} with bootstrap argument} \label{sec:proof_prop_bootstrap}

Proposition~\ref{prop: entry_bound_function_values} makes substantial use of Assumptions~\ref{assumption:Target},~\ref{ass: laundry_list}, and \ref{assumption:init}. We organize these assumptions into a single list for convenience here.%and the resolvent of $\AA \AA^T$ for reference in this section.
% We also need some \emph{a priori} delocalization information about the matrix.  This is most easily encapsulated in a single resolvent bound.
%Recall that $R(z)$ is the resolvent of $\AA^T \AA$ is $R(z) =  (z-\AA^T\AA)^{-1}$ and Assumption~\ref{ass: laundry_list} states:
\paragraph{Assumption.} \textit{
  Let $\Omega$ be a positively oriented smooth contour enclosing $[-1,1]$ of length at most $100\pi$ and contained in the complex disk of radius $3$.
  Suppose spectral norm of $\AA^T\AA$ is bounded by $1$ with high probability.
  Suppose there is a $\theta \in (0,\tfrac 12)$ and an event $\mathcal{G}$ that holds with high
  probability on which
  \begin{enumerate}
    \item 
      \(
	\max_{z \in \Omega} \max_{1 \leq i \leq n} |\ee_i^T R(z; \AA^T\AA) \xx_0| \leq n^{\theta-1/2}.
      \)
    \item
      \(
	\max_{z \in \Omega} \max_{1 \leq i \leq n} |\ee_i^T R(z; \AA\AA^T) \bb| \leq n^{\theta-1/2}.
      \)
    \item
      \(
	\max_{z \in \Omega} \max_{1 \leq i \neq j \leq n} |\ee_i^T R(z; \AA\AA^T) \ee_j^T| \leq n^{\theta-1/2}.
      \)
    \item
      \(
      \max_{z \in \Omega} \max_{1 \leq i \leq n} |\ee_i^T R(z; \AA\AA^T) \ee_i - \tfrac 1n\tr R(z; \AA\AA^T)| \leq n^{\theta-1/2}.
      \)
  \end{enumerate}
}

We note that by contour integration, an analytic function $f$ of the matrix $\AA^T\AA$ recovers the same estimate, up to constants that depend on the function, e.g.\ 
\begin{equation}\label{eqE:rnext}
  \max_{z \in \Omega} \max_{1 \leq i \leq n} |\ee_i^T f(\AA^T\AA) \nnu_0| \leq 50 \max_{|z| \leq 3} |f(z)| n^{\theta-1/2}.
\end{equation}
Going forward, we shall suppose that the constant $\epsilon$ in the definition of $\vartheta$ is taken to be much smaller than $\theta.$

We also need the following consequence of Assumption \ref{ass: laundry_list}
\begin{lemma}\label{lem:weirdo}
  Let $\ZZ = (\SSigma^T\UU^T (\II_n - \ee_k\ee_k^T)\UU\SSigma)$ for any $1 \leq k \leq n$.
  Then $\ZZ$ is matrix of norm at most $\|\SSigma\|^2$ and 
  there is a constant $C=C(\Omega)$
  \[
    \max_{z \in \Omega} \max_{1 \leq i \leq n} |\ee_i^T \UU \SSigma R(z; \ZZ)\SSigma^T\UU^T\bb| \leq C n^{\theta-1/2}
  \]
  as well as
  \[
    \max_{z \in \Omega} \max_{1 \leq i \neq j \leq n} |\ee_i^T \UU \SSigma R(z; \ZZ)\SSigma^T\UU^T\ee_j^T| \leq C n^{\theta-1/2}.
  \]
\end{lemma}
\begin{proof}
The norm bound on $\ZZ$ follows simply from bounding the projection matrix $\UU^T(\II_n-\ee_k\ee_k^T)\UU$ in norm by $1$.  The two displayed bounds will follow from Assumption \ref{ass: laundry_list} with the same constant $C$ and from the same argument, and so we show just the first bound.

First observe that by rotation invariance, it suffices to show
  \[
    \max_{z \in \Omega} \max_{1 \leq i \leq n} |\ee_i^T \AA R(z; \AA^T(\II_n - \ee_k\ee_k^T)\AA)\AA^T\bb| \leq C n^{\theta-1/2}.
  \]
From the Sherman--Morrison--Woodbury formula,
\begin{equation}\label{eqE:SMW}
R(z; \AA^T(\II_n - \ee_k\ee_k^T)\AA)
=
R(z; \AA^T\AA) 
- 
\frac{R(z; \AA^T\AA)(\AA^T\ee_k\ee_k^T\AA)R(z; \AA^T\AA)}{1+\ee_k^T\AA R(z; \AA^T\AA) \AA^T \ee_k}.
\end{equation}
Furthermore, the meromorphic matrix curve $\AA R(z; \AA^T\AA) \AA^T$ can be checked (when expanding around $z=\infty$) to be
\[
\AA R(z; \AA^T\AA) \AA^T
=-\II_n+z R(z; \AA\AA^T).
\]
Hence from Assumption \ref{ass: laundry_list}, there is a constant $C$ so that
  \[
    \max_{z \in \Omega} \max_{1 \leq i \leq n} |\ee_i^T \AA R(z; \AA^T\AA)\AA^T\bb| \leq C n^{\theta-1/2}.
  \]
  as well as
  \[
    \max_{z \in \Omega} \max_{1 \leq i \neq j \leq n} |\ee_i^T \AA R(z; \AA^T\AA)\AA^T\ee_j| \leq C n^{\theta-1/2}.
  \]
Likewise, for the on--diagonal case, we have for some other $C$
  \[
    \max_{z \in \Omega} \max_{1 \leq i \leq n} |\ee_i^T \AA R(z; \AA^T\AA)\AA^T\ee_i| \leq \max_{z \in \Omega}
    \|R(z; \AA\AA^T)\|\|\SSigma\|^2
    \leq C.
  \]
  We also need the upper bound:
  \[
  |1+\ee_k^T \AA R(z; \AA^T\AA)\AA^T\ee_k|^{-1}
  =
  |z\ee_k^T R(z; \AA\AA^T)\ee_k|^{-1}
  \leq
  \bigl(|z| \times \min_{x\in [0,\|\AA\|_{\text{op}}^2]}|z-x|\bigr)^{-1},
  \]
  which holds on the event $\mathcal{G}.$
  Combining these estimates with \eqref{eqE:SMW}, the desired bounds follow.
  
\end{proof}

% \begin{proposition} \label{prop: entry_bound_function_values} Under Assumption \ref{ass: laundry_list}, for any $T \ge 0$, any $\alpha > \theta$,
% \begin{equation}
%     \sup_{0 \le t \le T} \sup_{1 \le i \le n} \big (\ee_i^T (\UU \SSigma \nnu_t^{\vartheta} - \bb) \big )^2 \le n^{2\alpha -1}
% %    \quad\text{and}\quad
% %    \sup_{0 \le t \le T} \sup_{1 \le i \le n} \big (\ee_i^T \nnu_t^{\vartheta} \big )^2 \le n^{2\alpha -1}
% \end{equation}
% with overwhelming probability (conditioned on $\mathcal{G}$).
% \end{proposition}

We turn to the proof of Proposition \ref{prop: entry_bound_function_values},
which we recall stated that for any $T \ge 0$, any $\alpha > \theta$,
\begin{equation*}
    \sup_{0 \le t \le T} \sup_{1 \le i \le n} \big (\ee_i^T (\UU \SSigma \nnu_t^{\vartheta} - \bb) \big )^2 \le n^{2\alpha -1}
%    \quad\text{and}\quad
%    \sup_{0 \le t \le T} \sup_{1 \le i \le n} \big (\ee_i^T \nnu_t^{\vartheta} \big )^2 \le n^{\alpha -1}
\end{equation*}
with overwhelming probability.

\begin{proof}[Proof of Proposition \ref{prop: entry_bound_function_values}]
The following proof is an extension of \cite[Proposition 15]{paquetteSGD2021} beyond the orthogonally invariant case and to the case of nonzero regularization parameter.
Let $\hbar=\hbar_\alpha$ be the stopping time (c.f.\ \eqref{eq:stopping_time_hbar})
\begin{equation*} 
  \hbar \defas \inf \big \{ t \le \vartheta \, : \, \max_{1 \le i \le n} ( \ee_i^T (\UU \SSigma \nnu_t - \bb))^2 > n^{-1 + 2\alpha}\}.% \quad \text{or} \quad \max_{1 \le i \le n} (\ee_i^T\nnu_t)^2 > n^{-1+2\alpha} \big \}.
\end{equation*}
%The proof of the second case from the first 

The strategy here is a \emph{bootstrap} argument.  We show inductively that if $\hbar_{\alpha'}$ does not occur with overwhelming probability, then  $\hbar_{\alpha''}$ does not occur with overwhelming probability for some smaller $\alpha'' > \theta+\varepsilon.$  To begin the induction we note that by taking $\alpha=1+\epsilon$ for some small $\epsilon>0,$ the claim is vacuous, as under norm on $\nnu_t$ (and $\SSigma$ and $\bb$), controlling these entries follows deterministically and trivially. To simplify the notation, we introduce $\gamma_t^{\hbar} \defas \gamma(t \wedge \hbar)$ and $\nnu_t^{\hbar} \defas \nnu_{t \wedge \hbar}$.

We divide the jumps $\left\{ \tau_k \right\}$ of the underlying Poisson process into two types $\left\{\tau_{k,1}, \tau_{k,2} \right\}$, those for which a coordinate not equal to $i$ is chosen and those for which coordinate $i$ is chosen.  These are independent Poisson processes, and we let $N_{t,1}$ and $N_{t,2}$ be the counting functions of the number of jumps from either type.

Define $\ZZ = (\SSigma^T\UU^T (\Id_n - \ee_i\ee_i^T)\UU\SSigma +\delta (\tfrac{n-1}{n}) \Id_n)$ which we observe is positive semi-definite and let 
\[
  V_k \defas 
  \max_{0 \leq t \leq \tau_{k,2}} 
  \max_{ a \geq 0} |\ee_i^T (\UU \SSigma e^{-a \ZZ }\nnu_t^{\hbar} - \bb) |
  \quad
  \text{for all } k \in \N_0,
\]
where for $k=0$ we take $\tau_{0,2}=0$ by convention.
We shall show that with overwhelming probability that $V_{N_{T,2}+1} < n^{\beta-1/2}$ for $\beta =\tfrac{\varepsilon+\theta + 3\alpha}{4}$.
By a union bound over $i,$ we may then repeat the argument, having replaced in the definition of $\hbar$, $\alpha \to \tfrac{\varepsilon+\theta + 3\alpha}{4}$.  By iterating this argument finitely many times, we can approach any desired error above $\theta.$ 

We shall need a concentration bound for martingales, which is the continuous version of the Freedman inequality.
\begin{lemma}\label{lem:raremartingales}
  Suppose that $(Y_t : t \geq 0)$ is a c\`adl\`ag pure jump martingale with $Y_0 = 0$ adapted to filtration $(\mathcal{F}_t: t \geq 0)$.  Let $N_t$ be the counting function of the jumps of $Y_t$, and suppose for some bounded stopping time $T > 0$
\[
\int_0^T \mathbb{E} [ \bigl(\Delta Y_{t}\bigr)^2~\vert~\Delta N_{t}=1, \mathcal{F}_{t}-] \leq 1
\quad
\text{and}
\quad
|\Delta Y_{t}||\Delta N_{t}| \leq 1.
\]
Then for any $\epsilon,T > 0$ with overwhelming probability
\[
\sup_{0 \leq t \leq T} |Y_t| \leq n^{\epsilon}.
\]
\end{lemma}

We remove the jumps of the second kind from $\nnu_t^{\hbar} - \nnu_{\tau_{k-1,2}}^{\hbar}$, by setting
\[
\omega_t^{(k)} \defas 
\int\limits_{ (\tau_{k-1,2}, t]}
\Delta \nnu_s^{\hbar} dN_{s,1},
\]
i.e.\ it simply disregards those increments of $\nnu$ in which the jump is made in the $i$-th coordinate.
Then we have $\omega_{\tau_{k,2}}^{(k)} = \nnu_{\tau_{k,2}-}^{\hbar} - \nnu_{\tau_{k-1,2}}^{\hbar}$.
The problem of controlling $\nnu_t^{\hbar}$ on $[\tau_{k-1,2},\tau_{k,2}]$ can be reduced to control of $\omega,$
as at the endpoint,
\[
  \nnu_{\tau_{k,2}}^{\hbar}
  =
  \nnu_{\tau_{k,2}-}^{\hbar}
  %\omega_{\tau_{k,2}}^{(k)}
  -\gamma_{\tau_{k,2}-}^{\hbar}
  \bigl(
  \SSigma^T \UU^T \ee_i \ee_i^T ( \UU \SSigma \nnu^{\hbar}_{\tau_{k,2}-} - \bb) + \tfrac{\delta}{n} \nnu^{\hbar}_{\tau_{k,2}-} 
  \bigr).
\]
Thus we conclude, after left multiplying through by $\ee_i^T \UU \SSigma$ and subtracting $\bb$,
\[
  \ee_i^T
  ( \UU \SSigma \nnu^{\hbar}_{\tau_{k,2}} - \bb)
  =
  \bigl(1-
  \gamma_{\tau_{k,2}-}^{\hbar}
  \ee_i^T
  \UU \SSigma 
  \SSigma^T \UU^T \ee_i
  \bigr)
  \ee_i^T
  ( \UU \SSigma \nnu^{\hbar}_{\tau_{k,2}-} - \bb)
  -\gamma_{\tau_{k,2}-}^{\hbar}
  \ee_i^T
  \UU \SSigma 
  \bigl(
  \tfrac{\delta}{n} \nnu^{\hbar}_{\tau_{k,2}-} 
  \bigr),
\]
and it follows that we can give a representation
\begin{equation}\label{eqE:Vk1}
  V_k
  \leq \widehat{\gamma} \delta n^{\theta-1} \|\SSigma\| 
  +
  \max_{\tau_{k-1,2} \leq t \leq \tau_{k,2}} 
  \max_{ a \geq 0} |\ee_i^T (\UU \SSigma e^{-a \ZZ} (\omega_t^{(k)} +\nnu^{\hbar}_{\tau_{k-1,2}})|.
\end{equation}

We give a martingale decomposition of $\omega_t^{(k)}$ (c.f.\ \eqref{eq: statistics_nu} with linear $q$) 
\[
\omega_t^{(k)}
=
- \int\limits^t_{\tau_{k-1,2}} 
\gamma(s) \bigl( 
\SSigma^T \UU^T (\Id_n - \ee_i \ee_i^T) ( \UU \SSigma \nnu^{\hbar}_{s} - \bb)
+ \delta (\tfrac{n-1}{n})\nnu^{\hbar}_{s} \bigr) \dif s
+\mathcal{M}_t,
\quad
\text{ for } t \in [\tau_{k-1,2},\hbar].
\]
In particular, using $\nnu_s^{\hbar} = \omega_s^{(k)} + \nnu_{\tau_{k-1,2}}^{\hbar}$ for all $s < \tau_{k,2}$ and introducing an integrating factor, we conclude for any $t \in [\tau_{k-1,2},\hbar]$,
\begin{equation}\label{eqE:integrated}
  \begin{aligned}
  \omega_t^{(k)}
    =
    &-\bigl(1- e^{- \ZZ (\Gamma(t)-\Gamma(\tau_{k-1,2}))}\bigr)\nnu^{\hbar}_{\tau_{k-1,2}} 
    +\int\limits_{\tau_{k-1,2}}^t \gamma(s) e^{-\ZZ (\Gamma(t)-\Gamma(s))}\SSigma^T\UU^T (\Id_n - \ee_i \ee_i^T) \bb \dif s 
    +e^{- \ZZ \Gamma(t)} {\XX}_{t}, %\\
%    &+\gamma\cdot\!\!\!\int\limits_{\tau_{k-1,2}}^t 
%    e^{-\gamma(\SSigma^T\SSigma +\delta)(t-s)}
%    \biggl(
%    \SSigma^T \UU^T \ee_i \ee_i^T ( \UU \SSigma \nnu^{\hbar}_{s} - \bb)
%    +\tfrac{\delta}{n} \nnu^{\hbar}_{s}
%    \biggr)
%    \dif s 
  \end{aligned}
\end{equation}
where 
\[
   {\XX}_{t}
   =
   \int\limits_{\tau_{k-1,2}}^t
   e^{\ZZ \Gamma(s)}\cdot
   \dif \mathcal{M}_{s}%^{\text{grad}}.
\]

%\[
%  \begin{aligned}
%  (\ee_i^T (\UU \SSigma \nnu_t^{\hbar} - \bb))^2 
%  \leq
%  &3(\ee_i^T (\UU \SSigma e^{-\gamma(\SSigma^T\SSigma + \delta)t \wedge \hbar}\nnu_0))^2 \\
%  +
%  &3\biggl(
%  \ee_i^T 
%  \biggl[\int_0^{t \wedge \hbar} \gamma \UU \SSigma e^{-\gamma (\SSigma^T\SSigma +\delta)(t\wedge \hbar-s)}\SSigma^T\UU^T\dif s - \Id \biggr]\bb
%  \biggr)^2 \\
%  +&3\biggl(
%  \ee_i^T 
%  (\UU \SSigma e^{-\gamma(\SSigma^T\SSigma + \delta)t\wedge \hbar} \XX_t^{\hbar})
%  \biggr)^2.
%  \end{aligned}
%\]
We now substitute \eqref{eqE:integrated} into the expressions we wish to control, namely
 \eqref{eqE:Vk1}.  Before doing so, we observe that the first integral can be simplified.
Evaluating the integral produces
\[
  \biggl[\int_{\tau_{k-1,2}}^t \gamma(s) e^{- \ZZ (\Gamma(t)-\Gamma(s))}\SSigma^T\UU^T(\Id_n - \ee_i \ee_i^T)\dif s\biggr]\bb
 =
 \biggl[( \ZZ^{-1} )(\Id_d - e^{- \ZZ(\Gamma(t)-\Gamma(\tau_{k-1,2}))})\SSigma^T\UU^T (\Id_n - \ee_i \ee_i^T)\biggr]\bb.
\]
%\textcolor{red}{All entries of this need to be uniformly controlled using Assumption \ref{ass: laundry_list}} and using \eqref{eqE:rnext}, 
From Lemma \ref{lem:weirdo} and contour integration (as in \eqref{eqE:rnext}),
there is a constant $C(\delta)$ and an event of high probability $\mathcal{G} \in \mathcal{F}_0$ such that for $t \in [\tau_{k-1,2}, \hbar]$
\begin{equation}\label{eqE:Vk2}
  \begin{aligned}
  \max_{t \in [\tau_{k-1,2},\tau_{k,2}]}
  \max_{ a \geq 0} |\ee_i^T (\UU \SSigma e^{-a \ZZ} (\omega_t^{(k)} +\nnu^{\hbar}_{\tau_{k-1,2}})|
  \leq
  &V_{k-1}
  +
  C(\delta) n^{\theta-1/2} \\
  +
  &
  \max_{t \in [\tau_{k-1,2},\tau_{k,2}]}
  \max_{ a \geq 0}
  |\ee_i^T(\UU \SSigma e^{-\ZZ (a+\Gamma(t))} \XX_t)|.
\end{aligned}
\end{equation}

Using \eqref{eqE:integrated}, and the intervening argument we have that for $t \leq \hbar$
\begin{equation}\label{eqE:XXbnd}
  \|e^{-\ZZ \Gamma(t)} \XX_t\|
  \leq
  C(\delta) n^{\theta}.
\end{equation}
Hence if $|\Gamma(t)-\Gamma(s)| \leq \frac{1}{n}$ and $s > t$
\begin{equation}
\begin{aligned}
  \max_{ a \geq 0}
  |\ee_i^T(\UU \SSigma e^{-\ZZ (a+\Gamma(t))} \XX_t)
  -
  \ee_i^T(\UU \SSigma e^{- \ZZ (a+\Gamma(s))} \XX_t)|
  &\leq
  C(\delta)n^{\theta} \|\SSigma\| \| e^{-(\Gamma(s)-\Gamma(t)) \ZZ} - \Id_n \|\\
  &\leq
  \widehat{\gamma} C(\delta)n^{\theta-1} \|\SSigma\| \|\ZZ\|.
\end{aligned}
\end{equation}
It follows that if we let $t_j = (j/n) \cdot1/{\widehat{\gamma}} + \tau_{k-1,2}$ for all $j \in \N,$ then
\begin{equation}\label{eqE:Vk3}
  \begin{aligned}
    V_{k}
    \leq V_{k-1} + C(\delta,\Sigma) n^{\theta-1/2}
    +\max_{j : t_j < \tau_{k,2}+1/n}
    \max_{t \in [\tau_{k-1,2},t_j]}
    \max_{a \geq 0}
    |\ee_i^T(\UU \SSigma e^{- \ZZ (a+\Gamma(t_j))} \XX_t^{\hbar})|.
\end{aligned}
\end{equation}

We may further relax the set using \eqref{eqE:XXbnd} to take $a$ in mesh of the curve $e^{\ZZ \log(1-t)}$ for $t \in [0,1]$ which has $\|\cdot\|$--norm spacing $(1/n).$  This can be done with $n^2$ points: for each eigenvalue $\lambda > 0$ of $\ZZ$, the graph of $(1-t)^\lambda$ can be discretized using $n$ points; doing this for each eigenvalue, and then taking the union of all these mesh points gives the desired set.  Let $\mathcal{A}$ be this mesh set.  Then we conclude it suffices to show that for any $a > 0$, any $t > \tau_{k-1,2},$ with overwhelming probability
\begin{equation}\label{eqE:EE}
    \max_{s \in [\tau_{k-1,2},t]}
    |\ee_i^T(\UU \SSigma e^{- \ZZ (a+\Gamma(t))} \XX_s^{\hbar})|
    \leq n^{\beta-1/2}/(\log n)
  \end{equation}
with overwhelming probability, as then it follows from \eqref{eqE:Vk3} that on the event $N_{T,2} \leq \log n$ (which has overwhelming probability) 
\[
  \max_{0 \leq t \leq T} 
   |\ee_i^T (\UU \SSigma \nnu_t^{\hbar} - \bb) |
    \leq
    V_{N_{T,2}+1}
    \leq
    (\log n+1)
    \bigl( C(\delta,\SSigma) + n^{\beta-1/2}/\log n).
\]
Hence for all $n$ sufficiently large (depending on $\beta$ and $\theta$), the proof would be complete.

So turning to the needed martingale inequality, with $Y_s \defas \ee_i^T(\UU \SSigma e^{-\ZZ (a+\Gamma(t))} \XX_s^{\hbar})$, we observe that
\[
\langle Y_t \rangle 
\leq
n
   \int\limits_{\tau_{k-1,2}}^t
   \mathbb{E}_\PP
   \biggl( \ee_i^T \UU \SSigma e^{-\ZZ (a+\Gamma(t)-\Gamma(s))}  
   \SSigma^T \UU^T \PP ( \UU \SSigma \nnu_{s} - \bb) %+ \tfrac{\delta}{n} \nnu_{s} 
   \biggr)^2 \dif s,
\]
here we take the expectation with respect to $\PP$, which is chosen uniformly over all entries not equal to $i$.  We have also subtracted a non-degenerate constant from the inside of the square (recall that the true predictable quadratic variation would have the conditional variance, which corresponds to subtracting the optimal constant the conditional mean).
Hence
\[
\langle Y_t \rangle 
\leq
   \int\limits_{\tau_{k-1,2}}^{t\wedge \hbar}
   \sum_{j \neq i}
   \biggl( \ee_i^T \UU \SSigma e^{-\ZZ (a+\Gamma(t)-\Gamma(s))}  
   \SSigma^T \UU^T \ee_j\biggr)^2 \biggl(\ee_j^T ( \UU \SSigma \nnu_{s}^{\hbar} - \bb) %+ \tfrac{\delta}{n} \nnu_{s} 
   \biggr)^2 \dif s.
\]
From Lemma \ref{lem:weirdo}
\[
\langle Y_t \rangle 
\leq 
%\|\SSigma\|_{op}^2
(\tau_{k,2}-\tau_{k-1,2})\|\UU \SSigma \nnu_{s}^{\hbar} - \bb\|^2 n^{2\theta-1}
\leq 
(\tau_{k,2}-\tau_{k-1,2})
n^{\varepsilon+2\theta-1}.
\]
As with overwhelming probability $(\tau_{k,2}-\tau_{k-1,2}) \leq n^{\alpha-\theta},$ we conclude that with overwhelming probability 
\[
\langle Y_t \rangle 
\leq n^{\varepsilon + \theta + \alpha - 1}.
\]
We also need a bound on the largest jumps of the martingale.  Indeed we have at a jump $s$
\[
|\Delta Y_s| =
\biggl| \ee_i^T(\UU \SSigma e^{-\ZZ (a+\Gamma(t)-\Gamma(s))}  
   \SSigma^T \UU^T \PP ( \UU \SSigma \nnu_{s}^{\hbar} - \bb) %+ \tfrac{\delta}{n} \nnu_{s} 
   \biggr|
\leq n^{-1/2+\alpha}
\max_{j \neq i}
\biggl|
\ee_i^T \UU \SSigma e^{- \ZZ (a+\Gamma(t)-\Gamma(s))}  
\SSigma^T \UU^T \ee_j
\biggr|.
\]
Thus by Lemma \ref{lem:weirdo}, we conclude
\[
|\Delta Y_s| \leq
C(\SSigma)n^{\alpha+\theta-1}.
\]
Hence we apply
Lemma \ref{lem:raremartingales} to the martingale $Y_s
\cdot n^{1/2-(\varepsilon + \theta+\alpha)}$, we conclude that with $\beta = (\theta+\epsilon + 3\alpha)/4$,
\[
\max_{0 \leq s \leq t} |Y_t| \leq n^{-1/2+\beta}.
\]
Hence
\eqref{eqE:EE} holds as desired.
\end{proof}

\subsubsection{Proof of Proposition~\ref{prop:martingale_errors_SGD}} \label{sec:proof_prop_martingale_sgd}
Recall, Proposition~\ref{prop: entry_bound_function_values} allows us to adopt a stopping time $\hbar$, defined as 
\begin{equation}
    \hbar \defas \inf \big \{ t \le \vartheta \, : \, \max_{1 \le i \le n} ( \ee_i^T (\UU \SSigma \nnu_t - \bb))^2 > n^{-1 + 2\alpha}  \big \}. 
\end{equation}
Under this stopping time, we can show that Proposition~\ref{prop:martingale_errors_SGD} holds:

% With this proposition in hand, we can bound the martingale errors.

% \begin{proposition}[Martingale error bounds] \label{prop:martingale_errors_SGD} Suppose $0 < \alpha < 2$ and $ \varepsilon \le \alpha$. Let $q$ be a quadratic with $\|\nabla^2 q(\xx)\|$ and $\|\nabla q(0)\|_2$ bounded independent of $n$. For any $T > 0$, with overwhelming probability, 
% \begin{equation} \begin{gathered}
%     \sup_{0 \le t \le T} | \mathcal{M}_{t \wedge \vartheta}^{\text{grad}}(q) | \le C(\delta, \gamma, \SSigma, \UU^T \bb, T) (\|\nabla^2 q \| + \|\nabla q(0)\|_2) n^{2 \alpha - 1/2}\\
%     \text{and} \quad \sup_{0 \le t \le T} | \mathcal{M}_{t \wedge \vartheta}^{\text{quad}}(q) | \le \gamma^2 (\delta +1)^2 T^{1/2} \| \nabla^2 q\| n^{2\alpha - 1/2},
% \end{gathered} \end{equation}
% where $C(\delta, \gamma, \SSigma, \UU^T \bb, T)$ is a constant depending only on the norms of the inputs and independent of $n$ and $q$.
% \end{proposition}

\begin{proof}[Proof of Proposition~\ref{prop:martingale_errors_SGD}] Assume that $\alpha \ge \varepsilon$ in the stopping times $\hbar$ and $\vartheta$ We begin by proving the result for the martingale that arises due to the Hessian of $q$. We will use Part (ii) of Lemma~\ref{lem:general_martingale_bound}.  As with overwhelming probability this does not occur, it suffices to show a bound for the stopped process $\nnu_{t}^{\hbar} \defas \nnu_{t \wedge \hbar}$, that is, $\hbar = \vartheta$ with overwhelming probability.

To simplify notation, we let $Y_t^{\hbar}(q) \defas \mathcal{M}_{t \wedge \hbar}^{\text{quad}}(q)$ and $\gamma_t^{\hbar} \defas \gamma(t \wedge \hbar)$. The jumps of this martingale are given by (see \eqref{eq:martingale_delta})
\begin{equation} \begin{aligned} \label{eq:stopping_time_quad}
    &\Delta Y_t^\hbar (q) = (\gamma_{\tau_k-}^\hbar)^2 (\SSigma^T \UU^T \PP_k(\UU \SSigma \nnu_{\tau_k-}^{\hbar} - \bb) - \tfrac{\delta}{n} \nnu_{\tau_k-}^{\hbar} )^T (\nabla^2 q) \big (\SSigma^T \UU^T \PP_k(\UU \SSigma \nnu_{\tau_k-}^{\hbar} - \bb) - \tfrac{\delta}{n} \nnu_{\tau_k-}^{\hbar} \big )\\
    &= \underbrace{(\gamma_{\tau_k-}^\hbar)^2 (\SSigma^T \UU^T \PP_k(\UU \SSigma \nnu_{\tau_k-}^{\hbar}-\bb) )^T (\nabla^2 q) (\SSigma^T \UU^T \PP_k(\UU \SSigma \nnu_{\tau_k-}^{\hbar} - \bb))}_{\text{(a)}} + \underbrace{(\gamma_{\tau_k-}^\hbar)^2 \tfrac{\delta^2}{n^2} (\nnu_{\tau_k-}^{\hbar})^T (\nabla^2 q) ) \nnu_{\tau_k-}^{\hbar}}_{\text{(b)}} \\
    & - \underbrace{2 (\gamma_{\tau_k-}^\hbar)^2 \tfrac{ \delta}{n} (\SSigma^T \UU^T \PP_k(\UU \SSigma \nnu_{\tau_k-}^{\hbar} - \bb))^T (\nabla^2 q) \nnu_{\tau_k-}^{\hbar}}_{\text{(c)}}.
\end{aligned}
\end{equation}
For the compensator of $[\YY_t]$, we compute $\displaystyle \mathcal{C}_t = \lim_{\varepsilon \downarrow 0} \varepsilon^{-1} \mathbb{E}[|\Delta Y^{\hbar}_{t + \varepsilon}|^2| \mathcal{F}_t]$. To avoid the unnecessarily long expressions, we bound each term that appears in the expected square of \eqref{eq:stopping_time_quad} separately. Note that $\widehat{\mathscr{L}}(\nnu_{\tau_k}^{\hbar}) \le C(\SSigma, \bb) \|\nnu_{\tau_k-}^{\hbar}\|^2 \le C(\SSigma, \bb) n^{\varepsilon}$, where $C(\SSigma, \bb)$ is a constant. We begin with $\text{(a)}^2$:
\begin{equation} \begin{aligned} \label{eq:a^2}
    \lim_{\varepsilon \downarrow 0} \varepsilon^{-1} & \mathbb{E}[\text{(a)}^2 | \mathcal{F}_t] = (\gamma_{\tau_k-}^{\hbar})^4 \sum_{i=1}^n \big [ (\UU \SSigma \nnu_{\tau_k-}^{\hbar} - \bb)^T \ee_i \ee_i^T \UU \SSigma (\nabla^2 q) \SSigma^T \UU^T \ee_i (\ee_i^T (\UU \SSigma \nnu_{\tau_k-}^{\hbar} - \bb)) \big ]^2\\
    &= (\gamma_{\tau_k-}^{\hbar})^4 \sum_{i=1}^n (\ee_i^T (\UU \SSigma \nnu_{\tau_k-}^{\hbar}-\bb))^4 (\ee_i^T \UU \SSigma (\nabla^2 q) \SSigma^T \UU^T \ee_i)^2\\
    &\le \widehat{\gamma}^4 \|\nabla^2 q\|^2 \max_{1 \le i \le n} \|\SSigma^T \UU^T \ee_i\|_2^4 \max_{1 \le i \le n} |\ee_i^T(\UU \SSigma \nnu_{\tau_k-}^{\hbar}-\bb)|^2 \sum_{i=1}^n |\ee_i^T(\UU \SSigma \nnu_{\tau_k-}^{\hbar}-\bb)|^2\\
    &\le \widehat{\gamma}^4 \|\nabla^2 q\|^2 n^{2\alpha-1} \widehat{\mathscr{L}}(\nnu_{\tau_k-}^{\hbar})\\
    &\le \widehat{\gamma}^4 C(\SSigma, \bb) \|\nabla^2 q\|^2 n^{2\alpha -1 + \varepsilon} 
\end{aligned}
\end{equation}
% where $\Delta \nnu_{\tau_k}^{\hbar} = \gamma ( \SSigma^T \UU^T \PP_k(\UU \SSigma \nnu_{\tau_k-}^{\hbar} - \bb) - \tfrac{\delta}{n} \nnu_{\tau_k-}^{\hbar})$. 
% A simple computation yields $|\Delta Y_{\tau_k}^{\hbar}(q)| \le \|\nabla^2 q\| \|\Delta \nnu_{\tau_k}^{\hbar}\|^2$ where we have assumed that $\|\nabla^2 q\|$ is bounded independent of $n$. Consequently, it suffices to bound $\|\Delta \nnu_{\tau_k}^{\hbar}\|^2_2$, using the stopping time $\hbar$ in \eqref{eq:stopping_time_quad}
% \begin{align*}
%     \|\Delta \nnu_{\tau_k}^{\hbar} \|^2 &\le \gamma^2 \|\SSigma^T \UU^T \ee_{i_k} \ee_{i_k}^T (\UU \SSigma \nnu_{\tau_k-}^{\hbar} - \bb) \|^2 + \frac{\gamma^2 \delta^2}{n^2} \|\nnu_{\tau_k-}^{\hbar}\|^2_2\\
%     & \qquad + \frac{2\gamma^2 \delta}{n} \|\nnu_{\tau_k-}^{\hbar}\|_2 \|\SSigma^T \UU^T \ee_{i_k} \ee_{i_k}^T (\UU \SSigma \nnu_{\tau_k-}^{\hbar} - \bb)\|_2 \\
%     &\le \gamma^2 \max_{i = 1, \hdots, n} \|\SSigma^T \UU^T \ee_i\|^2_2 \, \cdot \, \max_{i = 1, \hdots, n} |\ee_i^T (\UU \SSigma \nnu_{\tau_k}^{\hbar} - \bb)|^2 + \frac{\gamma^2 \delta^2}{n^2} \|\nnu_{\tau_k}^{\hbar}\|^2_2 \\
%     &\qquad + \frac{2 \gamma^2 \delta}{n} \|\nnu_{\tau_k-}^{\hbar}\|_2 \max_{i = 1, \hdots, n} \|\SSigma^T \UU^T \ee_i\|_2  \, \cdot \, \max_{i=1, \hdots, n} |\ee_i^T (\UU \SSigma \nnu_{\tau_k}^{\hbar} - \bb)|\\
%     &\le \gamma^2 n^{\alpha-1} + \gamma^2 \delta^2 n^{\varepsilon - 2} + 2\gamma^2 \delta n^{\alpha + \varepsilon/2 - 2} \le \gamma^2 (\delta + 1)^2 n^{\alpha-1}.
% \end{align*}
Here we used that $\|\SSigma^T \UU^T \ee_i\|_2$ is the $i$-th row of $\AA$ which we have normalized to be $1$ and the stopping time $\hbar$ on the entries of $\UU \SSigma \nnu_{\tau_k}-\bb$. 

We can do the same for the other terms in $\displaystyle \lim_{\varepsilon \downarrow 0} \varepsilon^{-1} \mathbb{E}[ |\Delta Y_{t + \varepsilon}^{\hbar}|^2 | \mathcal{F}_t]$ which we state below:

\begin{align*}
    \lim_{\varepsilon \downarrow 0} \varepsilon^{-1} & \mathbb{E}[\text{(a)} \cdot \text{(b)} | \mathcal{F}_t] \\
     &= (\gamma_{\tau_k-}^\hbar)^4 n^{-2}\delta^2 \sum_{i=1}^n (\ee_i^T (\UU \SSigma \nnu_{\tau_k-}^{\hbar} - \bb) )^2 \ee_i^T \UU \SSigma (\nabla^2 q) \SSigma^T \UU^T \ee_i (\nnu_{\tau_k-}^{\hbar})^T (\nabla^2 q) \nnu_{\tau_k-}^{\hbar}\\
    &\le \widehat{\gamma}^4 n^{-2} \delta^2 \|\nabla^2 q\|^2 \max_{1 \le i \le n} \|\SSigma^T \UU^T \ee_i\|^2_2 \|\nnu_{\tau_k-}^{\hbar}\|^2 \widehat{\mathscr{L}}(\nnu_{\tau_k-}^{\hbar})\\
    &\le \widehat{\gamma}^4 \delta^2 C(\SSigma, \bb) \|\nabla^2 q\|^2 n^{2\varepsilon - 2},\\
    \lim_{\varepsilon \downarrow 0} \varepsilon^{-1} & \mathbb{E}[ \text{(a)} \cdot \text{(c)} | \mathcal{F}_t]\\
    &= -2 (\gamma_{\tau_k-}^\hbar)^4 \delta n^{-1} \sum_{i=1}^n (\ee_i^T(\UU \SSigma \nnu_{\tau_k-}^{\hbar}- \bb))^ 3 (\ee_i^T \UU \SSigma (\nabla^2 q) \SSigma^T \UU^T \ee_i ) (\ee_i^T \UU \SSigma (\nabla^2 q) \nnu_{\tau_k-}^{\hbar})\\
    &\le 2 \widehat{\gamma}^4 \delta \|\nabla^2 q\|^2 \widehat{\mathscr{L}}(\nnu_{\tau_k-}^\hbar) \|\nnu_{\tau_k-}^{\hbar}\|_2 \max_{1 \le i \le n} \|\SSigma^T \UU^T \ee_i\|_2^3 n^{\alpha-3/2}\\
    &\le 2 \widehat{\gamma}^4 \delta C(\SSigma, \bb) \|\nabla^2 q\|^2 n^{\alpha + 3\varepsilon/2 - 3/2},
\end{align*}
\begin{align*}
    \lim_{\varepsilon \downarrow 0} \varepsilon^{-1} \mathbb{E}[\text{(b)} \cdot \text{(c)} | \mathcal{F}_t] &= -2(\gamma_{\tau_k-}^\hbar)^4 \delta^3 n^{-3} \sum_{i=1}^n (\ee_i^T (\UU \SSigma \nnu_{\tau_k-}^{\hbar}-\bb) ) (\ee_i^T \UU \SSigma (\nabla^2 q) \nnu_{\tau_k-}^{\hbar}) (\nnu_{\tau_k-}^{\hbar} (\nabla^2 q) \nnu_{\tau_k-}^{\hbar})\\
    &\le 2 \widehat{\gamma}^4 \delta^3 \|\nabla^2 q\|^2 \|\nnu_{\tau_k-}^{\hbar}\|_2^3 \max_{1 \le i \le n} \|\UU^T \SSigma^T \ee_i\|_2 n^{\alpha-5/2}\\
    &\le 2 \widehat{\gamma}^4 \delta^3 C(\SSigma, \bb)  \|\nabla^2 q\|^2 n^{\alpha + 3/2 \varepsilon - 5/2},\\
    \lim_{\varepsilon \downarrow 0} \varepsilon^{-1} \mathbb{E}[\text{(b)}^2 | \mathcal{F}_t] &= (\gamma_{\tau_k-}^\hbar)^4 \delta^4 n^{-3} \big ( (\nnu_{\tau_k-}^{\hbar})^T (\nabla^2 q) \nnu_{\tau_k-}^{\hbar} \big )^2\\
    &\le \widehat{\gamma}^4 \delta^4 n^{-3} \|\nnu_{\tau_k-}^{\hbar}\|_2^4 \|\nabla^2 q\|^2 \\
    &\le \widehat{\gamma}^4 \delta^4 \|\nabla^2 q\|^2 n^{2 \varepsilon - 3},\\
    \lim_{\varepsilon \downarrow 0} \varepsilon^{-1} \mathbb{E}[(c)^2 | \mathcal{F}_t] &= 4 (\gamma_{\tau_k-}^\hbar)^4 \delta^2 n^{-2} \sum_{i=1}^n (\ee_i^T (\UU \SSigma \nnu_{\tau_k-}^{\hbar}-\bb) )^2 (\ee_i^T \UU \SSigma (\nabla^2 q) \nnu_{\tau_k-}^{\hbar})^2\\
    &\le 4 \widehat{\gamma}^4 \delta^2 n^{-2} \max_{1 \le i \le n} \{ \|\UU^T \SSigma^T \ee_i\|^2 \} \|\nabla^2 q\|^2 \|\nnu_{\tau_k-}^{\hbar}\|^2 \widehat{\mathscr{L}}(\nnu_{\tau_k-}^\hbar)\\
    &\le 4 \widehat{\gamma}^4 \delta^2 \|\nabla^2 q\|^2 C(\SSigma, \bb) n^{2\varepsilon - 2}.
\end{align*}
It immediately follows from these bounds that 
\begin{equation} \begin{aligned}
    &\mathcal{C}_t \defas \lim_{\varepsilon \downarrow 0} \varepsilon^{-1} \mathbb{E}[|\Delta Y_{t + \varepsilon}^{\hbar}|^2 | \mathcal{F}_t] \le C(\delta, \widehat{\gamma}, \SSigma, \bb) \|\nabla^2 q\|^2 n^{2 \alpha - 1 + \varepsilon}\\
    \Rightarrow \quad &\sup_{0 \le t \le T} \int_0^t \mathcal{C}_s \, \dif s \le C(\delta, \widehat{\gamma}, \SSigma, \bb) \|\nabla^2 q\|^2 n^{2\alpha -1 + \varepsilon}. 
\end{aligned} \end{equation}
In order to apply Part (ii) of Lemma~\ref{lem:general_martingale_bound}, we also need to bound $\displaystyle \max_{0 \le t \le T} |\Delta Y_t^{\hbar}|$. A simple computation using \eqref{eq:stopping_time_quad} yields the following bounds:
\begin{align*}
    |\text{(a)}| &\le \widehat{\gamma}^2 \max_{1 \le i \le n} |\ee_i^T (\UU \SSigma \nnu_{\tau_k-}^{\hbar} - \bb) |^2 \|\nabla^2 q\| \max_{1 \le i \le n} \|\SSigma^T \UU^T \ee_i\|^2 \le \widehat{\gamma}^2 \|\nabla^2 q\|^2 n^{2\alpha-1}\\
    |\text{(b)}| &\le \widehat{\gamma}^2 \delta^2 n^{-2} \|\nabla^2 q\| \|\nnu_{\tau_k-}^{\hbar}\|^2 \le \widehat{\gamma}^2 \delta^2 \|\nabla^2 q\| n^{\varepsilon -2}\\
    |\text{(c)}| &\le 2 \widehat{\gamma}^2 \delta n^{-1} \max_{1 \le i \le n} |\ee_i^T (\UU \SSigma \nnu_{\tau_k-}^{\hbar}-\bb)| \|\nabla^2 q\| \|\nnu_{\tau_k-}^{\hbar}\|_2 \max_{1 \le i \le n} \|\SSigma^T \UU^T \ee_i\|_2 \le 2 \widehat{\gamma}^2 \delta n^{\alpha -3/2 + \varepsilon /2}.
\end{align*}
It follows that $\sup_{0 \le t \le T} |\Delta Y_t^{\hbar}| \le C(\delta, \widehat{\gamma}) \|\nabla^2 q\| n^{2 \alpha -1}$. Using Part (ii) of Lemma~\ref{lem:general_martingale_bound} with $\widehat{\varepsilon} = \varepsilon/2$, we get that
\[
    \sup_{0 \le t \le T} |Y_t^{\hbar}(q)| \le C(\delta, \widehat{\gamma}, \SSigma, \bb) T^{1/2} \|\nabla^2 q\| n^{\alpha - 1/2 + \varepsilon}
\]
with overwhelming probability.
% It follows that $|\Delta Y_{\tau_k}^{\hbar}(q)|^2 \le  \gamma^4 (\delta+1)^4 \|\nabla^2 q\|^2 n^{2\alpha - 2}$.  Consequently, the quadratic variation $[Y_t]$ is 
% \[ 
%     [Y_T(q)] = \sum_{k=1}^{N_T} (\Delta Y_{\tau_k}^{\hbar}(q))^2 \le \gamma^4 (\delta+1)^4 \|\nabla^2 q\|^2 n^{2\alpha-1} T \tfrac{N_T}{nT}. 
% \]
% Using Part (i) of Lemma~\ref{lem:general_martingale_bound} with $\varepsilon = \alpha$, we have that
% \[
%     \sup_{0\le t \le T} |Y_t^{\hbar}(q)| \le \gamma^2 (\delta +1)^2 \| \nabla^2 q\| T^{1/2} n^{2\alpha - 1/2}
% \]
% with overwhelming probability. 

Next, we turn to proving the martingale that arises due to the gradient where we will use Part (ii) of Lemma~\ref{lem:general_martingale_bound}. As before, we adopt a stopping time $\hbar$ defined (for some $\alpha > 0$) as \eqref{eq:stopping_time_hbar} and we show a bound for the stopped process $\nnu_t^{\hbar} \defas \nnu_{t \wedge \hbar}$.

To simplify the notation, we let $Y_t(q) \defas \mathcal{M}_t^{\text{grad}}(q)$. The jumps of this martingale are given by (see \eqref{eq:martingale_delta})
\begin{align*}
    \Delta Y_{\tau_k}(q) &= \Delta \mathcal{M}_{\tau_k}^{\text{grad}}(q) = \nabla q(\nnu_{\tau_k-})^T \Delta \nnu_{\tau_k}\\
    &= -\gamma_{\tau_k-}^\hbar ( (\nabla^2 q) \nnu_{\tau_k-} + \nabla q(0))^T \big [ \SSigma^T \UU^T \PP_k (\UU \SSigma \nnu_{\tau_k-} - \bb) + \tfrac{\delta}{n} \nnu_{\tau_k-} \big ].
\end{align*}
We perform the Doob decomposition and compute the compensator for the quadratic variation $[Y_t^{\hbar}(q)]$, that is $\sum_{k=1}^{N_t} (\Delta Y_{\tau_k}^{\hbar}(q))^2$. For the compensator of $[Y_t]$, we compute 
\begin{align}
    \mathcal{C}_t &= \lim_{\varepsilon \downarrow 0} \varepsilon^{-1} \mathbb{E}[|\Delta Y_{t + \varepsilon}^{\hbar}|^2 | \mathcal{F}_t] \nonumber \\
    &\le \frac{\widehat{\gamma}^2 \delta^2}{n} \big | ( (\nabla^2 q) \nnu_t^{\hbar} + \nabla q(0) )^T \nnu_t^{\hbar} \big |^2 + n \widehat{\gamma}^2 \mathbb{E}\big [ \big | ( (\nabla^2 q) \nnu_t^{\hbar} + \nabla q(0) )^T \SSigma^T \UU^T \PP_k(\UU \SSigma \nnu_t^{\hbar} - \bb) \big |^2  | \mathcal{F}_t\big ] \nonumber \\
    & \qquad + 2 \widehat{\gamma}^2 \delta \big | ( (\nabla^2 q) \nnu_t^{\hbar} + \nabla q(0) )^T \nnu_t^{\hbar} ( (\nabla^2 q) \nnu_t^{\hbar} + \nabla q(0) )^T \mathbb{E} [ \SSigma^T \UU^T \PP_k ( \UU \SSigma \nnu_t^{\hbar} - \bb) | \mathcal{F}_t] \big | \nonumber \\
    &\le \frac{2\widehat{\gamma}^2 \delta^2}{n} (\|\nabla^2 q\|^2 \|\nnu_t^{\hbar}\|^4_2 + \|\nabla q(0) \|^2 \|\nnu_t^{\hbar}\|^2 ) + \widehat{\gamma}^2 \sum_{i=1}^n |((\nabla^2 q) \nnu_t^{\hbar} + \nabla q(0))^T \SSigma^T \UU^T \ee_i|^2 | \ee_i^T(\UU \SSigma^T \nnu_t^{\hbar} - \bb)|^2 \nonumber \\
    & + \frac{2\widehat{\gamma}^2 \delta}{n} (\|\nabla^2 q\| \|\nnu_{t}^{\hbar}\|^2_2 + \|\nabla q(0)\|_2 \|\nnu_t^{\hbar}\|_2) | ( (\nabla^2 q) \nnu_t^{\hbar} + \nabla q(0) )^T \SSigma^T (\SSigma \nnu_t^{\hbar} - \UU^T \bb) | \nonumber \\
    &\le  \frac{2\widehat{\gamma}^2 \delta^2}{n} (\|\nabla^2 q\|^2 \|\nnu_t^{\hbar}\|^4_2 + \|\nabla q(0)\|^2 \|\nnu_t^{\hbar}\|^2 ) \label{eq:term1} \\
    & \qquad  + \widehat{\gamma}^2 \max_{i =1, \hdots, n} |\ee_i^T(\UU \SSigma^T \nnu_t^{\hbar} - \bb)|^2  \|\UU \SSigma ( (\nabla^2 q) \nnu_t^{\hbar} + \nabla q(0) )\|^2_2 \label{eq:term2} \\
    & \qquad + \frac{2 \widehat{\gamma
    }^2 \delta}{n} \bigg ( \|\nabla^2 q\| \|\nnu_t^{\hbar}\|^2 + \|\nabla q(0)\|_2 \|\nnu_t^{\hbar}\|_2 \bigg ) \label{eq:term3} \\
    &\qquad \, \,  \cdot \bigg ( \|(\nabla^2 q) \SSigma^T \SSigma\| \|\nnu_t^{\hbar}\|^2 + \|\SSigma^T \SSigma (\nabla q(0))\|_2 \|\nnu_t^{\hbar}\|_2 + \|\SSigma^T \UU^T\bb\|_2 \|\nabla q(0)\|_2 + \| (\nabla^2 q) \SSigma^T \UU^T \bb\|_2 \|\nnu_t^{\hbar}\|_2 \bigg ) \nonumber .
\end{align}
We will bound each term in the final inequality independently. For \eqref{eq:term1}, we use the definition of the stopping time 
\begin{align*}
    \frac{2\widehat{\gamma}^2 \delta^2}{n} (\|\nabla^2 q\|^2 \|\nnu_t^{\hbar}\|^4_2 + \|\nabla q(0)\|^2 \|\nnu_t^{\hbar}\|^2 ) &\le 2\widehat{\gamma}^2 \delta^2 ( \|\nabla^2 q\|^2 n^{-1 + 2\varepsilon} + \|\nabla q(0)\|^2 n^{-1 +  \varepsilon})\\
    &\le 2\widehat{\gamma}^2 \delta^2 (\|\nabla^2 q\|^2 + \|\nabla q(0)\|^2) n^{-1 + 2 \varepsilon}.
\end{align*}
For the second term \eqref{eq:term2}, we have that 
\begin{align*}
    \widehat{\gamma}^2 \max_{i =1, \hdots, n} |\ee_i^T(\UU \SSigma^T \nnu_t^{\hbar} - \bb)|^2  \|\UU \SSigma ( (\nabla^2 q) \nnu_t^{\hbar} + \nabla q(0))\|^2_2 &\le 2\widehat{\gamma}^2 n^{-1 + 2\alpha} (\|\SSigma (\nabla^2 q)\|^2 \|\nnu_t^{\hbar}\|^2_2 + \|\nabla q(0) \|^2)\\
    &\le 2\widehat{\gamma}^2 n^{-1 + 2 \alpha + \varepsilon} (\|\SSigma (\nabla^2 q)\|^2 + \|\nabla q(0)\|^2).
\end{align*}
For the third term \eqref{eq:term3}, we have that
\begin{align*}
    \frac{2 \widehat{\gamma}^2 \delta}{n} \bigg ( \|\nabla^2 q\| \|\nnu_t^{\hbar}\|^2 & + \|\nabla q(0)\|_2 \|\nnu_t^{\hbar}\|_2 \bigg ) \le 2 \widehat{\gamma}^2 \delta ( \| \nabla^2 q\| n^{-1 + \varepsilon} + \|\nabla q(0) \|_2 n^{-1 + 1/2 \varepsilon})\\
   \text{and} \, \,  \|(\nabla^2 q) \SSigma^T \SSigma\| &\|\nnu_t^{\hbar}\|^2 + \|\SSigma^T \SSigma (\nabla q(0))\|_2 \|\nnu_t^{\hbar}\|_2 + \|\SSigma^T \UU^T\bb\|_2 \|\nabla q(0)\|_2 + \| (\nabla^2 q) \SSigma^T \UU^T \bb\|_2 \|\nnu_t^{\hbar}\|_2\\
    &\le C(\delta, \SSigma, \bb) (n^{\varepsilon} + n^{\varepsilon/2} + 1) (\|\nabla^2 q\| + \|\nabla q(0)\|_2),
\end{align*}
where $C( \delta, \SSigma, \UU^T \bb)$ is constant depending on the norms of the input. Putting this together, we deduce that 
\begin{equation} \begin{gathered}\label{eq:compensator_grad}
    \mathcal{C}_t \le C(\delta, \widehat{\gamma}, \SSigma,  \bb) (\|\nabla^2 q\|^2 + \|\nabla q(0)\|^2_2) n^{-1 + 2 \alpha + \varepsilon}\\
    \Rightarrow \quad \int_0^T \mathcal{C}_s \, \dif s \le C(\delta, \widehat{\gamma}, \SSigma, \UU^T \bb)  (\|\nabla^2 q\|^2 + \|\nabla q(0)\|^2_2) T n^{-1 + 2 \alpha + \varepsilon}.
\end{gathered} \end{equation}
This equation \eqref{eq:compensator_grad} satisfies the condition that the compensator for the quadratic variation $\langle Y_t \rangle$ in Part (ii) of Lemma~\ref{lem:general_martingale_bound}. It remains to show that $\displaystyle \sup_{0 \le t \le T} |\Delta Y_t|$ is bounded. For this, we compute as followed
\begin{align*}
    \widehat{\gamma}^{-1} |\Delta Y_t^{\hbar}| &\le \frac{\delta}{n} |((\nabla^2 q) \nnu_{t-}^{\hbar} + \nabla q(0))^T \nnu_{t-}^{\hbar}| + |((\nabla^2 q) \nnu_{t-}^{\hbar} + \nabla q(0))^T \SSigma^T \UU^T \ee_{i_k} \ee_{i_k} (\UU \SSigma \nnu_{t-}^{\hbar} - \bb) |\\
    & \le \frac{\delta}{n} \big ( \|\nabla^2 q\| \|\nnu_{t-}^{\hbar}\|^2 + \|\nabla q(0)\|_2 \|\nnu_{t-}^{\hbar}\|_2  \big )\\
    & \quad + \max_{1 \le i \le n} |((\nabla^2 q) \nnu_{t-}^\hbar + \nabla q(0) )^T \SSigma^T \UU^T \ee_i| \, \, \max_{1 \le i \le n} |\ee_i^T (\UU \SSigma \nnu_{t-}^{\hbar} - \bb)| \\
    &\le  \frac{\delta}{n} \big ( \|\nabla^2 q\| \|\nnu_{t-}^{\hbar}\|^2 + \|\nabla q(0)\|_2 \|\nnu_{t-}^{\hbar}\|_2  \big )\\
    & \quad +  \|(\nabla^2 q) \nnu_{t-}^{\hbar} + \nabla q(0) \|_2 \max_{1 \le i \le n} \big \{ \|\SSigma^T \UU^T \ee_i\|_2 \big \} \, \, \max_{1 \le i \le n} \big \{ |\ee_i^T (\UU \SSigma \nnu_{t-}^{\hbar} - \bb)| \big \}.
\end{align*}
We bound each of the terms independently. For the first term, we use the stopping time $\hbar$ to deduce that 
\begin{align*}
    \frac{\delta}{n} \big ( \|\nabla^2 q\| \|\nnu_{t-}^{\hbar}\|^2 + \|\nabla q(0)\|_2 \|\nnu_{t-}^{\hbar}\|_2  \big ) \le \delta (\|\nabla^2 q\| + \|\nabla q(0)\|_2) n^{-1 + \varepsilon}.
\end{align*}
For the second term, we first observe that $\|\SSigma^T \UU^T \ee_i\|_2$ is the $i$th row sum of $\AA$ which we have normalized to be $1$. We simply bound the other quantities using the stopping time $\hbar$, that is,
\begin{align*}
    \|(\nabla^2 q) \nnu_{t-}^{\hbar} + \nabla q(0) \|_2 \max_{1 \le i \le n} \big \{ \|\SSigma^T \UU^T \ee_i\|_2 \big \} & \, \, \max_{1 \le i \le n} \big \{ |\ee_i^T (\UU \SSigma \nnu_{t-}^{\hbar} - \bb)| \big \}\\
    &\le (\|\nabla^2 q\| + \|\nabla q(0)\|_2) n^{-1 + 2\alpha + \varepsilon/2}.
\end{align*}
Therefore, we conclude that 
\begin{align*}
    \sup_{0 \le t \le T} \widehat{\gamma}^{-1} |\Delta Y_t^{\hbar}| \le (\delta + 1)(\|\nabla^2 q\| + \|\nabla q(0)\|_2) (n^{-1 + \varepsilon} + n^{-1 + 2\alpha + \varepsilon/2}).
\end{align*}
We apply Part (ii) of Lemma~\ref{lem:general_martingale_bound} with $\widehat{\varepsilon} = \varepsilon/2$ to conclude the result. 
\end{proof}

\section{The random features model}
\label{sec:random_features}

In this section we prove the claims made for the random features model. Recall the definitions in Sec.~\ref{sec:random_features_intro}, in particular
\eq{
    \AA \deq \sigma(\XX \WW / \sqrt{n_0}) / \sqrt{n} = \sigma(\ZZ \Sigma^{1/2} \WW / \sqrt{n_0}) / \sqrt{n}.
}
There are two complementary resolvent matrices:
\eq{\label{eq_two_resolvents}
    \gcal\deq (\AA^\top \AA  - zI)^{-1} 
    \quad\text{and} \quad \gcaltil\deq (\AA\AA^\top  - zI)^{-1}.
}
These resolvents can be controlled for certain spectral arguments $z$. We now define such an allowable set that also suffices for the contour integrals we considered earlier.
\begin{definition}
    We define an allowable set for the argument $z$:
    \eq{
        \mathfrak{C} \deq \ha{z = E +i\eta: |E| \leq C n^\delta \text{ and } c n^{-\delta} \leq \eta \leq C n^\delta },
    }
    for real $E$ and $\eta$ and some $n$-independent positive constants $C$ and $c$.
\end{definition}

We show that with high probability, Assum.~\ref{ass: laundry_list} holds:
\begin{proposition}\label{prop_rf_resolvent}
    Suppose $z\in \mathfrak{C}$. Then with high probability,
    \al{
        \max_{1\leq\alpha<\beta\leq m} |\gcal_{\alpha\beta} | &\leq C n^{\delta-1/2} , \\
        \max_{1\leq a<b\leq n_1} |\gcaltil_{ab} | &\leq C n^{\delta-1/2} , \\
        \max_{1\leq\alpha\leq m} |\gcal_{\alpha\alpha} - \ntr(\gcal) | &\leq C n^{\delta-1/2} , \\
        \text{and}\quad\max_{1\leq a\leq n_1} |\gcaltil_{aa} - \ntr(\gcaltil) | &\leq C n^{\delta-1/2}.
    }
\end{proposition}

To prove Prop.~\ref{prop_rf_resolvent}, we will need several algebraic identities involving the resolvents of Eq.~\eqref{eq_two_resolvents}. As in \cite{schnelli}, these can be studied simultaneously by linearizing the problem using the block the matrix
\eq{\label{eq_rf_linearization}
    H \deq  \begin{bmatrix}
        -z I        &    \AA^\top \\
        \AA   &   -I 
    \end{bmatrix}.
}

\begin{remark}
    We follow the convention that Greek letters $\alpha,\beta,...$ are used to index the columns of $\AA$ (and range over $\h{1,\ldots,d}$) and Latin characters $a,b,...$ are used to index the rows of $\AA$ (and range over $\h{1,\ldots,n}$). This convention also applies to $H$, which is an $(d+n)$-dimensional square matrix, in the follow sense: the first $d$ dimensions will use $\alpha$ and $\beta$ and the last $n$ dimensions will use $a$ and $b$. This means that for these block matrices $a$ and $b$ range over $\h{d+1,\ldots,d+n}$. Another convention is to use a colon, ``:'', to indicate all indices. That is, for a matrix $M$, we would denote the $i$th row as $M_{i:}$ and the $i$th column as $M_{:i}$.
\end{remark}

We define $R$ as the inverse of $H$. By the Schur complement formula, one easily finds a connection between $R$ and the resolvents of Eq.~\eqref{eq_two_resolvents}:
\eq{\label{eq_gcal_from_h}
    \gcal = [ R_{\alpha\beta} ]_{1\leq\alpha,\beta\leq d} \quad\text{and}\quad \gcaltil = z^{-1}[ R_{a b} ]_{d+1\leq a,b\leq d + n}.
}
In particular, the Stieltjes transforms can be recovered from $R$, since
\eq{\label{eq_stieltjes}
    s_d(z) \deq   \ntr \gcal = \frac{1}{d}\sum_{\alpha=1}^{d} R_{\alpha\alpha}(z)
     \quad \text{and} \quad  \tilde{s}_d(z) \deq  \ntr  \gcaltil = \frac{z^{-1}}{n}\sum_{a=d+1}^{d+n} R_{a a}(z) .
}

To proceed, we introduce some quadratic forms in rows and columns of $\AA$. Shortly, we will see how they appear via the Schur complement formula. For $i\in\h{1,\ldots, d+n}$ define $R^{(i)}$ as the inverse of the minor $H^{(i)}$, i.e. both the $i$th row and column of $H$ is removed. For $i,j\in\h{1,\ldots, d+n}$ and $i\neq j$, we define $R^{(ij)}$ similarly by removing both the $i$th and $j$th rows and columns. In contrast, for $\alpha\in\{1,\ldots,d\}$ we define $\AA^{(\alpha)}$ as $\AA$ with only the $\alpha$th column removed and for $a\in\{1,\ldots,n\}$ we define $\AA^{(a)}$ as $\AA$ with only the $a$th row removed. Whether a row or column is removed should be clear from the use of a Latin or Greek index. Again these can be extended to multiple indices. The resolvents of Eq.~\eqref{eq_two_resolvents} can be naturally defined for these minors of $\AA$:
\eq{
    \gcalm{i}\deq ((\AA^{(i)})^\top \AA^{(i)}  - zI)^{-1} \quad \text{and}\quad \gcaltilm{k}\deq (\AA^{(k)}(\AA^{(k)})^\top  - zI)^{-1},
}
where $i=a,(ab)$ and $k=\alpha,(\alpha\beta)$. Analogous equations to Eq.~\eqref{eq_gcal_from_h} are easily derived. 

We define the quadratic forms:
\al{\label{def_theta_phi}
    \tilde{\theta}^{(\alpha)} &\deq  \sum_{a,b=1}^{n} \AA_{a \alpha}\AA_{b \alpha}  R^{(\alpha)}_{d+a, d+b} = z\sum_{a,b=1}^{n} \AA_{a \alpha}\AA_{b \alpha}  \gcaltilm{\alpha}_{a b}, \\
    \theta^{(a)} &\deq  \sum_{\alpha,\beta=1}^{d} \AA_{a\alpha}\AA_{a\beta} R^{(a)}_{\alpha\beta} = \sum_{\alpha,\beta=1}^{d} \AA_{a\alpha}\AA_{a\beta} \gcalm{a}_{\alpha\beta}, \\
     \tilde{\phi}^{(\alpha \beta)} &\deq   \sum_{a,b=1}^{n} \AA_{a\alpha}  \AA_{b \beta } R_{d+a,d+b}^{(\alpha \beta)} = z \sum_{a,b=1}^{n} \AA_{a\alpha}  \AA_{b \beta } \gcaltilm{\alpha \beta}_{ab} \\
    \text{and} \quad \phi^{(ab)} &\deq   \sum_{\alpha,\beta=1}^{d} \AA_{a\alpha}  \AA_{b \beta } R^{(ab)}_{\alpha\beta} =  \sum_{\alpha,\beta=1}^{d} \AA_{a\alpha}  \AA_{b \beta } \gcalm{ab}_{\alpha\beta}.
}
Then by the Schur complement formula, we have for the diagonal entries of $R$
\eq{\label{eq_schur_diag}
    \frac{1}{R_{\alpha\alpha}}  = -z -  \tilde{\theta}^{(\alpha)} \quad
    \text{and}\quad\frac{1}{R_{aa}}  =-1 -  \theta^{(a)}.
}
Similarly, for the off-diagonal entries of $R$, we have
\al{\label{eq_schur_offdiag}
    R_{\alpha\beta} &= R_{\alpha\alpha} R_{\beta\beta}^{(\alpha)} \tilde{\phi}^{(\alpha\beta)} =  \gcal_{\alpha\alpha}\gcalm{\alpha}_{\beta\beta} \tilde{\phi}^{(\alpha\beta)} \\
    \text{and}\quad R_{ab} &= R_{aa} R_{bb}^{(a)} \phi^{(ab)} = z^2\gcaltil_{aa}\gcaltilm{a}_{bb} \phi^{(ab)} .
}
Note $R^{(a)}$ is independent of the $a$th row of $\ZZ$ and $R^{(\alpha)}$ is independent of the $\alpha$th column of $\WW$. This motivates studying the concentration of these quadratic forms about deterministic quantities. To do this we first need to bound the norms of the minors of $\gcal$ and $\gcaltil$.

\begin{lemma}\label{lem_g_norm_bounds}
    For the operator norm, we have
    \eq{\label{eq_op_bound}
        \max_{a,b,\alpha,\beta}\max\h{ \norm{\gcalm{a}}_{\text{op}}, \norm{\gcaltilm{\alpha} }_{\text{op}},\norm{\gcalm{ab} }_{\text{op}},  \norm{\gcaltilm{\alpha\beta} }_{\text{op}} } \leq \frac{1}{\Im(z)}
    }
    and, for the Frobenius norm, we have
    \eq{\label{eq_fro_bound}
        \max_{a,b,\alpha,\beta}\max \h{ \norm{\gcalm{a}}^2, \norm{\gcaltilm{\alpha}}^2,  \norm{\gcalm{ab}}^2, \norm{\gcaltilm{\alpha\beta}}^2 } \leq \frac{Cn}{\Im(z)^2}.
    }
\end{lemma} 
\begin{proof}
    The bound in Eq.~\eqref{eq_op_bound} is trivial, since the spectrum of $(\AA^{(i)})^\top \AA^{(i)}$ is real. For Eq.~\eqref{eq_fro_bound}, we use the Ward identity (see Sec.~7.1.3 of \cite{yau2017dynamical} or Eq.~(5.21) of \cite{schnelli}).
\end{proof}

Define the matrices
\eq{
    \Sigma^\WW_{ab} \deq n\E_\WW A_{a\alpha} A_{b\alpha} \text{ and } \Sigma^\ZZ_{\alpha\beta} \deq n\E_\ZZ \AA_{a\alpha} \AA_{a\beta},
}
and note that $\Sigma^\WW_{ab}$ does not depend on the choice of $\alpha$ nor does $\Sigma^\ZZ$ depend on $a$. Note
\eq{\label{eq_alpha_tilde_mean}
    \E_{\WW_{:\alpha}} \tilde{\theta}^{(\alpha)} = \frac{1}{n}\sum_{a,b=1}^{n} \Sigma^\WW_{ab}  R^{(\alpha)}_{d+a, d+b} = z\bar{\tr}\pa{ \Sigma^\WW  \gcaltilm{\alpha} }
}
and
\eq{
    \E_{\ZZ_{a:}} {\theta}^{(a)} = \frac{1}{n}\sum_{\alpha,\beta=1}^{d} \Sigma^\ZZ_{\alpha\beta}  R^{(a)}_{\alpha, \beta} = \frac{d}{n}\bar{\tr}\pa{ \Sigma^\ZZ  \gcalm{a} }.
}

With these definitions in hand, we are able to define an event $\mathcal{T}$ on which the resolvent can be controlled. The goal of Sec.~\ref{sec_quadratic_forms_rf} is to show that this event occurs with high probability (Prop.~\ref{prop:RF_goodT}), and Sec.~\ref{sec_completed_proof_rf} completes the argument by showing Prop.~\ref{prop_rf_resolvent} holds on $\mathcal{T}$. We define this typical event as
\al{\label{def_good_event}
    \mathcal{T}\deq &\ha{\norm{\Sigma^\ZZ}_{\text{op}} \leq Cn^\delta} \cap \ha{\norm{\Sigma^\WW}_{\text{op}} \leq Cn^\delta} \\
    &\cap \ha{ \max_{a}\norm{\AA_{a:}} \leq Cn^\delta } \cap \ha{ \max_{\alpha}\norm{\AA_{:\alpha}} \leq Cn^\delta }\cap\ha{ \norm{\AA}_{\text{op}}\leq Cn^\delta } \\
    &\cap \ha{\max_{a} \absa{{\theta}^{(a)} - \E_{\ZZ_{a:}} {\theta}^{(a)} } \leq Cn^{\delta-1s/2} } \cap \ha{\max_{\alpha} \absa{\tilde{\theta}^{(\alpha)} - \E_{\WW_{:\alpha}} \tilde{\theta}^{(\alpha)} } \leq Cn^{\delta-1/2} } \label{def_good_event2}\\
    &\cap \ha{\max_{a\neq b} \absa{{\phi}^{(ab)}} \leq Cn^{\delta-1/2} } \cap \ha{\max_{\alpha\neq\beta} \absa{\tilde{\phi}^{(\alpha\beta)} } \leq Cn^{\delta-1/2} }.\label{def_good_event3}
}
Recall, the $\theta$ and $\phi$ terms are quadratic forms in $\AA$ that are defined in Eq.~\eqref{def_theta_phi}.

\begin{remark}
    Throughout this section we use $C$ to denote an arbitrarily large, $n_0$-independent, positive constant, which can increase from line to line. For example, if $X\leq C$ and $Y\leq C$, we will simply write $XY\leq C$, since $C^2$ is still some $n_0$-independent constant. This approach is valid as such replacements only occur a finite number of times. Similarly, $c$ and $\delta$ denote some arbitrarily small, $n_0$-independent, positive constants that can decrease from line to line. 
\end{remark}

\subsection{$\mathcal{T}$ occurs with high probability}

We start with some general properties of the random features model and concentration properties for some good events.
\begin{lemma}
There are constants $\mathfrak{a},\mathfrak{b},\mathfrak{c},\mathfrak{d} \in \R$ such that
\[
\begin{aligned}
&\int_\R
\sigma(x\sqrt{1+\epsilon}) e^{-x^2/2}\frac{\dif x}{\sqrt{2\pi}} = \mathfrak{a}\epsilon +  \mathcal{O}_{\epsilon \to 0}(\epsilon^2), \\
&\int_\R
\sigma^2(x\sqrt{1+\epsilon}) e^{-x^2/2}\frac{\dif x}{\sqrt{2\pi}} = \mathfrak{b} +
\mathfrak{c}\epsilon
+\mathcal{O}_{\epsilon \to 0}(\epsilon^2),\\
&\int_\R
x\sigma(x\sqrt{1+\epsilon}) e^{-x^2/2}\frac{\dif x}{\sqrt{2\pi}} = \mathfrak{d} +  \mathcal{O}_{\epsilon \to 0}(\epsilon). \\
\end{aligned}
\]
\end{lemma}
\begin{proof}
    From the exponential growth of the derivative of $\sigma$ we conclude that for any $\sigma$ such that $\Exp \sigma^2( a Z) < \infty$ for any $a \in \R$.
\end{proof}

For this section, we will require that the norm of $\SSigma^{1/2}\WW/\sqrt{n_0}$ is in control, specifically that 
\[
\begin{aligned}
\mathcal{N} =
&\biggl\{
\|\SSigma^{1/2}\WW/\sqrt{n_0}\|_{\text{op}} \leq C,
\,\,
\max_{ij} |\bigl(\ZZ \SSigma^{1/2} \WW/\sqrt{n_0}\bigr)_{ij}| \leq (\log n_0)^{1-\epsilon}
\biggr\} \\
&\cap 
\biggl\{
\| \Id_d - \tfrac{1}{n_0}\WW^T \SSigma \WW\|_\infty \leq n^{c-1/2}, \, 
|\tr \bigl(\Id_d - \tfrac{1}{n_0}\WW^T \SSigma \WW\bigr)| \leq \log n_0
\biggr\}
\end{aligned}
\]
for an unimportant large constant $C>0$ and any $\epsilon \in (0,\tfrac 12)$.  
\begin{lemma}\label{lem:RF_goodN}
Suppose $\overline{tr}(\Sigma)=1$.
\[
\Pr( \mathcal{N}) \geq 1 - e^{-\omega(\log n_0)}.
\]
\end{lemma}
\begin{proof}
There are 4 events to control.  Write $(I),(II),(III),(IV)$ for the four events considered in $\mathcal{N}$
\paragraph{(I).} Norm bounds and boundedness in norm of $\SSigma$.
\paragraph{(II).} Bernstein's inequality for subexponential random variables.
\paragraph{(III).} Hanson Wright + union bound.
\paragraph{(IV).} Bernstein + trace cyclicity.
\end{proof}

\subsubsection{Concentration for quadratic forms in random features}\label{sec_quadratic_forms_rf}

The random features matrix $\FF \coloneqq \sigma(\XX\WW/\sqrt{n_0})$ has independent rows \emph{conditionally} on $\WW$.
We show the following concentration inequality on quadratic forms.
\begin{proposition}\label{prop:RF_qf}
Suppose $\overline{tr} \SSigma =1$.
Suppose that $\GG$, which may measurably depend on $\WW$, but is otherwise deterministic and that $u_1,u_2$ are any two rows from the random features matrix $\FF$.  Let $\mathfrak{u}$ be the $2 \times n$ matrix with rows given by $u_1$ and $u_2$.  Then for any $\epsilon>0$
\[
\Pr
\biggl(
\biggl\{
\big\|\mathfrak{u} \GG \mathfrak{u}^T - \Exp [\mathfrak{u} \GG \mathfrak{u}^T \, | \, \WW]|\big\|
\geq t
\biggr\}
\cap
\mathcal{N}
~\bigg\vert~ \WW
\biggl)
\leq 
2
\exp\biggl(
-D_n\min\biggl\{
\frac{t^2}{\|\GG\|^2_{HS} \|\SSigma\|},
\frac{t}{\|\GG\|}
\biggr\}
\biggr),
\]
where $D_n = C e^{-4C_1(\log n_0)^{1-\epsilon}}$ with $C_1$ the constant from Assumption \ref{ass:RF_sigma}. In particular if $C_1 = 0$ (i.e. $\sigma$ is Lipschitz), then we may take $D_n = C$.  
The same statement holds if we instead replace the conditional expectation $\Exp [\mathfrak{u} \GG \mathfrak{u}^T | \WW]$ by 
%(\textcolor{red}{needs verification})
 \[
    \biggl(
        (\mathfrak{b}-\mathfrak{d})\tr(\GG)
    +\mathfrak{d}\tr\bigl( \tfrac{1}{n_0} \GG \WW^T \SSigma \WW\bigr)
    \biggr)
    \begin{bmatrix}
    1 & 0 \\
    0 & 1
    \end{bmatrix},
    \]
    after increasing the constants.
%\(
%\mathfrak{a}^2 v^T A v
%\,\text{where}
%\,
%v = \operatorname{diag}(\Id - \tfrac{1}{n_0}W^T\Sigma W).
%\)
%The same bound holds unconditionally on $W$ as well, taking expectation on both sides.  
%We may also replace the row $u$ by the same, though with the replacement $\sigma \to \varrho$
\end{proposition}

To prove this, we make use of a result of \cite{Adamczak}. Following \cite{Adamczak}, say that a random vector $\XX \in \R^n$ has the \emph{convex concentration property with constant $K$} if for every convex $1$--Lipschitz function $\varphi : \R^n \to \R$, the value $\varphi(\XX)$ is subgaussian with constant $K$.  Then from \cite{Adamczak}, the following theorem holds.
\begin{theorem}\label{thm:adamczak}
Let $\XX$ be a mean zero random vector in $\R^n$ with the convex concentration property with constant $K$.  There is a universal constant $C >0$ so that for any $n\times n$ matrix $\GG$ and every $t > 0$
\[
\Pr
\biggl(
\biggl\{
|\XX \GG \XX^T - \Exp [\XX \GG \XX^T]|
\geq t
\biggr\}
\biggl)
\leq 
2\exp\biggl(
-C\min\biggl\{
\frac{t^2}{K^4 \|\GG\|^2_{HS} \|\operatorname{Cov}(\XX)\|},
\frac{t}{K^2 \|\GG\|}
\biggr\}
\biggr).
\]
\end{theorem}

We also will rely on Gaussian concentration, namely:
\begin{lemma}\label{lem:GC}
For a Lipschitz function $F: \R^n \to \R$, and an iid standard normal $Z$ in $\R^n$, and for all $t \geq 0$
\[
\Pr ( |F(Z) - \Exp F(Z)| \geq t) \leq 2\exp\biggl( -\frac{t^2}{2\|\nabla F\|^2_\infty}\biggr).
\]
\end{lemma}
\noindent For convenience we will also use the subgaussian norm $\|\cdot\|_{\psi_2}$ which is equivalent up to universal constants to the optimal variance proxy in a Gaussian tail bound for a random variable i.e.
\[
\| X \|_{\psi_2} 
\asymp
\inf \{ V  > 0 : \forall~t > 0~\Pr( |X| > t) \leq 2 e^{-t^2/V^2}\}
\]

The following is the proof of Proposition \ref{prop:RF_qf}.
\begin{proof} We proceed in steps.\paragraph{Step 1: truncation.} 
    The activation function $\sigma$ is not Lipschitz.  With $x_0$ given by $|(\log n_0)^{1-\epsilon}|,$ define a new activation function $\varrho$ by
    \[
        \varrho(x) = \begin{cases}
        \sigma(x_0), & \text{if } x > x_0, \\
        \sigma(-x_0), & \text{if } x < -x_0 \\
        \sigma(x) & \text{otherwise}.
        \end{cases}.
    \]
    This activation function is Lipschitz with constant at most $C_0 e^{C_1(\log n_0)^{1-\epsilon}}$.  
    
    Without loss of generality we may represent $\uu_j = \sigma(\XX_j \WW/\sqrt{n_0})$, with $\XX_j=\ZZ_j \SSigma^{1/2}$ the first row of $\XX$.  Let $\vv_j = \varrho(\XX_j \WW/\sqrt{n_0})$ and $\mathfrak{v}$ have rows $\vv_1$ and $\vv_2$.  Then on $\mathcal{N},$ $\mathfrak{u}=\mathfrak{v}$.  Moreover using Cauchy--Schwarz
    \[
    \big\|\Exp[\mathfrak{u} \GG \mathfrak{u}^T - \Exp \mathfrak{v} \GG \mathfrak{v}^T ~|~ \WW]\big\|
    =
    \big\|\Exp\bigl( \mathfrak{u} \GG \mathfrak{u}^T (1-\mathbf{1}_{\mathcal{N}}) ~|~ \WW \bigr)\big\|
    \leq \|\GG\| \sqrt{\Exp \|\mathfrak{u}\|_{\text{HS}}^4 \times (1-\Pr(\mathcal{N}| \WW))}.
    \]
    The conditional probability of $\mathcal{N}^c$ decays faster than any power of $n_0,$ and hence we have that the coefficient of $\|\GG\|$ decays faster than any power of $n_0$.  In particular the difference is (deterministic) subgaussian with constant that vanishes faster than $\|\GG\|$ and $\|\GG\|_{\text{HS}}$ (and hence also $\|\GG\|_{\text{HS}} \|\SSigma\|$, which is only larger by Assumption \ref{ass:RF_cov}).  It follow that it suffices to prove 
    \[
    \Pr
    \biggl(
    \biggl\{
    \big\|\mathfrak{v} \GG \mathfrak{v}^T - \Exp [\mathfrak{v} \GG \mathfrak{v}^T | \WW]|\big\|
    \geq t
    \biggr\}
    \cap
    \mathcal{N}
    ~\bigg\vert~ \WW
    \biggl)
    \leq 
    2
    \exp\biggl(
    -D_n\min\biggl\{
    \frac{t^2}{\|\GG\|^2_{\text{HS}} \|\SSigma\|},
    \frac{t}{\|\GG\|}
    \biggr\}
    \biggr),
    \]
    as by adjusting constants the desired claim follows after adjusting the constant $D_n$.
    
    \paragraph{Step 2: removing the mean.}
    We do one further conditioning step and remove the mean of $\mathfrak{v}$.  Thus we define a row vector $\mathfrak{w}$ by
    \[
    \mathfrak{w} = \mathfrak{v} - \Exp[ \mathfrak{v} \, | \,  \WW ] = \mathfrak{v} - 
    \begin{bmatrix}
    \Exp[ \varrho(\XX_1 \WW/\sqrt{n_0}) | \WW ]\\
    \Exp[ \varrho(\XX_2 \WW/\sqrt{n_0}) | \WW ]\\
    \end{bmatrix}
    \eqqcolon
    \mathfrak{v} - \begin{bmatrix} 1 \\ 1 \end{bmatrix} \Delta .
    \]
    The law of $(\ZZ_j \SSigma^{1/2} \WW/\sqrt{n_0} : j=1,2)$ conditionally on $\WW$ is multivariate Gaussian.  
    We observe that for a fixed vector $\xx \in \R^{n_0}$, the function $F:(\ZZ_i) \mapsto \ww_i \xx$ is Lipschitz.  To determine its constant, we observe that (taking WLOG $i=1$)
    \[
    \nabla_{Z} (F(\ZZ_1)) 
    = \sum_{j=1}^{n_0} \nabla_{Z}\varrho( {\textstyle \sum \ZZ_1 \SSigma^{1/2} \cdot \WW^T_j})\xx_j
    = \sum_{j=1}^{n_0} \WW^T_j\SSigma^{1/2}\nabla_{Z}\varrho'( {\textstyle \sum \ZZ_1 \SSigma^{1/2} \cdot \WW^T_j})\xx_j.
    \]
    Thus from Lemma \ref{lem:GC} we have the simple bound for subgaussian norm $\|\cdot\|_{\psi_2}$
    \begin{equation}\label{eq:RFxw}
    \|\ww_i\xx\|_{\psi_2}
    \leq
        \|\WW\|_{\text{op}} \times \|\SSigma^{1/2}\|_{\text{op}} \times Ce^{C_1(\log n_0)^{1-\epsilon}} \times \|\xx\|.
    \end{equation}
    
    The entries of $\ZZ_1 \SSigma^{1/2} \WW/\sqrt{n_0}$ have variances which are uniformly, on the event $\mathcal{N}$, close to $1+\mathcal{O}(n_0^{c-1/2})$.  Thus with $\eta_j = 1- (\WW^T \SSigma \WW)_{jj}/n_0$ and for a standard normal $z$
    \[
    \Delta_{j}
    =
    \Exp( \varrho( z\sqrt{1-\eta_j}) \, | \, \WW)
    =
    \mathfrak{a}\eta_j + \mathcal{O}(n_0^{2c-1}).
    \]
    Note that the second order terms $\mathcal{O}(n_0^{2c-1})$ are so small that the vector 
    \[
        \|(\DDelta - \mathfrak{a}\eeta) \GG\| = \|\GG\|_{\text{op}}\mathcal{O}(n_0^{c-1/2}).
    \]

    In particular, taking the inner product of this vector with $\ww_1$ or $\ww_2$, conditionally on $\WW$ gives a subgaussian random variable with constant given by (see \eqref{eq:RFxw}) $\|\GG\|_{\text{op}} \mathcal{O}(\|\varrho\|_{\text{lip}}n_0^{c-1/2})$ for all $n_0$ sufficiently large.  Likewise the quadratic form
    \[
    |(\DDelta - \mathfrak{a}\eeta) \GG (\DDelta - \mathfrak{a}\eeta)^T|
    \leq \|\GG\|_{\text{op}}\mathcal{O}(n_0^{2c-1/2}),
    \]
    is similarly small. 
    As we may expand,
    \[
        \vv \GG \vv^T 
        =
        \ww \GG \ww^T
        + 2\mathfrak{a}\eeta \GG \ww^T
        + \mathfrak{a}^2\eeta \GG \eeta^T
        + 2(\DDelta - \mathfrak{a}\eeta) \GG \ww^T
        + 2\mathfrak{a}(\DDelta - \mathfrak{a}\eeta) \GG \eeta^T
        +(\DDelta - \mathfrak{a}\eeta) \GG (\DDelta - \mathfrak{a}\eeta)^T,
    \]
    and all terms with $(\DDelta - \mathfrak{a}\eeta)$ can be dispensed,
    it follows that it suffices to prove
    \[
    \Pr
    \biggl(
    \biggl\{
    \big\|\mathfrak{w} \GG \mathfrak{w}^T - \Exp [\mathfrak{w} \GG \mathfrak{w}^T | \WW]|\big\|
    \geq t
    \biggr\}
    \cap
    \mathcal{N}
    ~\bigg\vert~ \WW
    \biggl)
    \leq 
    2
    \exp\biggl(
    -D_n\min\biggl\{
    \frac{t^2}{\|\GG\|^2_{\text{HS}} \|\SSigma\|},
    \frac{t}{\|\GG\|}
    \biggr\}
    \biggr),
    \]
    and a subgaussian bound on $\mathfrak{a}\eeta\GG \ww^T_i$.  
    For this last part, from \eqref{eq:RFxw}, using that on $\mathcal{N}$ the norm of $\|\eeta\|$ is bounded, we have that 
    \(
    \|\mathfrak{a}\eeta \GG \ww_i^T\|_{\psi_2} \leq Ce^{C_1(\log n_0)^{1-\epsilon}}.
    \)
    This also shows that we may use $\Exp \big [ \mathfrak{w} \GG \mathfrak{w}^T \, | \, \WW \big ] + \mathfrak{a}^2\eeta \GG \eeta^T \JJ$ with $J_{ij}=1$ for all $i,j$ in place of the expectation $\Exp \big [ \mathfrak{v} \GG \mathfrak{v}^T \, | \, \WW \big ]$.
    
    \paragraph{Step 3: concentration of the quadratic form.}
    
    We just need to establish the convex concentration property for $w_1$ and $w_2$.  Taking a convex $1$-Lipschitz function $g : \R^{n_0} \to \R,$ we have
    \[
    \nabla_Z g(\ww_1) 
    = \sum_{j=1}^{n_0} \nabla_{Z}\varrho( {\textstyle \sum \ZZ_1 \SSigma^{1/2} \cdot \WW^T_j}) \partial_j g(w)
    = \sum_{j=1}^{n_0} \WW^T_j \SSigma^{1/2}\nabla_{Z}\varrho'( {\textstyle \sum \ZZ_1 \SSigma^{1/2} \cdot \WW^T_j})\partial_j g(w).
    \]
    Bounding in norm, we have
    \[
    \|\nabla_Z g(\ww_1) \| \leq \|\WW\|_{\text{op}} \times \|\SSigma^{1/2}\|_{\text{op}} \times Ce^{C_1(\log n_0)^{1-\epsilon}} \times \| \nabla g \|_\infty,
    \]
    which as $g$ is $1$-Lipschitz, gives the convex concentration property with $Ce^{C_1(\log n_0)^{1-\epsilon}}$ for some possibly larger $C$.  
    
    This property further extends to the concatenation $\ww_3$ of the vectors $\ww_1$ and $\ww_2$ as a random vector in $\R^{2n_0}$.
    Then entry-by-entry, we bound can produce tail bounds for the random $2\times 2$ matrix
    \(
    \mathfrak{w} \GG \mathfrak{w}^T - \Exp [\mathfrak{w} \GG \mathfrak{w}^T \, | \, \WW].
    \)
    For the on-diagonal entries, we may directly use the concentration of the quadratic form Theorem \ref{thm:adamczak}.
    For the off-diagonal entries, we can use the representation
    \[
    \bigl(\mathfrak{w} \GG \mathfrak{w}^T\bigr)_{12}
    =
    \ww_3 \begin{bmatrix}
    0 & \GG \\
    0 & 0
    \end{bmatrix}
    \ww_3^T,
    \]
    to which we may again apply Theorem \ref{thm:adamczak}.
    
    \paragraph{Step 4: Simplifying the mean.}
    
    Combining the arguments in the previous steps show that we may use $\Exp \big [ \mathfrak{w} \GG \mathfrak{w}^T \, | \, \WW \big ] + \mathfrak{a}^2\eeta \GG \eeta^T \JJ$ in place of the expectation $\Exp \big [ \mathfrak{u} \GG \mathfrak{u}^T \, | \, \WW \big ]$.  This additional term can be bounded by $|\eeta \GG \eeta^T| \leq \|\GG\|_{\text{op}} \|\eeta\|^2$, which is negligible. The diagonal contributions of $\ww_1 \GG \ww_1^T$ we simplify using 
    \[
     \sum_{j=1}^{n_0} G_{jj} \Exp \big [ w_{1j}^2 | \WW \big ]
     =
     \sum_{j=1}^{n_0} G_{jj} (\mathfrak{b} + \mathfrak{c} \eta_j + \mathcal{O}(\eta_j^2)).
    \]
    As the norm of $\eeta$ is at most logarithmic, we bound the correction term by $\|\GG\|_{\text{HS}} \mathcal{O}(\|\eeta\|)$, which is therefore negligible in subgaussian norm in comparison to the quadratic form fluctuations. For off-diagonal contributions, we use that for $j \neq k$ %(\textcolor{red}{needs verification})
    \[
        \Exp \big [ w_{1j}w_{1k} | \WW \big ] = \mathfrak{d} \frac{(\WW^T \SSigma \WW)_{jk}}{n_0} + \mathcal{O}\biggl( n_0^{c-1/2} 
        \times
        \frac{(\WW^T \SSigma \WW)_{jk}}{n_0}
        \biggr).
    \]
    Thus summing all off diagonal terms, 
    \[
    \Exp \big [ \ww_1 \GG \ww_1^T | \WW \big]
    =
    (\mathfrak{b}-\mathfrak{d})\tr(\GG)
    +\mathfrak{d}\tr\biggl( \tfrac{1}{n_0} \GG \WW^T \SSigma \WW \biggr)
    +\mathcal{O}\biggl( \log(n_0)\biggr) \|\GG\|_{\text{HS}}.
    \]
    The same holds for the other $\Exp[ \ww_2 \GG \ww_2^T | \WW]$ term.  The cross term $\Exp[ \ww_1 \GG \ww_2^T | \WW]$ vanishes by conditional independence.  We conclude that
    \[
    \Exp\bigl[ \mathfrak{w} \GG \mathfrak{w}^T | \WW\bigr ]
    =
    \biggl(
        (\mathfrak{b}-\mathfrak{d})\tr(\GG)
    +\mathfrak{d}\tr\bigl( \tfrac{1}{n_0} \GG \WW^T \SSigma \WW\bigr)
    \biggr)
    \begin{bmatrix}
    1 & 0 \\
    0 & 1
    \end{bmatrix}
    +\mathcal{O}(\log n_0) \|\GG\|_{\text{HS}}.
    \]
    
\end{proof}

Finally, we can show that $\mathcal{T}$ occurs with high probability. 
\begin{proposition}\label{prop:RF_goodT}
\eq{
    \Pr( \mathcal{T}) \geq 1 - e^{-\omega(\log n_0)}.
}
\end{proposition}
\begin{proof}
    Using the argument in Step 4 above, we observe
    \eq{
        \Sigma^\WW_{ab} = (\mathfrak{b}-\mathfrak{d})\text{diag}\pa{\frac{\ZZ\Sigma \ZZ^\top}{n_0}} + \mathfrak{d}\frac{\ZZ\Sigma \ZZ^\top}{n_0} + E,
    }
    where $\norm{E}_{\text{op}}\to0$. We assume $\norm{\Sigma}_{\text{op}}\leq C$ and standard arguments show $\norm{\ZZ}_{\text{op}}\leq Cn^\delta$ with high probability, so we can conclude that $\norm{\Sigma^\WW}_{\text{op}}\leq Cn^\delta$ with high probability also. The same argument works for $\Sigma^\ZZ$.
    
    Bounding $\norm{\AA_{a:}}$ can be achieved by applying Prop.~\ref{prop:RF_qf} with $\GG=\Id$. Then since there are only $n$ such events, the union bound can be used to control the maximum. Similarly, $\norm{\AA_{:\alpha}}$ can be controlled. Prop.~\ref{prop:RF_qf} is also used in a standard $\epsilon$-net argument to bound $\norm{\AA}_{\text{op}}$. 
    
    When we consider $\tilde{\theta}^{(\alpha)}$, we condition on everything except the $\alpha$th column of $\WW$, in which case $\gcaltilm{\alpha}$ is deterministic. We can then apply Prop.~\ref{prop:RF_qf} and control $\gcaltilm{\alpha}$ deterministically with Lem.~\ref{lem_g_norm_bounds} to obtain the bound stated in $\mathcal{T}$. Similarly, $\theta^{(a)}$, $\phi^{(ab)}$, and $\tilde{\phi}^{(\alpha\beta)}$ can be controlled. There are only $\mathcal{O}(n^2)$ events in lines \eqref{def_good_event2} and \eqref{def_good_event3}, so the union bound shows their intersection also occurs with high probability.

\end{proof}

\subsection{Completing the proof of Prop.~\ref{prop_rf_resolvent} conditional on $\mathcal{T}$}\label{sec_completed_proof_rf}

We start with a short lemma we will need later.
% \begin{lemma}\label{lem_im_part}
%     Let $R(z)$ be the resolvent of some matrix $H$ that has real eigenvalues, and let $\Sigma$ be a nonnegative definite matrix. Then
%     \eq{
%         \Im \pa{\tr\pa{ \Sigma R(z) } } \geq0
%     }
% \end{lemma}
% \begin{proof}
%     Let $\lambda_1,\ldots,\lambda_n$ and $v_1,\ldots, v_n$ be the eigenvalues of $H$. Then
%     \eq{
%         \tr\pa{ \Sigma R(z) } = \sum_i \frac{1}{\lambda_i - z} \tr\pa{ \Sigma v_i v_i^\top } = \sum_i \frac{1}{\lambda_i - z}  v_i^\top \Sigma  v_i .
%     }
%     Note that $\Im\pa{ \frac{1}{\lambda_i-z} } \geq0 $ and $v_i^\top \Sigma  v_i\geq0$ to complete the proof.
% \end{proof}
\begin{lemma}\label{lem_im_part}
For $z\in\mathfrak{C}$, $G$ the resolvent of a nonnegative definite matrix $H$ such that $\norm{H}_{\text{op}}\leq Cn^{\delta}$, and nonnegative definite matrix $M$ such that $cn^{-\delta}\leq\tr(M)\leq Cn^{\delta}$, 
\eq{
    \absa{1 + \tr(M G)} \geq Cn^{-\delta}.
}
In particular, on the event $\mathcal{T}$, we have
\eq{\label{eq_thetas_bounded}
    \min\ha{\min_{a}\absa{1+\theta^{(a)}}, \min_{\alpha}\absa{1+z^{-1}\tilde{\theta}^{(\alpha)}} } \geq Cn^{-\delta}.
}
\end{lemma}
\begin{proof}
    Let $\lambda_1,\ldots,\lambda_n$ and $v_1,\ldots, v_n$ be the eigenvalues of $H$, then
    \eq{
        1 + \tr(M G) = 1 + \sum_{i=1}^d \frac{1}{\lambda_i-z} v_i^* M v_i.
    }
    Note $v_i^* M v_i\geq0$ and $\sum_{i=1}^d v_i^* M v_i = \tr(M)$.

    We now consider three cases for the value of the real part of $z$. 
    \paragraph{(i).}  Assume $\Re(z)\leq0$: We have,
    \eq{
       \absa{1 + \tr(M G)} \geq \Re\pa{1 + \tr(M G)}
        \geq 1 + \sum_{i=1}^d \frac{\lambda_i-\Re(z)}{\pa{\lambda_i-\Re(z)}^2+\Im(z)^2} v_i^* M v_i
      \geq 1,
    }
    since $\min_i\lambda_i\geq 0$.
    
    \paragraph{(ii).} Assume $\Re(z)\geq 2\tr(M)+\max_i\lambda_i$: Then 
    \eq{
        \sum_{i=1}^d \frac{\lambda_i-\Re(z)}{\pa{\lambda_i-\Re(z)}^2+\Im(z)^2} v_i^* M v_i \geq -\sum_{i=1}^d \frac{1}{\Re(z)-\lambda_i} v_i^* M v_i \geq -\frac{1}{2 \tr(M) }\sum_{i=1}^d  v_i^* M v_i=-\frac{1}{2}.
    }
    Thus, $\absa{ 1 + \tr(M G) } \geq 1/2$.

    \paragraph{(iii).} Assume $0\leq\Re(z)\leq 2\tr(M)+\max_i\lambda_i$: Finally,
    \al{
       \absa{1 + \tr(M G)} &\geq \Im\pa{\tr(M G)}\\ 
        &=  \sum_{i=1}^d \frac{\Im(z)}{\pa{\lambda_i-\Re(z)}^2+\Im(z)^2} v_i^* M v_i\\
        & \geq \sum_{i=1}^d \frac{\Im(z)}{4\max\ha{\max_i\lambda_i,\tr(M)}^2+\Im(z)^2} v_i^* M v_i \\
        &\geq \frac{C}{\max\ha{\max_i\lambda_i,\tr(M)}^2} \min\h{\Im(z), 1/\Im(z)} \tr(M)\\
        &\geq Cn^{-\delta}
    }
    by our assumptions on $H$, $M$, and $z$.
        
    For Eq.~\eqref{eq_thetas_bounded}, set $M_{\alpha\beta}\deq A_{a\alpha} A_{a\beta}$ and $H\deq (\AA^{(a)})^\top \AA^{(a)}$. Then note $\tr(M)=\norm{A_{a:}}^2$, which is approximately 1 on $\mathcal{T}$. Similarly, the spectral norm of $(\AA^{(a)})^\top \AA^{(a)}$ is bounded on $\mathcal{T}$ as it is a minor of $\AA^\top \AA$.
    The proof for $\absa{1+z^{-1}\tilde{\theta}^{(\alpha)}}$ follows in identical fashion except we use $M_{ab}\deq A_{a\alpha} A_{b\alpha}$ and $H\deq \AA^{(\alpha)} (\AA^{(\alpha)})^{\top}$. 
\end{proof}

We now return to the proof of the section's main lemma.
\begin{proof}[Proof of Prop.~\ref{prop_rf_resolvent}]

\textbf{Diagonal entries of the resolvents.} We must first show
\eq{
    \absa{ \tilde{\theta}^{(\alpha)} -z \bar{\tr}( \Sigma^\WW \gcaltil ) } \leq C n^{\delta-1/2}.
}

Since we are on the event $\mathcal{T}$,
\al{
    \absa{ \sum_{a,b=1}^{n} z(\AA_{a \alpha}\AA_{b \alpha} -\Sigma_{ab}^{\WW} /n ) \gcaltilm{\alpha}_{a b} } &\leq  Cn^{\delta-1/2}
}
by Eq.~\eqref{eq_alpha_tilde_mean}. Similarly,
\al{
    \absa{ \sum_{\alpha,\beta=1}^{d} (\AA_{a \alpha}\AA_{a \beta} -\Sigma_{\alpha\beta}^{\ZZ}/n ) \gcalm{a}_{\alpha \beta} } &\leq  Cn^{\delta-1/2}.
}

Next, Holder's inequality implies
\eq{
    \left|\tr\left(  \Sigma^{\WW} (  \gcaltil -  \gcaltilm{\alpha} )  \right) \right| \leq || \Sigma^{\WW} ||_{\text{op}} ||  \gcaltil -  \gcaltilm{\alpha} ||_1.
}
Note this is just the trace, not the normalized trace. Moreover, $|| \Sigma^{\WW} ||_{\text{op}}  \leq C n_1^{\delta}$ on $\mathcal{T}$. Using Woodbury's identity and Lem.~\ref{lem_im_part}, we see
\al{
    \norma{  \gcaltil -  \gcaltilm{\alpha} }_1 &= \norma{ \gcaltilm{\alpha}\AA_{:\alpha}  \frac{1}{1+z^{-1}\tilde{\theta}^{(\alpha)}} \AA_{:\alpha}^\top \gcaltilm{\alpha} }_1 \\
    & \leq \norma{ \gcaltilm{\alpha} }_{\text{op}}^2 \norma{ \AA_{:\alpha} }^2 \absa{ \frac{1}{1+z^{-1}\tilde{\theta}^{(\alpha)}}} \\
    &\leq C n^{\delta},
}
where $\AA_{:\alpha}$ is the $\alpha$th column of $\AA$. 
An identical argument works for $\gcal -  \gcalm{a}$:
\al{
    \norma{  \gcal -  \gcalm{a} }_1 &= \norma{ \gcalm{a}\AA_{a:}^\top  \frac{1}{1+\theta^{(a)}} \AA_{a:} \gcalm{a} }_1 \\
    & \leq \norma{ \gcalm{a} }_{\text{op}}^2 \norma{ \AA_{a:} }^2 \absa{ \frac{1}{1+\theta^{(a)}}} \\
    &\leq C n^{\delta}.
}

Define $\tilde{\epsilon}_\alpha\deq z\ntr(\Sigma^\WW \gcaltil) - \tilde{\theta}^{(\alpha)}$ and $\epsilon_a\deq (d/n)\ntr(\Sigma^\ZZ \gcal) - \theta^{(a)}$. Using Eq.~\eqref{eq_schur_diag}, for any $\beta\in\{1,\ldots, m\}$ we see
\al{
    \gcal_{\beta\beta} = \frac{1}{-z-z\ntr(\Sigma^\WW \gcaltil ) + \epsilon_\beta}.
}
Then expanding in $\tilde{\epsilon}_\beta$, we have
\eq{
    \gcal_{\beta\beta} - \frac{1}{m}\sum_{\alpha=1}^{m} \gcal_{\alpha\alpha}(z) \leq \frac{1}{\absa{z}^2} \frac{\max_{\beta} \abs{\tilde{\epsilon}^{(\beta)} } }{ \absa{ 1 +\ntr\pa{ \Sigma^\WW \gcaltil } }^2 } 
    \leq Cn^{\delta-1/2},
}
since we have bounded $\abs{\tilde{\epsilon}^{(\beta)} } $ above and by using Lem.~\ref{lem_im_part}.

Similarly, for any $b\in\{1,\ldots, n_1\}$ 
\eq{
    \gcaltil_{bb} - \frac{1}{n_1}\sum_{a=1}^{n_1} \gcaltil_{a a}(z) \leq Cn^{\delta-1/2}
}
by expanding Eq.~\eqref{eq_schur_diag} in $\epsilon_b$.

\textbf{Off-diagonal entries of the resolvents.} Note that we may assume $\E_\WW[ \tilde{\phi}^{(\alpha\beta)} ] =\E_\ZZ[ \phi^{(ab)} ] = 0$ for $\alpha\neq\beta$ and $a\neq b$ (see Step 2 in the proof of Prop.~\ref{prop:RF_qf}). Then using Eq.~\eqref{eq_schur_offdiag} and the fact we are conditioning on the event $\mathcal{T}$,
\eq{
    \absa{\gcal_{\alpha\beta} }=  \absa{\gcal_{\alpha\alpha}}\absa{\gcalm{\alpha}_{\beta\beta}}  \absa{\tilde{\phi}^{(\alpha\beta)}} \leq Cn^{\delta-1/2} .
}
We can then bound $\absa{\gcal_{ab} }$ similarly.

\end{proof}

\bibliographystyle{plainnat}
\bibliography{references}

\end{document}